\documentclass[11pt]{article}
\usepackage{amsmath,amsthm,amssymb}
\usepackage[usenames,dvipsnames]{xcolor}
\usepackage{enumerate}
\usepackage{graphicx}
\usepackage{cite}
\usepackage{comment}
\usepackage{oands}
\usepackage{tikz}
\usepackage{changepage}
\usepackage{bbm}
\usepackage{mathtools}
\usepackage[margin=.99in]{geometry}

\usepackage[pagewise,mathlines]{lineno}
\usepackage{appendix}
\usepackage{stmaryrd} 
\usepackage{multicol}
\usepackage{microtype}
\usepackage[colorinlistoftodos]{todonotes}

\usepackage[pdftitle={Existence and uniqueness of the Liouville quantum gravity metric},
  pdfauthor={Ewain Gwynne and Jason Miller},
colorlinks=true,linkcolor=NavyBlue,urlcolor=RoyalBlue,citecolor=PineGreen,bookmarks=true,bookmarksopen=true,bookmarksopenlevel=2,unicode=true,linktocpage]{hyperref}

\theoremstyle{plain}
\newtheorem{thm}{Theorem}[section]
\newtheorem{cor}[thm]{Corollary}
\newtheorem{lem}[thm]{Lemma}
\newtheorem{prop}[thm]{Proposition}
\newtheorem{conj}[thm]{Conjecture}

\def\@rst #1 #2other{#1}
\newcommand\MR[1]{\relax\ifhmode\unskip\spacefactor3000 \space\fi
  \MRhref{\expandafter\@rst #1 other}{#1}}
\newcommand{\MRhref}[2]{\href{http://www.ams.org/mathscinet-getitem?mr=#1}{MR#2}}

\theoremstyle{definition}
\newtheorem{defn}[thm]{Definition}
\newtheorem{remark}[thm]{Remark}
\newtheorem{prob}[thm]{Problem}

\numberwithin{equation}{section}

\newcommand{\dsb}{\begin{adjustwidth}{2.5em}{0pt}
\begin{footnotesize}}
\newcommand{\dse}{\end{footnotesize}
\end{adjustwidth}}

\newcommand{\ssb}{\begin{adjustwidth}{2.5em}{0pt}}
\newcommand{\sse}{\end{adjustwidth}}

\newcommand{\aryb}{\begin{eqnarray*}}
\newcommand{\arye}{\end{eqnarray*}}
\def\alb#1\ale{\begin{align*}#1\end{align*}}
\def\allb#1\alle{\begin{align}#1\end{align}}
\newcommand{\eqb}{\begin{equation}}
\newcommand{\eqe}{\end{equation}}
\newcommand{\eqbn}{\begin{equation*}}
\newcommand{\eqen}{\end{equation*}}

\newcommand{\BB}{\mathbbm}
\newcommand{\ol}{\overline}
\newcommand{\ul}{\underline}
\newcommand{\op}{\operatorname}

\newcommand{\frk}{\mathfrak}
\newcommand{\eqD}{\overset{d}{=}}
\newcommand{\ep}{\varepsilon}
\newcommand{\rta}{\rightarrow}

\newcommand{\wt}{\widetilde}
\newcommand{\wh}{\widehat} 
\newcommand{\mcl}{\mathcal}

\newcommand{\bdy}{\partial}
\newcommand{\rng}{\mathring}

\newcommand{\dtm}{1} 
\newcommand{\dmetric}{{\mathfrak d}}

\let\originalleft\left
\let\originalright\right
\renewcommand{\left}{\mathopen{}\mathclose\bgroup\originalleft}
\renewcommand{\right}{\aftergroup\egroup\originalright}

\title{Existence and uniqueness of the\\ Liouville quantum gravity metric for $\gamma \in (0,2)$}
\date{ }
\author{Ewain Gwynne and Jason Miller \\ {\it University of Cambridge}}

\begin{document}

\maketitle 

\begin{abstract}
We show that for each $\gamma \in (0,2)$, there is a unique metric (i.e., distance function) associated with $\gamma$-Liouville quantum gravity (LQG).  More precisely, we show that for the whole-plane Gaussian free field (GFF) $h$, there is a unique random metric $D_h$ associated with the Riemannian metric tensor ``$e^{\gamma h} (dx^2 + dy^2)$" on $\mathbb C$ which is characterized by a certain list of axioms: it is locally determined by $h$ and it transforms appropriately when either adding a continuous function to $h$ or applying a conformal automorphism of $\BB C$ (i.e., a complex affine transformation). Metrics associated with other variants of the GFF can be constructed using local absolute continuity.

The $\gamma$-LQG metric can be constructed explicitly as the scaling limit of \emph{Liouville first passage percolation} (LFPP), the random metric obtained by exponentiating a mollified version of the GFF. Earlier work by Ding, Dub\'edat, Dunlap, and Falconet (2019) showed that LFPP admits non-trivial subsequential limits. This paper shows that the subsequential limit is unique and satisfies our list of axioms. 
In the case when $\gamma = \sqrt{8/3}$, our metric coincides with the $\sqrt{8/3}$-LQG metric constructed in previous work by Miller and Sheffield, which in turn is equivalent to the Brownian map for a certain variant of the GFF.
For general $\gamma \in (0,2)$, we conjecture that our metric is the Gromov-Hausdorff limit of appropriate weighted random planar map models, equipped with their graph distance.  We include a substantial list of open problems. 
\end{abstract}

\tableofcontents

\section{Introduction}
\label{sec-intro}

\subsection{Overview}
\label{sec-overview}

Fix $\gamma \in (0,2)$, let $U\subset\BB C$ be an open domain, and let $h$ be the Gaussian free field (GFF) on $U$, or some minor variant thereof.  The \emph{$\gamma$-Liouville quantum gravity (LQG)} surface described by $(U,h)$ is formally the random two-dimensional Riemannian manifold with metric tensor
\begin{equation}
\label{eqn:lqg_def}
e^{\gamma h}\, (dx^2 +dy^2),	
\end{equation}
where $dx^2+dy^2$ is the Euclidean Riemannian metric tensor. 

LQG surfaces were first introduced non-rigorously in the physics literature by Polyakov~\cite{polyakov-qg1,polyakov-qg2} as a canonical model of a random Riemannian metric on $U$.  Another motivation to study LQG surfaces is that they describe the scaling limit of random planar maps. The special case when $\gamma=\sqrt{8/3}$ (called ``pure gravity") corresponds to uniformly random planar maps, including uniform triangulations, quadrangulations, etc. Other values of $\gamma$ (sometimes referred to as ``gravity coupled to matter") correspond to random planar maps weighted by the partition function of an appropriate statistical mechanics model on the map, for example the uniform spanning tree for $\gamma=\sqrt 2$ or the Ising model for $\gamma=\sqrt 3$. 
 
The definition~\eqref{eqn:lqg_def} of LQG does not make literal sense since $h$ is only a distribution, not a function, so it does not have well-defined pointwise values and cannot be exponentiated.  Nevertheless, it is known that one can make sense of the associated volume form $\mu_h = e^{\gamma h(z)} \,dz$ (where $dz$ denotes Lebesgue measure) as a random measure on $U$ via various regularization procedures~\cite{kahane,shef-kpz,rhodes-vargas-review}. 
One such regularization procedure is as follows. 
For $s > 0$ and $z,w\in\BB C$, let $p_s (z,w) = \frac{1}{2\pi s} \exp\left( - \frac{|z-w|^2}{2s} \right)$ be the heat kernel, and note that $p_s (z,\cdot)$ approximates a point mass at $z$ when $s$ is small. For $\ep >0$, we define a mollified version of the GFF by
\eqb \label{eqn-gff-convolve}
h_\ep^*(z) := (h*p_{\ep^2/2})(z) = \int_{U} h(w) p_{\ep^2/2} (z,w) \, dw ,\quad \forall z\in U  ,
\eqe
where the integral is interpreted in the sense of distributional pairing. 
One can then define the $\gamma$-LQG measure $\mu_h$ as the a.s.\ weak limit~\cite{kahane,shef-kpz,rhodes-vargas-review,berestycki-gmt-elementary,shamov-gmc}
\eqb \label{eqn-measure-construct}
\lim_{\ep\rta 0} \ep^{\gamma^2/2} e^{\gamma h_\ep^*(z)} \,dz .
\eqe

By~\cite[Proposition 2.1]{shef-kpz}, the measure $\mu_h$ is conformally covariant: if $\phi  : \wt U \rta U$ is a conformal map and we set
\eqb \label{eqn-lqg-coord}
\wt h := h\circ\phi + Q\log|\phi'|, \quad\text{where} \quad Q = \frac{2}{\gamma} + \frac{\gamma}{2} ,
\eqe
then a.s.\ $\mu_h(\phi(A)) = \mu_{\wt h}(A)$ for each Borel set $A\subset\BB C$.  This leads one to define a \emph{$\gamma$-LQG surface} as an equivalence class of pairs $(U,h)$, with two such pairs $(U,h)$ and $(\wt U , \wt h)$ declared to be equivalent if there is a conformal map $\phi : \wt U \rta U$ for which $h$ and $\wt h$ are related as in~\eqref{eqn-lqg-coord}.
We think of two equivalent pairs as representing different parameterizations of the same random surface.
The conformal covariance property of $\mu_h$ says that this measure is intrinsic to the quantum surface --- it does not depend on the particular equivalence class representative.
 
In order for $\gamma$-LQG to be a reasonable model of a ``random two-dimensional Riemannian manifold", one also needs a random metric\footnote{Throughout this paper, the term ``metric" will be used to mean a distance function, rather than a metric tensor. We will not prove anything rigorous about metric tensors. The metric tensor~\eqref{eqn:lqg_def} is introduced only for context.} (distance function) $D_h$ on $U$ which is in some sense obtained by exponentiating $h$ and which satisfies a conformal covariance property analogous to that of the $\gamma$-LQG area measure.  Moreover, this metric should be the scaling limit of the graph distance on random planar maps with respect to the Gromov-Hausdorff topology.  Constructing a metric on $\gamma$-LQG is a much more difficult problem than constructing the measure $\mu_h$. Indeed, any natural regularization scheme for LQG distances involves minimizing over a large collection of paths, which results in a substantial degree of non-linearity. 

Prior to this work, a $\gamma$-LQG metric has only been constructed in the special case when $\gamma = \sqrt{8/3}$ in a series of works by Miller and Sheffield~\cite{lqg-tbm1,lqg-tbm2,lqg-tbm3}. 
In this case, for certain special choices of the pair $(U,h)$, the random metric space $(U,D_h)$ agrees in law with a \emph{Brownian surface}, such as the Brownian map~\cite{legall-uniqueness,miermont-brownian-map} or the Brownian disk~\cite{bet-mier-disk}. These Brownian surfaces are continuum random metric spaces which arise as the scaling limits of uniform random planar maps with respect to the Gromov-Hausdorff topology.
Miller and Sheffield's construction of the $\sqrt{8/3}$-LQG metric does not use a direct regularization of the field $h$.  Instead, they first construct a candidate for $\sqrt{8/3}$-LQG metric balls using a process called \emph{quantum Loewner evolution}, which is built out of the Schramm-Loewner evolution with parameter $\kappa=6$ ($\op{SLE}_6$), then show that there is a metric which corresponds to these balls.

In this paper, we will construct a $\gamma$-LQG metric for all $\gamma \in (0,2)$ via an explicit regularization procedure analogous to~\eqref{eqn-measure-construct}.  We will also show that this metric is uniquely characterized by a list of natural properties that any reasonable notion of a metric on $\gamma$-LQG should satisfy, so is in some sense the only ``correct" metric on $\gamma$-LQG. For simplicity, we will mostly restrict attention to the whole-plane case, but metrics associated with GFF's on other domains can be easily constructed via restriction and/or absolute continuity (see Remark~\ref{remark-other-domains}). 
In contrast to \cite{lqg-tbm1,lqg-tbm2,lqg-tbm3}, the present work will make no use of $\op{SLE}$. 
Furthermore, we do not a priori have an ambient metric space to compare to (such as the Brownian map in the case $\gamma=\sqrt{8/3}$) and we do not have any sort of exact solvability, i.e., we do not know the exact laws of any observables related to the metric.

We now describe how our metric is constructed. 
It is shown in~\cite{dg-lqg-dim}, building on~\cite{dzz-heat-kernel,ghs-map-dist}, that for each $\gamma \in (0,2)$, there is an exponent $d_\gamma > 2$ which describes distances in various discrete approximations of $\gamma$-LQG. \emph{A posteriori}, once the $\gamma$-LQG metric is constructed, one can show that $d_\gamma$ is its Hausdorff dimension~\cite{gp-kpz}. The value of $d_\gamma$ is not known explicitly except in the case when $ \gamma = \sqrt{8/3}$, in which case we know that $ d_{\sqrt{8/3}}=4$ (see Problem~\ref{prob-dimension}). We refer to~\cite{dg-lqg-dim,gp-lfpp-bounds} for bounds for $d_\gamma$ and some speculation about its possible value. 
For $\gamma \in (0,2)$, we define 
\eqb \label{eqn-xi}
\xi = \xi_\gamma := \frac{\gamma}{d_\gamma} .
\eqe

We say that a random distribution $h$ on $\BB C$ is a \emph{whole-plane GFF plus a continuous function}
if there exists a coupling of $h$ with a random continuous function $f : \BB C\rta \BB R$ such that the law of $h-f$ is that of a whole-plane GFF.
We similarly define a \emph{whole-plane GFF plus a bounded continuous function}, except we require that $f$ is bounded.\footnote{The reason why we sometimes restrict to bounded continuous functions is to ensure that the convolution with the whole-plane heat kernel is finite (so $D_h^\ep$ is defined) and that the results about subsequential limits of LFPP in~\cite{dddf-lfpp,lqg-metric-estimates} are applicable.}
Note that the whole-plane GFF is defined only modulo a global additive constant, but these definitions do not depend on the choice of additive constant.  By definition, a whole-plane GFF plus a continuous function is well-defined as a distribution, not just modulo additive constant. For example, a whole-plane GFF with a particular choice of additive constant can be viewed as a whole-plane GFF plus a continuous function.

If $h$ is a whole-plane GFF plus a bounded continuous function, we define $h^*_\ep(z)$ for $\ep > 0$ and $z\in\BB C$ as in~\eqref{eqn-gff-convolve} for our given choice of $h$.
For $z,w\in\BB C$ and $\ep> 0$, we define the \emph{$\ep$-LFPP metric} by 
\eqb \label{eqn-lfpp}
D_h^\ep(z,w) := \inf_{P : z\rta w} \int_0^1 e^{\xi h_\ep^*(P(t))} |P'(t)| \,dt 
\eqe
where the infimum is over all piecewise continuously differentiable paths from $z$ to $w$. One should think of LFPP as the metric analog of the approximations of the LQG measure in~\eqref{eqn-measure-construct}.\footnote{One can also consider other variants of LFPP, defined using different approximations of the GFF, but we consider $h_\ep^*$ here since this is the approximation for which tightness is proven in~\cite{dddf-lfpp}. If we knew tightness and some basic properties of the subsequential limiting metrics for LFPP defined using a different approximation of the GFF, then Theorem~\ref{thm-stronger-uniqueness} below would show that these variants of LFPP also converge to the $\gamma$-LQG metric.} The intuitive reason why we look at $e^{\xi h_\ep^*(z)}$ instead of $e^{\gamma h_\ep^*(z)}$ to define the metric is as follows. By~\eqref{eqn-measure-construct}, we can scale LQG areas by a factor of $C>0$ by adding $\gamma^{-1}\log C$ to the field. By~\eqref{eqn-lfpp}, this results in scaling distances by $C^{\xi/\gamma} = C^{1/d_\gamma}$, which is consistent with the fact that the ``dimension" should be the exponent relating the scaling of areas and distances.

Let $\frk a_\ep$ be the median of the $D_h^\ep$-distance between the left and right boundaries of the unit square in the case when $h$ is a whole-plane GFF normalized so that its circle average\footnote{See~\cite[Section 3.1]{shef-kpz} for the basic properties of the circle average process. Even though we define LFPP using truncation with the heat kernel, we will always fix the additive constant for the whole-plane GFF using the circle average.}
 over $\bdy\BB D$ is zero. 
We do not know the value of $\frk a_\ep$ explicitly, but see Corollary~\ref{cor-constant}. 
It was shown by Ding, Dub\'edat, Dunlap, and Falconet~\cite{dddf-lfpp} that the laws of the metrics $\frk a_\ep^{-1} D_h^\ep$ are tight w.r.t.\ the local uniform topology on $\BB C\times\BB C$, and every possible subsequential limit induces the Euclidean topology on $\BB C$  (see also the earlier tightness results for small $\gamma > 0$~\cite{ding-dunlap-lqg-fpp,df-lqg-metric} and for Liouville graph distance, a related model, for all $\gamma\in(0,2)$~\cite{ding-dunlap-lgd}).  Subsequently, it was shown by Dub\'edat, Falconet, Gwynne, Pfeffer, and Sun~\cite{lqg-metric-estimates}, using~\cite[Corollary 1.8]{local-metrics} (a general criterion for a local metric to be determined by the GFF), that every subsequential limit can be realized as a measurable function of $h$, so in fact the metrics $\frk a_\ep^{-1}  D_h^\ep$ admit subsequential limits in probability.
One of the main results of this paper gives the uniqueness of this subsequential limit. 

\begin{thm}[Convergence of LFPP] \label{thm-lfpp}
The random metrics $\frk a_\ep^{-1} D_h^\ep $ converge in probability w.r.t.\ the local uniform topology on $\BB C\times \BB C$ to a random metric on $\BB C$ which is a.s.\ determined by $h$.
\end{thm}

It is natural to define the limiting metric from Theorem~\ref{thm-lfpp} to be the $\gamma$-LQG metric associated with $h$. 
However, this definition is not entirely satisfactory since it is a priori possible that there are other natural ways to construct a metric on $\gamma$-LQG which do not yield the same result as the one in Theorem~\ref{thm-lfpp}. For example, Theorem~\ref{thm-lfpp} does not yet tell us that the limit of LFPP coincides with the metric of~\cite{lqg-tbm1,lqg-tbm2,lqg-tbm3} in the case when $\gamma=\sqrt{8/3}$. 

We will therefore define a $\gamma$-LQG metric in terms of a list of axioms (see Section~\ref{sec-strong-uniqueness} just below). We will show that (a) the metric of Theorem~\ref{thm-lfpp} satisfies these axioms and (b) there is at most one metric satisfying these axioms for each $\gamma \in (0,2)$. Taken together, these statements tell us that the metric of Theorem~\ref{thm-lfpp} is the \emph{only} reasonable metric that one can put on $\gamma$-LQG.

An important feature of our proofs is that they can be read with essentially no knowledge of the (substantial) existing literature on LQG.
Aside from basic properties of the GFF (as discussed, e.g., in~\cite{shef-gff} and the introductory sections of~\cite{ss-contour,ig1,ig4}), the only prior works which this paper relies on are~\cite{dddf-lfpp,local-metrics,lqg-metric-estimates,gm-confluence}. All of the results which we need from these papers are reviewed in Section~\ref{sec-prelim}.  

Our results open up many important new research directions in the theory of LQG. 
We have included in Section~\ref{sec-open-problems} a substantial list of open problems related to the $\gamma$-LQG metric. 
\bigskip

\noindent\textbf{Acknowledgments.} We thank three anonymous referees for helpful comments on an earlier version of this paper. We thank Jian Ding, Julien Dub\'edat, Alex Dunlap, Hugo Falconet, Josh Pfeffer, Scott Sheffield, and Xin Sun for helpful discussions. EG was supported by a Herchel Smith fellowship and a Trinity College junior research fellowship. JM was supported by ERC Starting Grant 804166.

\subsection{Axiomatic characterization of the $\gamma$-LQG metric}
\label{sec-strong-uniqueness}

To state our list of axioms precisely, we will need some preliminary definitions concerning metric spaces. In what follows, we let $(X,\dmetric )$ be a metric space.  
\medskip

\noindent
For $A,B\subset X$, we define
\eqbn
\dmetric(A,B) := \inf_{x\in A , y\in B} \dmetric(x,y) .
\eqen
\medskip

\noindent
A \emph{curve} in $X$ is a continuous function $P : [a,b] \rta X$.  For a curve $P$, the \emph{$\dmetric$-length} of $P$ is defined by 
\eqbn
\op{len}\left( P ; \dmetric  \right) := \sup_{T} \sum_{i=1}^{\# T} \dmetric(P(t_i) , P(t_{i-1})) 
\eqen
where the supremum is over all partitions $T : a= t_0 < \dots < t_{\# T} = b$ of $[a,b]$. Note that the $\dmetric$-length of a curve may be infinite.
\medskip

\noindent
For $Y\subset X$, the \emph{internal metric of $\dmetric$ on $Y$} is defined by
\eqb \label{eqn-internal-def}
\dmetric(x,y ; Y)  := \inf_{P \subset Y} \op{len}\left(P ; \dmetric \right) ,\quad \forall x,y\in Y 
\eqe 
where the infimum is over all paths $P$ in $Y$ from $x$ to $y$. 
Then $\dmetric(\cdot,\cdot ; Y)$ is a metric on $Y$, except that it is allowed to take infinite values.  
\medskip
 
\noindent
We say that $(X,\dmetric)$ is a \emph{length space} if for each $x,y\in X$ and each $\ep > 0$, there exists a curve of $\dmetric$-length at most $\dmetric(x,y) + \ep$ from $x$ to $y$. 
\medskip

\noindent
A \emph{continuous metric} on an open domain $U\subset\BB C$ is a metric $\dmetric$ on $U$ which induces the Euclidean topology on $U$, i.e., the identity map $(U,|\cdot|) \rta (U,\dmetric)$ is a homeomorphism. 
We equip the space of continuous metrics on $U$ with the local uniform topology for functions from $U\times U$ to $[0,\infty)$ and the associated Borel $\sigma$-algebra.
We allow a continuous metric to satisfy $\dmetric(u,v) = \infty$ if $u$ and $v$ are in different connected components of $U$.
In this case, in order to have $\dmetric^n\rta \dmetric$ w.r.t.\ the local uniform topology we require that for large enough $n$, $\dmetric^n(u,v) = \infty$ if and only if $\dmetric(u,v)=\infty$.
\medskip

\noindent
Let $\mcl D'(\BB C)$ be the space of distributions (generalized functions) on $\BB C$, equipped with the usual weak topology.
For $\gamma \in (0,2)$, a \emph{(strong) $\gamma$-Liouville quantum gravity (LQG) metric} is a measurable function $h\mapsto D_h$ from $\mcl D'(\BB C)$ to the space of continuous metrics on $\BB C$ such that the following is true whenever $h$ is a whole-plane GFF plus a continuous function.
\begin{enumerate}[I.]
\item \textbf{Length space.} Almost surely, $(\BB C , D_h)$ is a length space, i.e., the $D_h$-distance between any two points of $\BB C$ is the infimum of the $D_h$-lengths of $D_h$-continuous paths (equivalently, Euclidean continuous paths) between the two points. \label{item-metric-length0}
\item \textbf{Locality.} Let $U\subset\BB C$ be a deterministic open set. 
The internal metric $D_h(\cdot,\cdot ; U)$ is a.s.\ determined by $h|_U$.  \label{item-metric-local0}
\item \textbf{Weyl scaling.} Let $\xi$ be as in~\eqref{eqn-xi} and for each continuous function $f : \BB C\rta \BB R$, define \label{item-metric-f0}  
\eqb  \label{eqn-metric-f}
(e^{\xi f} \cdot D_h) (z,w) := \inf_{P : z\rta w} \int_0^{\op{len}(P ; D_h)} e^{\xi f(P(t))} \,dt , \quad \forall z,w\in \BB C ,
\eqe 
where the infimum is over all continuous paths from $z$ to $w$ parameterized by $D_h$-length. Then a.s.\ $ e^{\xi f} \cdot D_h = D_{h+f}$ for every continuous function $f : \BB C\rta\BB R$.  
\item \textbf{Coordinate change for translation and scaling.} For each fixed deterministic $r > 0$ and $z\in\BB C$, a.s.\ \label{item-metric-coord0}
\eqb
 D_h \left( ru + z , r v + z \right) = D_{h(r\cdot + z)  +Q\log r}(u,v)  , \: \forall u,v\in\BB C \quad \text{where} \quad Q =\frac{2}{\gamma} + \frac{\gamma}{2} .
\eqe    
\end{enumerate}

Let us briefly discuss why the above axioms are natural. 
Recall that $\gamma$-LQG should be the random Riemannian metric with metric tensor $e^{\gamma h} (dx^2+dy^2)$. 
Axiom~\ref{item-metric-length0} is simply the LQG analog of the statement that for a true Riemannian metric, the distance between two points can be defined as the infima of the lengths of paths connecting them. 
In a similar vein, Axiom~\ref{item-metric-local0} corresponds to the fact that for a smooth Riemannian metric, the lengths of paths are determined locally by the Riemannian metric tensor. 
Axiom~\ref{item-metric-f0} is just expressing the fact that the metric is obtained by exponentiating $\xi h$, so adding a continuous function $f$ to $h$ results in re-scaling the metric length measure on paths by $e^{\xi f}$. 

Axiom~\ref{item-metric-coord0} is the metric analog of the conformal coordinate change formula~\eqref{eqn-lqg-coord} for the $\gamma$-LQG area measure, but restricted to translations and scalings. This axiom together with Corollary~\ref{cor-rotational-invariance} says that $D_h$ depends only on the LQG surface $(\BB C , h)$, not on the particular choice of parameterization. We will prove a conformal covariance property for the $\gamma$-LQG metric w.r.t.\ conformal automorphisms between arbitrary domains, directly analogous to the conformal covariance of the $\gamma$-LQG area measure, in~\cite{gm-coord-change}.

\begin{thm}[Existence and uniqueness of the LQG metric] \label{thm-strong-uniqueness} 
Fix $\gamma \in (0,2)$. There is a $\gamma$-LQG metric $D$ such that the limiting metric of Theorem~\ref{thm-lfpp} is a.s.\ equal to $D_h$ whenever $h$ is a whole-plane GFF plus a bounded continuous function. 
Furthermore, the $\gamma$-LQG metric is unique in the following sense. If $D$ and $\wt D$ are two $\gamma$-LQG metrics, then there is a deterministic constant $C>0$ such that if $h$ is a whole-plane GFF plus a continuous function, then a.s.\ $D_h = C \wt D_h$. 
\end{thm}

Theorem~\ref{thm-strong-uniqueness} justifies us in referring to \emph{the} $\gamma$-LQG metric. Technically speaking there is a one-parameter family of such metrics, which differ by a global deterministic multiplicative constant. But, one can fix the constant in various ways to get a single canonically defined metric. For example, we can require that the median distance between the left and right boundaries of the unit square is 1 for the metric associated with a whole-plane GFF normalized so that its circle average over $\bdy\BB D$ is zero (the limiting metric in Theorem~\ref{thm-lfpp} has this normalization). 

Theorem~\ref{thm-strong-uniqueness} is related to Shamov's axiomatic characterization of Gaussian multiplicative chaos (GMC) measures, such as the $\gamma$-LQG measure~\cite[Corollary 5]{shamov-gmc}. Shamov's result says that a subcritical 
GMC measure associated with a field $X$ is uniquely characterized by how it transforms when we add to $X$ a function in the Cameron-Martin space. Weyl scaling (Axiom~\ref{item-metric-f0}) is the metric analog of this property. 
Unlike in Shamov's characterization we need other properties besides just Weyl scaling to characterize the LQG metric, most notably some sort of uniform control of the metric at different Euclidean scales (in the above list of axioms this is provided by Axiom~\ref{item-metric-coord0}, but this axiom can be weakened, see Section~\ref{sec-weak-uniqueness}).
 
In Axiom~\ref{item-metric-coord0} in the definition of a strong $\gamma$-LQG metric, we did not require that the metric is invariant under rotations of $\BB C$. 
It turns out that rotational invariance is implied by the other axioms. See Remark~\ref{remark-rotational-invariance} below for an intuitive explanation of why this is the case.

\begin{cor}[Rotational invariance] \label{cor-rotational-invariance}
If $\gamma\in(0,2)$ and $D$ is a $\gamma$-LQG metric then $D$ is rotationally invariant, i.e., if $\omega \in \BB C$ with $|\omega | =1$ and $h$ is a whole-plane GFF plus a continuous function, then a.s.\ $D_h(u,v) = D_{h(\omega\cdot)}(\omega^{-1} u ,\omega^{-1} v)$ for all $u,v\in\BB C$. 
\end{cor}
\begin{proof}
Define $D_h^{(\omega)}(u,v) :=  D_{h(\omega\cdot)}(\omega^{-1} u ,\omega^{-1} v)$.
It is easily verified that $D^{(\omega)}$ is a strong LQG metric, so Theorem~\ref{thm-strong-uniqueness} implies that there is a deterministic constant $C >0$ such that a.s.\ $D^{(\omega)}_h = C D_h$ whenever $h$ is a whole-plane GFF plus a continuous function.
To check that $C = 1$, consider a whole-plane GFF $h$ normalized so that its circle average over $\bdy\BB D$ is $0$. Then the law of $h$ is rotationally invariant, so $\BB P[D_h(0,\bdy\BB D) > R] = \BB P[D_h^{(\omega)}(0,\bdy\BB D) > R]$ for every $R > 0$. 
Therefore $C  =1$. 
\end{proof}

It is easy to check that the metric constructed in~\cite{lqg-tbm1,lqg-tbm2,lqg-tbm3} satisfies the axioms for a $\sqrt{8/3}$-LQG metric;
see~\cite[Section 2.5]{gms-poisson-voronoi} for a careful explanation of why this is the case. 
Consequently, Theorem~\ref{thm-strong-uniqueness} implies the following. 

\begin{cor}[Equivalence with the construction of~\cite{lqg-tbm1,lqg-tbm2,lqg-tbm3}] \label{cor-tbm-equivalence}
The $\sqrt{8/3}$-LQG metric constructed in~\cite{lqg-tbm1,lqg-tbm2,lqg-tbm3} agrees with the limiting metric of Theorem~\ref{thm-lfpp} (equivalently, the metric of Theorem~\ref{thm-strong-uniqueness}) up to a deterministic global scaling factor.
\end{cor}

The present work does not use the results of~\cite{lqg-tbm1,lqg-tbm2,lqg-tbm3}, but also does not supersede these results. 
Indeed, without these works it is not at all clear how to link the $\sqrt{8/3}$-LQG metric constructed in the present article to Brownian surfaces, and thereby to uniform random planar maps. 

There are a number of properties of the $\gamma$-LQG metric which are already known. 
It is shown in~\cite[Section 3.1]{lqg-metric-estimates} that one has superpolynomial concentration for the $D_h$-distance between two disjoint compact, connected sets which are not singletons (e.g., the inner and outer boundaries of an annulus or two opposite sides of a rectangle). 
Building on this,~\cite{lqg-metric-estimates} computes the optimal H\"older exponents between $D_h$ and the Euclidean metric, in both directions, and establishes moment bounds for various distance quantities (see also Section~\ref{sec-a-priori-estimates}). 
Confluence properties for $D_h$-geodesics analogous to the ones known for the Brownian map~\cite{legall-geodesics} are proven in~\cite{gm-confluence} (see also Section~\ref{sec-confluence-prelim}). 
It is shown in~\cite{mq-geodesics} that $D_h$-geodesics are conformally removable and their laws are mutually singular with respect to Schramm-Loewner evolution curves.  After the appearance of this paper, the work~\cite{gp-kpz} proved that $D_h$ satisfies a version of the KPZ formula~\cite{kpz-scaling,shef-kpz} and the work~\cite{afs-metric-ball} proved a concentration result for the LQG mass of a $D_h$-metric ball.

\begin{remark}[Metrics associated with other fields] \label{remark-other-domains}
Theorem~\ref{thm-strong-uniqueness} gives us a canonical $\gamma$-LQG metric associated with a whole-plane GFF plus a continuous function. It is not hard to see that one can also define the metric if $h$ is equal to a whole-plane GFF plus a continuous function plus a finite number of logarithmic singularities of the form $-\alpha\log|\cdot - z|$ for $z\in\BB C$ and $\alpha  < Q$; see~\cite[Theorem 1.10 and Proposition 3.17]{lqg-metric-estimates}. 

We can also define metrics associated with GFF's on proper sub-domains of $\BB C$. To this end, let $U\subset\BB C$ be open and let  $h$ be a whole-plane GFF. Due to Axiom~\ref{item-metric-local0}, we can define for each open set $U\subset \BB C$ the metric $D_{h|_U} := D_h(\cdot,\cdot ;U)$ as a measurable function of $h|_U$. 
We can write $h|_U = \rng h^U + \frk h^U$, where $\rng h^U$ is a zero-boundary GFF on $U$ and $\frk h^U$ is a random harmonic function on $U$ independent from $\rng h^U$.
In the notation~\eqref{eqn-metric-f}, we define
\eqb
D_{\rng h^U} := e^{-\xi \frk h^U} \cdot D_{h|_U} .
\eqe
Note that this is well-defined even though $\frk h^U$ does not extend continuously to $\bdy U$, since the definition of $D_{h|_U}$ involves only paths contained in $U$.
It is easily seen from Axioms~\ref{item-metric-local0} (locality) and~\ref{item-metric-f0} (Weyl scaling) that $D_{\rng h^U}$ is a measurable function of $\rng h^U$: indeed, if we are given an open set $V\subset U$ with $\ol V\subset U$, choose a smooth compactly supported bump $f : U\rta [0,1]$ which is identically equal to 1 on $V$. Then Axiom~\ref{item-metric-local0} applied to the field $h - f \frk h^U$ implies that the internal metric of $D_{\rng h^U}$ on $V$, which equals $D_{h-f\frk h^U}(\cdot,\cdot ; V)$, is determined by $(h-f\frk h^U)|_V = \rng h^U|_V$. Letting $V$ increase to all of $U$ gives the desired measurability of $D_{\rng h^U}$ w.r.t.\ $\rng h^U$. 
This defines the $\gamma$-LQG metric for a zero-boundary GFF. 

By Axiom~\ref{item-metric-f0}, we can also define the metric $D_{\wt h}$ in the case when $\wt h = \rng h^U + f$ is a zero-boundary GFF plus a continuous function on $U$, namely $D_{\wt h} := e^{\xi f} D_{\rng h^U}$. 
It is shown in~\cite{gm-coord-change} that this metric satisfies a conformal coordinate change relation analogous to the one satisfied by the $\gamma$-LQG measure (as discussed just below~\eqref{eqn-lqg-coord}).

We expect that for a fixed proper subdomain $U\subset\BB C$ there is an analogous formulation and characterization of the LQG metric on $U$.  However, we will not formulate such a result here.  We emphasize that the LQG metric on $U$ is determined by the LQG metric on $\BB C$, and moreover the LQG metric on $U$ determines the LQG metric on $\BB C$ due to Axiom~\ref{item-metric-local0} (locality) and the local absolute continuity between GFF's on different domains.  It is not hard to show using the results of~\cite{dddf-lfpp} that for, say, a zero-boundary GFF $\rng h^U$ on $U$, the metric $D_{\rng h^U}$ is the limit in law of LFPP on $U$ w.r.t.\ the topology of uniform convergence on compact subsets of $U\times U$: see, e.g., the arguments of~\cite[Section 2.2]{lqg-metric-estimates}.
\end{remark}

\begin{remark}[Why rotational invariance is unnecessary] \label{remark-rotational-invariance}
At a first glance, it may seem surprising that one does not need rotational invariance to uniquely characterize the LQG metric in Theorem~\ref{thm-strong-uniqueness}.
Indeed, one can define variants of LFPP which are not rotationally invariant by working with a stretched version of the Euclidean metric. For example, for a given $A > 1$ one can replace~\eqref{eqn-lfpp} by
\eqb \label{eqn-lfpp-stretched}
D_{h,A}^\ep(z,w) := \inf_{P : z\rta w} \int_0^1 e^{\xi h_\ep^*(P(t))} \sqrt{P_1'(t)^2 + A P_2'(t)^2} \,dt 
\eqe
where the infimum is over all piecewise continuously differentiable paths $P = (P_1,P_2)$ from~$z$ to~$w$.  The arguments of this paper and its predecessors apply verbatim with $D_{h,A}^\ep$ in place of $D_h^\ep$.  In particular, $D_{h,A}^\ep$ converges in probability to (a deterministic constant times) the $\gamma$-LQG metric and hence satisfies the rotational invariance property of Corollary~\ref{cor-rotational-invariance}. This is despite the fact that the metrics~\eqref{eqn-lfpp-stretched} do \emph{not} satisfy this rotational invariance property.

Here is an intuitive explanation for this phenomenon.  First, we note that $D_{h,A}^\ep$ is bi-Lipschitz equivalent with respect to $D_{h,1}^\ep  = D_h^\ep$ for each $\ep > 0$, with a deterministic bi-Lipschitz constants.  Therefore in a subsequential limit as $\ep \to 0$, we obtain two metrics $D_{h,A}$ and $D_h = D_{h,1}$ which are bi-Lipschitz equivalent with deterministic bi-Lipschitz constants.  Suppose that $P$ is a $D_h$-geodesic connecting $z$ and $w$. Using the confluence of geodesics results from~\cite{gm-confluence}, one can show that (very roughly speaking) for distinct times $s,t\in [0,D_h(z,w)]$, the restrictions of $h$ to small neighborhoods of $P(s)$ and $P(t)$ are approximately independent; see the outline of Section~\ref{sec-geodesic-iterate} in Section~\ref{sec-outline} below for details.  Moreover, since $P$ is a fractal type curve, it has no local notion of direction, so one expects that the law of $h$ restricted to a small neighborhood of $P(t)$ does not depend very strongly on $t$ or on the endpoints $z,w$ of $P$.  If we fix $n \in \BB N$ and let $0 = t_0 < \cdots t_n = D_h(z,w)$ be equally spaced times, we can approximate the $D_{h,A}$-length of $P$ by
\[ \sum_{j=1}^n D_{h,A}(P(t_{j-1}),P(t_j) ).\]
The above considerations suggest that each of the random variables $D_{h,A}(P(t_{j-1}),P(t_j))$ has approximately the same distribution and is bounded above and below by deterministic constants times $t_j-t_{j-1}$.  From law of large numbers type considerations, it follows that the $D_{h,A}$-length of $P$ is a deterministic constant times the $D_h$-length of $P$, where the constant does not depend on the endpoints of $P$.

Knowing that the $D_{h,A}$-length of every $D_h$ geodesic is a constant times its $D_h$-length (and vice-versa) does not immediately imply that $D_h$ is equal to a constant times $D_{h,A}$.  This is because if $P_n$ is a sequence of paths which converge uniformly to $P$, then it is not necessarily true that $\op{len}(P_n;D_{h,A})$ converges to $\op{len}(P;D_{h,A})$.  For this and other reasons, we will argue in a somewhat different manner than we have indicated above, though our arguments will still be based on the bi-Lipschitz equivalence of metrics and approximate independence statements for the local behavior of a geodesic at different times.  We will explain the general strategy in Section~\ref{sec-outline} in more detail.
\end{remark}

\subsection{Conjectured random planar map connection}
\label{sec-rpm-connection}

As noted above, the $\gamma$-LQG metric should describe the large scale behavior of the graph metric for random planar maps.  Since our $\gamma$-LQG metric is in some sense canonical, it is natural to make the following conjecture. 

\begin{conj} \label{conj-rpm-limit}
For each $\gamma \in (0,2)$, random planar maps in the $\gamma$-LQG universality class, equipped with their graph distance, converge in the scaling limit  with respect to the Gromov-Hausdorff topology to $\gamma$-LQG surfaces equipped with the $\gamma$-LQG metric constructed in Theorem~\ref{thm-lfpp} (see also Remark~\ref{remark-other-domains}). 
\end{conj}

Examples of planar map models to which Conjecture~\ref{conj-rpm-limit} should apply include random planar maps weighted by the number of spanning trees ($\gamma = \sqrt 2$), the Ising model partition function ($\gamma=\sqrt 3$), the number of bipolar orientations ($\gamma=\sqrt{4/3}$; \cite{kmsw-bipolar}), or the Fortuin-Kasteleyn model partition function ($\gamma \in (\sqrt 2 , 2)$; \cite{shef-burger}). Another class of models is the so-called \emph{mated-CRT maps}, which are defined for all $\gamma\in (0,2)$; see~\cite{wedges,ghs-map-dist,gms-tutte}. 

For $\gamma =\sqrt{8/3}$, Conjecture~\ref{conj-rpm-limit} has already been proven for many different uniform-type random planar maps. 
The reason for this is that we know that our $\sqrt{8/3}$-LQG metric is equivalent to the metric of~\cite{lqg-tbm1,lqg-tbm2,lqg-tbm3} (Corollary~\ref{cor-tbm-equivalence}); which in turn is equivalent to a Brownian surface, such as the Brownian map, for certain special $\sqrt{8/3}$-LQG surfaces~\cite[Corollary 1.5]{lqg-tbm2}; which in turn is the scaling limit of uniform random planar maps of various types~\cite{legall-uniqueness,miermont-brownian-map}. 

Conjecture~\ref{conj-rpm-limit} has not been proven for any random planar map model for $\gamma\not=\sqrt{8/3}$.  However, we already have a relationship between the continuum LQG metric and graph distances in random planar maps at the level of exponents for all $\gamma \in (0,2)$.  Indeed, the quantity $d_\gamma$ appearing in~\eqref{eqn-xi} describes several exponents associated with random planar maps, such as the ball volume exponent~\cite{ghs-map-dist,dg-lqg-dim} and the displacement exponent for simple random walk on the map~\cite{gm-spec-dim,gh-displacement}.  It is proven in~\cite{gp-kpz} that~$d_\gamma$ is the Hausdorff dimension of~$D_h$. 

Conjecture~\ref{conj-rpm-limit} can be made somewhat more precise by specifying exactly what type of $\gamma$-LQG surface should arise in the scaling limit. For random planar maps with the topology of the sphere (resp.\ disk, plane, half-plane) this surface should be the quantum sphere (resp.\ quantum disk, $\gamma$-quantum cone, $\gamma$-quantum wedge). See~\cite{wedges} for precise definitions of these quantum surfaces. Equivalent definitions of the quantum sphere and quantum disk, respectively, can be found in~\cite{dkrv-lqg-sphere,hrv-disk} (see~\cite{ahs-sphere,cercle-quantum-disk} for a proof of the equivalence). Some planar map models have been proven to converge to these quantum surfaces, for general $\gamma  \in (0,2)$, with respect to topologies which do not encode the metric structure explicitly. Examples of such topologies include convergence in the so-called \emph{peanosphere sense}~\cite{shef-burger,wedges} and convergence of the counting measure on vertices to the $\gamma$-LQG measure when the planar map is embedded appropriately into the plane~\cite{gms-tutte}.

\subsection{Weak LQG metrics and a stronger uniqueness statement}
\label{sec-weak-uniqueness}

We will prove Theorem~\ref{thm-lfpp} and~\ref{thm-strong-uniqueness} simultaneously by establishing a uniqueness statement for metrics under a weaker list of axioms, which are satisfied for both the strong LQG metrics considered in Section~\ref{sec-strong-uniqueness} and for subsequential limits of LFPP (as is shown in~\cite{dddf-lfpp,lqg-metric-estimates}). 

Let $\mcl D'(\BB C)$ be the space of distributions as in Section~\ref{sec-strong-uniqueness}. 
A \emph{weak $\gamma$-LQG metric} is a measurable function $h\mapsto D_h$ from $\mcl D'(\BB C)$ to the space of continuous metrics on $\BB C$ such that the following is true whenever $h$ is a whole-plane GFF plus a continuous function.
\begin{enumerate}[I.]
\item \textbf{Length space.} Almost surely, $(\BB C , D_h)$ is a length space, i.e., the $D_h$-distance between any two points of $\BB C$ is the infimum of the $D_h$-lengths of $D_h$-continuous paths (equivalently, Euclidean continuous paths) between the two points. \label{item-metric-length}
\item \textbf{Locality.} Let $U\subset\BB C$ be a deterministic open set. 
The internal metric $D_h(\cdot,\cdot ; U)$ is a.s.\ determined by $h|_U$.  \label{item-metric-local}
\item \textbf{Weyl scaling.} If we define $e^{\xi f} \cdot D_h$ as in~\eqref{eqn-metric-f}, then a.s.\ $ e^{\xi f} \cdot D_h = D_{h+f}$ for every continuous function $f : \BB C\rta\BB R$.   \label{item-metric-f}     
\item \textbf{Translation invariance.} For each fixed deterministic $z \in \BB C$, a.s.\ $D_{h(\cdot + z)} = D_h(\cdot + z , \cdot+z)$.  \label{item-metric-translate}
\item \textbf{Tightness across scales.} Suppose $h$ is a whole-plane GFF and for $z\in\BB C$ and $r>0$ let $h_r(z)$ be the average of $h$ over the circle $\bdy B_r(z)$. For each $r > 0$, there is a deterministic constant $\frk c_r > 0$ such that the set of laws of the metrics $\frk c_r^{-1} e^{-\xi h_r(0)} D_h (r \cdot , r\cdot)$ for $r > 0$ is tight (w.r.t.\ the local uniform topology). Furthermore, the closure of this set of laws w.r.t.\ the Prokhorov topology is contained in the set of laws on continuous metrics on $\BB C$ (i.e., every subsequential limit of the laws of the metrics $\frk c_r^{-1} e^{-\xi h_r(0)} D_h (r \cdot  , r \cdot )$ is supported on metrics which induce the Euclidean topology on $\BB C$). Finally, there exists  \label{item-metric-coord} 
$\Lambda > 1$ such that for each $\delta \in (0,1)$, 
\eqb \label{eqn-scaling-constant}
\Lambda^{-1} \delta^\Lambda \leq \frac{\frk c_{\delta r}}{\frk c_r} \leq \Lambda \delta^{-\Lambda} ,\quad\forall r  > 0.
\eqe 
\end{enumerate}

Axioms~\ref{item-metric-length} through~\ref{item-metric-f} for a weak LQG metric are identical to the corresponding axioms for a strong LQG metric.
Axiom~\ref{item-metric-translate} for a weak LQG metric is equivalent to Axiom~\ref{item-metric-coord0} (coordinate change) for a strong LQG metric with $r=1$.
Axiom~\ref{item-metric-coord} for a weak $\gamma$-LQG metric is a substitute for the exact scale invariance property given by Axiom~\ref{item-metric-coord0} for a strong LQG metric. This axiom implies the tightness of various functionals of $D_h$. For example, if $U\subset\BB C$ is open and $K\subset U$ is compact, then the laws of 
\eqb
\left( \frk c_r^{-1} e^{-\xi h_r(0)} D_h (r K , r\bdy U) \right)^{-1} \quad \text{and} \quad \frk c_r^{-1} e^{-\xi h_r(0)} \sup_{u,v\in r K} D_h (  u ,  v ; r U )
\eqe
as $r$ varies are tight. It is shown in~\cite[Theorem 1.5]{lqg-metric-estimates} that for any weak $\gamma$-LQG metric, one in fact has the following stronger version of~\eqref{eqn-scaling-constant}: 
\eqb \label{eqn-scaling-constant-stronger}
\frac{\frk c_{\delta r}}{\frk c_r} = \delta^{\xi Q + o_\delta(1)}, \quad\text{ uniformly over all $r>0$}.
\eqe 

By the scale invariance of the law of the whole-plane GFF, modulo additive constant, Axiom~\ref{item-metric-coord0} for a strong LQG metric immediately implies Axiom~\ref{item-metric-coord} for a weak LQG metric with $\frk c_r = r^{\xi Q }$, for $Q$ as in~\eqref{eqn-lqg-coord}. Indeed, using Axiom~\ref{item-metric-coord0} and then Axiom~\ref{item-metric-f0} for a strong $\gamma$-LQG metric shows that
\eqb \label{eqn-strong-implies-weak}
r^{-\xi Q} e^{-\xi h_r(0)} D_h (r \cdot , r\cdot)
= r^{-\xi Q} e^{-\xi h_r(0)} D_{h(r\cdot )  +Q\log r} 
= D_{h(r\cdot) - h_r(0)} 
\eqD D_h .
\eqe
Hence every strong $\gamma$-LQG metric is a weak $\gamma$-LQG metric. 

It is shown in~\cite[Theorem 1.2]{lqg-metric-estimates} that every subsequential limit in probability of the LFPP metrics $D_h^\ep$ of~\eqref{eqn-lfpp} is of the form $D_h$ where $D$ is a weak $\gamma$-LQG metric. 
Consequently, the following theorem contains both Theorem~\ref{thm-lfpp} and Theorem~\ref{thm-strong-uniqueness}. 
 
\begin{thm}[Strong uniqueness of weak LQG metrics] \label{thm-stronger-uniqueness}
Let $\gamma \in (0,2)$. 
Every weak $\gamma$-LQG metric is a strong $\gamma$-LQG metric. 
In particular, by Theorem~\ref{thm-strong-uniqueness}, such a metric exists for each $\gamma \in (0,2)$ and if $D$ and $\wt D$ are two weak $\gamma$-LQG metrics, then there is a deterministic constant $C>0$ such that if $h$ is a whole-plane GFF plus a continuous function, then a.s.\ $D_h = C \wt D_h$. 
\end{thm}

It turns out that all of our main results are easy consequences of the following statement, which superficially seems to be weaker that Theorem~\ref{thm-stronger-uniqueness}. 

\begin{thm}[Weak uniqueness of weak LQG metrics] \label{thm-weak-uniqueness}
Let $\gamma \in (0,2)$ and let $D$ and $\wt D$ be two weak $\gamma$-LQG metrics which have the \emph{same} values of $\frk c_r$ in Axiom~\ref{item-metric-coord}. 
There is a deterministic constant $C > 0$ such that if $h$ is a whole-plane GFF plus a continuous function, then a.s.\ $D_h = C \wt D_h$. 
\end{thm} 

Most of the paper is devoted to the proof of Theorem~\ref{thm-weak-uniqueness}. Let us now explain how Theorem~\ref{thm-weak-uniqueness} implies the other main theorems stated above. We first establish the first statement of Theorem~\ref{thm-stronger-uniqueness}.

\begin{lem} \label{lem-weak-to-strong}
Every weak $\gamma$-LQG metric is a strong $\gamma$-LQG metric.
\end{lem}
\begin{proof}[Proof of Lemma~\ref{lem-weak-to-strong} assuming Theorem~\ref{thm-weak-uniqueness}]
Suppose that $D $ is a weak $\gamma$-LQG metric.
For $b >0$, we define 
\eqb \label{eqn-weak-to-strong-wtD}
D^{(b)}_h(\cdot,\cdot) :=   D_{h(\cdot/b)} (b\cdot , b\cdot)  .
\eqe
We claim that $D^{(b)}$ is a weak $\gamma$-LQG metric with the same scaling constants $\frk c_r$ as $D$. 
It is easily verified that $D^{(b)}$ satisfies Axioms~\ref{item-metric-length} through~\ref{item-metric-translate} in the definition of a weak $\gamma$-LQG metric.
To check Axiom~\ref{item-metric-coord} (tightness across scales), we compute for $r>0$:
\alb
\frk c_r^{-1} e^{-\xi h_r(0)} D^{(b)}_h (r \cdot , r\cdot)
&= \frk c_r^{-1} e^{-\xi h_r(0)}  D_{h(\cdot/b)} ( b r  \cdot , b r \cdot) \notag \\
&= \left( \frac{\frk c_{b r}}{\frk c_r}    e^{-\xi ( h_r(0) - h_{b r}(0) )}     \right)   \frk c_{b r}^{-1} e^{-\xi h_{b r}(0)}  D_{h(\cdot/b)} ( b r  \cdot , b r \cdot) .
\ale
In the case when $h$ is a whole-plane GFF, the random variable $h_r(0) - h_{b r}(0)$ is centered Gaussian with variance $\log b^{-1}$~\cite[Section 3.1]{shef-kpz}. 
By~\eqref{eqn-scaling-constant}, $\frk c_{b r}/\frk c_r$ is bounded above by a constant depending only on $b$ (not on $r$). 
Axiom~\ref{item-metric-coord} (tightness across scales) for $D$ applied with $h(\cdot/b)$ in place of $h$ and $b r$ in place of $r$ therefore implies that the laws of the metrics $\frk c_r^{-1} e^{-\xi h_r(0)} D^{(b)}_h (r \cdot , r\cdot)$ are tight in the case when $h$ is a whole-plane GFF, and that every subsequential limit of the laws of these metrics is supported on metrics (not pseudometrics). 

Hence we can apply Theorem~\ref{thm-weak-uniqueness} with $\wt D = D^{(b)}$ to get that for each $b >0$, there is a deterministic constant $\frk k_b >0$ such that whenever $h$ is a whole-plane GFF plus a continuous function, a.s.\ $D_h^{(b)}  = \frk k_b D_h$. 
We now argue that $\frk k_b$ is a power of $b$. 

For $b_1,b_2 > 0$, we have $D^{(b_1b_2)} =  ( D^{(b_1)} )^{(b_2)}$, which implies that a.s.\ $D_h^{(b_1b_2)} = \frk k_{b_2} D_h^{(b_1)} = \frk k_{b_1} \frk k_{b_2} D_h$. Therefore, 
\eqb \label{eqn-weak-constants-mult}
\frk k_{b_1b_2} = \frk k_{b_1} \frk k_{b_2} . 
\eqe

It is also easy to see that $\frk k_b$ depends continuously on $b$. 
Indeed, by Axiom~\ref{item-metric-f} (Weyl scaling) and since $h(\cdot/b) - h_{1/b}(0) \eqD h$, we have $e^{-\xi h_{1/b}(0)} D_h^{(b)}(\cdot/b,\cdot/b) \eqD D_h$. By the continuity of $(z,w) \mapsto D_h(z,w)$ and $r\mapsto h_r(0)$, it follows that $D_h^{(b)} \rta D_h$ in law as $b\rta 1$.
This gives the continuity of $b\mapsto \frk k_b$ at $b = 1$. Using~\eqref{eqn-weak-constants-mult} then gives the desired continuity in general. 

The relation~\eqref{eqn-weak-constants-mult} and the continuity of $b\mapsto \frk k_b$ (actually, just Lebesgue measurability is enough) imply that $\frk k_b = b^\alpha$ for some $\alpha\in\BB R$.
Equivalently, for $b > 0$, a.s.\ 
\eqb \label{eqn-weak-to-strong-switch}
D_h(b\cdot, b\cdot)  =  b^{-\alpha} D_{h(b\cdot)}(\cdot,\cdot) .
\eqe 

For a whole-plane GFF, $h(b\cdot) - h_b(0) \eqD h$. By Axiom~\ref{item-metric-f} (Weyl scaling) and the definition of $\frk k_b$, 
\eqb
b^\alpha e^{-\xi h_b(0)}  D_h(b\cdot,b\cdot)  =   D_{h(b\cdot) - h_b(0)} \eqD D_h .  
\eqe
Therefore, Axiom~\ref{item-metric-coord} holds for $D$ with $\frk c_r = r^{-\alpha}$. By~\eqref{eqn-scaling-constant-stronger}, we get that $\alpha = -\xi Q$.
Hence for $b > 0$, we have (using Axiom~\ref{item-metric-f} in the first equality) 
\eqb
D_{h(\cdot/b) + Q\log(1/b)}(b\cdot, b\cdot) = b^{-\xi Q} D_h^{(b)} =   D_h  .
\eqe
Therefore, $D$ is a strong LQG metric. 
\end{proof}

\begin{proof}[Proof of Theorems~\ref{thm-lfpp}, \ref{thm-strong-uniqueness}, and~\ref{thm-stronger-uniqueness} assuming Theorem~\ref{thm-weak-uniqueness}]
By Lemma~\ref{lem-weak-to-strong}, every weak $\gamma$-LQG metric is a strong $\gamma$-LQG metric.
By~\eqref{eqn-strong-implies-weak}, every strong LQG metric satisfies the axioms in the definition of a weak $\gamma$-LQG metric with $\frk c_r = r^{\xi Q}$. 
We can therefore apply Theorem~\ref{thm-weak-uniqueness} to get that there is at most one strong LQG metric. 
This completes the proof of the uniqueness parts of Theorems~\ref{thm-strong-uniqueness} and~\ref{thm-stronger-uniqueness}.

As for existence, we recall that~\cite[Theorem 1.2]{lqg-metric-estimates} (building on~\cite{dddf-lfpp}) shows that for every sequence of $\ep$'s tending to zero, there is a weak $\gamma$-LQG metric $D$ and a subsequence along which the re-scaled LFPP metrics $\frk a_\ep^{-1} D_h^\ep$ converge in probability to $D_h$, whenever $h$ is a whole-plane GFF plus a bounded continuous function. By the uniqueness part of Theorem~\ref{thm-stronger-uniqueness}, $D$ is in fact a strong $\gamma$-LQG metric and any two different subsequential limiting metrics differ by a deterministic multiplicative constant factor. 
Recall that $\frk a_\ep$ is the median $D_h^\ep$-distance between the left and right boundaries of the unit square in the case when $h$ is a whole-plane GFF normalized so that $h_1(0) = 0$. 
Hence for any subsequential limiting metric the median $D_h$-distance between the left and right boundaries of the unit square is 1. 
 Therefore, the multiplicative constant factor is 1, so the subsequential limit of $D_h^\ep$ in probability is unique.
 This gives Theorem~\ref{thm-lfpp} and the existence parts of Theorems~\ref{thm-strong-uniqueness} and~\ref{thm-stronger-uniqueness}. 
\end{proof}

Finally, we note that our results give non-trivial information about the approximating LFPP metrics from~\eqref{eqn-lfpp}.
Indeed, let $\{\frk a_\ep\}_{\ep > 0}$ be the scaling constants from Theorem~\ref{thm-lfpp}.
It is shown in~\cite[Theorem 1.5]{dg-lqg-dim} that $\frk a_\ep = \ep^{1-\xi Q + o_\ep(1)}$. 
Using Theorem~\ref{thm-lfpp}, we obtain the following stronger form of this relation. 

\begin{cor} \label{cor-constant}
The function $\ep\mapsto \frk a_\ep$ is regularly varying with exponent $ 1-\xi Q $, i.e., for every $C > 0$ one has $\lim_{\ep\rta 0} \frk a_{C\ep}/\frk a_\ep = C^{ 1-\xi Q }$. 
\end{cor}

We expect, but do not prove here, that in fact Theorem~\ref{thm-lfpp} holds with $\frk a_\ep = \ep^{1-\xi Q}$.

\begin{proof}[Proof of Corollary~\ref{cor-constant}]
It is shown in~\cite[Lemma 2.14]{lqg-metric-estimates} that for any sequence of $\ep$'s tending to zero along which the re-scaled LFPP metrics $\frk a_\ep^{-1} D_h^\ep$ converge in law, also $\frk a_{C\ep}/\frk a_\ep$ converges (the limit is $C  \frk c_{1/C}$, with $\frk c_{1/C}$ as in Axiom~\ref{item-metric-coord} (tightness across scales) for the limiting weak $\gamma$-LQG metric).
By Theorem~\ref{thm-lfpp}, $\frk a_\ep^{-1} D_h^\ep$ converges in probability as $\ep\rta 0$, so in fact $\frk a_{C\ep}/\frk a_\ep$ converges, not just subsequentially. 
This means that $\frk a_{C\ep}$ is regularly varying with some exponent $\alpha > 0$. 
Since $\frk a_\ep = \ep^{1-\xi Q +o_\ep(1)}$, we must have $\alpha  =1-\xi Q$. 
\end{proof}

\subsection{Outline}
\label{sec-outline}

As explained above, to prove our main results it remains only to prove Theorem~\ref{thm-weak-uniqueness}. We emphasize that unlike many results in the theory of LQG, this paper does not build on a large amount of external input. Rather, we will only use some results from the papers~\cite{dddf-lfpp,local-metrics,lqg-metric-estimates,gm-confluence}, which can be taken as black boxes. All of the externally proven results which we will use are reviewed in Section~\ref{sec-prelim}.

Throughout this outline and the rest of the paper, we will use (without comment) the following two basic facts about $D_h$-geodesics when $D$ is a weak $\gamma$-LQG metric and $h$ is a whole-plane GFF. 
\begin{itemize}
\item Almost surely, for every $z,w\in\BB C$, there is at least one $D_h$-geodesic from $z$ to $w$. This follows from~\cite[Corollary 2.5.20]{bbi-metric-geometry} and the fact that $(\BB C , D_h)$ is a boundedly compact length space (i.e., closed bounded subsets are compact; see~\cite[Lemma 3.8]{lqg-metric-estimates}). 
\item For each fixed $z,w\in\BB C$, the $D_h$-geodesic from $z$ to $w$ is a.s.\ unique. This follows from, e.g., the proof of~\cite[Theorem 1.2]{mq-geodesics} (see also~\cite[Lemma 2.2]{gm-confluence}). 
\end{itemize}
 
In the remainder of this section we give a very rough idea of the proof of Theorem~\ref{thm-weak-uniqueness}. 
There are a number of technicalities involved, which we will gloss over in order to make the central ideas as transparent as possible. 
Consequently, some of the statements in this subsection are not exactly accurate without additional caveats. 
More detailed (and more precise) outlines can be found at the beginnings of the individual sections and subsections.

We first comment briefly on the role of the axioms in the proof. Axiom~\ref{item-metric-local} (locality) shows that the metric is compatible with the long-range independence and domain Markov properties of the GFF. These properties will be used in several places of our proofs (see Section~\ref{sec-gff-ind}). Axiom~\ref{item-metric-f} (Weyl scaling) has two main uses. First, it implies that adding a constant $C$  to the field scales distances by a factor of $e^{\xi C}$. This is important since the law of the GFF is only scale and translation invariant modulo additive constant. Second, it allows us to show that certain distance-related events occur with positive probability by adding a smooth bump function $h$ and noting that this affects the law of the GFF in an absolutely continuous way (see the outline of Section 5 below).
Axioms~\ref{item-metric-translate} (translation invariance) and~\ref{item-metric-coord} (tightness across scales) are often used together to get estimates for the restriction of the metric to the Euclidean ball of radius $r$ centered at $z$ which are uniform over all possible points $z$ and radii $r$. We will sometimes also use Axiom~\ref{item-metric-translate} by itself, with $r$ fixed, when we need more precise information than just up-to-constants estimates.
\medskip 

\noindent\textbf{Main idea of the proof.} Suppose $D$ and $\wt D$ are two weak $\gamma$-LQG metrics as in Theorem~\ref{thm-weak-uniqueness} and let $h$ be a whole-plane GFF.  
As explained in Proposition~\ref{prop-lqg-metric-bilip}, it follows from a general theorem for local metrics of the Gaussian free field~\cite[Theorem 1.6]{local-metrics} that $D_h$ and $\wt D_h$ are bi-Lipschitz equivalent, i.e.,
\eqb \label{eqn-max-min-def}
c_* := \inf\left\{ \frac{\wt D_h(u,v)}{D_h(u,v)} : u,v\in\BB C ,\: u\not=v\right\}   > 0
\quad\text{and} \quad
C_* := \sup\left\{ \frac{\wt D_h(u,v)}{D_h(u,v)} : u,v\in\BB C ,\: u\not= v\right\} < \infty .
\eqe 
It is easily seen that $c_*$ and $C_*$ are a.s.\ equal to deterministic constants (Lemma~\ref{lem-max-min-const}).
We identify $c_*$ and $C_*$ with these constants (which amounts to re-defining $c_*$ and $C_*$ on an event of probability zero).
To prove Theorem~\ref{thm-weak-uniqueness} we will show that $c_* = C_*$. 

The basic idea of the proof of this fact is as follows. Suppose by way of contradiction that $c_* < C_*$. Then for any $c' \in (c_* , C_*)$ there a.s.\ exist distinct points $u,v\in\BB C$ such that $\wt D_h(u,v) \leq c' D_h(u,v)$. 
In Section~\ref{sec-attained} (see outline below), using translation invariance of the GFF, modulo additive constant, and the local independence properties of the GFF,  we will deduce from this that the following is true. There exists $\ul\beta  , \ul p\in (0,1)$, depending only on the laws of $D_h$ and $\wt D_h$, such that for each $c' \in (c_* , C_*)$ there are many small values of $r> 0$ (how small depends on $c'$) for which 
\eqb \label{eqn-outline-ulG}
\BB P\left[ \text{$\exists u, v \in B_{r }(0)$ s.t.\ $|u - v | \geq  \ul\beta r$ and $\wt D_h(u,v) \leq c' D_h(u,v) $} \right] \geq \ul p , 
\eqe
where $B_{r }(0)$ is the Euclidean ball of radius $r $ centered at 0. 
By interchanging the roles of $D_h$ and $\wt D_h$, we can similarly find $\ol\beta  , \ol p\in (0,1)$, depending only on the laws of $D_h$ and $\wt D_h$, such that for each $C' \in (c_* , C_*)$, there are many small values of $r>0$ (how small depends on $C'$) for which
\eqb \label{eqn-outline-olG}
\BB P\left[ \text{$\exists u, v \in B_{r }(0)$ s.t.\ $|u - v | \geq  \ol\beta r$ and $\wt D_h(u,v) \geq C' D_h(u,v) $} \right] \geq \ol p . 
\eqe
See Section~\ref{sec-attained} for precise statements. 
The reason why the bounds only hold for ``many" choices of $r  > 0$, instead of for all $r > 0$, is that we only have tightness across scales (Axiom~\ref{item-metric-coord}), not exact scale invariance. 
We will use~\eqref{eqn-outline-ulG} to deduce a contradiction to~\eqref{eqn-outline-olG}.

Consider a $D_h$-geodesic $P$ between two fixed points $\BB z , \BB w \in\BB C$. 
Using~\eqref{eqn-outline-ulG} and a local independence argument for different segments of $P$ (which is explained in the outlines of Sections~\ref{sec-geodesic-iterate} and~\ref{sec-geodesic-shortcut} below), one can show that it holds with superpolynomially high probability as $\delta \rta 0$ (i.e., except on an event of probability decaying faster than any positive power of $\delta$), at a rate which is uniform over the choice of $\BB z$ and $\BB w$, that the following is true.
There are times $0 < s < t < D_h(\BB z,\BB w)$ such that $\wt D_h(P(s) , P(t)) \leq c' (t-s)$ and $D_h(P(s) ,P(t)) \geq \delta D_h(\BB z,\BB w)$.
By the definition~\eqref{eqn-max-min-def} of $C_*$, the $\wt D_h$-distance from $\BB z $ to $P(s)$ is at most $C_* s$ and the $\wt D_h$-distance from $P(t)$ to $\BB w$ is at most $C_* (D_h(\BB z,\BB w) -t)$. Combining these facts shows that with superpolynomially high probability as $\delta \rta 0$, 
\eqb \label{eqn-outline-less}
\wt D_h(\BB z, \BB w) \leq (C_* - (C_*-c') \delta ) D_h(\BB z, \BB w) .
\eqe

We now let $\ol\beta$ be as in~\eqref{eqn-outline-olG} and fix a large constant $q > 1$. 
For any $r > 0$, we can take a union bound to get that with probability tending to 1 as $\delta \rta 0$, at a rate which is uniform in $r$, the bound~\eqref{eqn-outline-less} holds simultaneously for all $\BB z ,\BB w \in \left(\delta^q r \BB Z^2\right) \cap B_{r }(0)$.
Now consider an arbitrary pair of points $\BB z , \BB w \in B_{r }(0)$ with $|\BB z - \BB w| \geq \ol\beta r $.
Let $\BB z' ,\BB w' \in \left( r \delta^q \BB Z^2\right) \cap B_{r }(0)$ be the points closest to $\BB z$ and $\BB w$, respectively.
By the bi-H\"older continuity of $D_h$ and $\wt D_h$ w.r.t.\ the Euclidean metric~\cite[Theorem 1.7]{lqg-metric-estimates}, if we choose $q$ sufficiently large, in a manner depending only on the H\"older exponents (i.e., only on $\gamma$), then $|D_h(\BB z , \BB w) - D_h(\BB z' ,\BB w')|$ and $|\wt D_h(\BB z ,\BB w) - \wt D_h(\BB z' ,\BB w')|$ are much smaller than $\delta D_h(\BB z,\BB w)$. 
From this, we infer that with probability tending to 1 as $\delta \rta 0$, at a rate which is uniform in $r$, the bound~\eqref{eqn-outline-less} holds simultaneously for all $\BB z,\BB w \in B_{r }(0)$ with $|\BB z-\BB w| \geq \ol\beta r$. 
If $\delta$ is chosen sufficiently small so that this probability is at least $1 - \ol p/2$, we get a contradiction to~\eqref{eqn-outline-olG} with $C ' = C_* - (C_*-c') \delta$.

The purpose of Sections~\ref{sec-attained}, \ref{sec-geodesic-iterate}, and~\ref{sec-geodesic-shortcut} is to fill in the details of the above argument. These three sections are mostly independent from one another: only the main theorem/proposition statements at the beginning of each section are used in later sections. 
\medskip

\noindent\textbf{Section~\ref{sec-attained}: bounds for ratios of distances at many scales.} 
The purpose of Section~\ref{sec-attained} is to prove (more quantitative versions of) the bounds~\eqref{eqn-outline-ulG} and~\eqref{eqn-outline-olG} stated above.
Since we are only working with a weak $\gamma$-LQG metric, not a strong $\gamma$-LQG metric, we do not have exact scale invariance, just tightness across scales (Axiom~\ref{item-metric-coord}). Consequently, if $c' \in (c_* , C_*)$, then we cannot necessarily say that pairs of points $u,v$ for which $\wt D_h(u,v) \leq c' D_h(u,v)$ exist with uniformly positive probability over different Euclidean scales. That is, it could in principle be that for every small fixed $\ul\beta > 0$, the probability that there exists $u,v\in B_{r }(0)$ with  $\wt D_h(u,v) \leq c' D_h(u,v)$ and $|u-v| \geq \ul\beta r$ is very small for some values of $r > 0$. 
However, we can say that such pairs of points exist with uniformly positive probability for a suitably ``dense" set of scales $r$ via an argument which proceeds (very roughly) as follows. 

Let $\ul\beta  , \ul p \in (0,1)$ be small and suppose by way of contradiction that there is a sequence $r_k \rta 0$ such that $r_{k+1} / r_k$ is bounded above and below by deterministic constants and the following is true. For each $k$, it holds with probability at least $1-\ul p$ that $\wt D_h(u,v) \geq c' D_h(u,v)$ for every pair of points $u,v\in B_{r_k }(0)$ for which $|u-v| \geq \ul\beta r_k$. Using the translation invariance of the metric (Axiom~\ref{item-metric-translate}) and the local independence properties of the GFF (in particular, Lemma~\ref{lem-annulus-iterate} below), we see that if $\ul\beta,\ul p$ are sufficiently small (how small depends only on the laws of $D_h$ and $\wt D_h$, not on $c'$ or $r_k$), then the following is true.  We can cover any fixed compact subset of $\BB C$ by Euclidean balls of the form $B_{r_k}(z)$ with the property that $\wt D_h(u,v) \geq c' D_h(u,v)$ for every pair of points $u \in \bdy B_{(1-\ul\beta)r_k}(z)$ and $v \in \bdy B_{r_k}(z)$.  By considering the times when a $\wt D_h$-geodesic between two fixed points of $\BB C$ crosses an annulus $B_{r_k}(z) \setminus B_{(1-\ul\beta)r_k}(z)$ for $z$ as above, we get that a.s.\ $\inf_{z,w\in \BB C} \wt D_h(z,w) / D_h(z,w)  \geq c''  $ for a constant $c''  \in (c_* ,c')$. This contradicts the definition~\eqref{eqn-max-min-def} of $c_*$.

Hence the set of ``bad" scales $r$ for which points $u,v \in B_{r }(0)$ with $|u-v| \geq \ul\beta r$ and $\wt D_h(u,v) \leq c' D_h(u,v)$ are unlikely to exist cannot be too large, which means that the complementary set of ``good" scales for which such points exist with probability at least $\ul p$ has to be reasonably dense.  This leads to~\eqref{eqn-outline-ulG}. The bound~\eqref{eqn-outline-olG} follows by interchanging the roles of $D_h$ and $\wt D_h$.
\medskip

\begin{figure}[t!]
 \begin{center}
\includegraphics[scale=1]{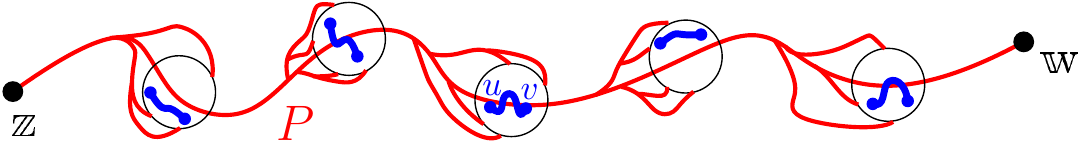}
\vspace{-0.01\textheight}
\caption{Illustration of the main ideas in Section~\ref{sec-geodesic-iterate}. Using results on confluence of geodesics from~\cite{gm-confluence}, we can show that there are many times $t$ at which the $D_h$-geodesic $P$ is \emph{stable}, in the sense that changing the behavior of the field in a small Euclidean ball around $P(t)$ does not result in a macroscopic change to the $D_h$-geodesic (the precise condition is given in~\eqref{eqn-ball-stab-event}). In particular, to produce such stable times we consider the metric ball growth started from $\BB z$ and use the confluence across a metric annulus from~\cite[Theorem 3.9]{gm-confluence} at a large number of evenly spaced radii.
 In fact, using the results of Section~\ref{sec-attained}, we can arrange that there are many such stable times whose corresponding balls contain a pair of points $u,v$ such that $\wt D_h(u,v) \leq c' D_h(u,v)$ and $|u-v|$ is comparable to the Euclidean radius of the ball. These pairs of points and the $\wt D_h$-geodesics between them are shown in blue. Using the results of Section~\ref{sec-geodesic-shortcut}, we can show that for each of these stable times, it holds with positive conditional probability given the past that $P$ gets close to the corresponding pair of points $u,v$. 
By a standard concentration inequality for Bernoulli sums, applied at the stable times, this shows that $P$ has to get close to at least one such pair of points $u,v$ with extremely high probability. 
}\label{fig-geo-outline}
\end{center}
\vspace{-1em}
\end{figure} 

\noindent\textbf{Section~\ref{sec-geodesic-iterate}: independence along an LQG geodesic.} Once we know that there are many pairs of points $u,v$ with $\wt D_h(u,v) \leq c' D_h(u,v)$, we want to use some sort of local independence to say that a $D_h$-geodesic $P$ is extremely likely to get close to at least one such pair of points (i.e., we need the $D_h$-distance from $P$ to each of $u$ and $v$ to be much smaller than $D_h(u,v)$). 
However, $D_h$-geodesics are highly non-local functionals of the field and do not satisfy any reasonable Markov property. 
So, techniques for obtaining local independence which may be familiar from the theory of SLE/GFF couplings~\cite{ss-contour,dubedat-coupling,ig1,ig2,ig3,ig4,shef-zipper,wedges} do not apply in our setting. 

Instead we need to develop a new set of techniques to obtain local independence at different points of $D_h$-geodesics.  See Figure~\ref{fig-geo-outline} for an illustration.
In fact, we will prove a general theorem (Theorem~\ref{thm-geo-iterate0}) which roughly speaking says the following. 
Suppose we are given events $\frk E_r^{\BB z , \BB w}(z)$ for $z , \BB z , \BB w \in\BB C$ and $r > 0$ with the following properties. 
The event $\frk E_r^{\BB z ,\BB w}(z)$ is determined by $h|_{B_r(z)}$ and the part of the $D_h$-geodesic $P^{\BB z,\BB w}$ from $\BB z$ to $\BB w$ which is contained in $B_r(z)$. Moreover, for each $z , \BB z ,\BB w \in \BB C$, the conditional probability of $\frk E_r^{\BB z , \BB w}(z)$  given $h|_{\BB C\setminus B_r(z)}$ and the event $\{ P^{\BB z,\BB w} \cap B_r(z)\not=\emptyset\}$ is a.s.\ bounded below by a deterministic constant. 
Then when $r$ is small it is very likely that for nearly every choice of $\BB z, \BB w \in\BB C$, the event $\frk E_r^{\BB z,\BB w}(z)$ occurs for at least one ball $B_r(z)$ hit by $P^{\BB z,\BB w}$.

We will eventually apply this theorem with $\frk E_r^{\BB z,\BB w}(z)$ given by, roughly speaking, the event that $P^{\BB z,\BB w}$ gets close to a pair of points $u,v \in B_r(z)$ with $\wt D_h(u,v) \leq c' D_h(u,v)$ and $|u-v| \geq \op{const} \times r$. This together with the triangle inequality and the bi-H\"older continuity of $D_h$ and $\wt D_h$ w.r.t.\ the Euclidean metric (to transfer from $|u-v| \geq \op{const} \times r$ to a lower bound for $D_h(u,v)$) will lead to~\eqref{eqn-outline-less}.
 
We will prove the above ``independence along a geodesic" theorem using the results on confluence of $D_h$-geodesics established in~\cite{gm-confluence}.
These results tell us that if $\BB z \in \BB C$ is fixed and $\BB w_1, \BB w_2 \in \BB C$ are close together, then the $D_h$-geodesics $P_1$ from $\BB z$ to $\BB w_1$ and $P_2$ from $\BB z$ to $\BB w_2$ typically agree until they get close to $\BB w_1$ and $\BB w_2$, i.e., $P_1|_{[0,\tau]} = P_2|_{[0,\tau]}$ for a time $\tau$ which is close to $D_h(\BB z,\BB w_1)$ (equivalently, to $D_h(\BB z, \BB w_2)$) when $D_h(\BB w_1,\BB w_2)$ is small. Note that this property is not true for geodesics for a smooth Riemannian metric, but it is true for geodesics in the Brownian map~\cite{legall-geodesics}.

Now fix $\BB z , \BB w$ and consider the $D_h$-geodesic $P = P^{\BB z,\BB w}$ from $\BB z$ to $\BB w$. 
The above confluence property applied with $\BB w_1 = P(t)$ for a typical time $t \in [0,D_h(\BB z,\BB w)]$ and $\BB w_2$ a point near $P(t)$ will allow us to show that with extremely high probability, there are many times $t\in [0, D_h(\BB z , \BB w)]$ at which $P$ is ``stable" in the following sense.
If we make a small modification to $h$ in a neighborhood of $P(t)$, then we will not change $P|_{[0,\tau]}$ for a time $\tau$ a little bit less than $t$. 
This allows us to say that events depending on the field in a small neighborhood of $P(t)$ have positive conditional probability given an initial segment of $P$. 
Applying this at a large number of evenly spaced times $t \in [0,D_h(\BB z,\BB w)]$ will show that it is extremely likely that the event $\frk E_r^{\BB z,\BB w}(z)$ discussed above occurs for at least one Euclidean ball $B_r(z)$ hit by $P$.
\medskip

\noindent\textbf{Section~\ref{sec-geodesic-shortcut}: an LQG geodesic gets close to a shortcut with positive probability.} 
Fix $\BB z,\BB w\in\BB C$ and let $P = P^{\BB z,\BB w}$ be the $D_h$-geodesic from $\BB z$ to $\BB w$.  By~\eqref{eqn-outline-ulG} and translation invariance (Axiom~\ref{item-metric-translate}) we know that there exists $\ul \beta,\ul p\in (0,1)$ such that if $c' \in (c_* , C_*)$, then there are many values of $r >0$ such that~\eqref{eqn-outline-ulG} holds with $z$ in place of $0$ (actually, we will use a variant of~\eqref{eqn-outline-ulG} which gives more precise information about the locations of $u$ and $v$; see Proposition~\ref{prop-attained-good'}).  In light of the results of Section~\ref{sec-geodesic-iterate}, we want to show that if we condition on $\{P\cap B_r(z) \not=\emptyset\}$, then the conditional probability that $P$ gets close to a pair of points $u,v$ as in~\eqref{eqn-outline-ulG} (with $z$ in place of $0$) is bounded below by a positive deterministic constant which does not depend on $r$ or $z$.
 
For a deterministic open set $U \subset \BB C$, one can prove that the $D_h$-geodesic $P$ enters $U$ with positive probability as follows. Consider a deterministic path from $\BB z$ to $\BB w$ and let $\phi$ be a smooth bump function which takes large values in a narrow ``tube" around this path and which vanishes outside a slightly larger tube. By Weyl scaling (Axiom~\ref{item-metric-f}), $D_{h-\phi}$ distances in the tube are much shorter than distances anywhere else. Hence the $D_{h-\phi}$-geodesic from $\BB z$ to $\BB w$ has to stay in the tube and hence has to enter $U$. Since the laws of $h$ and $h-\phi$ are absolutely continuous, we get that the $D_h$-geodesic enters $U$ with positive probability.

We will use a similar strategy to show that $P$ has positive conditional probability given $\{P\cap B_r(z)\not=\emptyset\}$ to get near a pair of points $u,v \in B_r(z)$ with $\wt D_h(u,v) \leq c' D_h(u,v)$ and $|u-v| \geq \ul\beta r$. However, additional complications arise.  For example, the region we want $P$ to enter (a small neighborhood of either $u$ or $v$) is random, which will be resolved by choosing a deterministic region which contains the $\wt D_h$-geodesic between $u$ and $v$ with positive probability. 
We also need to ensure that the condition $\wt D_h(u,v) \leq c' D_h(u,v)$ is not destroyed when we add our bump function.  To do this, we will need to make sure that the $\wt D_h$-geodesic between $u$ and $v$ is contained in the region where the bump function attains its largest possible value.
Another issue is that we need the bump function $\phi$ to be supported on a region of diameter of order $r \approx |u-v|$, so that its Dirichlet energy is bounded independently of $r$. In particular, this support cannot contain the starting and ending points $\BB z$ and $\BB w$ of the $D_h$-geodesic.
This will be resolved by growing the $D_h$-metric balls from $\BB z$ and $\BB w$ until they hit $B_{3r}(z)$ and choosing a bump function whose support approximates a path between the hitting points. 
\medskip

\noindent In \textbf{Section~\ref{sec-conclusion}}, we combine all of the above ingredients to conclude the proof of Theorem~\ref{thm-weak-uniqueness}, following the argument in the ``main ideas" section above.
\textbf{Section~\ref{sec-open-problems}} contains a list of open problems.

\begin{remark}[Proof for strong LQG metrics]
As explained above, we prove Theorem~\ref{thm-weak-uniqueness} instead of just proving Theorem~\ref{thm-strong-uniqueness} since subsequential limits of LFPP are only known to be weak LQG metrics, not strong LQG metrics. 
If we only wanted to prove Theorem~\ref{thm-strong-uniqueness}, we could make only a few minor simplifications to our proofs.
The most significant simplifications would be in Section~\ref{sec-attained}.  
In particular, similar arguments to the ones in Section~\ref{sec-attained} would give points $u,v$ such that $\wt D_h(u,v) = C_* D_h(u,v)$ instead of just $\wt D_h(u,v) \geq C' D_h(u,v)$ for $C'$ slightly less than $C_*$. Additionally, all of the results in Section~\ref{sec-attained} which are currently only proven to hold for ``at least $\mu\log_8 \ep^{-1}$ scales" could instead be shown to hold for all scales. This would allow us to eliminate the parameters $\mu,\nu,$ and $C'$ throughout the paper. We could of course also replace $\frk c_r$ by $r^{\xi Q}$ and eliminate the ``scale parameter" $\BB r$ throughout. This results in cosmetic simplifications in Sections~\ref{sec-geodesic-iterate} through~\ref{sec-conclusion}. 
\end{remark}

\begin{remark}[Relationship to~\cite{legall-uniqueness,miermont-brownian-map}]
It is natural to ask how our proof compares to the proofs of the Gromov-Hausdorff convergence of uniform quadrangulations to the Brownian map in~\cite{legall-uniqueness,miermont-brownian-map}. 
Both this paper and~\cite{legall-uniqueness,miermont-brownian-map} start from a tightness result and seek to show that the limiting object is unique. Moreover, all three papers rely crucially on confluence of geodesics (in the Brownian map setting, tightness is proven in~\cite{legall-topological} and confluence is proven in~\cite{legall-geodesics}). 
However, this is about the extent of the similarities.

In the Brownian map setting, one has an explicit a priori description of the conjectural limiting metric space $(X,\dmetric)$ in terms of the Brownian snake. 
In particular, there is a marked point $x_* \in X$ (which is a uniform sample from the area measure on the Brownian map) such that $\dmetric(x_* , x)$ can be described explicitly in terms of the Brownian snake.
Due to the convergence of discrete snakes to the Brownian snake and the Schaeffer bijection~\cite{cv-bijection,schaeffer-bijection}, one gets that any possible subsequential limit of uniform quadrangulations can be represented by a metric $\wt{\dmetric}$ on $X$ such that $\wt{\dmetric} \leq \dmetric$ and $\wt \dmetric(x_* , x) = \dmetric(x_*,x)$ for every $x \in X$ (see~\cite{marc-mokk-tbm}). The heart of the proof in each of~\cite{legall-geodesics,miermont-brownian-map} consists of using confluence to approximate a $\wt \dmetric$-geodesic by a concatenation of segments of $\wt \dmetric$-geodesics started from $x_*$ (the method of approximation in the two papers is quite different). 

In our setting, we do not have an a priori construction of the limiting object and we do not know a priori that any quantities related to two different weak LQG metrics are exactly equal. Instead, we have a coupling of our weak LQG metric to the GFF. We use confluence together with the Markov property of the GFF to get that far-away geodesic segments are nearly independent from each other. 
\end{remark}

\section{Preliminaries}
\label{sec-prelim}

In this subsection, we first introduce some basic (mostly standard) notation. We then review all of the results from~\cite{local-metrics,lqg-metric-estimates,gm-confluence} which we will need for the proof of Theorem~\ref{thm-weak-uniqueness}.  On a first read, the reader may wish to read only Sections~\ref{sec-notation} (which introduces notation) and~\ref{sec-bilip} (which proves the bi-Lipschitz equivalence of the metrics $D_h$ and $\wt D_h$ in Theorem~\ref{thm-bilip}) then refer back to the other subsections as needed. 

\subsection{Basic notation and terminology}
\label{sec-notation}

\subsubsection*{Integers}

\noindent
We write $\BB N = \{1,2,3,\dots\}$ and $\BB N_0 = \BB N \cup \{0\}$. 
For $a < b$, we define $[a,b]_{\BB Z}:= [a,b]\cap\BB Z$. 

\subsubsection*{Asymptotics}

\noindent
If $f  :(0,\infty) \rta \BB R$ and $g : (0,\infty) \rta (0,\infty)$, we say that $f(\ep) = O_\ep(g(\ep))$ (resp.\ $f(\ep) = o_\ep(g(\ep))$) as $\ep\rta 0$ if $f(\ep)/g(\ep)$ remains bounded (resp.\ tends to zero) as $\ep\rta 0$. We say that
\eqb \label{eqn-superpolynomial}
f(\ep) = o_\ep^\infty(\ep) \quad \text{if and only if} \quad f(\ep) = o_\ep(\ep^p) ,\: \forall p > 0. 
\eqe
We similarly define $O(\cdot)$ and $o(\cdot)$ errors as a parameter goes to infinity. 
\medskip

\noindent
If $f,g : (0,\infty) \rta [0,\infty)$, we say that $f(\ep) \preceq g(\ep)$ if there is a constant $C>0$ (independent from $\ep$ and possibly from other parameters of interest) such that $f(\ep) \leq  C g(\ep)$. We write $f(\ep) \asymp g(\ep)$ if $f(\ep) \preceq g(\ep)$ and $g(\ep) \preceq f(\ep)$. 
\medskip

\noindent
We often specify requirements on the dependencies on rates of convergence in $O(\cdot)$ and $o(\cdot)$ errors, implicit constants in $\preceq$, etc., in the statements of lemmas/propositions/theorems, in which case we implicitly require that errors, implicit constants, etc., in the proof satisfy the same dependencies. 
\medskip

\noindent
The parameter $\gamma$ is fixed throughout the paper. All implicit constants and rates of convergence are allowed to depend on $\gamma$, and this will not be stated explicitly.

\subsubsection*{Balls and annuli}

\noindent
For $z\in\BB C$ and $r>0$, we write $B_r(z)$ for the Euclidean ball of radius $r$ centered at $z$. We also define the open annulus
\eqb \label{eqn-annulus-def}
\BB A_{r_1,r_2}(z) := B_{r_2}(z) \setminus \ol{B_{r_1}(z)} ,\quad\forall 0 < r_r < r_2 < \infty .
\eqe

\noindent
For a metric space $(X,\dmetric)$ and $r>0$, we write $\mcl B_r(A;\dmetric)$ for the open ball consisting of the points $x\in X$ with $\dmetric(x,A) < r$.  
If $A = \{y\}$ is a singleton, we write $\mcl B_r(\{y\};\dmetric) = \mcl B_r(y;\dmetric)$.  
\medskip
  
\noindent
For a metric $\dmetric$ on $\BB C$, $r>0$, and $z\in\BB C$ we write $\mcl B_r^\bullet(z;\dmetric)$ for the \emph{filled metric ball} which is the union of $\ol{\mcl B_r(z;\dmetric)}$ and the bounded connected components of $\BB C\setminus \ol{\mcl B_r(z;\dmetric)}$.

\subsubsection*{Local sets}

Following~\cite[Lemma 3.9]{ss-contour}, if $(h,A)$ is a coupling of a whole-plane GFF and random compact set $A \subset\BB C$, we say that $A$ is a \emph{local set} for $h$ if for each open set $U\subset\BB C$, the event $\{A\cap U \not=\emptyset\}$ is conditionally independent from $h|_{\BB C\setminus U}$ given $h|_U$. If $A$ is determined by $h$ (which will be the case for all of the local sets we consider), this is equivalent to the statement that $A$ is determined by $h|_U$ on the event $\{A\subset U\}$.  The following lemma is a re-statement of~\cite[Lemma 2.1]{gm-confluence}.

\begin{lem}[\!\cite{gm-confluence}] \label{lem-ball-local} 
Let $D$ be a weak $\gamma$-LQG metric and let $h$ be a whole-plane GFF. Also let $z\in\BB C$ and let $\tau$ be a stopping time for the filtration generated by $(\mcl B_s^\bullet(z;D_h), h|_{\mcl B_s^\bullet(z;D_h)})$. Then $\mcl B_\tau^\bullet(z;D_h)$ is a local set for $h$. The same is true with closures of ordinary $D_h$-metric balls in place of filled $D_h$-metric balls. 
\end{lem}

\subsubsection*{General notational conventions}

We make some comments about how various symbols are used in order to help the reader follow the paper (we will not make any precise definitions here). 
\smallskip

\noindent
We use the symbols $\BB z,\BB w,z,w,u,v$ for points in $\BB C$. Typically, $\BB z,\BB w$ are fixed (often the endpoints of a geodesic), $z$ and $w$ are allowed to vary (e.g., over some open set) or are random, and $u,v$ are dummy variables appearing, e.g., in suprema/infima. 
\smallskip

\noindent
We use the symbols $p$ and $\BB p$ for probabilities. 
Typically, $\BB p$ is fixed throughout several lemmas, whereas $p$ is allowed to change more frequently. 
\smallskip

\noindent
The symbols $r$ and $\BB r$ denote Euclidean radii. 
Typically, $\BB r$ represents a fixed Euclidean scale. The reason why we need this is that we do not have exact scale invariance, only tightness across scales, so we often need to prove things at an arbitrary Euclidean scale, rather than just considering a single scale and then re-scaling.
The symbol $r$ is used for other Euclidean radii, which may depend on $\BB r$ and/or be random.
We use $s$ and $t$ for LQG radii. 
\smallskip

\noindent
The symbol $\ep$ typically denotes a small parameter which is independent from the Euclidean scale $\BB r$ (so $\ep \rta 0$ at a rate which does not depend on $\BB r$). 
The symbols $\mu$ and $\nu$ will always carry the same meaning as in the proposition statements in Section~\ref{sec-attained}: namely, we require that for any fixed $\BB r$ and any small enough $\ep$, there are at least $\mu\log_8 \ep^{-1}$ ``good" scales $r\in [\ep^{1+\nu} \BB r , \ep\BB r]$.

\subsection{Bi-Lipschitz equivalence of weak LQG metrics}
\label{sec-bilip}

In this subsection we explain why the results of~\cite{local-metrics} imply that any two weak $\gamma$-LQG metrics with the same scaling constants are bi-Lipschitz equivalent. 

\begin{prop} \label{prop-lqg-metric-bilip}
Let $h$ be a whole-plane GFF, let $\gamma \in (0,2)$, and let $D$ and $\wt D$ be two weak $\gamma$-LQG metrics, with the same scaling constants $\frk c_r$. 
There is a deterministic constant $C>0$ such that a.s.\
\eqb \label{eqn-lqg-metric-bilip} 
C^{-1} D_h(z,w) \leq \wt D_h(z,w) \leq C D_h(z,w) ,\quad\forall z,w\in\BB C .
\eqe
\end{prop}

Proposition~\ref{prop-lqg-metric-bilip} is a special case of a general theorem from~\cite{local-metrics} which tells us when two random metrics coupled with the same GFF are bi-Lipschitz equivalent. To state the theorem, we first recall some definitions. 
 
\begin{defn}[Jointly local metrics] \label{def-jointly-local}
Let $(h,D_1,\dots,D_n)$ be a coupling of the GFF $h$ with $n$ random continuous length metrics.
We say that $D_1,\dots,D_n$ are \emph{jointly local metrics} for $h$ if for any open set $V\subset \BB C$, the collection of internal metrics $\{ D_j(\cdot,\cdot;  V) \}_{j = 1,\dots,n}$ is conditionally independent from $(h|_{\BB C\setminus V} ,  \{ D_j(\cdot,\cdot;  U\setminus \ol V) \}_{j = 1,\dots,n}   )$ given $h|_V$.
\end{defn}

In the setting of Proposition~\ref{prop-lqg-metric-bilip}, the metrics $D_h$ and $\wt D_h$ are each local for $h$ due to Axiom~\ref{item-metric-local}. 
Since these metrics are each determined by $h$, they are conditionally independent given $h$. 
Therefore, we can apply~\cite[Lemma 1.4]{local-metrics} to get that $D_h$ and $\wt D_h$ are jointly local for $h$. 
 
\begin{defn}[Additive local metrics] \label{def-additive-local}
Let $(h,D_1,\dots,D_n)$ be a coupling of $h$ with $n$ random continuous length metric which are jointly local for $h$.
For $\xi \in\BB R$, we say that $ D_1,\dots,D_n $ are \emph{$\xi$-additive} for $h$ if for each $z\in \BB C$ and each $r> 0$ such that $B_r(z) \subset U$, the metrics $(e^{-\xi h_r(z)} D_1,\dots, e^{-\xi h_r(z)} D_n)$ are jointly local metrics for $h - h_r(z)$. 
\end{defn}

By Axiom~\ref{item-metric-f} (Weyl scaling), it follows that our metrics $D_h$ and $\wt D_h$ are jointly local for $h$. 
The following theorem is a special case of~\cite[Theorem 1.6]{local-metrics}.

\begin{thm}[\!\!\cite{local-metrics}] \label{thm-bilip}
Let $\xi \in\BB R$, let $h$ be a whole-plane GFF normalized so that $h_1(0) = 0$, and let $(h,D_h ,\wt D_h)$ be a coupling of $h$ with two random continuous metrics on $\BB C$ which are jointly local and $\xi$-additive for $h $.  
There is a universal constant $p \in (0,1)$ such that the following is true.
Suppose there is a constant $C>0$ such that (using the notation for annuli from~\eqref{eqn-annulus-def}), we have
\eqb \label{eqn-bilip}
\BB P\left[  \sup_{u,v \in \bdy B_r(z)} \wt D_h\left(u,v; \BB A_{r/2,2r}(z) \right) \leq C D_h(\bdy B_{r/2}(z) , \bdy B_r(z) ) \right] \geq p    ,\quad \forall z\in \BB C , \quad \forall r  > 0 .
\eqe 
Then a.s.\ $\wt D(z,w) \leq C D(z,w)$ for all $z,w\in\BB C$. 
\end{thm}

\begin{proof}[Proof of Proposition~\ref{prop-lqg-metric-bilip}]
By Axioms~\ref{item-metric-translate} and~\ref{item-metric-coord} for each of $D_h$ and $\wt D_h$, for any $p\in (0,1)$ we can find a constant $C_p > 1$ such that for each $z\in\BB C$ and each $r>0$, it holds with probability at least $p$ that
\eqb
 \sup_{u,v \in \bdy B_r(z)}  D_h\left(u,v; \BB A_{r/2,2r}(z) \right) \leq C_p \frk c_r e^{\xi h_r(z)}   ,
 \quad  D_h(\bdy B_{r/2}(z) , \bdy B_r(z) )\geq C_p^{-1} \frk c_r e^{\xi h_r(z)}   ,
\eqe
and the same is true with $\wt D_h$ in place of $h$. 
Therefore,~\eqref{eqn-bilip} holds with $C = C_p^2$ for each of the pairs $(D_h,\wt D_h)$ and $(\wt D_h , D_h)$. 
Theorem~\ref{thm-bilip} therefore implies Proposition~\ref{prop-lqg-metric-bilip} with $C=C_p^2$, where $p$ is as in Theorem~\ref{thm-bilip}.  
\end{proof}

\subsection{Local independence for the GFF}
\label{sec-gff-ind}

In many places throughout the paper, we will estimate various probabilities using the local independence properties of the GFF. We will do this using two different lemmas, which we state in this section. The first is a restatement of part of~\cite[Lemma 3.1]{local-metrics}. 

\begin{lem}[Iterating events in nested annuli] \label{lem-annulus-iterate}
Fix $0 < s_1<s_2 < 1$. Let $\{r_k\}_{k\in\BB N}$ be a decreasing sequence of positive numbers such that $r_{k+1} / r_k \leq s_1$ for each $k\in\BB N$ and let $\{E_{r_k} \}_{k\in\BB N}$ be events such that $E_{r_k} \in \sigma\left( (h-h_{r_k}(0)) |_{\BB A_{s_1 r_k , s_2 r_k}(0)  } \right)$ for each $k\in\BB N$. 
For $K\in\BB N$, let $N(K)$ be the number of $k\in [1,K]_{\BB Z}$ for which $E_{r_k}$ occurs. 
For each $a > 0$ and each $b\in (0,1)$, there exists $p = p(a,b,s_1,s_2) \in (0,1)$ and $c = c(a,b,s_1,s_2) > 0$ such that if  
\eqb \label{eqn-annulus-iterate-prob}
\BB P\left[ E_{r_k}  \right] \geq p , \quad \forall k\in\BB N  ,
\eqe 
then 
\eqb \label{eqn-annulus-iterate}
\BB P\left[ N(K)  < b K\right] \leq c e^{-a K} ,\quad\forall K \in \BB N. 
\eqe 
\end{lem}

We will only ever apply Lemma~\ref{lem-annulus-iterate} to say that $N(K) \geq 1$ with high probability, i.e., the choice of $b$ in~\eqref{eqn-annulus-iterate} will not matter for our purposes.

\begin{lem}[Iterating events in disjoint balls]  \label{lem-spatial-ind}
Let $h$ be a whole-plane GFF and fix $s > 0$. 
Let $n\in\BB N$ and let $\mcl Z$ be a collection of $\#\mcl Z = n$ points in $\BB C$ such that $|z-w| \geq 2(1+s)$ for each distinct $z,w\in\mcl Z$. 
For $z\in\mcl Z$, let $E_z$ be an event which is determined by $(h - h_{1+s}(z)) |_{B_1(z)}$. 
For each $p , q \in (0,1)$, there exists $n_* = n_*(s,p,q) \in \BB N$ such that if $\BB P[E_z] \geq p$ for each $z\in\mcl Z$, then
\eqbn
\BB P\left[ \bigcup_{z\in\mcl Z} E_z \right] \geq q ,\quad \forall n \geq n_* .
\eqen 
\end{lem}
\begin{proof}
Let $U :=  \bigcup_{z\in \mcl Z}  B_{1+s}(z)$ and let $\frk h$ be the harmonic part of $h|_{U}$. Since the balls $B_{1+ s}(z)$ for $z\in\mcl Z$ are disjoint, the Markov property of $h$ implies that the fields $(h-h_{1+s}(z))|_{B_{1+s}(z)}$ for $z\in\mcl Z$, and hence also the events $E_z$, are conditionally independent given $h|_{\BB C\setminus U}$ (equivalently, given $\frk h$). 

We will now compare the conditional law given $h|_{\BB C\setminus U}$ to the unconditional law. 
For $z\in\mcl Z$, let
\eqb
\frk M_z := \sup_{u \in B_{1+s/2}(z)} |\frk h(u) - \frk h(z)| .
\eqe
By a standard Radon-Nikodym derivative calculation for the GFF (see, e.g.,~\cite[Lemma 4.1]{mq-geodesics}) and the translation and scale invariance of the law of $h$, modulo additive constant, for each $\alpha >0$ there is a constant $C = C(\alpha,s) > 0$ such that the following is true. The conditional law given of $(h - h_{1+s}(z))|_{B_1(z)}$ given $h|_{\BB C\setminus U}$ is absolutely continuous with respect to its marginal law and if $H_z$ denotes the Radon-Nikodym derivative of the conditional law with respect to the marginal law, then a.s.\ 
\eqb \label{eqn-rn-moment}
\max\left\{ \BB E\left[ H_z^{ \alpha} \,|\,  h|_{\BB C\setminus U} \right] , \, \BB E\left[ H_r^{-\alpha} \,|\, h|_{\BB C\setminus U} \right] \right\}      
\leq C \exp\left(C   \frk M_z^2 \right) .
\eqe

Each $\frk M_z$ is an a.s.\ finite random variable. By the translation invariance of the law of $h$, modulo additive constant, the law of $\frk M_z$ does not depend on $z$. 
So, we can find a constant $A = A(s,q) > 0$ such that $\BB P[\frk M_z \leq A ] \geq 1 - (1-q)/4$ for each $z \in \mcl Z$. 
Then $\BB E[\#\{z\in\mcl Z : \frk M_z > A\}] \leq (1-q) n / 4$ so 
\eqb \label{eqn-spatial-ind-count}
\BB P\left[ \#\{z\in\mcl Z : \frk M_z \leq A\} \geq n/2 \right] \geq 1 - \frac{1-q}{2} .
\eqe 

Since $E_z$ is determined by $(h - h_{1+s}(z))|_{B_1(z)}$ and $\BB P[E_z] \geq p$ for each $z\in\mcl Z$,~\eqref{eqn-rn-moment} implies that there exists $\wt p = \wt p(p,  A) > 0$ such that on the event $\{\frk M_z \leq A\} $ (which is determined by $h|_{\BB C\setminus U}$), a.s.\ 
\eqb
\BB P\left[ E_z \,|\, h|_{\BB C\setminus U}  \right]  \geq \wt p. 
\eqe
Since the $E_z$'s are conditionally independent given $h|_{\BB C\setminus U}$, we see that a.s.\
\eqb \label{eqn-spatial-ind-iterate}
\BB P\left[ \bigcup_{z\in\mcl Z} E_z \,|\, h|_{\BB C\setminus U}  \right]  \geq 1 - \wt p^{\#\{z\in\mcl Z : \frk M_z \leq A\}}       .
\eqe
We now choose $n_*$ large enough that $1 - \wt p^{n_*/2} \geq 1 - (1-q)/2$ and combine~\eqref{eqn-spatial-ind-count} with~\eqref{eqn-spatial-ind-iterate}.
\end{proof}

\subsection{Estimates for weak LQG metrics}
\label{sec-a-priori-estimates}

In this subsection we review results from~\cite{lqg-metric-estimates} which we will need for the proofs of our main theorems. 
Throughout, $D$ denotes a weak $\gamma$-LQG metric and $h$ denotes a whole-plane GFF.  In particular, we state a bi-H\"older continuity bound for $D_h$ and the Euclidean metric (Lemma~\ref{lem-holder-uniform}), a bound for the $D_h$-diameters of squares (Lemma~\ref{lem-square-diam}), and bounds which prevent a $D_h$-geodesic from spending a long time near a line (Lemma~\ref{lem-line-path}), a circle (Lemma~\ref{lem-attained-long}), or the boundary of a $D_h$-metric ball (Lemma~\ref{lem-geo-bdy}).

All of the results which we state in this subsection involve a parameter $\BB r$, which controls the ``Euclidean scale" at which we are working. 
This parameter is necessary since we are only assuming tightness across scales (Axiom~\ref{item-metric-coord}) instead of exact scale invariance.
All estimates are required to be uniform in the choice of $\BB r$. 
Our first result, which follows from~\cite[Lemmas 3.20 and 3.22]{lqg-metric-estimates}, is a form of local H\"older continuity for the identity map $(\BB C , |\cdot|) \rta (\BB C , D_h)$ and its inverse. 

\begin{lem}[H\"older continuity] \label{lem-holder-uniform}
Fix a compact set $K\subset\BB C$ and exponents $\chi \in (0,\xi(Q-2))$ and $\chi' > \xi(Q+2)$. For each $\BB r > 0$, it holds with probability tending to $1$ as $a \rta 0$, at a rate which is uniform in $\BB r$, that for each $u,v\in \BB r K $ with $ |u-v| \leq a \BB r$, 
\eqb \label{eqn-holder-lower}
 D_h\left( u,v  \right) \geq   \frk c_{\BB r}  e^{ \xi h_{\BB r}(0)} \left|\frac{u-v}{\BB r} \right|^{\chi'} \quad \text{and} 
\eqe
\eqb \label{eqn-holder-upper}
D_h\left( u,v ; B_{2|u-v|}(u) \right)  \leq \frk c_{\BB r}  e^{ \xi h_{\BB r}(0)}  \left|\frac{u-v}{\BB r} \right|^\chi  .
\eqe
\end{lem}

We note that~\eqref{eqn-holder-upper} gives an upper bound for the $D_h$-distance from $u$ to $v$ \emph{along paths which stay in $B_{2|u-v|}(u)$}. This is slightly stronger than just an upper bound for $D_h(u,v)$. In Section~\ref{sec-geodesic-shortcut}, we will also need the following variant of~\eqref{eqn-holder-upper} which gives an upper bound for the $D_h$-internal diameters of Euclidean squares and is proven in~\cite[Lemma 3.20]{lqg-metric-estimates}. 

\begin{lem}[Internal diameters of Euclidean squares] \label{lem-square-diam}
Let $K$ and $\chi$ be as in Lemma~\ref{lem-holder-uniform}. 
For each $\chi \in (0,\xi(Q-2))$ and each $\BB r > 0$, it holds with probability tending to $1$ as $\ep\rta 0$, at a rate which is uniform in $\BB r$, that for each $k\in\BB N_0$ and each $2^{-k}\ep \BB r \times 2^{-k}\ep \BB r$ square $S$ with corners in $2^{-k}\ep \BB r \BB Z^2$ which intersects $\BB r K$, 
\eqb \label{eqn-square-diam}
 \sup_{u,v \in S} D_h\left( u,v ; S \right) \leq   \frk c_{\BB r} e^{ \xi h_{\BB r}(0)}  (2^{-k} \ep)^\chi .
\eqe
\end{lem}

In several places throughout the paper, we will want to prevent a $D_h$-geodesic from staying in small neighborhood of a fixed Euclidean path. 
The following lemma, which is a restatement of~\cite[Proposition 4.1]{lqg-metric-estimates}, will allow us to do this.

\begin{lem}[Lower bound for distances in a narrow tube] \label{lem-line-path}
Let $L\subset\BB C$ be a compact set which is either a line segment or an arc of a circle and fix $ b > 0$. 
For each $\BB r > 0$ and each $q > 0$, it holds with probability at least $1 -  \ep^{q^2/(2\xi^2) + o_\ep(1)}$ that
\eqb \label{eqn-line-path}
 \inf\left\{ D_h\left(u,v ; B_{\ep\BB r}(\BB r L ) \right) : u,v \in B_{\ep\BB r}(\BB r L ) , |u-v| \geq b\BB r \right\} \geq \ep^{  q + \xi Q - 1-\xi^2/2 } \frk c_{\BB r} e^{\xi h_{\BB r}(0)}  ,
\eqe
where the rate of the $o_\ep(1)$ depends on $L, b, q$ but not on $\BB r$.
\end{lem} 

By~\cite[Theorem 1.9]{ang-discrete-lfpp}, for each $\gamma \in (0,2)$ we have $1-\xi Q  \geq 0$, and hence $\xi Q - 1 - \xi^2/2 < 0$. Therefore, the power of $\ep$ on the right side of~\eqref{eqn-line-path} is negative for small enough $q$. Hence, Lemma~\ref{lem-line-path} implies that when $\ep$ is small and $u,v\in B_{\ep \BB r}(\BB r L)$ with $|u-v| \geq b\BB r$, it holds with high probability that $D_h\left(u,v; B_{\ep\BB r}(\BB rL)\right)$ is much larger than $D_h(u,v)$. In particular, a $D_h$-geodesic from $u$ to $v$ cannot stay in $B_{\ep\BB r}(L)$.
Lemma~\ref{lem-line-path} has the following useful corollary.
For the statement, we recall the notation for Euclidean annuli from~\eqref{eqn-annulus-def}.

\begin{lem}[Lower bound for distances in a narrow annulus] \label{lem-attained-long}
For each $S > s > 0$ and each $p\in (0,1)$, there exists $\alpha_* = \alpha_*(s,S,p) \in (1/2,1)$ such that for each $\alpha\in [\alpha_*,1)$, each $z\in\BB C$, and each $\BB r > 0$, 
\eqb \label{eqn-attained-long} 
\BB P\left[ \inf\left\{ D_h\left( u , v ; \BB A_{\alpha \BB r , \BB r}(z) \right)  :  u , v \in \BB A_{\alpha \BB r ,\BB r }(z) , D_h(u,v) \geq s \frk c_{\BB r} e^{\xi h_{\BB r}(z)} \right\} \geq S  \frk c_{\BB r} e^{\xi h_{\BB r}(z)}  \right] \geq p . 
\eqe 
\end{lem}
\begin{proof}
By Weyl scaling (Axiom~\ref{item-metric-f}), the event in~\eqref{eqn-attained-long} does not depend on the choice of additive constant for $h$.
By Axiom~\ref{item-metric-translate} (translation invariance) and the translation invariance of the law of $h$ modulo additive constant, the probability of this event does not depend on $z$. 
By Axiom~\ref{item-metric-coord} (tightness across scales), we can find $b = b(s)  > 0$ such that with probability at least $1-(1-p)/2$, any points $u,v\in B_{\BB r}(0)$ with $D_h(u,v) \geq s \frk c_{\BB r} e^{\xi h_{\BB r}(0)}$ satisfy $|u-v| \geq b \BB r$. Combining with Lemma~\ref{lem-line-path} (with $\ep = 1-\alpha$ and $L =\bdy\BB D$) concludes the proof.
\end{proof}

Finally, we record a lemma which prevents $D_h$-geodesics from spending a long time near the boundary of a $D_h$-metric ball which is needed in Section~\ref{sec-annulus-choose}. 
The lemma is a re-statement of~\cite[Proposition 4.3]{lqg-metric-estimates}. 

\begin{lem}[Geodesics cannot spend a long time near metric ball boundary] \label{lem-geo-bdy}
For each $ M  > 0$ and each $\BB r >0$, it holds with probability $1-o_\ep^\infty(\ep)$ as $\ep\rta 0$, at a rate which is uniform in the choice of $\BB r$, that the following is true.
For each $s > 0$ for which $\mcl B_s(0;D_h)\subset B_{\ep^{-M} \BB r}(0)$ and each $D_h$-geodesic $P$ from 0 to a point outside of $\mcl B_s(0;D_h)$,  
\eqb \label{eqn-geo-bdy}
\op{area}\left( B_{\ep \BB r}(P) \cap B_{\ep \BB r}\left(\bdy\mcl B_s(0;D_h) \right) \right) \leq \ep^{2 - 1/ M} \BB r^2,
\eqe
where $\op{area}$ denotes 2-dimensional Lebesgue measure.  
\end{lem}

\subsection{Confluence of geodesics}
\label{sec-confluence-prelim}

In this subsection we will review some facts about $D_h$-geodesics which are proven in~\cite{gm-confluence}. These facts are used only in Section~\ref{sec-stab}. 
For $z\in \BB C$, $r >0$, and $n\in\BB N$ we define the radii $\rho_r^n(z)$ as in~\cite[Equation (3.13)]{gm-confluence}.
The radius $\rho_r^n(z)$ is the $n$th smallest $t \in \{2^k r\}_{k\in\BB N}$ for which a certain event in $\sigma( (h-h_{6r}(z))|_{\BB A_{2r,5r}(z)})$ occurs.  Roughly speaking, the event in question tells us that if we fix $\BB z \in\BB C$ and $t > 0$ such that the filled LQG metric ball $\mcl B_t^\bullet(\BB z ;D_h)$ intersects $B_r(z)$, then with constant-order conditional probability given $(\mcl B_t^\bullet(\BB z ;D_h) , h|_{\mcl B_t^\bullet(\BB z ;D_h)})$, no $D_h$-geodesic from outside of $\mcl B_t^\bullet(\BB z ;D_h) \cup B_{5r}(z)$ can enter $B_r(z)$ before hitting $\mcl B_t^\bullet(\BB z ; D_h)$ (the precise definition of the event is given in~\cite[Section 3.2]{gm-confluence}).
We will not need the precise definition of $\rho_r^n(z)$ here, only a few facts which we will review in this subsection.  

We have $\rho_r^n(z) \geq 6r$ and $\rho_r^n(z)$ is a stopping time for the filtration generated by $h|_{B_{6 t}(z)}$ for $t\geq r$. 
The following is immediate from~\cite[Lemma 3.4]{gm-confluence}, the translation invariance of the law of $h$, modulo additive constant, and Axiom~\ref{item-metric-translate} (translation invariance). 
  
\begin{lem}[Bounds for radii used to control geodesics] \label{lem-clsce-all} 
There is a constant $\eta > 0$ depending only on the choice of metric such that the following is true. 
If we abbreviate 
\eqb \label{eqn-rho-abbrv}
\rho_{ \BB r , \ep}(z) := \rho_{\ep\BB r}^{\lfloor \eta\log\ep^{-1} \rfloor}(z) ,
\eqe  
then for each compact set $K\subset\BB C$, each $\BB r > 0$, and each $\BB z\in\BB C$, it holds with probability $1 - O_\ep(\ep^2)$ (at a rate depending on $K$, but not on $\BB r$ or $\BB z$) that
\eqb
\rho_{ \BB r , \ep}(z) \leq  \ep^{1/2} \BB r , \quad\forall z\in \left(\frac{\ep \BB r}{4} \BB Z^2 \right) \cap B_{\ep \BB r}( \BB r K + \BB z) .
\eqe
\end{lem}

Henceforth fix $\eta$ as in Lemma~\ref{lem-clsce-all} and let $\rho_{\BB r,\ep}(z)$ be as in~\eqref{eqn-rho-abbrv}.  
For $\BB r > 0$, $\ep > 0$, and a compact set $K\subset\BB C$, we define 
\eqb \label{eqn-extra-radius-eucl}
R_{\BB r}^\ep(K) :=  6 \sup\left\{ \rho_{\BB r,\ep}(z) : z\in \left( \frac{\ep \BB r}{4} \BB Z^2 \right) \cap B_{\ep \BB r}\left( K \right) \right\} +\ep \BB r .
\eqe
Since $\rho_{\BB r,\ep}(z)$ is a stopping time for the filtration generated by $h|_{B_{6 t}(z)}$ for $t\geq r$, each $\rho_{\BB r,\ep}(z)$ for $z\in \left( \frac{\ep \BB r}{4} \BB Z^2 \right) \cap B_{\ep \BB r}\left( K \right)$ is a.s.\ determined by $R_{\BB r}^\ep(K)$ and the restriction of $h$ to $B_{R_{\BB r}^\ep(K)}(K)$. 
Lemma~\ref{lem-clsce-all} shows that for each fixed choice of $K$, $\BB P[ R_{\BB r}^\ep(\BB r K + \BB z) \leq (6\ep^{1/2}  +\ep) \BB r ]$ tends to 1 as $\ep\rta 0$, uniformly over all $\BB z\in\BB C$ and $\BB r >0$.  

Recall from Section~\ref{sec-notation} that $\mcl B_s^\bullet(\BB z ; D_h) $ for $\BB z \in\BB C$ and $s > 0$ denotes the filled $D_h$-ball of radius $s$ centered at $\BB z$.  
Throughout the rest of this subsection we fix $\BB z \in\BB C$ and abbreviate $\mcl B_s^\bullet := \mcl B_s^\bullet(\BB z;  D_h)$. 
For $s >0$, define
\eqb \label{eqn-extra-radius}
\sigma_{s,\BB r}^\ep = \sigma_{s,\BB r}^\ep(\BB z) := \inf\left\{ s' > s :   B_{R_{\BB r}^\ep(\mcl B_s^\bullet)}(\mcl B_s^\bullet) \subset \mcl B_{s'}^\bullet  \right\}.
\eqe  
We observe that if $\tau$ is a stopping time for $\left\{ \left( \mcl B_t^\bullet , h|_{\mcl B_t^\bullet} \right) \right\}_{t\geq 0}$, then so is $\sigma_{\tau ,  \BB r }^\ep$. 
The following lemma is used to prevent $D_h$-geodesics from getting near a specified boundary point of a $D_h$-metric ball. It is an immediate consequence of~\cite[Lemma 3.6]{gm-confluence} (which is the case when $\BB z = 0$) together with the translation invariance of the law of $h$, modulo additive constant, and Axiom~\ref{item-metric-translate} (translation invariance). 
 
\begin{lem}[Geodesics are unlikely to get near a specified point of $\bdy\mcl B_\tau^\bullet$] \label{lem-geo-kill-pt}
There exists $\alpha  >0$, depending only on the choice of metric, such that the following is true. 
Let $r > 0$, let $\tau$ be a stopping time for the filtration generated by $\left\{ \left( \mcl B_s^\bullet , h|_{\mcl B_s^\bullet} \right) \right\}_{s \geq 0}$,
and let $x\in\bdy\mcl B_\tau^\bullet$ and $\ep\in (0,1)$ be chosen in a manner depending only on $( \mcl B_\tau^\bullet  , h|_{\mcl B_\tau^\bullet} )$. 
There is an event $ G_x^\ep \in \sigma\left(\mcl B_{  \sigma_{\tau,r}^\ep}^\bullet  , h|_{B_{  \sigma_{\tau,r}^\ep}^\bullet} \right)$ with the following properties.
\begin{enumerate}[A.]
\item If $ R_r^\ep(\mcl B_\tau^\bullet) \leq   \op{diam} \mcl B_\tau^\bullet$ and $G_x^\ep $ occurs, then no $D_h$-geodesic from $\BB z$ to a point in $\BB C\setminus \mcl B_{\sigma_{\tau,\BB r}^\ep}^\bullet$ can enter $B_{\ep r}(x) \setminus \mcl B_\tau^\bullet$. \label{item-geo-event-kill-pt}
\item There is a deterministic constant $C_0 > 1$ depending only on the choice of metric such that a.s. $\BB P\left[     G_x^\ep \,|\, \mcl B_\tau^\bullet  , h|_{\mcl B_\tau^\bullet} \right] \geq 1 -  C_0 \ep^\alpha$. \label{item-geo-event-prob-pt}
\end{enumerate}
\end{lem} 

We will now state a confluence property for LQG geodesics started from $\BB z$.  
Each point $x\in \bdy \mcl B_s^\bullet $ lies at $D_h$-distance exactly $s$ from $\BB z$, so every $D_h$-geodesic from $\BB z$ to $x$ stays in $\mcl B_s^\bullet$. For some atypical points $x$ there might be many such $D_h$-geodesics. But, it is shown in~\cite[Lemma 2.4]{gm-confluence} that there is always a distinguished $D_h$-geodesic from $\BB z$ to $x$, called the \emph{leftmost geodesic}, which lies (weakly) to the left of every other $D_h$-geodesic from $\BB z$ to $x$ if we stand at $x$ and look outward from $\mcl B_s^\bullet$. The following is~\cite[Theorem 1.4]{gm-confluence}. 

\begin{thm}[Confluence of geodesics across a metric annulus] \label{thm-finite-geo0}
Almost surely, for each $0 < t < s < \infty$ there is a finite set of $D_h$-geodesics from $\BB z$ to $\bdy\mcl B_t^\bullet $ such that every leftmost $D_h$-geodesic from $\BB z$ to $\bdy\mcl B_s^\bullet$ coincides with one of these $D_h$-geodesics on the time interval $[0,t]$. In particular, there are a.s.\ only finitely many points of $\bdy\mcl B_t^\bullet $ which are hit by leftmost $D_h$-geodesics from $\BB z$ to $\bdy\mcl B_s^\bullet $. 
\end{thm}

Combined with~\cite[Lemma 2.7]{gm-confluence}, Theorem~\ref{thm-finite-geo0} tells us that we can decompose $\bdy\mcl B_s^\bullet$ into a finite union of boundary arcs such that for any points $x,y\in\bdy\mcl B_s^\bullet$ which lie in the same arc, the leftmost $D_h$-geodesics from $\BB z$ to $x$ and from $\BB z$ to $y$ coincide in the time interval $[0,t]$. 
We will need a more quantitative version of Theorem~\ref{thm-finite-geo0} which gives us stretched exponential concentration for the number of such arcs if we truncate on a certain high-probability regularity event.  
To this end, we define  
\eqb \label{eqn-tau_r-def}
\tau_r(\BB z) := D_h(\BB z , \bdy B_r(\BB z) )  = \inf\left\{ s > 0 : \mcl B_s^\bullet \not\subset B_r(\BB z) \right\} ,\quad\forall r  > 0 .
\eqe
We also fix $\chi \in (0, \xi(Q-2))$, chosen in a manner depending only on $\xi$ and $Q$, so that by Lemma~\ref{lem-holder-uniform} $D_h$ is a.s.\ locally $\chi$-H\"older continuous w.r.t.\ the Euclidean metric. 
For $\BB r  >0$ and $a \in (0,1)$, we define $\mcl E_{\BB r}^{\BB z}(a)  $ to be the event that the following is true.
\begin{enumerate}
\item \emph{(Comparison of $D_h$-balls and Euclidean balls)} $B_{a \BB r}(\BB z) \subset \mcl B_{\tau_{\BB r}}^\bullet$ and $ \tau_{3\BB r} - \tau_{2\BB r} \geq a \frk c_{\BB r} e^{ \xi h_{\BB r}(0)}$.   \label{item-quantum-ball-contained0} 
\item \emph{(One-sided H\"older continuity)} $\frk c_{\BB r}^{-1} e^{-\xi h_r(0)} D_h(u,v) \leq    \left( \frac{ |u - v| }{\BB r} \right)^\chi$ for each $u,v \in B_{4 \BB r}(0)$ with $|u-v|/\BB r \leq a$. \label{item-holder-cont0}
\item \emph{(Bounds for radii used to control geodesics)} The radii of Lemma~\ref{lem-clsce-all} satisfy $\rho_{\BB r,\ep}(z) \leq \ep^{1/2}\BB r $ for each $ z\in \left( \frac{\ep \BB r}{4} \BB Z^2 \right) \cap B_{4 \BB r}(\BB z) $ and each dyadic $\ep \in (0,a]$.    \label{item-good-radii-small0}
\end{enumerate}

It is easy to see that $\BB P[\mcl E_{\BB r}^{\BB z}(a)] \rta 1$ as $a\rta 0$, uniformly over the choice of $\BB r$ and $\BB z$: in particular, this follows from \cite[Lemma 3.8]{gm-confluence} (which is the case when $\BB z=0$) and Axiom~\ref{item-metric-translate}.  
We will in fact show in Section~\ref{sec-iterate-reg} that with high probability, $\mcl E_{\BB r}^{\BB z}(a)$ occurs \emph{simultaneously} for all $\BB z$ in a fixed bounded open subset of $\BB C$. 
The following more quantitative version of Theorem~\ref{thm-finite-geo0} is~\cite[Theorem 3.9]{gm-confluence}.

\begin{thm}[Quantitative confluence of geodesics] \label{thm-finite-geo-quant}
For each $a \in (0,1)$, there is a constant $b_0 > 0$ depending only on $a$ and constants $b_1 , \beta > 0$ depending only on the choice of metric $D$ such that the following is true. 
For each $\BB z\in\BB C$, each $\BB r >0$, each $N\in\BB N$, and each stopping time $\tau$ for $\{(\mcl B_s^\bullet , h|_{\mcl B_s^\bullet})\}_{s\geq 0}$ with $\tau \in [\tau_{\BB r}(\BB z) ,\tau_{2\BB r}(\BB z) ]$ a.s., the probability that $\mcl E_{\BB r}^{\BB z}(a)$ occurs and there are more than $N$ points of $\bdy\mcl B_{\tau}^\bullet$ which are hit by leftmost $D_h$-geodesics from $\BB z$ to $\bdy\mcl B_{\tau +  N^{-\beta} \frk c_{\BB r} e^{\xi h_{\BB r}(\BB z)} }^\bullet$ is at most $b_0 e^{-b_1 N^\beta}$. 
\end{thm}

\section{The optimal bi-Lipschitz constant}
\label{sec-attained}

Throughout this section, we assume that we are in the setting of Theorem~\ref{thm-weak-uniqueness}, so that $D$ and $\wt D$ are two weak $\gamma$-LQG metrics with the same scaling constants. We also let $h$ be a whole-plane GFF. We know from Proposition~\ref{prop-lqg-metric-bilip} that $D_h$ and $\wt D_h$ are a.s.\ bi-Lipschitz equivalent. 
We define the optimal bi-Lipschitz constants $c_*$ and $C_*$ as in~\eqref{eqn-max-min-def}.
Since $D_h$ and $\wt D_h$ are a.s.\ bi-Lipschitz equivalent (Proposition~\ref{prop-lqg-metric-bilip}), a.s.\ $0 < c_* \leq C_* < \infty$. 

\begin{lem} \label{lem-max-min-const}
Each of $c_*$ and $C_*$ is a.s.\ equal to a deterministic constant.
\end{lem}
\begin{proof}
We will prove the statement for $C_*$; the statement for $c_*$ is proven in an identical manner.
Suppose $C > 0$ is such that $\BB P[C_*  >C] > 0$. We will show that in fact $\BB P[C_* >C] = 1$.
 
There is some large deterministic $R > 0$ such that with positive probability, there are points $u,v\in B_R(0)$ such that $\wt D_h(u,v) / D_h(u,v) > C$. Since each of $D_h$ and $\wt D_h$ induces the Euclidean topology on $\BB C$, after possibly increasing $R$, we can arrange that with positive probability, there are points $u,v\in B_R(0)$ such that 
\eqb \label{eqn-max-min-event}
\wt D_h(u,v) / D_h(u,v) > C , \quad D_h(u,v) \leq D_h(u,\bdy B_R(0)) ,\quad \text{and} \quad \wt D_h(u,v) \leq \wt D_h(u,\bdy B_R(0)) .
\eqe
The condition that $D_h(u,v) \leq D_h(u,\bdy B_R(0))$ is equivalent to the condition that $v$ is contained in the $D_h$-metric ball of radius $D_h(u,\bdy B_R(0))$ centered at $u$. By Axiom~\ref{item-metric-local} (locality), it follows that $h|_{B_R(0)}$ a.s.\ determines $ D_h(u,\bdy B_R(0))$ for every $u\in B_R(0)$ and hence also $h|_{B_R(0)}$ determines all of the $D_h$-metric balls of radius $D_h(u,\bdy B_R(0))$ centered at points of $B_R(0)$. 
Similar considerations hold with $\wt D_h$ in place of $D_h$. 
Therefore, the event that there exist $u,v\in B_R(0)$ such that~\eqref{eqn-max-min-event} holds is determined by $h|_{B_R(0)}$. 
In fact, by Axiom~\ref{item-metric-f} (Weyl scaling) this event is determined by $h|_{B_R(0)}$ \emph{viewed modulo additive constant}, since adding a constant to $h$ results in scaling $D_h$ and $\wt D_h$ by the same constant factor.

For $z\in\BB C$, let $E(z)$ be the event that there exist points $u,v\in B_R(z)$ such that~\eqref{eqn-max-min-event} holds with $B_R(z)$ in place of $B_R(0)$. 
Then $E(z)$ is determined by $h|_{B_R(z)}$, viewed modulo additive constant. 
By Axiom~\ref{item-metric-translate} (translation invariance) and the translation invariance of the law of $h$, modulo additive constant, the probability of $E (z)$ does not depend on $z$. 
The event that $E(z)$ occurs for infinitely many $z\in\BB Z^2$ is determined by the tail $\sigma$-algebra generated by $h|_{\BB C\setminus B_r(z)}$, viewed modulo additive constant, as $r\rta \infty$. This tail $\sigma$-algebra is trivial, so we get that a.s.\ $E(z)$ occurs for infinitely many $z\in\BB C$. This means that in fact $\BB P[C_* > C] = 1$, so $C_*$ is a.s.\ equal to a deterministic constant. 
\end{proof}

We henceforth re-define each of $c_*$ and $C_*$ on an event of probability zero so that they are deterministic.  The main goal of this section is to show that there are many values of $r > 0$ for which it holds with uniformly positive probability that there are points $\BB z, \BB w \in \BB C$ such that $|\BB z| , |\BB w|,$ and $|\BB z - \BB w|$ are all of order $r$ and $\wt D_h(\BB z  ,\BB w) /D_h(\BB z , \BB w)$ is close to $C_*$ (resp.\ $c_*$).
To quantify this, we introduce the following events. 
For $r > 0$, $C' \in (0,C_*]$, and $\beta \in (0,1)$, define 
\eqb \label{eqn-max-event}
\ol G_r(C',\beta) := \left\{ \text{$\exists \BB z , \BB w  \in B_{r }(0)$ s.t.\ $|\BB z  -  \BB w   | \geq \beta r$ and $\wt D_h(\BB z , \BB w  ) \geq C'  D_h(\BB z , \BB w  )$} \right\} .
\eqe 
For $c' \geq c_*$, we similarly define
\eqb \label{eqn-min-event}
\ul G_r(c',\beta) := \left\{ \text{$\exists \BB z , \BB w   \in B_{r }(0)$ s.t.\ $|\BB z -  \BB w  | \geq \beta r$ and $\wt D_h(\BB z , \BB w  ) \leq c'  D_h(\BB z , \BB w  )$} \right\} .
\eqe  
 
It is easy to see from the definition~\eqref{eqn-max-min-def} of $C_*$ that for each fixed $r > 0$ and $C' \in (0,C_*)$, there exists $p , \beta \in (0,1)$ (allowed to depend on $C'$ and $r$) such that $\BB P[\ol G_r(C',\beta)]\geq p$.\footnote{By the definition~\eqref{eqn-max-min-def} of $C_*$, there exists some $p , \beta\in (0,1)$ and $R > 0$ (allowed to depend on $r$) such that with probability at least $p$, there exists $\BB z,\BB w \in B_{Rr}(0)$ such that $|\BB z - \BB w| \geq \beta r$ and $\wt D_h(\BB z , \BB w  ) \geq C'  D_h(\BB z , \BB w  )$. We need to replace $B_{Rr}(0)$ by $B_r(0)$. 
By possibly replacing $\BB z$ and $\BB w$ by a pair of points along a $D_h$-geodesic from $\BB z$ to $\BB w$, we can arrange that in fact $|\BB z-\BB w| = \beta r$. 
We can cover $B_{Rr}(0)$ by at most a $\beta , R$-dependent constant number $N$ of Euclidean balls of the form $B_r(z)$ for $z\in B_{R r}(0)$ such that any two points $\BB z,\BB w \in B_{Rr}(0)$ with $|\BB z - \BB w| = \beta r$ are contained in one of these balls.  
By Weyl scaling (Axiom~\ref{item-metric-f}), the translation invariance of the law of $h$ modulo additive constant, and Axiom~\ref{item-metric-translate},
the probability that there exists $\BB z,\BB w\in B_r(z)$ with $|\BB z - \BB w| \geq \beta r$ and $\wt D_h(\BB z , \BB w  ) \geq C'  D_h(\BB z , \BB w  )$ does not depend on $z$.
By a union bound, it therefore follows that $\BB P[\ol G_r(C',\beta)] \geq p/N$. \label{footnote-G-prob}}
Since we are working with weak LQG metrics, which are not known to be exactly invariant under spatial scaling, it is \emph{not} clear a priori that $p$ and $\beta$ can be taken to be uniform in the choice of $r$. It is also not clear a priori that $p$ and $\beta$ can be chosen independently of $C'$.  
Similar considerations apply for $\ul G_r(c',\beta)$. 
We will establish that one can choose $p$ and $\beta$ independently of $C'$ and $r$ provided $r$ is restricted to lie in a suitably ``dense" subset of $(0,1)$, in the following sense.

\begin{prop} \label{prop-attained-max}
For each $0 < \mu < \nu < 1$, there exists $\ol\beta = \ol\beta(\mu,\nu) \in (0,1)$ and $\ol p = \ol p(\mu,\nu) \in (0,1)$ such that for each $C' \in (0,C_*)$ and each sufficiently small $\ep > 0$ (depending on $C'$), there are at least $\mu\log_8\ep^{-1}$ values of $r \in [\ep^{1+\nu}  ,\ep   ] \cap \{8^{-k} : k\in\BB N\}$ for which $\BB P[\ol G_r(C' , \ol\beta )] \geq \ol p$.   
\end{prop}

\begin{prop} \label{prop-attained-min}
For each $0 < \mu < \nu < 1$, there exists $\ul\beta = \ul\beta(\mu,\nu) \in (1/2,1)$ and $\ul p = \ul p(\mu,\nu) \in (0,1)$ such that for each $c'  > c_*$ and each sufficiently small $\ep > 0$ (depending on $c'$), there are at least $\mu\log_8\ep^{-1}$ values of $r \in [\ep^{1+\nu}  ,\ep   ] \cap \{8^{-k} : k\in\BB N\}$ for which $\BB P[\ul G_r(c' , \ul\beta )] \geq \ul p$.   
\end{prop}
 
We emphasize that the parameters $\ol\beta,\ol p$ in Proposition~\ref{prop-attained-max} (resp.\ the parameters $\ul\beta,\ul p$ in Proposition~\ref{prop-attained-min}) \emph{do not} depend on $C'$ (resp.\ $c'$). The only thing which depends on $C'$ (resp.\ $c'$) is how small $\ep$ has to be in order for the conclusion of the proposition statement to hold.  

\subsection{Quantitative versions of Propositions~\ref{prop-attained-max} and~\ref{prop-attained-min}}
\label{sec-attained-statement}

We will need more quantitative versions of Propositions~\ref{prop-attained-max} and~\ref{prop-attained-min} which differ from the original proposition statements in two important ways. First, instead of starting at a constant-order scale, we will start at some given scale $\BB r > 0$ for which we have an a priori lower bound on $\BB P[\ol G_{\BB r}(C'',\beta)]$ for some $C'' \in (0,C_*)$ and $\beta\in (0,1)$ (or $\BB P[\ul G_{\BB r}(c'',\beta)]$ for some $c'' > c_*$ and $\beta\in (0,1)$).
We will then produce many radii in $[\ep^{1+\nu}\BB r, \ep\BB r]$ instead of in $[\ep^{1+\nu} , \ep]$. 
The reason for introducing $\BB r$ is that we only have tightness across scales (Axiom~\ref{item-metric-coord}) instead of true scale invariance.
Second, instead of just lower bounding the probability of $\ol G_r(C' ,  \beta )$ or $\ul G_r(c',\beta)$, we will obtain a lower bound for the probability of a smaller event which is more complicated, but also more useful. 
Let us begin by stating a more quantitative version of Proposition~\ref{prop-attained-max}.

\begin{prop} \label{prop-attained-good}
For each $0 < \mu < \nu < 1$, there exists $\alpha_* = \alpha_*(\mu,\nu) \in (1/2,1)$ and $p = p(\mu,\nu) \in (0,1)$ such that for each $\alpha \in [\alpha_*,1)$ and each $C' \in (0,C_*)$, there exists $C''  = C''(\alpha,C',\mu,\nu ) \in (C' , C_*)$ such that for each $\beta \in (0,1)$, there exists $\ep_0 = \ep_0(\beta,\alpha,C',\mu,\nu ) > 0$ such that the following holds for each $\BB r > 0$ for which $\BB P[\ol G_{\BB r}(C'',\beta)] \geq \beta$ and each $\ep \in (0,\ep_0]$. 
\begin{enumerate}[(A)] 
\item There are at least $\mu\log_8\ep^{-1}$ values of $r \in [\ep^{1+\nu} \BB r ,\ep \BB r ] \cap \{8^{-k} \BB r : k\in\BB N\}$ for which the following holds with probability at least $p$. 
There exists $u \in \bdy B_{\alpha r}(0)$ and $v \in \bdy B_r(0)$ such that 
\eqb \label{eqn-attained-good}
\wt D_h(u,v) \geq C' D_h(u,v)  
\eqe
and the $D_h$-geodesic from $u$ to $v$ is unique and is contained in $\ol{\BB A_{\alpha r , r}(0)}$.  \label{item-attained-good} 
\end{enumerate}
\end{prop}
 
The event described in~\eqref{item-attained-good} is contained in $\ol G_r(C' , 1-\alpha)$, so if~\eqref{item-attained-good} holds for some $\BB r > 0$ then there are at least $\mu\log_8 \ep^{-1}$ values of $r \in [\ep^{1+\nu} \BB r , \ep \BB r] \cap \{8^{-k} : k\in\BB N\}$ such that 
\eqbn
\BB P[\ol G_r(C' , 1 - \alpha)] \geq p .
\eqen
Furthermore, as explained in Footnote~\ref{footnote-G-prob}, the definition~\eqref{eqn-max-min-def} of $C_*$ implies that for any $C'' \in (0,C_*)$, there exists some $\beta \in (0,1)$ such that $\BB P[\ol G_1(C'',\beta)] \geq \beta$. Therefore, Proposition~\ref{prop-attained-good} applied with $\BB r =1$ implies Proposition~\ref{prop-attained-max} with $\ol\beta =1-\alpha$ and $\ol p = p$. 
 
By the symmetry between our hypotheses on $\wt D_h$ and $D_h$, Proposition~\ref{prop-attained-good} implies the analogous statement with the roles of $D_h$ and $\wt D_h$ interchanged, which reads as follows.

\begin{prop} \label{prop-attained-good'}
For each $0 < \mu < \nu < 1$, there exists $\alpha_* = \alpha_*(\mu,\nu) \in (1/2,1)$ and $p = p(\mu,\nu) \in (0,1)$ such that for each $\alpha \in [\alpha_*,1)$ and each $c'  > c_*$, there exists $c''  = c''(\alpha,c',\mu,\nu) \in (c_* , c')$ such that for each $\beta \in (0,1)$, there exists $\ep_0 = \ep_0(\alpha,\beta,c',\mu,\nu ) > 0$ such that the following holds for each $\BB r > 0$ for which $\BB P[\ul G_{\BB r}(c'',\beta)] \geq \beta$ and each $\ep \in (0,\ep_0]$. 
\begin{enumerate}[(A')] 
\item There are at least $\mu\log_8\ep^{-1}$ values of $r \in [\ep^{1+\nu} \BB r ,\ep \BB r ] \cap \{8^{-k} \BB r : k\in\BB N\}$ for which it holds with probability at least $p$ that the following is true. There exists $u \in \bdy B_{\alpha r}(0)$ and $v \in \bdy B_r(0)$ such that 
\eqb \label{eqn-attained-good'}
\wt D_h(u,v) \leq c' D_h(u,v)  
\eqe
and the $\wt D_h$-geodesic from $u$ to $v$ is unique and is contained in $\ol{\BB A_{\alpha r , r}(0)}$.  \label{item-attained-good'} 
\end{enumerate}
\end{prop}

As in the case of Proposition~\ref{prop-attained-good}, Proposition~\ref{prop-attained-good'} immediately implies Proposition~\ref{prop-attained-min}.

To prove Proposition~\ref{prop-attained-good}, we will (roughly speaking) prove the contrapositive.

\begin{prop} \label{prop-attained-bad}
For each $0 < \mu < \nu < 1$, there exists $\alpha_* = \alpha_*(\mu,\nu) \in (1/2,1)$ and $p = p(\mu,\nu) \in (0,1)$ such that for each $\alpha\in [\alpha_*,1)$ and each $C' \in (0,C_*)$, there exists $C''  = C''(\alpha,C',\mu,\nu ) \in (C' , C_*)$ such that for each $\beta \in (0,1)$, there exists $\ep_0 = \ep_0(\alpha,\beta,C',\mu,\nu)  > 0$ such that if $\BB r > 0$ and there exists $\ep \in (0,\ep_0]$ satisfying the condition~\eqref{item-attained-bad} just below, then $\BB P[\ol G_{\BB r}(C'',\beta) ]  < \beta$. 
\begin{enumerate}[(A)] 
\setcounter{enumi}{1}
\item There are at least $(\nu-\mu) \log_8\ep^{-1}$ values of $r \in [\ep^{1+\nu} \BB r ,\ep \BB r ] \cap \{8^{-k} \BB r : k\in\BB N\}$ for which it holds with probability at least $1- p$ that the following is true. For each $u\in \bdy B_{\alpha r}(0)$ and $v\in\bdy B_r(0)$ for which the $D_h$-geodesic from $u$ to $v$ is unique and is contained in $\ol{\BB A_{\alpha r , r}(0)}$, one has \label{item-attained-bad}
\eqb \label{eqn-attained-bad}
\wt D_h(u,v) \leq C' D_h(u,v) .
\eqe 
\end{enumerate}
\end{prop}

\begin{proof}[Proof of Proposition~\ref{prop-attained-good}, assuming Proposition~\ref{prop-attained-bad}]
Assume we are given $0<\mu<\nu<1$ and let $\alpha_* , p $ be chosen as in Proposition~\ref{prop-attained-bad}. 
Also fix $\alpha\in [\alpha_*,1)$, $C' \in (0,C_*)$, and $\beta \in (0,1)$ and let $C''$ and $\ep_0$ be chosen as in Proposition~\ref{prop-attained-bad}.
For $\BB r,  \ep > 0$, let $\mcl K_{\BB r}^\ep :=  [\ep^{1+\nu} \BB r , \ep \BB r] \cap \left\{ 8^{-k}  \BB r : k\in\BB N \right\} $
and note that $\#\mcl K_{\BB r}^\ep =  \lfloor \nu \log_8 \ep^{-1} \rfloor$.  

If~\eqref{item-attained-good} does not hold for some $\BB r  > 0$ and $\ep \in (0,\ep_0]$, then there are \emph{fewer} than $\mu\log_8 \ep^{-1}$ values of $k\in \mcl K_{\BB r}^{\ep}$ for which the last sentence of~\eqref{item-attained-good} holds with probability at least $p$. 
For such a choice of $\BB r$ and $\ep$, there are at least $(\nu-\mu)\log_8\ep^{-1}$ values of $k\in \mcl K_{\BB r}^{\ep}$ for which the last sentence of~\eqref{item-attained-bad} holds with probability at least $1-p$. That is,~\eqref{item-attained-bad} holds for the pair $(\BB r , \ep)$. 
By Proposition~\ref{prop-attained-bad}, this means that $\BB P[\ol G_{\BB r}(C'',\beta) ]  < \beta$. 
Hence we have proven the contrapositive of Proposition~\ref{prop-attained-good}. 
\end{proof}

\subsection{Proof of Proposition~\ref{prop-attained-bad}}
\label{sec-attained-proof}

As explained in Section~\ref{sec-attained-statement}, to prove all of the propositions statements from earlier in this section it remains only to prove Proposition~\ref{prop-attained-bad}. The basic idea of the proof is as follows. 
If we assume that~\eqref{item-attained-bad} holds for a small enough choice of $p\in (0,1)$ (depending only on $\mu$ and $\nu$), then we can use Lemma~\ref{lem-annulus-iterate} to cover space by Euclidean balls of the form $B_{r/2}(z)$ for $r\in [\ep^{1+\nu} \BB r , \ep \BB r]$ with the following property.
For each $u \in \bdy B_{\alpha r}(z)$ and each $v\in \bdy B_{r}(z)$ such that the $D_h$-geodesic from $u$ to $v$ is unique and is contained in $\ol{\BB A_{\alpha r , r}(z)}$, we have $\wt D_h(u,v) \leq C' D_h(u,v)$.
By considering the times when a $D_h$-geodesic between two fixed points $\BB z , \BB w \in \BB C$ crosses the annulus $\BB A_{\alpha r , r}(z)$ for such a $z$ and $r$, we will be able to show that $\wt D_h(\BB z , \BB w) \leq C'' D_h(\BB z,\BB w)$ for a suitable constant $C'' \in (C' , C_*)$. 
Applying this to an appropriate $\beta$-dependent collection of pairs of points $(\BB z, \BB w)$ will show that $\BB P[\ol G_{\BB r}(C'',\beta) ]  < \beta$. 
The reason why we need to make $\alpha$ close to 1 is to ensure that the events we consider depend on $h$ in a sufficiently ``local" manner (see the discussion just after the definition of $\mathsf E_r(z)$ below). 

Let us now define the events to which we will apply Lemma~\ref{lem-annulus-iterate}. 
For $z \in \BB C$, $r > 0$, and parameters $\alpha \in (1/2,1)$, $ A > 1$ and $C' \in (0,C_*)$, let $\mathsf E_r(z) = \mathsf E_r(z; \alpha,A,C')$ be the event that the following is true.
\begin{enumerate}
\item \emph{(Comparison of $D_h$ and $\wt D_h$)} For each $u \in \bdy B_{\alpha r}(z)$ and each $v\in \bdy B_{r}(z)$ such that the $D_h$-geodesic from $u$ to $v$ is unique and is contained in $\ol{\BB A_{\alpha r , r}(z)}$, we have $\wt D_h(u,v) \leq C' D_h(u,v)$. \label{item-attained-dist}
\item \emph{(Lower bound for paths in $\BB A_{\alpha r , r}(z)$)} If $u \in \bdy B_{\alpha r}(z)$ and $v\in \bdy B_r(z)$ are such that either $D_h(u,v) > D_h(u , \bdy\BB A_{r/2,2r}(z))$ or $  \wt D_h(u,v) > \wt D_h(u , \bdy\BB A_{r/2,2r}(z))$, then each path from $u$ to $v$ which stays in $\ol{\BB A_{\alpha r , r}(z)}$ has $D_h$-length strictly larger than $D_h\left( u , v ; \BB A_{r/2,2r}(z) \right)$. \label{item-attained-long}
\item \emph{(Distance around $\BB A_{\alpha r , r}(z)$)} There is a path in $\BB A_{\alpha r , r}(z)$ which disconnects the inner and outer boundaries of $\BB A_{\alpha r , r}(z)$ and has $D_h$-length at most $A D_h\left(\bdy B_{\alpha r}(z) , \bdy B_{  r}(z) \right)$. \label{item-attained-around}
\end{enumerate}
Condition~\ref{item-attained-dist} is the main point of the event $\mathsf E_r(z)$, as discussed just above. 
The purpose of condition~\ref{item-attained-long} is to ensure that $\mathsf E_r(z)$ is determined by $h|_{\BB A_{r/2,2r}(z)}$. Without this condition, we would not necessarily be able to tell whether a path in $\ol{\BB A_{\alpha r ,r}(z)}$ is a $D_h$-geodesic without seeing the field outside of $\BB A_{r/2,2r}(z)$ (see Lemma~\ref{lem-attained-msrble}). 
The purpose of condition~\ref{item-attained-around} is as follows. If a $D_h$-geodesic between two points outside of $B_r(z)$ enters $B_{\alpha r}(z)$, then it must cross the path from condition~\ref{item-attained-around} twice. This means that it can spend at most $A D_h\left(\bdy B_{\alpha r}(z) , \bdy B_{  r}(z) \right)$ units of time in $B_{\alpha r}(z)$ since otherwise the path from condition~\ref{item-attained-around} would provide a shortcut, which would contradict the definition of a geodesic. If we assume~\eqref{item-attained-bad}, this fact will eventually allow us to force a $D_h$-geodesic to spend a positive fraction of its time tracing segments between points $u,v$ with $\wt D_h(u,v) \leq C' D_h(u,v)$. 
   
We want to use Lemma~\ref{lem-annulus-iterate} to argue that if~\eqref{item-attained-bad} holds, then with high probability there are many values of $z\in\BB C$ such that $\mathsf E_r(z)$ occurs for some $r\in [\ep^{1+\nu} \BB r , \ep \BB r]$. 
We first check the measurability condition in Lemma~\ref{lem-annulus-iterate}

\begin{lem} \label{lem-attained-msrble}
For each $z\in\BB C$ and $r>0$,  
\eqb \label{eqn-attained-msrble}
\mathsf E_r(z) \in \sigma\left(  (h-h_{4r}(z)) |_{   \BB A_{r/2,2r}(z)  } \right) .
\eqe
\end{lem}
\begin{proof}
By Axiom~\ref{item-metric-f} (Weyl scaling) subtracting $h_{4r}(0)$ from $h$ results in scaling $D_h$ and $\wt D_h$ by the same factor, so does not affect the occurrence of $\mathsf E_r(z)$.  
Hence it suffices to prove~\eqref{eqn-attained-msrble} with $h|_{\BB A_{r/2,2r}(z)}$ in place of $(h-h_{4r}(z)) |_{   \BB A_{r/2,2r}(z)  }$. 
From Axiom~\ref{item-metric-f}, it is obvious that condition~\ref{item-attained-around} in the definition of $\mathsf E_r(z)$ (distance around $\BB A_{\alpha r , r}(z)$) is determined by $h |_{   \BB A_{r/2,2r}(z)  }$. 

For $u \in \bdy B_{\alpha r}(z)$ and $v\in \bdy B_{r}(z)$, we can determine whether $D_h(u,v) > D_h(u , \bdy\BB A_{r/2,2r}(z))$ from the internal metric $D_h\left(\cdot,\cdot ;\BB A_{r/2,2r}(z)  \right)$: indeed, $D_h(u , \bdy\BB A_{r/2,2r}(z))$ is clearly determined by this internal metric and $D_h(u,v) \leq D_h(u , \bdy\BB A_{r/2,2r}(z))$ if and only if $v$ is contained in the $D_h$-ball of radius $ D_h(u , \bdy\BB A_{r/2,2r}(z))$ centered at $v$, which is contained in $\ol{\BB A_{r/2,2r}(z)}$.  
Similar considerations hold with $\wt D_h$ in place of $D_h$. 
Hence condition~\ref{item-attained-long} in the definition of $\mathsf E_r(z)$ (lower bound for paths in $\BB A_{\alpha r , r}(z)$) is determined by $h|_{\BB A_{r/2,2r}(z)}$. 

If $P$ is a path from $ u \in \bdy B_{\alpha r}(z)$ to $v\in \bdy B_r(z)$ which stays in $\ol{\BB A_{\alpha r ,r}(z)}$, then $P$ is a $D_h$-geodesic if and only if $\op{len}(P ; D_h) = D_h(u,v)$. 
Hence if condition~\ref{item-attained-long} holds, then $P$ cannot be a $D_h$-geodesic unless $D_h(u,v) \leq D_h(u , \bdy\BB A_{r/2,2r}(z) )$ and $ \wt D_h(u,v) \leq \wt D_h(u , \bdy\BB A_{r/2,2r}(z) )$ (note that $D_h(u,v ; \BB A_{r/2,2r}(z)) \geq D_h(u,v)$), in which case we can tell whether $P$ is a $D_h$-geodesic from the restriction of $h$ to the $D_h$-metric ball of radius $D_h(u , \bdy\BB A_{r/2,2r}(z) )$ centered at $u$, which in turn is determined by $h|_{\BB A_{r/2,2r}(z)}$. 
Furthermore, on the event that $D_h(u,v) \leq D_h(u , \bdy\BB A_{r/2,2r}(z) )$ and $ \wt D_h(u,v) \leq \wt D_h(u , \bdy\BB A_{r/2,2r}(z) )$, both $D_h(u,v)$ and $\wt D_h(u,v)$ are determined by $h|_{\BB A_{r/2,2r}(z)}$. 
Therefore, the intersection of conditions~\ref{item-attained-dist} (comparison of $D_h$ and $\wt D_h$) and~\ref{item-attained-long} in the definition of $\mathsf E_r(z)$ is determined by $h|_{\BB A_{r/2,2r}(z)}$.  
Hence we have proven~\eqref{eqn-attained-msrble}.  
\end{proof}

We now show that~\eqref{item-attained-bad} implies a lower bound for $\BB P[\mathsf E_r(z)]$ for some values of $r \in [\ep^{1+\nu} \BB r , \ep \BB r]$.

\begin{lem} \label{lem-shorter-annulus}
For each $0 < \mu < \nu < 1$ and each $q > 0$, there exists $\alpha_* \in (1/2,1)$ and $p\in (0,1)$  depending only on $q , \mu,\nu $ such that for each $\alpha \in [\alpha_*,1)$, there exists $A = A(\alpha,q,\mu,\nu)  > 1$ such that the following is true for each $C'\in (0,C_*)$. 
If $\BB r > 0$ and $\ep \in (0,1)$ such that~\eqref{item-attained-bad} holds for the above choice of $p,\alpha , C'$, then  
\eqb \label{eqn-shorter-annulus}
\BB P\left[ \text{$\mathsf E_r(z)$ occurs for at least one $r\in [\ep^{1+\mu} \BB r , \ep \BB r] \cap \{8^{-k} \BB r :k\in\BB N\}$} \right] \geq 1 -  O_\ep(\ep^q ) ,\quad\forall z \in\BB C  ,
\eqe 
at a rate which is uniform over the choices of $z$ and $\BB r$. 
\end{lem} 
\begin{proof}  
Assume~\eqref{item-attained-bad} is satisfied for some choice of $\BB r ,\ep , p , \alpha,C'$ and let $r_1,\dots,r_K \in [\ep^{1+\mu} \BB r , \ep \BB r] \cap \{8^{-k} \BB r :k\in\BB N\}$ be the values of $r$ from~\eqref{item-attained-bad}, enumerated in decreasing order.
Note that $K\geq (\nu-\mu) \log_8\ep^{-1}$ by assumption.   
By Lemma~\ref{lem-attained-msrble}, we can apply Lemma~\ref{lem-annulus-iterate} to find that there exists $\wt p = \wt p(q,\mu , \nu) \in (0,1)$ such that if 
\eqb \label{eqn-shorter-annulus-show}
\BB P[\mathsf E_{r_k}(z)] \geq \wt p ,\quad \forall z\in\BB C ,  \quad \forall k\in [1,K ]_{\BB Z} ,
\eqe
then~\eqref{eqn-shorter-annulus} holds. It therefore suffices to choose $p$, $\alpha_*$, and $A$ in an appropriate manner depending on $\wt p$ so that if~\eqref{item-attained-bad} holds, then~\eqref{eqn-shorter-annulus-show} holds.

By tightness across scales (Axiom~\ref{item-metric-coord}), we can find $S > s > 0$ depending on $\wt p$ such that for each $z\in\BB C$ and $r>0$, it holds with probability at least $1 - (1-\wt p)/4$ that  
\eqb
  D_h\left( \bdy B_r(z) , \bdy \BB A_{r/2 , 2r}(z)  \right)   \geq  s \frk c_r e^{\xi h_r(z)} \quad \text{and} \quad 
\sup_{u,v\in \BB A_{3 r /4 , r}(z)} D_h\left( u ,v ; \BB A_{r/2,2r}(z)\right) \leq S \frk c_r e^{\xi h_r(z)} ,
\eqe 
and the same is true with $\wt D_h$ in place of $D_h$. 
Since $ \BB A_{\alpha r , r}(z) \subset  \BB A_{3 r /4 , r}(z) $ for any choice of $\alpha \in [3/4,1)$, Lemma~\ref{lem-attained-long} with the above choice of $s$ and $S$ gives an $\alpha_* \in [3/4,1)$ depending on $\wt p$ such that for each $\alpha\in [\alpha_*,1)$, $z \in \BB C$, and $r>0$, condition~\ref{item-attained-long} (lower bound for paths in $\BB A_{\alpha r , r}(z)$) in the definition of $\mathsf E_r(z)$ holds with probability at least $1-(1-\wt p)/3$.
 
Now suppose $\alpha \in [\alpha_*,1)$. We can again apply Axiom~\ref{item-metric-coord} (tightness across scales) to find that there exists $A > 1$ depending on $\alpha$ and $\wt p$ such that for each $z\in\BB C$ and $r>0$, condition~\ref{item-attained-around} (distance around $\BB A_{\alpha r , r}(z)$) in the definition of $\mathsf E_r(z)$ occurs with probability at least $1 - (1 - \wt p)/3$.
 
If~\eqref{item-attained-bad} holds for the above choice of $\alpha$ and with $p < (1-\wt p)/3$, then for each $z\in\BB C$ and each $k\in [1,K ]_{\BB Z}$, condition~\ref{item-attained-dist} (comparison of $D_h$ and $\wt D_h$) in the definition of $\mathsf E_{r_k }(z)$ holds with probability at least $1- (1-\wt p)/3$. Combining the three preceding paragraphs shows that~\eqref{eqn-shorter-annulus-show} holds.  
\end{proof}

\begin{lem} \label{lem-shorter-annulus-all}
There is a $q>1$ depending only on $\mu,\nu $ such that if $p,\alpha_*$, $\alpha \in [\alpha_*,1)$, and $A$ is chosen as in Lemma~\ref{lem-shorter-annulus} for this choice of $q$, then the following is true for each $C'\in (0,C_*)$. 
If~\eqref{item-attained-bad} holds for some $\BB r > 0$ and $\ep \in (0,1)$ and for this choice of $p,\alpha,C'$, then for each open set $U\subset\BB C$, it holds with probability tending to 1 as $\ep\rta 0$ (at a rate which is uniform in $\BB r$) that for $z\in \BB r U$, there exists $r   \in [\ep^{1+\nu} \BB r , \ep \BB r]\cap\{8^{-k} \BB r :k\in\BB N\}$ and $w\in \left(\frac{\ep^{1+\nu} \BB r}{100}  \BB Z^2 \right) \cap (\BB r U)$ such that $z\in B_{r/2}(w)$ and $\mathsf E_r(w)  $ occurs. 
\end{lem}
\begin{proof} 
Upon choosing $q$ sufficiently large, this follows from Lemma~\ref{lem-shorter-annulus} and a union bound over all $w \in  \left(\frac{\ep^{1+\nu} \BB r}{100}  \BB Z^2 \right) \cap (\BB r U)$. 
\end{proof}

\begin{figure}[t!]
 \begin{center}
\includegraphics[scale=1]{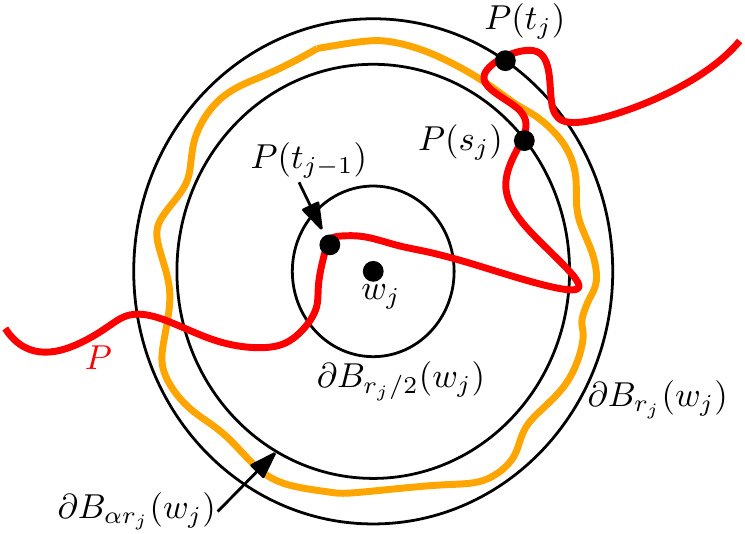}
\vspace{-0.01\textheight}
\caption{Illustration of the proof of Proposition~\ref{prop-attained-bad}. The $D_h$-geodesic $P$ from $\BB z$ to $\BB w$ along with one of the balls $B_{r_j}(w_j)$ hit by $P$ for which $\mathsf E_{r_j}(w_j)$ occurs are shown. The time $t_j$ is the first time that $P$ exits $B_{r_j}(w_j)$ after time $t_{j-1}$ and the time $s_j$ is the last time before $t_j$ at which $P$ hits $\bdy B_{\alpha r_j}(w_j)$. Condition~\ref{item-attained-dist} in the definition of $\mathsf E_{r_j}(w_j)$ shows that $\wt D_h(P(s_j) , P(t_j)) \leq C' (t_j - s_j)$. The orange path comes from condition~\ref{item-attained-around} in the definition of $\mathsf E_{r_j}(w_j)$, and its $D_h$-length is at most $A D_h(\bdy B_{\alpha r_j}(w_j) , \bdy B_{r_j}(w_j)) \leq A (t_j - s_j)$. Since $P$ crosses this orange path both before time $t_{j-1}$ and after time $s_j$ and $P$ is a $D_h$-geodesic, we have that $s_j - t_{j-1} \leq A (t_j - s_j)$. This shows that the intervals $[s_j , t_j]$ occupy a uniformly positive fraction of the total $D_h$-length of $P$, which in turn allows us to show that $\wt D_h(\BB z , \BB w) \leq C'' D_h(\BB z ,\BB w)$ for a constant $C'' \in (C' , C_*)$ depending only on $C' , A$. 
}\label{fig-attained}
\end{center}
\vspace{-1em}
\end{figure} 

\begin{proof}[Proof of Proposition~\ref{prop-attained-bad}]
See Figure~\ref{fig-attained} for an illustration of the proof.
\medskip

\noindent\textit{Step 1: setup.}
Let $p, \alpha_*$, $\alpha\in [\alpha_*,1)$, and $A > 1$ be chosen as in Lemma~\ref{lem-shorter-annulus-all}. Also fix
\eqb \label{eqn-C''-choice}
C'' \in \left( C'  + \frac{A}{ A+1 }(C_* - C')  , C_* \right) ,
\eqe 
and note that we can choose $C''$ in a manner depending only on $\alpha,C',\mu,\nu$ (since $A$ depends only on $\alpha,\mu,\nu$).

We will show that there exists $\ep_0 = \ep_0(\beta ,\alpha,C',\mu,\nu) > 0$ such that if $\BB r > 0$, $\ep \in (0,\ep_0]$, and~\eqref{item-attained-bad} holds for these values of $\BB r , \ep,p,\alpha$, then with probability greater than $1  - \beta$,  
\eqb \label{eqn-lower-ratio}
\wt D_h(\BB z , \BB w) \leq C'' D_h(\BB z,\BB w) \quad \text{$\forall \BB z,\BB w \in B_{\BB r }(0)$ with $|\BB z-\BB w| \geq \beta \BB r$}.
\eqe
In other words, $\BB P[\ol G_{\BB r}(C'',\beta)^c] > 1 - \beta$, as required. 

By Axiom~\ref{item-metric-coord} (tightness across scales), there is some large bounded open set $U\subset \BB C$ depending only on $\beta$ such that for each $\BB r >0$, it holds with probability at least $1-\beta/2$, the $D_h$-diameter of $B_{\BB r}(0)$ is smaller than the $D_h$-distance from $B_{\BB r}(0)$ to $\bdy(\BB r U)$, in which case every $D_h$-geodesic between points of $B_{\BB r}(0)$ is contained in $\BB r U$. 
Henceforth fix such a choice of $U$.
Let $F_{\BB r}^\ep$ be the event that every $D_h$-geodesic between points of $B_{\BB r}(0)$ is contained in $\BB r U$ and the event of Lemma~\ref{lem-shorter-annulus-all} with the above choices of $\alpha,A,C',$ and $U$, so that $\BB P[F^\ep] \geq 1 - \beta / 2 - o_\ep(1)$, uniformly in $\BB r$, under the assumption~\eqref{item-attained-bad}. 
\medskip

\noindent\textit{Step 2: covering a $D_h$-geodesic with paths of short $\wt D_h$-length. }
To prove~\eqref{eqn-lower-ratio}, we consider points $\BB z,\BB w \in B_{\BB r}(0) \cap \BB Q^2$ with $|\BB z-\BB w| \geq \beta \BB r$ and let $P : [0,D_h(\BB z , \BB w )] \rta \BB C$ be the (a.s.\ unique) $D_h$-geodesic from $\BB z$ to $\BB w$. 
Let $t_0  = 0$ and inductively let $t_j$ for $j\in\BB N$ be the smallest time $t \geq t_{j-1}$ at which $P$ exits a Euclidean ball of the form $B_{r}(w)$ for $w\in \left(\frac{\ep^{1+\nu} \BB r}{100} \BB Z^2 \right) \cap (\BB r U)$ and $r\in [\ep^{1+\nu} \BB r , \ep \BB r] \cap\{8^{-k} \BB r :k\in\BB N\}$ such that $P(t_{j-1}) \in B_{r/2}(w)$ and $\mathsf E_r(w)$ occurs; or let $t_j = D_h(\BB z , \BB w )$ if no such $t$ exists.  
If $t_j  < D_h(\BB z , \BB w )$, let $w_j$ and $r_j$ be the corresponding values of $w$ and $r$. 
Also let $s_j$ be the last time before $t_j$ at which $P$ hits $\bdy B_{\alpha r_j}(w)$, so that $s_j \in [t_{j-1} , t_j]$ and $P([s_j , t_j]) \subset \ol{ \BB A_{\alpha r_j , r_j}(w_j)}$.
Finally, define
\eqb
\ul J := \max\left\{ j\in \BB N : |\BB z - P(t_{j-1} )| <  2 \ep \BB r \right\}
\quad \text{and} \quad
\ol J := \min\left\{j\in\BB N : |\BB w - P(t_{j+1} )|  <  2 \ep \BB r \right\} .
\eqe
The reason for the definitions of $\ul J$ and $\ol J$ is that $\BB z , \BB w \notin B_{ r_j}(w_j)$ for $j \in [\ul J , \ol J]_{\BB Z}$ (since $r_j \leq \ep \BB r$ and $P(t_j) \in B_{r_j}(w_j)$). 
By the definition of $F_{\BB r}^\ep$, on this event we have $t_j  < D_h(\BB z,\BB w)$ and $|P(t_{j-1}) - P(t_j)| \leq 2 \ep \BB r$ whenever $|\BB w - P(t_{j-1})| \geq  \ep \BB r$. 
Therefore, on $F_{\BB r}^\ep$,  
\eqb \label{eqn-endpoint-contain}
P(t_{\ul J}) \in B_{4\ep \BB r}(\BB z) \quad \text{and} \quad P(t_{\ol J}) \in B_{4\ep\BB r}(\BB w) .
\eqe

Since $P$ is a $D_h$-geodesic, for $j\in [\ul J , \ol J]_{\BB Z}$ also $P|_{[s_j , t_j]}$ is a $D_h$-geodesic from $P(s_j) \in \bdy\mcl B_{\alpha r_j}(w_j)$ to $P(t_j) \in \bdy B_{r_j}(w_j)$ and by definition this $D_h$-geodesic stays in $\ol{\BB A_{\alpha r_j , r_j}(w_j)}$. 
Moreover, $P|_{[s_j,t_j]}$ is the only $D_h$-geodesic from $P(s_j)$ to $P(t_j)$ since otherwise we could re-route $P$ along another such $D_h$-geodesic to contradict the uniqueness of the $D_h$-geodesic from $\BB z$ to $\BB w$.  

Combining this with condition~\ref{item-attained-dist} in the definition of $\mathsf E_{r_j}(w_j)$ (comparison of $D_h$ and $\wt D_h$), applied with $u = P(s_j)$ and $v = P(t_j)$, and the definition~\eqref{eqn-max-min-def} of $C_*$, we find that  
\eqb \label{eqn-increment-tilde-D}
\wt D_h\left( P(s_j) , P(t_j)  \right) \leq C' (t_j - s_j) 
\quad \text{and} \quad
\wt D_h\left( P(t_{j-1}) , P(s_j) \right) \leq C_* (s_j - t_{j-1}  ) ,\quad \forall j \in [\ul J ,\ol J]_{\BB Z}.
\eqe 

We will now argue that $ s_j - t_{j-1}$ is not too much larger than $t_j - s_j$. 
If $j \in [\ul J , \ol J]_{\BB Z}$, then since $r_j \leq \ep \BB r$ and $|P(t_j) -\BB z| \wedge |P(t_j) - \BB w| \geq 2\ep \BB r$, the $D_h$-geodesic $P$ must cross the annulus $\BB A_{\alpha r_j,r_j}(w_j)$ at least once before time $t_{j-1}$ and at least once after time $s_j$. 
By condition~\ref{item-attained-around} in the definition of $\mathsf E_{r_j}(w_j)$, there is a path disconnecting the inner and outer boundaries of this annulus with $D_h$-length at most $A D_h\left(\bdy B_{\alpha r_j}(w_j) , \bdy B_{  r_j}(w_j) \right)$. 
The geodesic $P$ must hit this path at least once before time $t_{j-1}$ and at least once after time $s_j$. 
Since $P$ is a geodesic and $P(s_j) \in \bdy B_{\alpha r_j}(w_j)$, $P(t_j) \in \bdy B_{ r_j}(w_j)$, it follows that 
\eqbn 
s_j - t_{j-1}  \leq A D_h\left(\bdy B_{\alpha r_j}(w_j) , \bdy B_{  r_j}(w_j) \right) \leq A (t_j - s_j) .
\eqen
Adding $A (s_j - t_{j-1})$ to both sides of this inequality, then dividing by $A+1$, gives
\eqb \label{eqn-increment-compare}
s_j - t_{j-1} \leq  \frac{A}{A+1} (t_j - t_{j-1}) .
\eqe
\medskip

\noindent\textit{Step 3: upper bound for $\wt D_h$. }
By combining the above relations, we get that on $F_{\BB r}^\ep$, 
\allb \label{eqn-attained-sum}
\wt D_h\left( B_{4 \ep \BB r}(\BB z) , B_{4 \ep \BB r}(\BB w)  \right) 
&\leq \sum_{j=\ul J+1}^{\ol J} \left( \wt D_h\left( P(t_{j-1}) , P(s_j) \right) +   \wt D_h\left( P(s_j) , P(t_j)  \right)  \right) \quad  \text{(by \eqref{eqn-endpoint-contain})} \notag \\
&\leq \sum_{j=\ul J+1}^{\ol J} \left( C_*( s_j - t_{j-1})  +   C' (t_j - s_j)  \right) \quad  \text{(by \eqref{eqn-increment-tilde-D})} \notag \\
&= \sum_{j=\ul J+1}^{\ol J} \left(    C' (t_j - t_{j-1} )  + (C_* - C') (s_j - t_{j-1}) \right) \notag \\
&\leq \left( C'  + \frac{A}{A+1}(C_* - C') \right) \sum_{j=\ul J+1}^{\ol J}  (t_j - t_{j-1} )  \quad \text{(by \eqref{eqn-increment-compare})} \notag \\
&\leq \left( C'  + \frac{A}{A+1}(C_* - C') \right) D_h(\BB z,\BB w) .
\alle

By~\eqref{eqn-C''-choice}, Axiom~\ref{item-metric-coord} (tightness across scales) for $D$ and $\wt D$, and the triangle inequality, it holds with probability tending to 1 as $\ep\rta 0$, uniformly in $r$, that
\eqb \label{eqn-sup-distance}
\left| \wt D_h(\BB z,\BB w) - \wt D_h\left( B_{4 \ep \BB r}(\BB z) , B_{4 \ep \BB r}(\BB w)  \right)  \right| \leq \frac{1}{100} \left(  C'' - \left( C'  + \frac{A}{A+1}(C_* - C')     \right) \right) D_h(\BB z,\BB w) 
\eqe 
simultaneously for all $\BB z,\BB w\in B_{  \BB r }(0)$ with $|\BB z-\BB w| \geq \beta \BB r$. 
By combining this with~\eqref{eqn-attained-sum} and recalling that $\BB P[F_{\BB r}^\ep] = 1-\beta/2 - o_\ep(1)$ uniformly in $\BB r$ if~\eqref{item-attained-bad} holds,
we get that if $\ep_0$ is chosen to be sufficiently small, in a manner which does not depend on $\BB r$, then if~\eqref{item-attained-bad} holds for $\BB r > 0$ and $\ep \in (0 , \ep_0]$, then it holds with probability at least $1-\beta$ that~\eqref{eqn-lower-ratio} holds simultaneously for each  $\BB z,\BB w\in B_{  \BB r}(0) \cap \BB Q^2$ with $|\BB z-\BB w| \geq \beta \BB r$. By the continuity of $D_h$ and $\wt D_h$, we can remove the requirement that $\BB z,\BB w\in\BB Q^2$ (which was only used to get the uniqueness of the $D_h$-geodesic from $\BB z$ to $\BB w$).  
\end{proof}

\section{Independence along a geodesic}
\label{sec-geodesic-iterate}

\newcommand{\confarcs}{{\mathcal I}}
\newcommand{\confpts}{\mathrm{Conf}}
\newcommand{\arcendpts}{\mathrm{EndPts}}
\newcommand{\stabE}{\mathrm{Stab}}

Let $h$ be a whole-plane GFF and let $D$ be a weak $\gamma$-LQG metric. The goal of this section is to prove the following general ``local independence" type result for events depending on a small segment of a $D_h$-geodesic.  We will first state a simplified version of our result which is easier to parse (Theorem~\ref{thm-geo-iterate0}), then state the full version (Theorem~\ref{thm-geo-iterate}). 

\begin{thm} \label{thm-geo-iterate0} 
Suppose we are given events $\frk E_r^{\BB z , \BB w}(z) \in \sigma(h)$ for $z\in\BB C$, $r > 0$, and $\BB z  ,\BB w \in \BB C$ and a deterministic constant $\Lambda > 1$ which satisfy the following properties, where here $P = P^{\BB z,\BB w}$ denotes the (a.s.\ unique) $D_h$-geodesic from $\BB z$ to $\BB w$. 
\begin{enumerate}  
\item (Measurability) The event $\frk E_r^{\BB z,\BB w}(z)$ is determined by $h|_{ B_r(z)}$ and the geodesic $P $ stopped at the last time it exists $B_r(z)$. \label{item-geo-iterate-msrble0} 
\item (Lower bound for $\BB P[\frk E_r^{\BB z, \BB w}(z)]$) If $\BB z,\BB w\in\BB C\setminus B_r(z)$, then a.s.\ \label{item-geo-iterate-compare0}
\allb \label{eqn-geo-iterate-compare0}
 \BB P\left[ \frk E_r^{\BB z, \BB w}(z)    \,\big|\, h|_{\BB C\setminus B_{  r}(z)} ,  \{P  \cap B_{ r}(z)\not=\emptyset\} \right] \geq \Lambda^{-1} .
\alle
\end{enumerate}
For each $\nu\in (0,1)$, $q > 0$, $\ell \in (0,1)$, and bounded open set $U\subset\BB C$, it holds with probability tending to 1 as $\ep \rta 0$, at a rate depending only on $U,q,\ell,\Lambda$, that for each $\BB z , \BB w \in \left( \ep^q  \BB Z^2 \right) \cap U$ with $|\BB z-\BB w|\geq \ell  $, there exists $z\in\BB C$ and $r\in [\ep^{1+\nu} , \ep]$ such that $ P^{\BB z,\BB w} \cap B_{ r}(z)\not=\emptyset$ and $\frk E_r^{\BB z,\BB w}(z)$ occurs. 
\end{thm}

We think of the parameter $q > 0$ in Theorem~\ref{thm-geo-iterate0} as large, so the conclusion of the theorem holds for all pairs $(\BB z , \BB w)$ in a fine mesh of $U$. 

Intuitively, the reason why Theorem~\ref{thm-geo-iterate0} is true is as follows.
The geodesic segments $P\cap B_r(z)$ and $P\cap B_r(w)$ are approximately independent from one another when $|z-w| $ is much larger than~$r$.  
When $r$ is small, we can cover $P$ by a large number of balls $B_r(z)$ whose corresponding center points $z$ lie at Euclidean distance much further than $r$ from one another. 
Using~\eqref{eqn-geo-iterate-compare0} and a general concentration inequality for independent random variables, one gets that for each fixed pair $(\BB z,\BB w)$, with high probability there exists $z\in\BB C$ such that $ P^{\BB z,\BB w} \cap B_{ r}(z)\not=\emptyset$ and $\frk E_r^{\BB z,\BB w}(z)$ occurs. 
One then takes a union bound over all pairs $\BB z , \BB w \in \left( r^q  \BB Z^2 \right) \cap U$. 

The above heuristic is not quite right since $D_h$-geodesics do not depend locally on the field, so $P\cap B_r(z)$ and $P\cap B_r(w)$ are not approximately independent when $|z-w|$ is much greater than $r$. Indeed, it is possible that changing what happens in $B_r(z)$ could affect the behavior of $P$ macroscopically even when $r$ is very small.
As a substitute for this lack of long-range independence, we will use the confluence of geodesics results from~\cite{gm-confluence}, as discussed in Section~\ref{sec-outline}, and only make changes to the field at places where the geodesics are ``stable'' in the sense that a microscopic change does not lead to macroscopic changes to $P$. 
The reason why we only get a statement which holds with probability tending to $1$ as $r\rta 0$ at the end of Theorem~\ref{thm-geo-iterate0} is that we need to truncate on a global regularity event in order to make confluence hold with high probability.

We will actually prove (and use) a more general version of Theorem~\ref{thm-geo-iterate0} which differs from Theorem~\ref{thm-geo-iterate0} in the following respects.
\begin{itemize}
\item We allow for more flexibility in the Euclidean radii involved in the various conditions, which is represented by constants $\{\lambda_i\}_{i =1,\dots,5}$ (for our particular application, the constants are chosen explicitly in~\eqref{eqn-constant-choice}).
\item We introduce events $E_r(z)$ which are determined by the restriction of $h$ to an annulus $\BB A_{\lambda_1 r , \lambda_4 r}(z)$ (for constants $\lambda_1 < \lambda_4$) and which are required to have probability close to 1. We replace~\eqref{eqn-geo-iterate-compare0} by a comparison between the conditional probabilities of $\frk E_r^{\BB z,\BB w}(z)$ and $E_r(z)$ given $h|_{\BB C\setminus B_{\lambda_3 r}(z)}$, for another constant $\lambda_3$. The occurrence of $E_r(z)$ can be thought of as the statement that ``$h|_{\BB A_{\lambda_1 r , \lambda_4 r}(z)}$ is sufficiently well behaved that $\frk E_r^{\BB z,\BB w}(z)$ has a chance to occur". 
\item We do not require our events to be defined for all $r>0$, but rather only for values of $r$ in a suitably ``dense" set $\mcl R \subset (0,\infty)$. The reason why we need to allow for this is that the results of Section~\ref{sec-attained} only hold for values of $r$ in a suitably dense set. 
\item We work with a given ``base scale" $\BB r > 0$ (e.g., we consider points in $\BB r U$ instead of in $U$) and we require our estimates to be uniform in the choice of $\BB r$.  The reason for this is that we have only assumed tightness across scales (Axiom~\ref{item-metric-coord}) instead of exact scale invariance.
\end{itemize}  

\begin{thm} \label{thm-geo-iterate}
There exists $\nu_*  \in (0,1)$ depending only on the choice of metric $D$ such that for each $0 < \mu < \nu \leq \nu_*$ and each $0 < \lambda_1 < \lambda_2 \leq \lambda_3 \leq \lambda_4 < \lambda_5 $, there exists $\BB p \in (0,1)$ such that the following is true. 
Suppose $\BB r > 0$ and we are given a small number $\ep_0  > 0$; a deterministic set of radii $\mcl R\subset (0,\ep_0]$; events $E_r(z) \in \sigma(h) $ for $z\in\BB C$ and $r\in\mcl R$; events $\frk E_r^{\BB z , \BB w}(z) \in \sigma(h)$ for $z\in\BB C$, $r\in\mcl R$, and $\BB z  ,\BB w \in \BB C$; and a deterministic constant $\Lambda > 1$ which satisfy the following properties. 
\begin{enumerate} 
\item (Density of $\mcl R$) For each $\ep \in (0,\ep_0]$, there exist $\lfloor \mu\log_8 \ep^{-1} \rfloor$ radii $r_1^\ep , \dots , r_{\lfloor \mu\log_8\ep^{-1} \rfloor}^\ep \in [\ep^{1+\nu} \BB r , \ep \BB r]\cap \mcl R$ such that $r_k^\ep /r_{k-1}^\ep \geq \lambda_4/\lambda_1$ for each $k=2,\dots,\lfloor \mu\log_8\ep^{-1} \rfloor$. \label{item-geo-iterate-dense}
\item (Measurability) For each $z\in\BB C$ and $r\in\mcl R$, $E_r(z)$ is determined by $(h - h_{\lambda_5 r}(z)) |_{\BB A_{\lambda_1 r , \lambda_4 r}(z)}$ for each $\BB z ,\BB w \in\BB C$, and $\frk E_r^{\BB z,\BB w}(z)$ is determined by $h|_{ B_{\lambda_4 r}(z) }$ and the (a.s.\ unique) $D_h$-geodesic from $\BB z$ to $\BB w$ stopped at the last time it exists $B_{\lambda_4 r}(z)$. \label{item-geo-iterate-msrble}
\item (Lower bound for $\BB P[E_r(z)]$) For each $z\in\BB C$ and $r\in\mcl R$, we have $\BB P[E_r(z)] \geq \BB p$. \label{item-geo-iterate-prob} 
\item (Comparison of $E_r(z)$ and $\frk E_r^{\BB z, \BB w}(z)$) Suppose $z\in\BB C$, $r\in\mcl R$, $\BB z , \BB w$ are distinct points of $\BB C\setminus B_{\lambda_4 r}(z)$, and $P = P^{\BB z , \BB w}$ is the $D_h$-geodesic from $\BB z$ to $\BB w$. Then a.s.\ \label{item-geo-iterate-compare}
\allb \label{eqn-geo-iterate-compare}
& \Lambda^{-1}  \BB P\left[ E_r(z) \cap \{P \cap B_{\lambda_2 r}(z)\not=\emptyset\} \,\big|\, h|_{\BB C\setminus B_{\lambda_3 r}(z)}   \right] \notag \\
&\qquad \leq \BB P\left[ \frk E_r^{\BB z, \BB w}(z) \cap \{P \cap B_{\lambda_2 r}(z)\not=\emptyset\} \,\big|\, h|_{\BB C\setminus B_{\lambda_3 r}(z)}   \right] .
\alle
\end{enumerate}
Under the above hypotheses, for each $q > 0$, $\ell \in (0,1)$, and bounded open set $U\subset\BB C$, it holds with probability tending to 1 as $\ep\rta 0$, at a rate depending only on $U,q,\ell,\mu,\nu , \{\lambda_i\}_{i=1,\dots,5} ,\ep_0,\Lambda$, that for each $\BB z , \BB w \in \left( \ep^q \BB r \BB Z^2 \right) \cap \left(\BB r U\right)$ with $|\BB z-\BB w|\geq \ell \BB r$, there exists $z\in\BB C$ and $r \in [\ep^{1+\nu} \BB r , \ep \BB r]$ such that $P^{\BB z , \BB w} \cap B_{\lambda_2 r}(z) \not=\emptyset$ and $\frk E_r^{\BB z,\BB w}(z)$ occurs. 
\end{thm}

Theorem~\ref{thm-geo-iterate0} is the special case of Theorem~\ref{thm-geo-iterate} where $\mcl R = (0,\infty)$; $\lambda_2 = \lambda_3    = \lambda_4 = 1$; $E_r(z)$ is the whole probability space; and $\BB r =1$.  
The parameter $\BB p$ in Theorem~\ref{thm-geo-iterate} will eventually be chosen to be sufficiently close to 1 that we can apply Lemma~\ref{lem-annulus-iterate} to cover a large region of space by balls $B_{\lambda_1 r}(z)$ for pairs $(z,r)$ such that $E_r(z)$ occurs (see Lemma~\ref{lem-reg-event-prob}). 
The events $E_r(z)$ and $\frk E_r^{\BB z,\BB w}(z)$ play very different roles in the statement of Theorem~\ref{thm-geo-iterate}. The event $\frk E_r^{\BB z,\BB w}(z)$ is the main event that we are interested in, and concerns a segment of the $D_h$-geodesic from $\BB z$ to $\BB w$. The event $E_r(z)$ is locally determined by $h$, has probability close to 1, and can be thought of as the event that the restriction of $h$ to the annulus $\BB A_{\lambda_1 r , \lambda_4 r}(z)$ is sufficiently regular that $\frk E_r^{\BB z,\BB w}(z)$ has a chance to occur.

The statement of Theorem~\ref{thm-geo-iterate} is easier to understand if one thinks of the particular setting in which we will apply it. 
Recall the optimal bi-Lipschitz constants from~\eqref{eqn-max-min-def}. 
For us, $E_r(z)$ will be the event that there exists a pair of points $u,v \in \BB A_{\lambda_1 r,\lambda_2 r}(z)$ at Euclidean distance of order $r$ from each other for which $\wt D_h(u,v) \leq c_2' D_h(u,v)$ for a constant $c_2' \in (c_* , C_*)$; and some regularity conditions hold which are needed to ensure that conditions~\ref{item-geo-iterate-msrble} and~\ref{item-geo-iterate-compare} in the theorem statement are satisfied.  
We will only be able to show that $\BB P[E_r(z)]$ is bounded below for a ``dense" set of scales $\mcl R$ as in condition~\ref{item-geo-iterate-dense} due to the results in Section~\ref{sec-attained}. 
The event $\frk E_r^{\BB z,\BB w}(z)$ will be the event that, roughly speaking, the $D_h$-geodesic $P^{\BB z,\BB w}$ gets close to $u,v$ and hence (by the triangle inequality) hits a pair of points $P(s),P(t)$ at Euclidean distance of order $r$ from each other for which $\wt D_h(P(s) , P(t) ) \leq c_2' D_h(P(s) , P(t) )$. 
More precisely, we will prove the following statement in Section~\ref{sec-geodesic-shortcut}. 

\begin{prop}
\label{prop-geo-event}
Assume (by way of eventual contradiction) that $c_* < C_*$. 
Let $0 < \mu < \nu \leq \nu_*$ and $c_* < c_1' < c_2' < C_*$. 
There exist universal constants $\{\lambda_i\}_{i=1,\dots,5} $ and parameters $b , \rho \in (0,1)$ depending only on $\mu,\nu$ such that the following is true. 
Let $\BB p$ be as in Theorem~\ref{thm-geo-iterate} for the above choice of $\mu,\nu,\{\lambda_i\}_{i=1,\dots,5}$ and let $c'' =c''(c_1',\mu,\nu) > c_*$ be as in Proposition~\ref{prop-attained-good'} with $c' = c_1'$. If $\beta\in (0,1)$ and $\BB r > 0$ are such that $\BB P[\ul G_{\BB r}(c'' , \beta)] \geq \beta$ (in the notation~\eqref{eqn-min-event}), then there exists $\ep_0 = \ep_0(\beta,c_1' , c_2',\mu,\nu) \in (0,1)$, a deterministic set of radii $\mcl R\subset (0,\ep_0]$, events $E_r(z)$ and $\frk E_r^{\BB z , \BB w}(z)$, and a deterministic constant $\Lambda = \Lambda(c_1' , c_2',\mu,\nu) > 1$ which satisfy the hypotheses of Theorem~\ref{thm-geo-iterate} with $\rho^{-1} \BB r$ in place of $\BB r$ and have the following additional property. 
Suppose $z \in \BB C$, $r  \in\mcl R$, and $\BB z , \BB w \in \BB C\setminus B_{\lambda_4 r}(z)$, and let $P = P^{\BB z,\BB w}$ be the $D_h$-geodesic from $\BB z$ to $\BB w$. If $\frk E_r^{\BB z , \BB w}(z)$ occurs, then there are times $0 < s < t <  |P|$ such that
\eqb \label{eqn-geo-event}
P([s,t]) \subset B_{\lambda_2 r}(z) ,\quad |P(s) - P(t)| \geq b r ,\quad \text{and} \quad 
\wt D_h(P(s) , P(t))  \leq c_2' D_h(P(s) , P(t)) .
\eqe
\end{prop}

Roughly speaking, Proposition~\ref{prop-geo-event} combined with Theorem~\ref{thm-geo-iterate} implies that the pairs of points $(u,v)$ such that $\wt D_h(u,v) \leq c_2' D_h(u,v)$ and $|u-v|$ is not too small are sufficiently dense that a typical $D_h$-geodesic is extremely likely to get close to such a pair of points. This will be applied in Section~\ref{sec-conclusion} to derive a contradiction to the definition~\eqref{eqn-max-min-def} of $C_*$ if we assume that $c_* < C_*$, and thereby to show that $c_* = C_*$.

\begin{remark} \label{remark-rho}
The reason for the parameter $\rho$ in Proposition~\ref{prop-geo-event} is as follows. If  $\BB P[\ul G_{\BB r}(c'' , \beta)] \geq \beta$, then Proposition~\ref{prop-attained-good'} gives us a parameter $p = p(\mu,\nu) \in (0,1)$ such that there are many values of $r \in [\ep^{1+\nu} \BB r, \ep\BB r]$ for which a certain event occurs with probability at least $p$. 
In Section~\ref{sec-geodesic-shortcut}, we will use the event of Proposition~\ref{prop-attained-good'} to build the event $E_r(z)$. 
In order to make $E_r(z)$ occur with probability at least $\BB p$ (which can be arbitrarily close to 1) instead of just probability $p$, we will consider lots of small Euclidean balls and argue (using Lemma~\ref{lem-spatial-ind}) that with probability at least $\BB p$ the event of Proposition~\ref{prop-attained-good'} occurs for at least one of these balls. In order to do this, we need the radius of the annulus involved in the definition of $E_r(z)$ to be a large deterministic constant factor times the radius of the balls involved in the event of Proposition~\ref{prop-attained-good'} (so that we can fit many such balls in the annulus). This factor is $\rho^{-1}$. 
\end{remark}

\subsection{Setup and outline}
\label{sec-iterate-setup}

Assume that we are in the setting of Theorem~\ref{thm-geo-iterate} for some $\BB r  >0$. 
To lighten notation, we will also impose the assumption that $\lambda_3 = 1$ (the proof when $\lambda_3\not=1$ is identical, just with extra factors of $\lambda_3$ in various subscripts). 
Let $U\subset \BB C$ be open and bounded and let $\ell > 0$, as in the conclusion of Theorem~\ref{thm-geo-iterate}. 
Also fix $\ep \in (0,\ep_0]$ and distinct points $\BB z , \BB w \in \BB r U$ with $|\BB z- \BB w|\geq 4 \ell \BB r$ (the reason for the factor of 4 here is to reduce factors of 4 elsewhere). Let
\eqb \label{eqn-iterate-geo-choice}
P  = P^{\BB z,\BB w} := \left(\text{$D_h$-geodesic from $\BB z$ to $\BB w$} \right) .
\eqe  
To lighten notation, throughout the rest of this section we will not include the parameters $\BB r , \ep , \BB z ,\BB w$ in the notation.
But, we will always require that all estimates are uniform in the choice of $\BB r$, $\BB z$, and~$\BB w$ (we will typically be sending $\ep\rta 0$). 
Since we will commonly be growing metric balls starting from $\BB z$, we also introduce the following abbreviations for $z\in\BB C$ and $r,s>0$: 
\eqb \label{eqn-iterate-abbrv}
\frk E_r(z) = \frk E_r^{\BB z,\BB w}(z) ,\quad
\mcl B_s^\bullet := \mcl B_s^\bullet(\BB z ; D_h) \quad\text{and} \quad
\tau_r := \tau_r(\BB z) = \inf\{s > 0 : \mcl B_s^\bullet \not\subset B_r(\BB z)\}, 
\eqe  
where here we recall that $\mcl B_s^\bullet(\BB z ;D_h)$ is the filled metric ball.

\begin{figure}[t!]
 \begin{center}
\includegraphics[scale=.6]{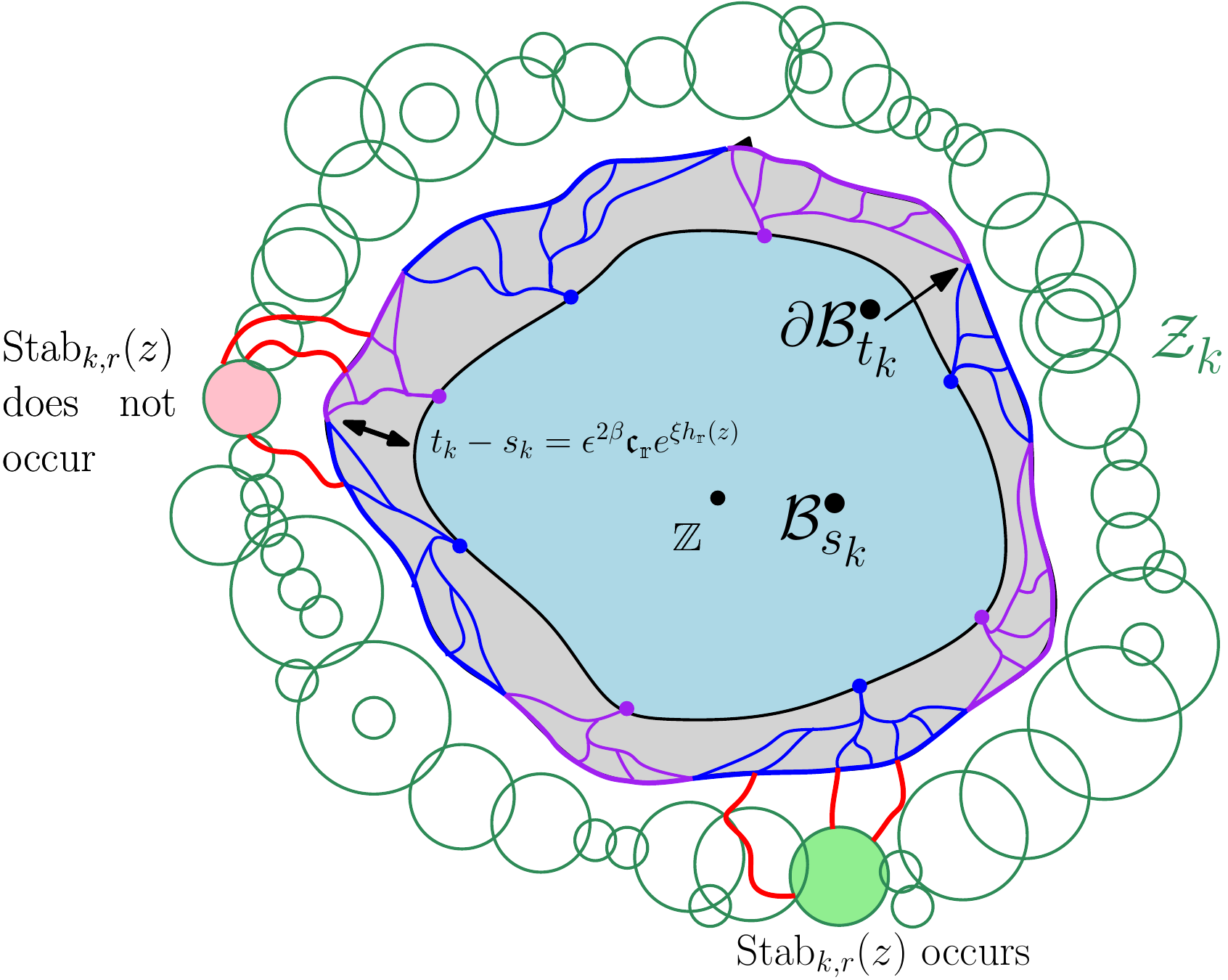}
\vspace{-0.01\textheight}
\caption{Illustration of the objects defined in Section~\ref{sec-iterate-setup}. The two filled LQG metric balls $\mcl B_{s_k}^\bullet \subset \mcl B_{t_k}^\bullet$ centered at $\BB z$ are shown, along with the set of points $\confpts_k \subset\bdy\mcl B_{s_k}^\bullet$ hit by leftmost $D_h$-geodesics from $\BB z$ to $\bdy\mcl B_{t_k}^\bullet $ (alternating blue and purple) and the set of arcs $\confarcs_k$ of $\bdy\mcl B_{t_k}^\bullet$ consisting of points whose leftmost $D_h$-geodesics hit the same point of $\confpts_k$. Several representative leftmost $D_h$-geodesics are shown for each such arc. We have also shown in green several of the balls $B_r(z)$ for $(z,r) \in \mcl Z_k$. Each such ball has radius in $[\ep^{1+\nu} \BB r , \ep \BB r]$ and its Euclidean distance from $\mcl B_{t_k}^\bullet$ is of order $\ep$. We have highlighted examples of one such ball $B_r(z)$ for which the event $\stabE_{k,r}(z)$ of~\eqref{eqn-ball-stab-event} occurs (light green), i.e., each of the red $D_h(\cdot,\cdot;\BB C\setminus \ol{B_r(z)})$-geodesics from $\BB z$ to points of $\bdy B_r(z)$ hit the same arc of $\confarcs_k$ (we have only shown the segments of these geodesics after they exit $\mcl B_{t_k}^\bullet$). We have also highlighted one ball for which $\stabE_{k,r}(z)$ does not occur (pink).  
}\label{fig-iterate-setup}
\end{center}
\vspace{-1em}
\end{figure} 

We now define several objects which we will work with throughout the rest of this section. 
See Figure~\ref{fig-iterate-setup} for an illustration. 
Fix $\beta \in (0,1)$ to be chosen later, in a manner depending only on $D$.
Define
\eqb \label{eqn-geo-iterate-times}
s_k := \tau_{\ell\BB r} + k \ep^\beta \frk c_{\BB r} e^{\xi h_{\BB r}(\BB z)}  \quad \text{and} \quad
t_k := s_k + \ep^{2\beta} \frk c_{\BB r} e^{\xi h_{\BB r}(\BB z)} \in [s_k ,s_{k+1}] , \quad \forall k \in \BB N_0 .
\eqe

By Theorem~\ref{thm-finite-geo0}, it is a.s.\ the case that for each $k\in \BB N_0$ there are only finitely many points of $\bdy\mcl B_{s_k }^\bullet$ which are hit by leftmost $D_h$-geodesics from $\BB z$ to $\bdy\mcl B_{t_k}^\bullet$. 
Let $\confpts_k \subset \bdy\mcl B_{s_{k} }^\bullet$ be the set of such points and let $\confarcs_{k}$ be the set of subsets of $\bdy\mcl B_{t_k}^\bullet$ of the form 
\eqb \label{eqn-arc-def}
\left\{y\in \bdy\mcl B_{t_k}^\bullet : \text{leftmost $D_h$-geodesic from $\BB z$ to $y$ passes through $x$}\right\} \quad \text{for} \quad x\in \confpts_{k}.
\eqe 
By~\cite[Lemma 2.7]{gm-confluence}, $\confarcs_{k}$ is a collection of disjoint arcs of $\bdy\mcl B_{t_k}^\bullet$ whose union is all of $\bdy\mcl B_{t_k}^\bullet$. We also note that by Axiom~\ref{item-metric-local} (locality), $\confarcs_k$ is determined by $\mcl B_{t_k}^\bullet$ and $h|_{\mcl B_{t_k}^\bullet}$.

For much of this section, we will work with the increasing filtration  
\eqb \label{eqn-nomax-filtration}
\mcl F_k := \sigma\left( \mcl B_{t_k}^\bullet ,h|_{\mcl B_{t_k}^\bullet} , P|_{[0,s_k]} \right) ,\quad\forall k\in\BB N_0. 
\eqe
Conditioning on all of $P|_{[0,s_k]}$ may seem rather extreme, but thanks to the confluence of geodesics this conditioning is a equivalent to a much tamer looking conditioning.

\begin{lem} \label{lem-geo-sigma-algebra}
We have the equivalent representation
\eqb \label{eqn-geo-sigma-algebra}
\mcl F_k  =  \sigma\left(\mcl B_{t_k}^\bullet , h|_{\mcl B_{t_k}^\bullet} , \text{arc of $\confarcs_k$ which contains $P(t_k)$}\right) .
\eqe
\end{lem}
\begin{proof}
On the event that the target point $\BB w$ of $P$ lies in $\mcl B_{t_k}^\bullet$, the path $P|_{[0,s_k]}$ is determined by $(\mcl B_{t_k}^\bullet ,h|_{\mcl B_{t_k}^\bullet}) $. 
On the complementary event $\{\BB w \notin \mcl B_{t_k}^\bullet\}$, we have $P(s_k ) \in \bdy\mcl B_{s_k}^\bullet$ and $P|_{[0,s_k]}$ is the a.s.\ unique $D_h(\cdot,\cdot; \mcl B_{t_k}^\bullet)$-geodesic from $\BB z$ to $P(s_k)$. 
Hence, on this event $P|_{[0,s_k]}$ is determined by $(\mcl B_{s_k }^\bullet ,h|_{\mcl B_{s_k }^\bullet} , P(s_k))$. 
Moreover, $P|_{[0,t_k]}$ is a.s.\ the unique (hence also leftmost) $D_h$-geodesic from $\BB z$ to $P(t_k)$, hence $P(s_{k} )$ is one of the points of $\confpts_{k}$. 
By the definition of $\confarcs_k$, this point is determined by which arc of $\confarcs_k$ contains $P(t_k)$.   
\end{proof}

We now introduce the set of Euclidean balls $B_r(z)$ which we will consider when trying to produce a ball for which $\frk E_r(z)$ occurs. 
With  $r_1^\ep , \dots , r_{\lfloor \mu\log_8\ep^{-1} \rfloor}^\ep \in [\ep^{1+\nu} \BB r , \ep \BB r]\cap \mcl R$ as in condition~\ref{item-geo-iterate-dense} from Theorem~\ref{thm-geo-iterate}, let $\mcl Z_k   $ for $k\in\BB N$ be the set of pairs $(z,r)$ such that 
\allb \label{eqn-good-annulus-set0}
 z\in \left( \frac{\lambda_1 \ep^{1+\nu} \BB r }{4} \BB Z^2 \right) \setminus \mcl B_{t_k}^\bullet ,
\quad r \in \left\{  r_1^\ep , \dots , r_{\lfloor \mu\log_8\ep^{-1} \rfloor}^\ep    \right\} ,
\quad \text{and} \quad
\op{dist}\left( z , \bdy\mcl B_{t_k}^\bullet \right) \in [ \lambda_4 \ep \BB r , 2 \lambda_4 \ep \BB r] .
\alle
Note that $\mcl Z_k \in \sigma\left( \mcl B_{t_k}^\bullet   \right)$. 

We want to say that with extremely high probability, there are many values of $k\in \BB N_0$ for which the event $\frk E_r(z)$ occurs for some $(z,r)\in\mcl Z_k$ such that $P\cap B_{\lambda_2 r}(z) \not=\emptyset$. 
We will do this by lower-bounding the conditional probability given $\mcl F_k$ that $\frk E_r(z)$ occurs for at least one $(z,r)\in\mcl Z_k$, then considering a polynomial (in $\ep$) number of values of~$k$ and applying a standard concentration inequality for binomial random variables. 

In order to say something useful about the conditional law given $\mcl F_k$ of what happens in one of the balls $B_r(z)$ for $(z,r)\in\mcl Z_k$, we need to know that making a small change to what happens in $B_r(z)$ does not affect which arc of $\confarcs_k$ contains $P(t_k)$. 
For $z\in\BB C$ and $r>0$, we therefore let $\stabE_{k,r}(z)$ be the event that $(z,r) \in \mcl Z_k$ and
\eqb \label{eqn-ball-stab-event}
 \text{Each $D_h(\cdot,\cdot ; \BB C\setminus\ol{ B_r(z)})$-geodesic from $\BB z$ to a point of $\bdy  B_r(z)$ hits $\bdy\mcl B_{t_k}^\bullet$ in the same arc of $\confarcs_k$}  .
\eqe
Here, by a $D_h(\cdot,\cdot ; \BB C\setminus \ol{B_r(z)} )$-geodesic from $\BB z$ to a point $x \in \bdy B_r(z)$ we mean a path from $\BB z$ to $x$ in $\BB C\setminus B_r(z)$ which has minimal $D_h$-length among all such paths and which does not hit $\bdy B_r(z)$ except at $x$. 
Note that such a geodesic need not exist for every point of $\bdy B_r(z)$. 
However, if $P$ is a $D_h$ geodesic started from $\BB z$ which enters $B_r(z)$, then $P$, stopped at the first time when it enters $B_r(z)$, is a $D_h(\cdot,\cdot ; \BB C\setminus \ol{B_r(z)} )$-geodesic from $\BB z$ to a point of $\bdy B_r(z)$. 

In Section~\ref{sec-stab}, we will use various quantitative results on confluence of geodesics from~\cite{gm-confluence} to show that with high probability $\stabE_{k,r}(z)$ occurs for most of the pairs $(z,r)\in\mcl Z_k$ such that $P$ enters $B_{\lambda_2 r}(z)$. 
The reason why the events $\stabE_{k,r}(z)$ are useful is the following lemma, which is used only in Section~\ref{sec-annulus-choose}. 

\begin{lem} \label{lem-ball-stab-msrble}
For each $z\in\BB C$, $r>0$, and $k\in\BB N_0$ the event $\stabE_{k,r}(z)$ of~\eqref{eqn-ball-stab-event} is a.s.\ determined by $h|_{\BB C\setminus B_r(z)}$.
Furthermore, on the event $\stabE_{k,r}(z) \cap \{P\cap B_r(z) \not=\emptyset\}$, both $(\mcl B_{t_k}^\bullet , h|_{\mcl B_{t_k}^\bullet})$ and the arc of $\confarcs_k$ which contains $P(t_k)$ are a.s.\ determined by $h|_{\BB C\setminus B_r(z)}$ and the indicator $\BB 1_{\stabE_{k,r}(z) \cap \{P\cap B_r(z) \not=\emptyset\}}$.
In particular, for any event $F\in\mcl F_k$ the event $F\cap \stabE_{k,r}(z) \cap \{P\cap B_r(z) \not=\emptyset\}$ is a.s.\ determined by $h|_{\BB C\setminus B_r(z)}$ and the indicator $\BB 1_{\stabE_{k,r}(z) \cap \{P\cap B_r(z) \not=\emptyset\}}$.
\end{lem}
\begin{proof}
Since $  \mcl B_{t_k}^\bullet $ is a local set for $h$ (Lemma~\ref{lem-ball-local}) and since balls $B_r(z)$ for $(z,r) \in \mcl Z_k$ are disjoint from $\mcl B_{t_k}^\bullet$, we find that $\{(z,r) \in \mcl Z_k\}$ is determined by $h|_{\BB C\setminus B_r(z)}$. 
Furthermore, $\left( \mcl B_{t_k}^\bullet , h|_{\mcl B_{t_k}^\bullet} \right)$ and hence also $\confarcs_k$ is determined by $h|_{\BB C\setminus B_r(z)}$ on the event $\{(z,r)\in\mcl Z_k\}$.
By Axiom~\ref{item-metric-local} (locality), it then follows that $\stabE_{k,r}(z)$ is determined by $h|_{\BB C\setminus B_r(z)}$. 

Since $\stabE_{k,r}(z) \subset \{(z,r) \in \mcl Z_k\}$, we already know that $\left( \mcl B_{t_k}^\bullet , h|_{\mcl B_{t_k}^\bullet} \right)$ is a.s.\ determined by $h|_{\BB C\setminus B_r(z)}$ on the event $\stabE_{k,r}(z)$. 
On the event $\{P\cap B_r(z) \not=\emptyset\}$, the $D_h$-geodesic $P$ stopped at the first time it enters $B_r(z)$ is a $D_h(\cdot,\cdot ; \BB C\setminus\ol{ B_r(z)})$-geodesic from $\BB z$ to a point of $\bdy  B_r(z)$. If $\stabE_{k,r}(z)$ occurs, then every such $D_h(\cdot,\cdot ; \BB C\setminus\ol{ B_r(z)})$-geodesic passes through the same arc of $\confarcs_k$, and we can see which arc this is by observing $h|_{\BB C\setminus B_r(z)}$. 
Therefore, on $\stabE_{k,r}(z) \cap \{P\cap B_r(z) \not=\emptyset\}$, the arc of $\confarcs_k$ which contains $P(t_k)$ is a.s.\ determined by $h|_{\BB C\setminus B_r(z)}$ and $\BB 1_{\stabE_{k,r}(z) \cap \{P\cap B_r(z) \not=\emptyset\}}$. 

The last statement of the lemma follows from the second statement and Lemma~\ref{lem-geo-sigma-algebra}. 
\end{proof}

We define the set of ``good" pairs
\eqb \label{eqn-good-annulus-set}
\mcl Z_k^E := \left\{ (z,r) \in \mcl Z_k : E_r(z) \cap \stabE_{k,r}(z) \cap \{P\cap B_{\lambda_2 r}(z) \not=\emptyset\} \: \text{occurs} \right\}  
\eqe
and the set of ``very good" pairs
\eqb \label{eqn-good-annulus-set'}
\mcl Z_k^{\frk E} := \left\{ (z,r) \in \mcl Z_k : \frk E_r(z) \cap \stabE_{k,r}(z) \cap \{P\cap B_{\lambda_2 r}(z) \not=\emptyset\} \: \text{occurs} \right\} . 
\eqe
The proof of Theorem~\ref{thm-geo-iterate0} is based on lower-bounding the conditional probability that $\mcl Z_k^{\frk E} \not=\emptyset$ given $\mcl F_k$, which allows us to say that the number of $k$ for which  $\mcl Z_k^{\frk E} \not=\emptyset$ stochastically dominates a binomial random variable.
To lower-bound $\BB P[\mcl Z_k^{\frk E}\not=\emptyset \,|\, \mcl F_k]$, we will first establish a lower bound for $\BB P[\mcl Z_k^{\frk E}\not=\emptyset \,|\, \mcl F_k]$ in terms of  $\BB P[\mcl Z_k^{ E}\not=\emptyset \,|\, \mcl F_k]$ using condition~\ref{item-geo-iterate-compare} in Theorem~\ref{thm-geo-iterate} (Section~\ref{sec-annulus-choose}).
We will then show that it is very likely that $\mcl Z_k^E \not=\emptyset$ for many values of $k$ (Section~\ref{sec-stab}). 
This will imply that it is very likely that there are many values of $k$ for which $\BB P[\mcl Z_k^{  E}\not=\emptyset \,|\, \mcl F_k]$ is bounded below, and hence there are many values of $k$ for which $\BB P[\mcl Z_k^{\frk E}\not=\emptyset \,|\, \mcl F_k]$ is bounded below (Section~\ref{sec-uncond-to-cond}).
We will now outline the rest of the proof of Theorem~\ref{thm-geo-iterate}. See Figure~\ref{fig-iterate-outline} for a schematic illustration of how the various results in this section fit together.
\medskip

\begin{figure}[t!]
 \begin{center}
\includegraphics[scale=.8]{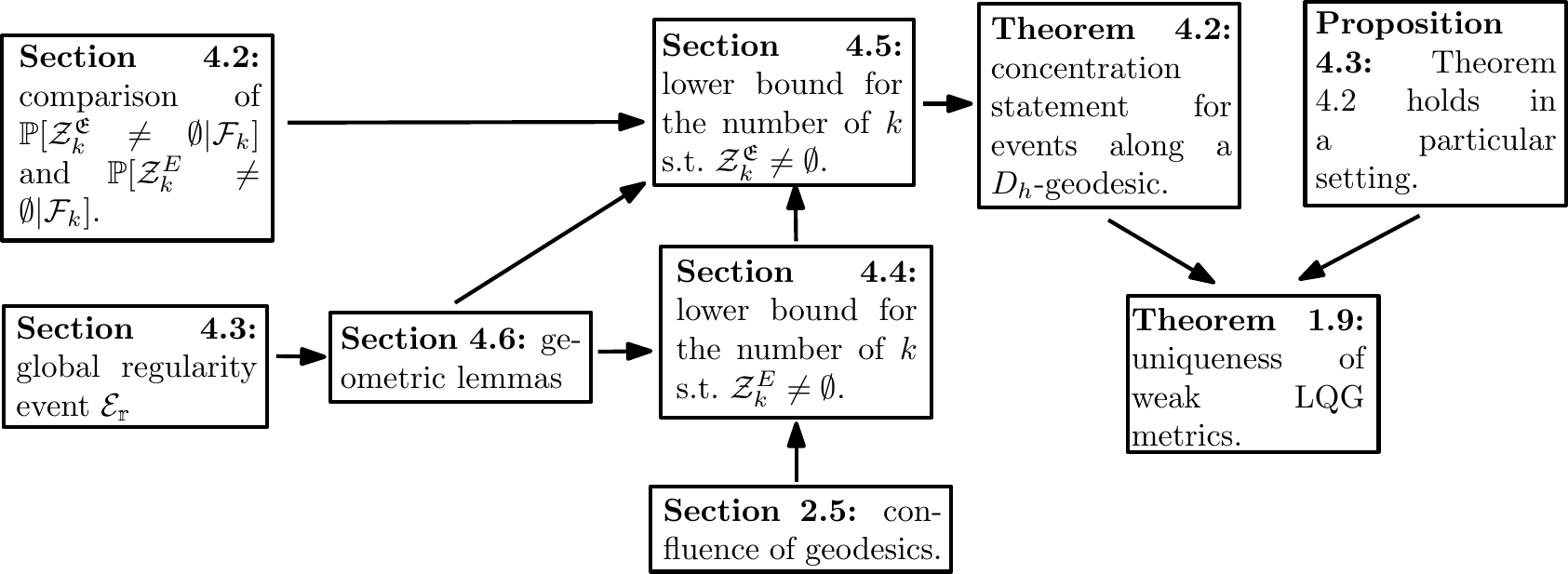}
\vspace{-0.01\textheight}
\caption{Schematic outline of Section~\ref{sec-geodesic-iterate}. An arrow between two sections/results means that the first is used in the proof of the second. Note that Proposition~\ref{prop-geo-event} is proven in Section~\ref{sec-geodesic-shortcut} and Theorem~\ref{thm-weak-uniqueness} is proven in Section~\ref{sec-conclusion}.}\label{fig-iterate-outline}
\end{center}
\vspace{-1em}
\end{figure} 

\noindent In \textbf{Section~\ref{sec-annulus-choose}}, we show that for each $k\in\BB N$, $\BB P[\mcl Z_k^{\frk E} \not=\emptyset | \mcl F_k]$ is bounded below by $\ep^{2\nu+o_\ep(1)} \BB P[\mcl Z_k^E\not=\emptyset |\mcl F_k]$, minus a small error. 
The reason why this is true is that~\eqref{eqn-geo-iterate-compare} together with Lemma~\ref{lem-ball-stab-msrble} allows us to lower-bound $\BB E[\#\mcl Z_k^{\frk E} \,|\, \mcl F_k]$ in terms of $\BB E[\#\mcl Z_k^E \,|\, \mcl F_k]$. 
Then, Lemma~\ref{lem-geo-bdy} along with a Paley-Zygmund type argument allows us to transfer from a lower bound for $\BB E[\#\mcl Z_k^{\frk E} \,|\, \mcl F_k]$ to a lower bound for $\BB P[\mcl Z_k^{\frk E} \not=\emptyset \,|\, \mcl F_k]$.
Here, one should think of $\nu$ as being small (relative to $\beta$), so that $\ep^{2\nu + o_\ep(1)}$ is not too much different from $\ep^{o_\ep(1)}$. 
\medskip

\noindent In \textbf{Section~\ref{sec-iterate-reg}}, we define a global regularity event $\mcl E_{\BB r}$ which we will truncate on for most of the rest of the proof and show that it occurs with high probability. 
This event includes various bounds for $D_h$-distances (e.g., H\"older continuity), but the most important condition is that $E_r(z)$ occurs for many pairs $(z,r)$.
To make the latter condition occur with high probability, we will make $\BB p$ sufficiently close to 1 to allow us to apply Lemma~\ref{lem-annulus-iterate} and a union bound.
\medskip

\noindent In \textbf{Section~\ref{sec-stab}}, we show that if we truncate on $\mcl E_{\BB r}$, then with very high probability there are many values of $k$ for which $\mcl Z_k^E\not=\emptyset$.
Since the definition of $\mcl E_{\BB r}$ already includes the condition that $E_r(z)$ occurs for many pairs $(z,r) \in \mcl Z_k$, 
the main difficulty here is showing that $\stabE_{k,r}(z)$ occurs for most of the pairs $(z,r)$ such that $P\cap B_{\lambda_2 r}(z)\not=\emptyset$.
This will be accomplished by applying the results on confluence of geodesics from~\cite{gm-confluence}, as reviewed in Section~\ref{sec-confluence-prelim}, and multiplying over $k$ to get concentration.
We will choose the parameter $\beta$ from~\eqref{eqn-geo-iterate-times} to be small so that we have enough ``room" between $\bdy\mcl B_{s_k}^\bullet$ and $\bdy \mcl B_{t_k}^\bullet$ for various confluence effects to occur. 
\medskip

\noindent In \textbf{Section~\ref{sec-uncond-to-cond}}, we will transfer from the statement that ``$\mcl Z_k^E\not=\emptyset$ for many values of $k$" to the statement that ``$\mcl Z_k^{\frk E}\not=\emptyset$ for many values of $k$". This will be accomplished using the result of Section~\ref{sec-annulus-choose} and an elementary probabilistic lemma (Lemma~\ref{lem-uncond-to-cond}) which allows us to convert between conditional and unconditional probabilities. We will then complete the proof of Theorem~\ref{thm-geo-iterate} by truncating on $\mcl E_{\BB r}$ and then taking a union bound over many pairs of initial and terminal points $\BB z, \BB w$.  
\medskip

\noindent In \textbf{Section~\ref{sec-geometric-lemma}}, we collect the proofs of some geometric lemmas which are stated in Sections~\ref{sec-stab} and~\ref{sec-uncond-to-cond}, but whose proofs are postponed to avoid distracting from the core of the argument. These geometric lemmas are used to control the behavior of $D_h$-geodesics on the regularity event $\mcl E_{\BB r}$.

\subsection{Comparison of $E_r(z)$ and $\frk E_r(z)$}
\label{sec-annulus-choose}

Recall the definitions of the filtration $\{\mcl F_k\}_{k\geq 0}$ from~\eqref{eqn-nomax-filtration}, the set of ``good" pairs $\mcl Z_k^E$ from~\eqref{eqn-good-annulus-set}, and the set of ``very good" pairs $\mcl Z_k^{\frk E}$ from~\eqref{eqn-good-annulus-set'}. 
The events $E_r(z)$ are easier to work with than the events $\frk E_r(z)$ since $E_r(z)$ has high probability and is determined locally by $h$. 
The goal of this subsection is to prove the following lemma, which will eventually allow us to transfer from a lower bound for the probability that $\mcl Z_k^E \not=\emptyset$ to a lower bound for the probability that $\mcl Z_k^{\frk E} \not=\emptyset$. 

\begin{lem} \label{lem-annulus-choose} 
Let $M >0$. On the event $\{ \mcl B_{t_k}^\bullet \subset B_{\ep^{-M}}(\BB z) \} \cap \{\BB w \notin B_{ 3\lambda_4 \ep \BB r}(\mcl B_{t_k}^\bullet)\}$, it holds except on an event of probability $o_\ep^\infty(\ep)$ that
\eqb \label{eqn-annulus-choose}
\BB P\left[ \mcl Z_k^{\frk E}\not=\emptyset \,|\, \mcl F_k \right] \geq \ep^{2\nu + o_\ep(1)} \BB P\left[ \mcl Z_k^E \not= \emptyset  \,|\, \mcl F_k \right]  - o_\ep^\infty(\ep)  , 
\eqe 
where the rates of the $o_\ep(1)$ and $o_\ep^\infty(\ep)$ are deterministic and depend only on $M,\mu,\nu,\{\lambda_i\}_{i=1,\dots,5}$. 
\end{lem}

Nothing from this section besides Lemma~\ref{lem-annulus-choose} is used in subsequent subsections.
Lemma~\ref{lem-annulus-choose} will eventually be a consequence of condition~\ref{item-geo-iterate-compare} of Theorem~\ref{thm-geo-iterate}, which together with Lemma~\ref{lem-ball-stab-msrble} allows us to compare the conditional expectations of $\#\mcl Z_k^{E}$ and $\#\mcl Z_k^{\frk E}$ given $\mcl F_k$. 
To transfer from a lower bound for the conditional expectation of $\#\mcl Z_k^{\frk E}$ to a lower bound for the probability that $\mcl Z_k^{\frk E}\not=\emptyset$, we will use a Paley-Zygmund type argument. For this purpose we need the following upper bound for $\#\mcl Z_k^{\frk E}$, which comes from Lemma~\ref{lem-geo-bdy} and Markov's inequality (to transfer from unconditional to conditional probability). 

\begin{lem} \label{lem-good-annulus-count}
Let $M > 0$ and $\zeta \in (0,1)$. Also let
\eqbn
\mcl Z_k(P) := \left\{(z,r) \in \mcl Z_k : P\cap B_r(z) \not=\emptyset\right\} ,
\eqen
so that $\mcl Z_k^{\frk E} \subset \mcl Z_k^E \subset \mcl Z_k(P)$.
 On the event $\{\mcl B_{t_k}^\bullet \subset B_{\ep^{-M} \BB r}(\BB z)\}$, it holds except on an event of probability $o_\ep^\infty(\ep)$ as $\ep\rta 0$ that
\eqb  \label{eqn-good-annulus-count}
  \BB E\left[ \#\mcl Z_k(P) \BB 1_{(\#\mcl Z_k(P) > \ep^{-2\nu-\zeta})} \,|\, \mcl F_k \right] = o_\ep^\infty(\ep) ,
\eqe 
where the rate of the $o_\ep^\infty(\ep)$ depends only on $M,\zeta,\mu,\nu,\{\lambda_i\}_{i=1,\dots,5}$. 
\end{lem}
\begin{proof}
By Lemma~\ref{lem-geo-bdy} (applied with $M\vee (2/\zeta)$ in place of $M$ and $4\lambda_4\ep$ in place of $\ep$), on the event $\{\mcl B_{t_k}^\bullet \subset B_{\ep^{-M} \BB r}(\BB z)\}$ it is extremely unlikely that $P$ spends a long time near $\bdy \mcl B_{t_k}^\bullet$: more precisely, it holds except on an event of probability $o_\ep^\infty(\ep)$ as $\ep\rta 0$ that
\eqb \label{eqn-use-geo-bdy}
\op{area}\left( B_{4\lambda_4 \ep \BB r}(P) \cap B_{4\lambda_4\ep \BB r}\left(\bdy\mcl B_{t_k}^\bullet \right) \right) \leq \ep^{2 - \zeta /2} \BB r^2 .
\eqe

By~\eqref{eqn-good-annulus-set0}, each ball $B_r(z)$ for $(z,r) \in \mcl Z_k$ is contained in $B_{4\lambda_4 \ep \BB r}\left(\bdy \mcl B_{t_k}^\bullet \right)$ and the maximal number of such balls which contain any given point of $\BB C$ is at most a constant (depending only on $M,\mu,\nu,\{\lambda_i\}_{i=1,\dots,5}$) times $\ep^{-2\nu} \log_8 \ep^{-1}$. 
By the definition of $\mcl Z_k(P)$, each ball $B_r(z)$ for $(z,r) \in \mcl Z_k(P)$ is contained in $B_{4\lambda_4 \ep \BB r}(P)$. 
Therefore, the left side of~\eqref{eqn-use-geo-bdy} is at least a constant times $\ep^{2 +  2\nu} (\log_8\ep^{-1})^{-1} \BB r^2  \# \mcl Z_k(P)$. 
From~\eqref{eqn-use-geo-bdy}, we now get that
\eqb \label{eqn-annulus-count-uncond}
\BB P\left[\# \mcl Z_k(P) > \ep^{-2\nu-\zeta} ,\,    \mcl B_{t_k}^\bullet \subset B_{\ep^{-M} \BB r}(\BB z) \right] = o_\ep^\infty(\ep) .
\eqe

Since $\{\mcl B_{t_k}^\bullet \subset B_{\ep^{-M} \BB r}(\BB z)\}\subset \mcl F_k$, we can apply Markov's inequality to deduce from~\eqref{eqn-annulus-count-uncond} that with probability $1-o_\ep^\infty(\ep)$, 
\eqb \label{eqn-annulus-count-cond}
  \BB P\left[ \# \mcl Z_k(P) > \ep^{-2\nu-\zeta}  \,|\, \mcl F_k \right] \BB 1_{(\mcl B_{t_k}^\bullet \subset B_{\ep^{-M} \BB r}(\BB z))} = o_\ep^\infty(\ep )         .
\eqe 
If $\mcl B_{t_k}^\bullet \subset B_{\ep^{-M} \BB r}(\BB z) $, then~\eqref{eqn-good-annulus-set0} implies that for each $(z,r) \in\mcl Z_k$,  
\eqbn
z\in \mcl Z_k \subset B_{(\ep^{-M} + 2\lambda_4\ep) \BB r}(\BB z) \cap \left( \frac{\lambda_1 \ep^{1+\nu}}{4} \BB Z^2\right) . 
\eqen
Since there are at most $\mu\log_8 \ep^{-1}$ possibilities for $r$, on the event $\{\mcl B_{t_k}^\bullet \subset B_{\ep^{-M} \BB r}(\BB z)\} $, we have the trivial upper bound 
\eqb \label{eqn-annulus-count-poly}
 \# \mcl Z_k(P) \leq \#\mcl Z_k \leq O_\ep \left( \ep^{-2M(1+\nu)} \log_8 \ep^{-1} \right) .
\eqe
Combining~\eqref{eqn-annulus-count-cond} and~\eqref{eqn-annulus-count-poly} gives~\eqref{eqn-good-annulus-count}.
\end{proof}
 
We will also need the following elementary probabilistic lemma which will be used in conjunction with Lemma~\ref{lem-ball-stab-msrble} to transfer from conditional probabilities given $h|_{\BB C\setminus B_r(z)}$ to conditional probabilities given $\mcl F_k$.

\begin{lem} \label{lem-cond-expectation}
Let $(\Omega,\mcl M , \BB P)$ be a probability space.  Let $\mcl F , \mcl G\subset\mcl M$ be sub-$\sigma$-algebras. 
Let $E \in\mcl M$ be an event such that $F\cap E \in \mcl G \vee \sigma(E)$ for each $F\in\mcl F$. 
Also let $G\in\mcl F\cap \mcl G$. 
Suppose $H_1,H_2 \in \mcl M$ are events and $\Lambda >0$ is a deterministic constant such that a.s.\    
\eqb \label{eqn-event-compare-G}
 \BB P\left[ H_1 \cap E \,|\, \mcl G\right] \BB 1_G  \leq \Lambda \BB P\left[H_2 \cap E \,|\, \mcl G \right] \BB 1_G .
\eqe
Then a.s.\
\eqb \label{eqn-event-compare-F}
 \BB P\left[ H_1 \cap E \,|\, \mcl F\right] \BB 1_G  \leq \Lambda \BB P\left[H_2 \cap E \,|\, \mcl F \right] \BB 1_G .
\eqe
\end{lem}
\begin{proof}
Let $\mcl G' := \mcl G\vee \sigma(E)$. On the event that $\BB P[E \,|\, \mcl G] > 0$, for any $H\in\mcl M$,
\eqb \label{eqn-add-event}
\BB P\left[ H \cap E \,|\, \mcl G' \right]
=  \frac{ \BB P\left[ H  \cap E \,|\, \mcl G \right] }{  \BB P\left[ E \,|\, \mcl G \right] } \BB 1_{E} .
\eqe 
On the event that $\BB P[E \,|\, \mcl G] = 0$, we instead have $\BB P\left[ H \cap E \,|\, \mcl G' \right] = 0$.
Applying~\eqref{eqn-add-event} with $H=H_1$ and with $H=H_2$ and plugging the results into~\eqref{eqn-event-compare-G} shows that a.s.\  
\eqb \label{eqn-event-compare-G'}
 \BB P\left[ H_1 \cap E \,|\, \mcl G' \right] \BB 1_G \leq \Lambda \BB P\left[H_2 \cap E \,|\, \mcl G' \right] \BB 1_G .
\eqe
We claim that for any $H\in\mcl M$, a.s.\ 
\eqb \label{eqn-cond-exp-nest}
\BB E\left[ \BB P\left[ H  \cap E \,|\, \mcl G' \right] \,|\, \mcl F \right] \BB 1_G = \BB P\left[ H \cap E \,|\, \mcl F \right] \BB 1_G . 
\eqe
Once~\eqref{eqn-cond-exp-nest} is proven, we can take the conditional expectations given $\mcl F$ of both sides of~\eqref{eqn-event-compare-G'} to get~\eqref{eqn-event-compare-F}. 
To prove~\eqref{eqn-cond-exp-nest}, let $F\in\mcl F$. By hypothesis, $F\cap E \in  \mcl G'$. 
Therefore,
\allb
\BB E\left[ \BB E\left[ \BB P\left[ H  \cap E \,|\, \mcl G' \right] \,|\, \mcl F \right] \BB 1_G  \BB 1_F     \right]
&= \BB E\left[ \BB P\left[ H\cap E \,|\, \mcl G' \right]  \BB 1_{F\cap G}     \right] \quad \text{(since $F \cap G \in\mcl F$)} \notag\\
&= \BB E\left[ \BB E\left[ \BB 1_{H\cap E} \,|\, \mcl G' \right] \BB 1_{F\cap G\cap E}   \right] \quad \text{(since $E\in\mcl G'$ and $\BB 1_E\BB 1_E = \BB 1_E$)} \notag\\
&= \BB E\left[ \BB 1_{H\cap E} \BB 1_{F\cap G\cap E}   \right] \quad \text{since $F\cap G \cap E \in\mcl G'$} \notag\\
&= \BB P\left[ H\cap F \cap G \cap E \right]  .
\alle
By the definition of conditional expectation, this implies~\eqref{eqn-cond-exp-nest}. 
\end{proof}

\begin{proof}[Proof of Lemma~\ref{lem-annulus-choose}]
Recall that we are assuming that $\lambda_3 = 1$, so that our hypothesis~\eqref{eqn-geo-iterate-compare} says that for $(z,r) \in \BB C \times \mcl R$ such that $\BB z , \BB w \notin B_{ \lambda_4 r}(z)$, 
\eqb \label{eqn-geo-iterate-compare'}
\BB P\left[ E_r(z) \cap \left\{P \cap B_{\lambda_2 r}(z)\not=\emptyset \right\} \,\big|\, h|_{\BB C\setminus B_{  r}(z)}   \right] 
\leq \Lambda \BB P\left[ \frk E_r(z) \cap \left\{P \cap B_{\lambda_2 r}(z)\not=\emptyset \right\} \,\big|\, h|_{\BB C\setminus B_{ r}(z)}   \right] .
\eqe 
Since $\stabE_{k,r}(z) \in \sigma\left( h|_{\BB C\setminus B_r(z)} \right)$ (Lemma~\ref{lem-ball-stab-msrble}), we infer from~\eqref{eqn-geo-iterate-compare'} and the definitions~\eqref{eqn-good-annulus-set} and~\eqref{eqn-good-annulus-set'} of $\mcl Z_k^E$ and $\mcl Z_k^{\frk E}$ that for each $(z,r) \in \BB C\times \mcl R$ such that $\BB z ,\BB w \notin B_{\lambda_4 r}(z)$, a.s.\ 
\eqb \label{eqn-annulus-events-out}
\BB P\left[ (z,r) \in \mcl Z_k^E \,\big|\, h|_{\BB C\setminus B_r(z)}   \right]  
 \leq \Lambda \BB P\left[ (z,r) \in \mcl Z_k^{\frk E} \,\big|\, h|_{\BB C\setminus B_r(z)}   \right] .
\eqe 

We will now deduce from~\eqref{eqn-annulus-events-out} and Lemma~\ref{lem-cond-expectation} that on $\{(z,r) \in \mcl Z_k\}\cap \{\BB w \notin B_{ 3\lambda_4 \ep \BB r}(\mcl B_{t_k}^\bullet)\}$, a.s.\ 
\eqb \label{eqn-annulus-events-cond}
\BB P\left[ (z,r) \in \mcl Z_k^E \,\big|\, \mcl F_k   \right] 
 \leq \Lambda \BB P\left[ (z,r) \in \mcl Z_k^{\frk E} \,\big|\,  \mcl F_k   \right] .
\eqe 
In particular, we will apply Lemma~\ref{lem-cond-expectation} with $\mcl F = \mcl F_k$, $\mcl G = \sigma\left(h|_{\BB C\setminus B_r(z)}\right)$, $E = \stabE_{k,r}(z) \cap \{P\cap B_r(z)\not=\emptyset\}$, $G = \{(z,r) \in \mcl Z_k\}\cap \{\BB w \notin B_{ 3\lambda_4 \ep \BB r}(\mcl B_{t_k}^\bullet)\}$, $H_1 = \{(z,r) \in \mcl Z_k^E \}$, and $H_2 = \{(z,r) \in \mcl Z_k^{\frk E}\}$. 

We check the hypotheses of Lemma~\ref{lem-cond-expectation} with the above choice of parameters, starting with the requirement that the event $G$ defined above belongs to $ \mcl F_k \cap\sigma\left(h|_{\BB C\setminus B_r(z)} \right) $. 
Indeed, it is clear from the definition~\eqref{eqn-good-annulus-set0} of $\mcl Z_k$ that $G\in \sigma(\mcl B_{t_k}^\bullet , h|_{\mcl B_{t_k}^\bullet})$.
By the definition~\eqref{eqn-geo-sigma-algebra} of $\mcl F_k$, we have $G \in \mcl F_k$. 
By the definition~\eqref{eqn-good-annulus-set0} of $\mcl Z_k$ and the locality of $\mcl B_{t_k}^\bullet$ (Lemma~\ref{lem-ball-local}), also $G\in \sigma\left(h|_{\BB C\setminus B_r(z)} \right) $. 
By the definition~\eqref{eqn-good-annulus-set0} of $\mcl Z_k$, if $G$ occurs with positive probability then $\BB z ,\BB w \notin B_{\lambda_4 r}(z)$, so in particular~\eqref{eqn-annulus-events-out} holds a.s.\ on $G$.
By Lemma~\ref{lem-ball-stab-msrble}, the intersection of any event in $\mcl F_k$ with $\stabE_{k,r}(z) \cap \{P\cap B_{ r}(z)\not=\emptyset\}$ is a.s.\ determined by $h|_{\BB C\setminus B_r(z)}$ and $\BB 1_{\stabE_{k,r}(z) \cap \{P\cap B_{ r}(z)\not=\emptyset\}}$. 
 We may therefore apply Lemma~\ref{lem-cond-expectation} to deduce~\eqref{eqn-annulus-events-cond} from~\eqref{eqn-annulus-events-out}. 
 
Summing~\eqref{eqn-annulus-events-cond} over all $(z,r)\in \mcl Z_k$ gives that on $ \{\BB w \notin B_{ 3\lambda_4 \ep \BB r}(\mcl B_{t_k}^\bullet)\}$,
\eqb \label{eqn-annulus-events-sum} 
\BB E\left[ \# \mcl Z_k^{\frk E} \,|\,  \mcl F_k   \right] 
\geq \Lambda^{-1} \BB E\left[ \# \mcl Z_k^E \,|\, \mcl F_k   \right]
\geq \Lambda^{-1} \BB P\left[  \mcl Z_k^E \not=\emptyset \,|\, \mcl F_k \right] .
\eqe 
By Lemma~\ref{lem-good-annulus-count}, for each $\zeta\in (0,1)$, on the event $\{ \mcl B_{t_k}^\bullet \subset B_{\ep^{-M}}(\BB z) \}$, it holds except on an event of probability $o_\ep^\infty(\ep)$ that 
\allb  \label{eqn-annulus-count-upper}
\BB E\left[ \# \mcl Z_k^{\frk E} \,|\,  \mcl F_k   \right]
&\leq \ep^{-2\nu -\zeta} \BB P\left[ 0 <  \# \mcl Z_k^{\frk E}  \leq \ep^{-2\nu-\zeta}  \,|\, \mcl F_k \right] 
 +  \BB E\left[ \#\mcl Z_k^{\frk E} \BB 1_{(\# \mcl Z_k^{\frk E} > \ep^{-2\nu-\zeta})}  \,|\, \mcl F_k \right] \notag \\ 
&\leq \ep^{-2\nu -\zeta} \BB P\left[   \mcl Z_k^{\frk E}  \not=\emptyset \,|\, \mcl F_k \right]  + o_\ep^\infty(\ep)  .
\alle
Combining~\eqref{eqn-annulus-events-sum} and~\eqref{eqn-annulus-count-upper} gives that on the event $\{ \mcl B_{t_k}^\bullet \subset B_{\ep^{-M}}(\BB z) \} \cap \{\BB w \notin B_{ 3\lambda_4 \ep \BB r}(\mcl B_{t_k}^\bullet)\} $, it holds except on an event of probability $o_\ep^\infty(\ep)$ that  
\eqb
\ep^{-2\nu -\zeta} \BB P\left[   \mcl Z_k^{\frk E}  \not=\emptyset \,|\, \mcl F_k \right]  + o_\ep^\infty(\ep) \geq \Lambda^{-1} \BB P\left[  \mcl Z_k^E \not=\emptyset \,|\, \mcl F_k \right] .
\eqe
Re-arranging this inequality and then sending $\zeta \rta 0$ sufficiently slowly as $\ep\rta 0$ gives~\eqref{eqn-annulus-choose}.
\end{proof}

\subsection{Global regularity event}
\label{sec-iterate-reg}

Throughout most of the rest of the proof of Theorem~\ref{thm-geo-iterate}, we will truncate on a global regularity event which we define in this subsection.
The parameter $\BB p \in (0,1)$ of Theorem~\ref{thm-geo-iterate} has to be chosen sufficiently close to 1 to allow us to apply Lemma~\ref{lem-annulus-iterate} to make the probability of one of the conditions in the event as close to 1 as we like. We emphasize that our global regularity event does \emph{not} depend on the particular choice of $\BB z,\BB w$ in~\eqref{eqn-iterate-geo-choice}. 

Fix bounded, connected open sets $U \subset V \subset \BB C$ and parameters $\nu , \ell  > 0$ ($\nu$, $U$, and $\ell$ are the parameters from Theorem~\ref{thm-geo-iterate}).
Also fix, once and for all, parameters $\chi \in (0,\xi (Q-2))$ and $\chi' > \xi (Q+2)$ as in Lemma~\ref{lem-holder-uniform}, chosen in a manner which depends only on $\gamma$ (we will not make the dependence on these parameters explicit). 
For $\BB r > 0$ and $a   \in (0,1)$, let $\mcl E_{\BB r}    = \mcl E_{\BB r}(a , \nu , \ell ,  U , V  )$ be the event that the following is true.
\begin{enumerate}
\item \emph{(Comparison of domains)} $\sup_{z,w\in \BB r U} D_h(z,w) \leq D_h(\BB r U , \BB r \bdy V)$. \label{item-open-set-contain} 
\item \emph{(Comparison of $D_h$-balls and Euclidean balls)} For each $z \in \BB C$ and $r>0$, let $\tau_r(z)$ be the smallest $t > 0$ for which the filled $D_h$-metric ball $B_t^\bullet(z;D_h)$ intersects $\bdy B_r(z)$, as in~\eqref{eqn-tau_r-def}.  Then for each $z\in B_{4\ell\BB r}(\BB r V)$, we have $B_{a\BB r}(z) \subset \mcl B_{\tau_{\ell \BB r}(z)}^\bullet(z ; D_h) $ and
\label{item-quantum-ball-diam}
\eqb \label{eqn-quantum-ball-diam} 
  \min\left\{\tau_{2 \ell \BB r}(z)  - \tau_{ \ell \BB r}(z)  ,  \tau_{3 \ell \BB r}(z)  - \tau_{2\ell \BB r}(z)   \right\} \geq a \max\left\{ \frk c_{ \BB r} e^{\xi h_{\BB r}(z)} ,  \frk c_{ \ell \BB r} e^{\xi h_{\ell \BB r}(z)} \right\} .
\eqe 
\item \emph{(H\"older continuity)} For each $z,w\in B_{4\ell\BB r}(\BB r V)$ with $|z-w| \leq a \BB r$,   \label{item-holder-cont}
\eqb \label{eqn-holder-cont}
\frk c_{\BB r}^{-1} e^{-\xi h_{\BB r}(0)} D_h(z,w) \geq \left| \frac{z-w}{\BB r} \right|^{\chi'} 
\quad \text{and} \quad 
\frk c_{\BB r}^{-1} e^{-\xi h_{\BB r}(0)} D_h\left( z , w ; B_{2|z-w|}(z) \right) \leq \left| \frac{z-w}{\BB r} \right|^\chi .
\eqe 
\item \emph{(Comparison of circle averages)} We have \label{item-circle-avg}
\eqb
\sup_{z\in \BB r V} |h_{\BB r}(z) - h_{\BB r}(0)| \leq a^{-1} .
\eqe 
\item \emph{(Existence of good annuli)} Define $r_1^\ep , \dots , r_{\lfloor \mu\log_8 \ep^{-1} \rfloor}^\ep \in [\ep^{1+\nu} \BB r , \ep \BB r]\cap \mcl R$ as in condition~\ref{item-geo-iterate-dense} from Theorem~\ref{thm-geo-iterate}. For each $\ep \in (0,a \BB r] \cap \{2^{-n}\BB r\}_{n\in\BB N}$ and each $z\in \left(\frac{\lambda_1 \ep^{1+\nu} \BB r}{4} \BB Z^2 \right) \cap B_{4\ell\BB r}(\BB r V)$, there exists at least one $r \in \{r_1^\ep,\dots,r_{\lfloor \mu\log_8\ep^{-1} \rfloor} \}$ for which $E_r(z)$ occurs. \label{item-good-radii-E}
\item \emph{(Bounds for radii used to control geodesics)} Define the radii $\rho_{\BB r,\ep}(z)$ for $\ep > 0$ and $z\in\BB C$ as in Lemma~\ref{lem-clsce-all} and the discussion just preceding it. For each $\ep \in (0,a  ] \cap \{2^{-n}  \}_{n\in\BB N}$ and each $z\in \left(\frac{\ep^{1+\nu} \BB r}{4} \BB Z^2 \right) \cap B_{4\ell\BB r}(\BB r V)$, we have $\rho_{\BB r,\ep}(z) \leq \ep^{1/2} \BB r$.   \label{item-good-radii-small} 
\end{enumerate}

We note that the upper bound in~\eqref{eqn-holder-cont} uses $D_h\left( z , w ; B_{2|z-w|}(z) \right) \geq D_h(z,w)$ instead of $D_h(z,w)$. 
We will need this slightly stronger upper bound for $D_h$-distances in the proof of Lemma~\ref{lem-holder-balls} below. 

\begin{remark} \label{remark-clsce-event}
Due to conditions~\ref{item-quantum-ball-diam}, \ref{item-holder-cont}, and~\ref{item-good-radii-small}, and since $\ell \in (0,1)$, for each $\BB z\in \BB r U$ the event $\mcl E_{\BB r}$ defined just above is contained in the event $\mcl E_{\ell\BB r}^{\BB z}(a )$ as defined just above Theorem~\ref{thm-finite-geo-quant} with $\ell\BB r$ in place of $\BB r$.
\end{remark}

\begin{lem} \label{lem-reg-event-prob}
There exists $\BB p \in (0,1)$ depending only on $\mu,\nu , \{\lambda_i\}_{i=1,\dots,5}$ such that under the hypotheses of Theorem~\ref{thm-geo-iterate}, the following is true. For each bounded open set $U\subset\BB C$, $\nu ,\ell \in (0,1)$, and $p\in (0,1)$, there exists a bounded open set $V\supset U$ and a parameter $a\in (0,1)$, depending only $U,\nu,\ell , p$, such that $\BB P[\mcl E_{\BB r}] \geq p$ for each $\BB r > 0$.
\end{lem}
\begin{proof} 
By Axiom~\ref{item-metric-coord} (tightness across scales), we can find a bounded open set $V\supset U$, depending only on $U$, such that condition~\ref{item-open-set-contain} (comparison of domains) in the definition of $\mcl E_{\BB r}$ holds with probability at least $1-(1-p)/6$. 
Again using Axiom~\ref{item-metric-coord}, we can find a small enough $a \in (0,1)$, depending on $\ell , V , p$, such that condition~\ref{item-quantum-ball-diam} (comparison of balls) holds with probability at least $1-(1-p)/6$. 
By Lemma~\ref{lem-holder-uniform}, after possibly shrinking $a$ we can further arrange that condition~\ref{item-holder-cont} (H\"older continuity) holds with probability at least $1-(1-p)/6$.
By the continuity of the circle average process and the scale invariance of the law of $h$, modulo additive constant, after possibly further shrinking $a$ we can arrange that condition~\ref{item-circle-avg} (comparison of circle averages) holds with probability at least $1- (1-p)/6$. 
By Lemma~\ref{lem-annulus-iterate}, conditions~\ref{item-geo-iterate-msrble} and~\ref{item-geo-iterate-prob} of Theorem~\ref{thm-geo-iterate}, and a union bound over all $z \in \left(\frac{\lambda_1 \ep^{1+\nu} \BB r}{4} \BB Z^2 \right) \cap V$, if $\BB p$ is chosen sufficiently close to 1, in a manner depending only on $\mu,\nu$, and $\{\lambda_i\}_{i=1,\dots,5}$, then the probability of condition~\ref{item-good-radii-E} (existence of good annuli) in the definition of $\mcl E_{\BB r}$ tends to 1 as $a \rta 0$, uniformly over the choice of $\BB r$. 
Therefore, after possibly further shrinking $a$, we can arrange that condition~\ref{item-good-radii-E} in the definition of $\mcl E_{\BB r}$ holds with probability at least $1 - (1-p)/6$. 
By Lemma~\ref{lem-clsce-all} and a union bound over values of $\ep \in  (0,a  ] \cap \{2^{-n}  \}_{n\in\BB N}$, after possibly further shrinking $a$ we can also arrange that condition~\ref{item-good-radii-small} (bounds for $\rho_{\BB r,\ep}(z)$) in the definition of $\mcl E_{\BB r}$ holds with probability at least $1-(1-p)/6$. 
\end{proof}

\subsection{Geodesic stability event occurs at many times}
\label{sec-stab}

Henceforth fix $p\in (0,1)$ (which we will eventually send to 1), a bounded open set $U\subset\BB C$, and $\ell \in (0,1)$ and let $V,a$ be as in Lemma~\ref{lem-reg-event-prob} for this choice of $p,U,\ell$ and the given values of $\mu,\nu$ from Theorem~\ref{thm-geo-iterate}. 
Let $\mcl E_{\BB r}$ be the event of Section~\ref{sec-iterate-reg} with this choice of parameters, so that $\BB P[\mcl E_{\BB r}] \geq p$. 
Define
\eqb \label{eqn-end-time}
 K := \lfloor a \ep^{-\beta} \rfloor  -1 ,
\eqe
where $\beta$ is as in~\eqref{eqn-geo-iterate-times}.
The significance of the value $K$ is that condition~\ref{item-quantum-ball-diam} (comparison of balls) in the definition of $\mcl E_{\BB r}$ implies that, in the notation~\eqref{eqn-geo-iterate-times},
\eqb \label{eqn-geo-iterate-upper}
s_{K+1} \leq \tau_{2 \ell \BB r} ,\quad \text{on $\mcl E_{\BB r}$}. 
\eqe
Recalling the parameter $\beta$ from~\eqref{eqn-geo-iterate-times} and the parameters $\chi  <\chi'$ as in condition~\ref{item-holder-cont} (H\"older continuity) in the definition of $\mcl E_{\BB r}$, we henceforth impose the requirement that
\eqb \label{eqn-beta-compare}
\beta \in (0, \chi / \chi') .
\eqe    
We will make our final choice of $\beta$ in Proposition~\ref{prop-stab} just below. 

Let $\mcl Z_k^E$ be as in~\eqref{eqn-good-annulus-set} and let $K$ be as in~\eqref{eqn-end-time}. 
The goal of this section is to show that with high probability there are many values of $k\in [0,K]_{\BB Z}$ for which $\mcl Z_k^E \not=\emptyset$.
In the next subsection, we will combine this with Lemma~\ref{lem-annulus-choose} to show that there are many values of $k\in [0,K]_{\BB Z}$ for which $\mcl Z_k^{\frk E}\not=\emptyset$. The following proposition is the main result of this subsection and is the only statement from this subsection which is referenced in Section~\ref{sec-uncond-to-cond}. 

\begin{prop} \label{prop-stab}
There are small constants $\beta,\theta \in (0,1)$ depending only on the choice of metric $D$ such that if we use this choice of $\beta$ in~\eqref{eqn-geo-iterate-times}, then on $\mcl E_{\BB r}$ it holds except on an event of probability decaying faster than any positive power of $\ep$, at a rate which is uniform in $\BB r, \BB z ,\BB w$, that there are at least $(1-\ep^\theta ) K$ values of $k \in [0,K]_{\BB Z}$ for which $\mcl Z_k^E \not=\emptyset$. 
\end{prop}

\begin{figure}[t!]
 \begin{center}
\includegraphics[scale=.8]{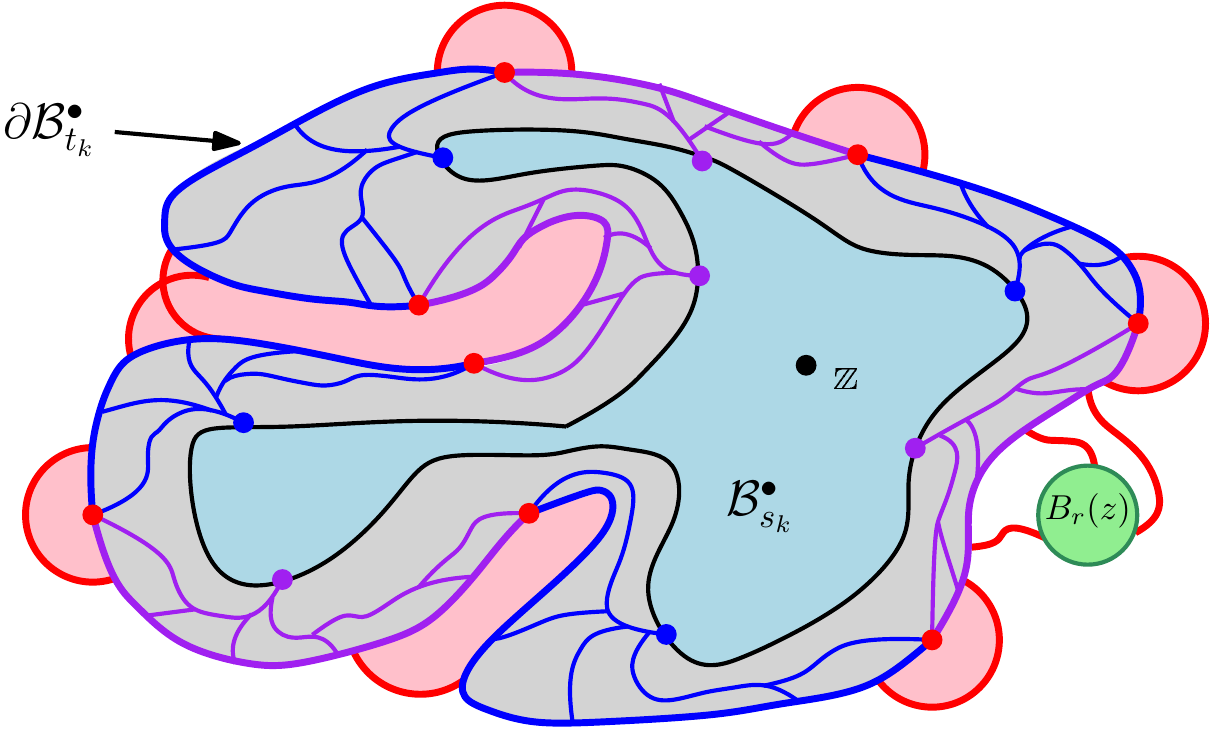}
\vspace{-0.01\textheight}
\caption{Illustration of the proof of Proposition~\ref{prop-stab}. The points in the set $\arcendpts_k$ of endpoints of arcs in $\confarcs_k$ are shown in red.  
We first use Theorem~\ref{thm-finite-geo-quant} to bound $\#\arcendpts_k = \#\confpts_k$. 
Lemma~\ref{lem-dc-set} allows us to choose for each $y\in \arcendpts_k$ a point $ z_y \in\bdy\mcl B_{t_k}^\bullet$ (not shown) such that an arc of $B_{16\ep^\kappa \BB r} (z_y)$ disconnects the set $\mcl C_y^{\ep^\kappa\BB r}$ (defined just after Lemma~\ref{lem-dc-set}) from $\infty$ in $\BB C\setminus \mcl B_{t_k}^\bullet $. The set which this arc disconnects from $\infty$, which contains $\mcl C_y^{\ep^\kappa\BB r}$, is shown in pink. 
Note that the sets $\mcl C_y^{\ep^\kappa\BB r}$ for different choices of $y$ are allowed to overlap. 
Lemma~\ref{lem-geo-kill-pt} and a union bound over $y\in\arcendpts_k$ shows that with high probability, for each $y\in\arcendpts_k$, no $D_h$-geodesic from $\BB z$ to $\bdy\mcl B_{s_{k+1}}^\bullet$ can enter any of the $\mcl C_y^{\ep^\kappa \BB r}$'s.  
This together with Lemma~\ref{lem-geo-disconnect} allows us to show that $\stabE_{k,r}(z)$ occurs for each $(z,r) \in \mcl Z_k$ such that $P\cap B_r(z)\not=\emptyset$.
One such ball $B_r(z)$ is shown in green and several segments of $D_h(\cdot,\cdot; \BB C\setminus \ol{B_r(z)})$-geodesics from $\BB z$ to points of $\bdy B_r(z)$ are shown in red.
}\label{fig-stab}
\end{center}
\vspace{-1em}
\end{figure} 

By condition~\ref{item-good-radii-E} (existence of good annuli) in the definition of $\mcl E_{\BB r}$, we already know that on this event, for each $k\in [0,K]_{\BB Z}$ there are many pairs $(z,r) \in \mcl Z_k$ for which $P\cap B_{\lambda_2 r}(z) \not=\emptyset$ and $E_r(z)$ occurs. The main point of this subsection is to show that there are many such pairs for which also the event $\stabE_{k,r}(z)$ of~\eqref{eqn-ball-stab-event} occurs. Roughly speaking, the idea of the proof is as follows; see Figure~\ref{fig-stab} for an illustration. If $P$ enters $B_r(z)$ but $\stabE_{k,r}(z)$ fails to occur, then $P$ has to get ``close" in some sense to one of the endpoints of one of the arcs in $\confarcs_k$.\footnote{Technically speaking, we are only able to show (using Lemma~\ref{lem-geo-disconnect} below) that $P$ has to enter a region which has one of these endpoints on its boundary and which can be disconnected from $\infty$ in $\BB C\setminus \mcl B_{t_k}^\bullet$ by a small set.} Indeed, otherwise H\"older continuity allows us to force all of the $D_h(\cdot,\cdot ; \BB C\setminus \ol{B_r(z)})$-geodesics from $\BB z$ to points of $\bdy B_r(z)$ to hit the same arc of $\confarcs_k$ as $P$. This is explained in Lemma~\ref{lem-geo-disconnect}. 

On the other hand, if we choose $\beta$ sufficiently small then results from~\cite{gm-confluence} (in particular, Theorem~\ref{thm-finite-geo-quant}) show that $\#\confarcs_k$ is extremely likely to be of smaller order than $\ep^{-\alpha}$, where $\alpha$ is the exponent from Lemma~\ref{lem-geo-kill-pt}. We can therefore apply that lemma once for each of the endpoints of the $\confarcs_k$'s and take a union bound to say that with polynomially high probability given $(\mcl B_{t_k}^\bullet  ,h|_{\mcl B_{t_k}^\bullet})$, no $D_h$-geodesic from $\BB z$ to a point at macroscopic distance from $\bdy\mcl B_{s_k}^\bullet$ can get near any of the endpoints of the $\confarcs_k$'s (Lemma~\ref{lem-stab-endpt}). The claimed superpolynomial concentration when we truncate on $\mcl E_{\BB r}$ comes from a standard concentration bound for independent Bernoulli random variables, provided we choose $\theta$ to be sufficiently small relative to $\alpha$.  

In order to quantify how close $D_h$-geodesics get to the endpoints of the $\confarcs_k$'s, we will need some deterministic definitions.
Let $U\subset\BB C$ be a connected domain such that $\BB C\setminus U$ is compact and connected. View $\bdy U$ as a collection of prime ends. 
If $X\subset U$, we define the \emph{prime end closure} $\op{Cl}'(X)$ to be the set of points in $z\in U\cup\bdy  U$ with the following property: if $\phi : U \cup \bdy U \rta \BB C\setminus  \BB D $ is a conformal map, then $\phi(z)$ lies in $\ol{\phi(X)}$. 
Following~\cite[Equation (2.19)]{gm-confluence}, for $z,w\in U\cup \bdy U $ we define
\eqb \label{eqn-d^U-def} 
d^U(z,w) = \inf\left\{\op{diam}(X) : \text{$X$ is a connected subset of $U$ with $z,w\in \op{Cl}'(X) $}\right\} ,
\eqe
where here $\op{diam}$ denotes the Euclidean diameter. Then $d^U$ is a metric on $U \cup \bdy U$ which is bounded below by the Euclidean metric on $\BB C$ restricted to $U \cup \bdy U$ and bounded above by the internal Euclidean metric on $U \cup \bdy U$. Note that $d^U$ is not a length metric.

\begin{figure}[t!]
 \begin{center}
\includegraphics[scale=.8]{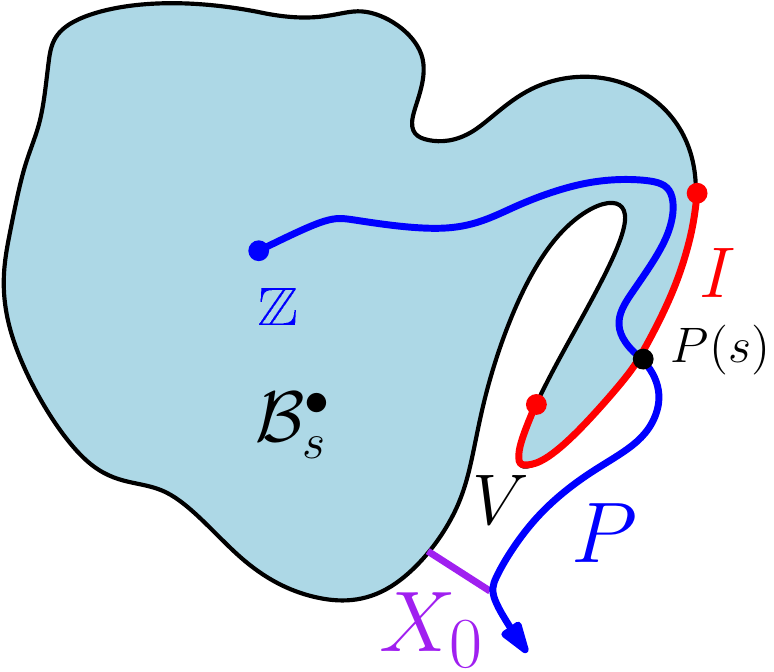}
\vspace{-0.01\textheight}
\caption{Illustration of the statement and proof of Lemma~\ref{lem-geo-disconnect}. If $P$ gets $d^{\BB C\setminus\mcl B_s^\bullet}$-close to $\bdy\mcl B_s^\bullet\setminus I$, then there is a set $X_0$ of small Euclidean diameter which intersects $P$ and $\bdy\mcl B_s^\bullet\setminus I$. Moreover the H\"older continuity condition in the definition of $\mcl E_{\BB r}$ implies that the Euclidean diameter of the segment of $P$ between $s$ and the first time it hits $X_0$ is small. The union $X$ of this segment and $X_0$ disconnects one of the endpoints of $I$ from $\infty$. 
}\label{fig-geo-disconnect}
\end{center}
\vspace{-1em}
\end{figure}

\begin{lem} \label{lem-geo-disconnect}
Almost surely, if $\mcl E_{\BB r}$ occurs then the following is true for every $s \in [0,\tau_{3\ell\BB r}]$, every $\ep \in (0,a]$, and every non-trivial proper connected arc $I \subset \bdy \mcl B_s^\bullet$ (i.e., $I$ is the image of an arc of $\bdy \BB D$ which is not a singleton or all of $\bdy\BB D$ under a conformal map $\BB C\setminus \ol{\BB D} \rta \BB C\setminus \mcl B_s^\bullet$). 
Let $P$ be a $D_h$-geodesic from $\BB z$ to a point outside of $\mcl B_s^\bullet$ which passes through $I$ and suppose that in the notation~\eqref{eqn-d^U-def}, we have $d^{\BB C\setminus \mcl B_s^\bullet}(P , \bdy\mcl B_s^\bullet\setminus I) \leq \ep \BB r$.
There is a connected set $X\subset  \BB C\setminus \mcl B_s^\bullet $ with Euclidean diameter at most $2 \ep^{\chi/\chi'}   \BB r$ such that $P(s)$ and at least one of the two endpoints of $I$ both lie in the prime end closure of the same bounded connected component of $\BB C\setminus (\mcl B_s^\bullet \cup X)$.  
\end{lem}
\begin{proof}
See Figure~\ref{fig-geo-disconnect} for an illustration of the statement and proof of the lemma. 
Assume that $\mcl E_{\BB r}$ occurs and let $s , I  ,\ep$, and $ P$ be as in the lemma statement. 
By hypothesis, for each $\delta \in (0,1)$ there is a connected set $X_0 \subset  \BB C\setminus \mcl B_s^\bullet $ which has Euclidean diameter at most $(\ep+\delta)\BB r$ and which satisfies $P\cap X_0\not=\emptyset$ and $\op{Cl}'(X_0) \cap (\bdy \mcl B_s^\bullet\setminus I) \not=\emptyset$. 
By possibly shrinking $X_0$, we can assume without loss of generality that $\op{Cl}'(X_0) \cap (\bdy \mcl B_s^\bullet\setminus I)$ is a single prime end, which is necessarily in $\mcl B_s^\bullet\setminus I$.

Let $t$ be the first time after $s$ at which $P$ hits $X_0$. 
By the upper bound in condition~\ref{item-holder-cont} (H\"older continuity) in the definition of $\mcl E_{\BB r}$, the $D_h$-diameter of $X_0$ is at most $(\ep+\delta)^{\chi} \frk c_{\BB r} e^{\xi h_{\BB r}(0)} $. 
Since $P$ is a $D_h$-geodesic, $P(t) \in X_0$, and $\op{Cl}'(X_0)$ contains a point of $\bdy\mcl B_s^\bullet$ (which implies that $D_h(X_0, \bdy\mcl B_s^\bullet) = 0$), it follows that $t-s \leq \text{($D_h$-diameter of $X_0$)} \leq (\ep+\delta)^{ \chi} \frk c_{\BB r} e^{\xi h_{\BB r}(0)} $. 
By the lower bound in condition~\ref{item-holder-cont} (H\"older continuity) in the definition of $\mcl E_{\BB r}$, the Euclidean diameter of $P([s,t])$ is at most $(\ep+\delta)^{ \chi/ \chi'} \BB r  $.
The set $X := X_0 \cup P((s,t])$ has Euclidean diameter at most $((\ep+\delta)^{\chi/ \chi'} + \ep + \delta ) \BB r$ and its prime end closure contains both the point $P(s) \in I$ and a point of $\bdy\mcl B_s^\bullet \setminus I$. 
Hence one of the connected components $V$ of $\BB C\setminus (\mcl B_s^\bullet \cup X)$ is bounded and contains an endpoint of $I$. 
Since $\op{Cl}'(X) $ intersects $I$ only at $P(s)$ (here we use that $\op{Cl}'(X_0) \cap \bdy\mcl B_s^\bullet$ is a single point), it follows that also $P(s) \in \bdy V$. 
We now conclude the proof by choosing $\delta$ to be sufficiently small (depending on $\ep$) so that $(\ep+\delta)^{\chi/ \chi'} + \ep + \delta \leq 2\ep^{\chi/\chi'}$. 
\end{proof}

We will eventually apply the contrapositive of Lemma~\ref{lem-geo-disconnect}, i.e., we will say that if $P$ does \emph{not} enter a region which contains one of the endpoints of $I$ and which is disconnected from $\infty$ in $\BB C\setminus \mcl B_s^\bullet$ by a set of small diameter, then $d^{\BB C\setminus \mcl B_s^\bullet}(P ; \bdy\mcl B_s^\bullet\setminus I)$ is bounded below. 
The following elementary deterministic lemma will be used in conjunction with Lemma~\ref{lem-geo-kill-pt} to prevent $P$ from entering such a region (we will apply the lemma with $\mcl K = \mcl B_{t_k}^\bullet$). 

\begin{lem} \label{lem-dc-set}
Let $\mcl K \subset \BB C$ be a compact connected set such that $\BB C\setminus \mcl K$ is connected and view $\bdy \mcl K$ as a collection of prime ends.
For $y\in\bdy \mcl K$ and $\ep > 0$, let $\mcl C_y^\ep$ be the set of points in $z \in  \BB C\setminus \mcl K $ such that the following is true.
There is a connected set $X\subset \BB C\setminus \mcl K $ (allowed to depend on $z$ and $y,\ep$) with Euclidean diameter at most $\ep$ such that $z$ and $y$ lie in the prime end closure of the same bounded connected component of $\BB C\setminus (X\cup \mcl K)$. 
Then there is a compact connected set $Y_y^\ep \subset \BB C\setminus \mcl K $ of Euclidean diameter at most $16\ep$ (depending only on $y,\ep$) such that $\mcl C_y^\ep$ is contained in the prime end closure of a single bounded connected component of $\BB C\setminus (Y_y^\ep \cup \mcl K)$.
\end{lem}

The proof of Lemma~\ref{lem-dc-set} is straightforward, but it takes a few paragraphs so we postpone it until Section~\ref{sec-geometric-lemma} to avoid interrupting the proof of Theorem~\ref{thm-geo-iterate}. The reader may want to refer to Figure~\ref{fig-dc-set} for an illustration of the definition of $\mcl C_y^\ep$.

Returning now to the setting of Proposition~\ref{prop-stab}, for $k\in [0,K]_{\BB Z}$ let $\arcendpts_k$ be the set of endpoints of the arcs in $\confarcs_k$.
As in Lemma~\ref{lem-dc-set}, for $\delta >0$ and $y\in \bdy \mcl B_{t_k}^\bullet$, we let $\mcl C_y^\delta$ be the set of points $z\in (\BB C\setminus \mcl B_{t_k}^\bullet) \cup \bdy \mcl B_{t_k}^\bullet$ with the following property: there is a compact connected set $X\subset\ol{\BB C\setminus \mcl B_{t_k}^\bullet}$ with Euclidean diameter at most $\delta$ such that $z$ and $y$ lie in the closure of the same bounded connected component of $\BB C\setminus (\mcl B_{t_k}^\bullet \cup X)$. 

\begin{lem} \label{lem-stab-endpt}
Fix $\kappa \in (0,1)$. If $\beta,\theta\in (0,1)$ are chosen sufficiently small, in a manner depending only on $\kappa$ and the choice of metric $D$, then on $\mcl E_{\BB r}$ it holds except on an event of probability decaying faster than any positive power of $\ep$, at a rate which is uniform in $\BB r$, that there are at least $(1-\ep^\theta ) K$ values of $k \in [0,K]_{\BB Z}$ for which the following is true.
In the notation introduced just above, no $D_h$-geodesic from $\BB z$ to $\bdy\mcl B_{s_{k+1}}^\bullet$ can enter $\bigcup_{y\in \arcendpts_k} \mcl C_y^{\ep^\kappa \BB r}$. 
\end{lem} 
\begin{proof} 
Fix parameters $\theta , \omega  \in (0,1)$ to be chosen later, in a manner depending only on $D$. 
We will first choose $\beta$ in a manner depending on $\omega ,D$ and then choose $\omega$ in a manner depending on $\kappa, D$, and then choose $\theta$ in a manner depending on $\beta  ,\omega$. In particular, we will take $\omega < \alpha \kappa /2$ and $\theta <  \min\{\omega , \beta/2 \} $.
The parameter $\kappa$ will be chosen in a manner depending only on $D$ in the proof of Proposition~\ref{prop-stab} below. 

We will first show, using Theorem~\ref{thm-finite-geo-quant}, that if $\beta$ is chosen to be sufficiently small (depending on $\omega,D$) then with extremely high probability on $\mcl E_{\BB r}$ one has for each $k\in [0,K]_{\BB Z}$ that $ \#\confpts_k \leq  \ep^{-\omega}$, which implies that $ \#\arcendpts_k \leq  \ep^{-\omega}$. 
We then show using Lemma~\ref{lem-geo-kill-pt} and a union bound over at most $\ep^{-\omega}$ elements of $\arcendpts$ that if $\omega$ is chosen to be sufficiently small (depending on the parameter $\alpha$ of Lemma~\ref{lem-geo-kill-pt}, which depends only on $D$), then for each $k$ it holds with conditional probability at least $1-\ep^\omega$ given $ \mcl B_{t_k}^\bullet , h|_{\mcl B_{t_k}^\bullet}$ that no $D_h$-geodesic from $\BB z$ to $\bdy\mcl B_{s_{k+1}}^\bullet$ can enter $\bigcup_{y\in \arcendpts_k} \mcl C_y^{\ep^\kappa \BB r}$.
Finally, we will use the Markovian structure of the GFF together with a standard concentration inequality for Bernoulli random variables to show that if $\theta$ is chosen to be sufficiently small then with extremely high probability this happens for at least $(1-\ep^\theta)K$ values of $k\in [0,K]_{\BB Z}$. 
\medskip

\noindent\textit{Step 1: bounding the number of confluence points.}   
Recall from Section~\ref{sec-iterate-setup} that $t_k = s_k + \ep^{2\beta} \frk c_{\BB r} e^{\xi h_{\BB r}(0)}$ and $\confpts_k$ is the set of points of $\bdy\mcl B_{s_k}^\bullet$ which are hit by leftmost $D_h$-geodesics from $\BB z$ to $\bdy\mcl B_{t_k}^\bullet$. 
Due to Remark~\ref{remark-clsce-event}, we can apply Theorem~\ref{thm-finite-geo-quant} (with $N = \lfloor \ep^{-\omega} \rfloor$ and $\tau = s_k$) to get that if $\beta$ is chosen sufficiently small, in a manner depending only on $\omega$ and $D$, then for each $k\in [0,K]_{\BB Z}$, the probability that $\mcl E_{\BB r}$ occurs and $ \#\confpts_k  >  \ep^{-\omega} $ decays faster than any positive power of $\ep$. By a union bound over $k$, 
\eqb \label{eqn-stab-interval-count}
\BB P\left[\mcl E_{\BB r} ,\, \max_{k\in [0,K]_{\BB Z}} \#\confpts_k  >  \ep^{-\omega} \right]  = o_\ep^\infty(\ep)  .
\eqe  
\medskip
 
\noindent\textit{Step 2: bounding the parameters from Lemma~\ref{lem-geo-kill-pt}.}  
Recall the radii $\rho_{\BB r,\ep}(z)$, which appear in Lemma~\ref{lem-clsce-all} and condition~\ref{item-good-radii-small} (bounds for $\rho_{\BB r,\ep}(z)$) in the definition of $\mcl E_{\BB r}$ (the precise definition of these radii is not needed here, only their role in Lemma~\ref{lem-geo-kill-pt}). 
To lighten notation, for $k\in [0,K]_{\BB Z}$ we define
\allb \label{eqn-stab-extra-times}
R_k 
&:= R_{\BB r}^{\ep^\kappa}(\mcl B_{t_k}^\bullet)
= 6 \max\left\{\rho_{\BB r,\ep^\kappa}(z) : z\in \left( \frac{\ep^\kappa \BB r}{4} \BB Z^2 \right) \cap B_{\ep^\kappa \BB r}\left(\mcl B_{t_k}^\bullet \right) \right\} +\ep\BB r
,\quad \text{as in~\eqref{eqn-extra-radius-eucl} and} \notag\\
\sigma_k 
&:= \sigma_{t_k,\BB r}^{\ep^\kappa}
= \inf\left\{ s' > s :   B_{R_k}(\mcl B_{t_k}^\bullet) \subset \mcl B_{s'}^\bullet  \right\} ,
\quad \text{as in~\eqref{eqn-extra-radius}} .
\alle
Note that we use~\eqref{eqn-extra-radius-eucl} and~\eqref{eqn-extra-radius} with $\ep^\kappa$ in place of $\ep$.   

On $\mcl E_{\BB r}$, we have $\mcl B_{t_k}^\bullet \subset \mcl B_{\tau_{2\ell\BB r}}^\bullet \subset B_{2\ell\BB r}(\BB z) \subset B_{4\ell\BB r}(\BB r V)$ for each $k\in [0,K]_{\BB Z}$ (see~\eqref{eqn-geo-iterate-upper}).
Hence we can apply condition~\ref{item-good-radii-small} (bounds for $\rho_{\BB r,\ep}(z)$) in the definition of $\mcl E_{\BB r}$ and the definition~\eqref{eqn-stab-extra-times} of $R_k$ to get that if $\ep$ is chosen sufficiently small, depending on $a$ and $\kappa$, then on $\mcl E_{\BB r}$, we have $R_k \leq  (6\ep^{\kappa/2} + \ep^{\kappa/2}) \BB r \leq 7\ep^{\kappa/2} \BB r$ for each $k\in [0,K]_{\BB Z}$. 
By combining this with the upper bound for  $D_h$-distances from condition~\ref{item-holder-cont} (H\"older continuity) in the definition of $\mcl E_{\BB r}$, we get that $B_{R_k}(\mcl B_{t_k}^\bullet) \subset \mcl B_{t_k + 7^{\chi } \ep^{\kappa\chi/2} \frk c_{\BB r} e^{\xi h_{\BB r}(0)}}^\bullet$. 
By this together with the definition of $\sigma_k$ and condition~\ref{item-circle-avg} (comparison of circle averages) in the definition of $\mcl E_{\BB r}$ (to replace $h_{\BB r}(0)$ with $h_{\BB r}(\BB z)$), on $\mcl E_{\BB r}$ we have $\sigma_k \leq t_k +   A \ep^{\kappa \chi/2} \frk c_{\BB r} e^{\xi h_{ \BB r }(\BB z)}  $, where $A = 7^\chi e^{\xi/a}$ is an unimportant constant.  

We henceforth assume that $\beta < \kappa \chi  /2$, so that by the conclusion of the preceding paragraph and the definition~\eqref{eqn-geo-iterate-times} of $t_k$ and $s_{k+1}$, for small enough $\ep \in (0,1)$ (how small depends only on $a , \beta,\kappa$), 
\eqb \label{eqn-stab-extra-radius}
\sigma_k \leq t_k + ( \ep^\beta - \ep^{2\beta} ) \frk c_{ \BB r } e^{\xi h_{ \BB r }(\BB z)}     = s_{k+1} ,\quad\forall k \in [0,K]_{\BB Z}  ,\quad \text{on $\mcl E_{\BB r}$} .
\eqe  
\medskip

\noindent\textit{Step 3: killing off geodesics near the endpoints with polynomially high probability.} 
Recall that $\arcendpts_k$ denotes the set of endpoints of arcs in $\confarcs_k$. 
We have $\#\arcendpts_k = \#\confarcs_k$. 
By Lemma~\ref{lem-dc-set}, each of the sets $\mcl C_y^{\ep^\kappa \BB r}$ can be disconnected from $\infty$ in $\BB C\setminus \mcl B_{t_k}^\bullet$ by a connected subset of $\BB C\setminus \mcl B_{t_k}^\bullet$ of Euclidean diameter at most $16\ep^\kappa\BB r$. 
By an argument as in~\eqref{eqn-stab-extra-radius}, if $\ep \in (0,1)$ is chosen sufficiently small (how small depends only on $\beta,\kappa$), then $B_{16\ep^\kappa\BB r}(\mcl B_{t_k}^\bullet) \subset \mcl B_{s_{k+1}}^\bullet$. 
We may therefore choose for each $y\in \arcendpts_k$ a point $z_y \in \bdy\mcl B_{t_k}^\bullet$, in a manner depending only on $(\mcl B_{t_k}^\bullet , h|_{\mcl B_{t_k}^\bullet})$, with the following property.
\begin{enumerate}[($*$)] \label{eqn-z-choice}
\item Every path in $\BB C\setminus \mcl B_{t_k}^\bullet$ from $\mcl C_y^{\ep^\kappa \BB r}$ to $\BB C\setminus \mcl B_{s_{k+1}}^\bullet$ must enter $B_{16\ep^\kappa \BB r}(z_y)$. 
\end{enumerate}

By Lemma~\ref{lem-geo-kill-pt} (applied with $\tau = t_k$ and $16\ep^\kappa$ in place of $\ep$), there are constants $C_0  > 1$ and $  \alpha > 0$, depending only on the choice of metric, and an event $G_y$ for each $y\in \arcendpts_k$ such that $G_y \in \sigma\left( \mcl B_{\sigma_k}^\bullet , h|_{\mcl B_{\sigma_k}^\bullet} \right)$ and the following is true.  
\begin{enumerate}[A.]
\item If $R_k   \leq   \op{diam}(\mcl B_{t_k}^\bullet)$ and $G_y$ occurs, then no $D_h$-geodesic from $\BB z$ to a point of $\BB C\setminus \mcl B_{\sigma_k}^\bullet$ can enter $B_{16\ep^\kappa \BB r }(z_y) \setminus \mcl B_{t_k}^\bullet  $.  \label{item-geo-stab-occur}
\item Almost surely, $\BB P[G_y |  \mcl B_{t_k}^\bullet , h|_{\mcl B_{t_k}^\bullet} ] \geq 1 - C_0 \ep^{\alpha\kappa}$. \label{item-geo-stab-prob}
\end{enumerate} 

Henceforth assume that $\omega  \in (0,\alpha \kappa /2)$. 
On the event $\{\#\confpts_k  \leq  \ep^{-\omega}\} $ (which is in $\sigma( \mcl B_{t_k}^\bullet , h|_{\mcl B_{t_k}^\bullet})$ and has high probability by~\eqref{eqn-stab-interval-count} and the fact that $\mcl E_{\BB r}$ has high probability), we can take a union bound over at most $\ep^{-\omega}$ elements of $\arcendpts_k$ to get 
\eqb \label{eqn-stab-union}
\BB P\left[ \bigcap_{y\in \arcendpts_k} G_y     \bigg|  \mcl B_{t_k}^\bullet , h|_{\mcl B_{t_k}^\bullet} \right] \geq 1 -  C_0 \ep^{\alpha \kappa -\omega}   .
\eqe
Since $\omega < \alpha\kappa/2$, the right side of~\eqref{eqn-stab-union} is at least $1-\ep^\omega$ for small enough $\ep \in (0,1)$ (how small depends only on $\alpha,\kappa,\omega,C_0$).
This implies that for each such $\ep$, 
\eqb \label{eqn-stab-union'}
\BB P\left[ \left( \bigcap_{y\in \arcendpts_k} G_y \right)^c ,\, \sigma_k \leq s_{k+1} ,\,  \#\confpts_k  \leq  \ep^{-\omega}  \bigg|  \mcl B_{t_k}^\bullet , h|_{\mcl B_{t_k}^\bullet} \right] \leq   \ep^{ \omega} .
\eqe
Note that we have added the additional event $\{\sigma_k \leq s_{k+1}\}$, for reasons which will become apparent just below.
\medskip

\noindent\textit{Step 4: independence across radii to get concentration.} 
The radius $\sigma_k$ is a stopping time for $\{(\mcl B_s^\bullet , h|_{\mcl B_s^\bullet})\}_{s\geq 0}$, so the event inside the conditional probability in~\eqref{eqn-stab-union'} belongs to $\sigma\left( \mcl B_{s_{k+1}}^\bullet , h|_{\mcl B_{s_{k+1}}^\bullet} \right)$. 
Since $t_{k+1} \geq s_{k+1}$, it therefore follows from~\eqref{eqn-stab-union'} that the number of $k \in [0,K]_{\BB Z}$ for which either $ \bigcap_{y\in \arcendpts_k} G_y $ occurs, $\sigma_k > s_{k+1}$, or $\#\confpts_k  > \ep^{-\omega}$ stochastically dominates a binomial distribution with $K$ trials and success probability $1 - \ep^{\omega}$. 
By Hoeffding's inequality, for any choice of $\theta\in (0,1)$ the probability that there are fewer than $(1-\ep^{\theta}) K$ such values of $k$ is at most 
\eqbn
\exp\left( - 2 (\ep^\theta - \ep^{\omega})^2 K \right) .
\eqen
Since $K = \lfloor a \ep^{-\beta} \rfloor  -1$ by~\eqref{eqn-end-time}, this last quantity decays faster than any positive power of $\ep$ provided we take $\theta \in (0 , \min\{\omega , \beta/2\})$. 

By~\eqref{eqn-stab-extra-radius}, on $\mcl E_{\BB r}$ we have $ \sigma_k \leq s_{k+1} $ for each $k\in [0,K]_{\BB Z}$.
By~\eqref{eqn-stab-interval-count}, if $\mcl E_{\BB r}$ occurs then except on an event of probability $o_\ep^\infty(\ep)$ we have $\#\confpts_k  \leq  \ep^{-\omega}$ for each $K\in [0,K]_{\BB Z}$. 
Combining these observations with the preceding paragraph shows that 
\eqb \label{eqn-stab-iterate}
\BB P\left[ \mcl E_{\BB r} ,\, \#\left\{  k \in [0,K]_{\BB Z} : \bigcap_{y\in \arcendpts_k} G_y \: \text{occurs} \right\} < (1-\ep^{\theta}) K \right] = o_\ep^\infty(\ep) .
\eqe

Recall that $R_k \leq 7 \ep^{\kappa/2}  \BB r$ on $\mcl E_{\BB r}$ (see this discussion just after~\eqref{eqn-stab-extra-times}).  As $t_k \geq \tau_{\ell\BB r}$ we have that $\op{diam} \mcl B_{t_k}^\bullet \geq \ell \BB r$.  By choosing $\ep > 0$ sufficiently small we can arrange so that $ \ell \BB r \geq 7 \ep^{\kappa/2} \BB r$.  That is, $R_k \leq   \op{diam}\mcl B_{t_k}^\bullet$ on $\mcl E_{\BB r}$ provided $\ep$ is chosen sufficiently small (in a manner depending only on $\kappa$ and $\ell$). 
Consequently,~\eqref{eqn-stab-iterate} together with property~\ref{item-geo-stab-occur} of $G_y$ show that on $\mcl E_{\BB r}$, it holds except on an event of probability $o_\ep^\infty(\ep)$ that there are at least $(1-\ep^\theta ) K$ values of $k \in [0,K]_{\BB Z}$ for which no $D_h$-geodesic from $\BB z$ to a point outside of $\mcl B_{\sigma_k}^\bullet$ can enter $\bigcup_{y\in\arcendpts_k} B_{16\ep^\kappa \BB r }(z_y) \setminus \mcl B_{t_k}^\bullet$. 
By~\eqref{eqn-stab-extra-radius}, this holds in particular for each $D_h$-geodesic from $\BB z$ to $\bdy\mcl B_{s_{k+1}}^\bullet$. 

A $D_h$-geodesic started from $\BB z$ can hit $\bdy\mcl B_{t_k}^\bullet$ at most once. Therefore, the defining property ($*$) of $z_y$, applied to the path $P|_{(t_k , |P|]}$, shows that for each $k$ as in the preceding paragraph, no $D_h$-geodesic from $\BB z$ to $\bdy\mcl B_{s_{k+1}}^\bullet$ can enter $\bigcup_{y\in \arcendpts_k} \mcl C_y^{\ep^\kappa \BB r}$. 
\end{proof}

To deduce Proposition~\ref{prop-stab} from Lemma~\ref{lem-stab-endpt}, we need some quantitative control on the $D_h(\cdot,\cdot ; \BB C\setminus \ol{B_r(z)})$-geodesics appearing in the definition~\eqref{eqn-ball-stab-event} of $\stabE_{k,r}(z)$. 
The needed control is provided by the following lemma.

\begin{lem} \label{lem-geodesic-diam}
If $\mcl E_{\BB r}$ occurs, then for each $k\in [0,K]_{\BB Z}$, each $(z,r) \in \mcl Z_k$, and each $D_h(\cdot,\cdot ; \BB C\setminus \ol{B_r(z)})$-geodesic $P' : [0,|P'|] \rta \BB C \setminus B_r(z)$ from $\BB z$ to a point of $\bdy B_r(z)$, we have
\eqb \label{eqn-geodesic-diam}
\op{diam} P'([t_k , |P'|]) \preceq \ep^{\chi/\chi'} \BB r ,
\eqe
with a deterministic implicit constant depending only on $a$ and $\lambda_4$, where $\op{diam}$ denotes Euclidean diameter.  
\end{lem}

Lemma~\ref{lem-geodesic-diam} is a straightforward consequence of the definition of $\mcl E_{\BB r}$. We postpone the proof until Section~\ref{sec-geometric-lemma}.

\begin{proof}[Proof of Proposition~\ref{prop-stab}] 
Let $\chi , \chi'$ be the H\"older exponents from condition~\ref{item-holder-cont} in the definition of $\mcl E_{\BB r} $ and let $\beta,\theta \in (0,1)$ be chosen so that the conclusion of Lemma~\ref{lem-stab-endpt} holds with $\kappa = \frac12 (\chi/\chi')^2$. 
By Lemma~\ref{lem-stab-endpt}, we only need to prove that if $\mcl E_{\BB r}$ occurs and $\ep \in (0,1)$ is chosen to be sufficiently small (in a deterministic manner which does not depend on $k$ or $\BB r$), then the following is true.
If $k\in [0,K]_{\BB Z}$ is such that no $D_h$-geodesic from $\BB z$ to $\bdy\mcl B_{s_{k+1}}^\bullet$ can enter $\bigcup_{y\in \arcendpts_k} \mcl C_y^{\ep^\kappa \BB r}$, then $\mcl Z_k^E \not=\emptyset$.  

Henceforth assume that $\mcl E_{\BB r}$ occurs and $k$ is as above.  Recall the definition~\eqref{eqn-good-annulus-set0} of $\mcl Z_k$.  By condition~\ref{item-good-radii-E} (existence of good annuli) in the definition of $\mcl E_{\BB r} $, each point of $\bdy B_{ 2\lambda_4 \ep \BB r}(\mcl B_{t_k}^\bullet) $ is contained in a Euclidean ball $B_{\lambda_2 r}(z)$ for some $(z,r) \in \mcl Z_k $ for which $E_r(z)$ occurs. 
By the definition~\eqref{eqn-good-annulus-set0}, each of these Euclidean balls has radius $r \leq \ep\BB r$, so is contained in $B_{4\lambda_4 \ep \BB r}(\mcl B_{t_k}^\bullet)$. 

Since $t_k \leq s_{k+1} \leq \tau_{3\ell\BB r}$ (by~\eqref{eqn-geo-iterate-upper}) and $|\BB z - \BB w| \geq 4\ell\BB r$, if $\ep$ is sufficiently small then the union of these Euclidean balls disconnects $\bdy\mcl B_{t_k}^\bullet$ from $\BB w$.  Therefore, $P$ must enter $B_{\lambda_2 r}(z)$ for some $(z,r) \in \mcl Z_k$ such that $E_r(z)$ occurs. 

We will now conclude the proof by showing that, in the notation~\eqref{eqn-ball-stab-event},
\eqb \label{eqn-ball-stab-occurs}
\text{$\stabE_{r,k}(z)$ occurs for every $(z,r) \in \mcl Z_k$ with $P\cap B_r(z) \not=\emptyset$.}
\eqe  
Recall that we are assuming that $k$ is such that no $D_h$-geodesic from $\BB z$ to $\bdy\mcl B_{s_{k+1}}^\bullet$ can enter $\bigcup_{y\in \arcendpts_k} \mcl C_y^{\ep^\kappa \BB r}$.
Since $s_{k+1} \leq \tau_{3\ell\BB r} \leq \tau_{|\BB z-\BB w|}$ we must have $|P| \geq s_{k+1}$, so $P$ passes through $\bdy\mcl B_{s_{k+1}}^\bullet$. 
Hence $P|_{[0,s_{k+1}]}$ cannot enter $\bigcup_{y\in \arcendpts_k} \mcl C_y^{\ep^\kappa \BB r}$.
Since $P$ does not re-enter $\mcl B_{s_{k+1}}^\bullet$ after time $s_{k+1}$ and $\bigcup_{y\in \arcendpts_k} \mcl C_y^{\ep^\kappa \BB r}\subset\mcl B_{s_{k+1}}^\bullet$, also $P$ cannot enter $\bigcup_{y\in \arcendpts_k} \mcl C_y^{\ep^\kappa \BB r}$.
From this and Lemma~\ref{lem-geo-disconnect} (applied in the contrapositive direction with $(\ep^\kappa /2)^{ \chi'/\chi}$ in place of $\ep$ and $I_k$ in place of $I$), we infer that if $I_k \in \confarcs_k$ is chosen so that $P(t_k) \in I_k$, then 
\eqb \label{eqn-stab-external-dist}
d^{\BB C\setminus \mcl B_{t_k}^\bullet}\left(P , \bdy\mcl B_{t_k}^\bullet\setminus I_k \right) \geq (\ep^\kappa /2)^{\chi'/\chi} \BB r .
\eqe

Now let $(z,r) \in \mcl Z_k $ with $P\cap B_r(z) \not=\emptyset$ and let $P' : [0,|P'|] \rta \BB C\setminus  B_r(z)$ be a $D_h(\cdot,\cdot ; \BB C\setminus \ol{B_r(z)} )$-geodesic from $\BB z$ to a point of $\bdy B_r(z)$. 
We will show that $P'(t_k) \in I_k$ for any possible choice of $P'$, which by definition implies that $\stabE_{r,k}(z)$ occurs. 
By Lemma~\ref{lem-geodesic-diam}, 
\eqb \label{eqn-stab-external-diam}
\op{diam}\left( P'([t_k , |P'|] )\right) \preceq \ep^{\chi/\chi'} \BB r  
\eqe
with a deterministic implicit constant depending only on $a$ and $\lambda_4$. 

Since $B_r(z) \subset \BB C\setminus \mcl B_{t_k}^\bullet$, the definition~\eqref{eqn-d^U-def} of $d^{\BB C\setminus \mcl B_{t_k}^\bullet}$ implies that the $d^{\BB C\setminus \mcl B_{t_k}^\bullet}$-diameter of $B_r(z)$ is the same as its Euclidean diameter, which is $2\ep\BB r$. 
Since $P\cap B_r(z) \not=\emptyset$ and $P'(|P'|) \in \bdy B_r(z)$, it follows from~\eqref{eqn-stab-external-diam} and the triangle inequality that for small enough $\ep\in (0,1)$, 
\eqb \label{eqn-stab-external-tri}
d^{\BB C\setminus\mcl B_{t_k}^\bullet}(P , P'(t_k) ) \preceq  \ep  \BB r + \ep^{\chi/\chi'} \BB r \preceq \ep^{\chi /\chi'} \BB r .   
\eqe  
Since $\kappa < (\chi/\chi')^2$, we infer that the left side of~\eqref{eqn-stab-external-tri} is strictly smaller than $(\ep^\kappa /2)^{\chi'/\chi} \BB r$ for small enough $\ep \in (0,1)$ (depending only on $a$ and $\lambda_4$).
By combining~\eqref{eqn-stab-external-dist} and~\eqref{eqn-stab-external-tri} we infer that $P'(t_k) \notin \bdy\mcl B_{t_k}^\bullet \setminus I_k$. 
Hence $P'(t_k) \in  I_k$. 
Since this holds for every choice of $P'$, we get that $\stabE_{r,k}(z)$ occurs, as required.
\end{proof}

\subsection{Transferring from $E_r(z)$ to $\frk E_r(z)$}
\label{sec-uncond-to-cond}

We now want to combine Lemma~\ref{lem-annulus-choose} and Proposition~\ref{prop-stab} to say that with high probability, there are many values of $k\in [0,K]_{\BB Z}$ for which there exists $(z,r)\in \mcl Z_k$ for which $\frk E_r(z)$ occurs. In particular, we will establish the following statement.

\begin{prop} \label{prop-nomax-quant} 
Let $\beta,\theta \in (0,1)$ be as in Proposition~\ref{prop-stab} and suppose we have chosen $\nu$ sufficiently small that $4\nu < \beta \wedge \theta$. 
Also let $\zeta\in (0,1)$ be a small ``error" parameter. If $\mcl E_{\BB r}$ occurs, then except on an event of probability $o_\ep^\infty(\ep)$, at a rate which is uniform in the choice of $\BB r,\BB z,\BB w$, there are at least $\ep^{2\nu+\zeta} K$ values of $k\in [0,K]_{\BB Z}$ for which $\mcl Z_k^{\frk E}\not=\emptyset$. 
\end{prop}

Lemma~\ref{lem-annulus-choose} gives a comparison of the conditional probabilities given $\mcl F_k$ of $\{\mcl Z_k^E\not=\emptyset\}$ and $\{\mcl Z_k^{\frk E}\not=\emptyset\}$ (the reason why we have this comparison is that condition~\ref{item-geo-iterate-compare} in Theorem~\ref{thm-geo-iterate} has a comparison of conditional probabilities).
On the other hand, Propositions~\ref{prop-stab} and~\ref{prop-nomax-quant} give statements which hold with high unconditional probability. To transfer between conditional and unconditional probabilities we will use the following elementary lemma. 

\begin{lem} \label{lem-uncond-to-cond}
Let $K \in\BB N$ and let $E_0,\dots,E_K$ be events (not necessarily independent). Also let $\mcl F_1 \subset \mcl F_1\subset \cdots\subset \mcl F_K$ be $\sigma$-algebras such that $E_k \in \mcl F_{k+1}$ for each $k\in [0,K-1]_{\BB Z}$. 
For $\alpha \in (0,1)$, $\delta \in (0,\alpha)$, and $m\in\BB N$, 
\eqb \label{eqn-uncond-to-cond}
\BB P\left[ \sum_{k=0}^K \BB 1_{\left( \BB P[E_k \,|\, \mcl F_k] \geq \alpha \right)} \leq  K -   m ,\, 
\sum_{k=0}^K \BB 1_{E_k} \geq K - (1 - \alpha-\delta) m 
\right]
\leq    e^{-2 \delta^2 m} . 
\eqe 
\end{lem}
\begin{proof} 
For $j \in \BB N$, let $\tau_j$ be the $j$th smallest $k \in [0 , K ]_{\BB Z}$ for which $\BB P[E_k \,|\, \mcl F_k] < \alpha$, or $\tau_j = K+1$ if no such $j$ exists. 
Then $\{\tau_j = k\} \in \mcl F_k$ for each $k \in [0 ,K]_{\BB Z}$ and
\eqb \label{eqn-uncond-to-cond-event}
\left\{ \sum_{k=0}^K \BB 1_{\left( \BB P[E_k \,|\, \mcl F_{k}] \geq \alpha \right)} \leq  K - m \right\} 
= \left\{ \sum_{k=0}^K \BB 1_{\left( \BB P[E_k \,|\, \mcl F_{k}] < \alpha \right)} \geq m  + 1 \right\} 
= \{ \tau_{m+1}  \leq K \} .
\eqe 

By the definition of the $\tau_j$'s, for each $j \in \BB N$,
\eqb
\BB P\left[ E_{\tau_j}^c  \,|\, \mcl F_{\tau_j} \right]  
 \geq 1 -  \alpha .
\eqe
Since $E_{\tau_{j'}} \in \mcl F_{\tau_{j-1}}$ for each $j' \leq j-1$, it follows that $\sum_{j=1}^{m+1} \BB 1_{E_{\tau_j}^c}$ stochastically dominates a binomial distribution with $m+1$ trials and success probability $1-\alpha$. 
By Hoeffding's inequality, for $m \in \BB N$ the probability that the number of $j \in [1,m+1]_{\BB Z}$ for which $E_{\tau_j}^c $ occurs is smaller than $(1 - \alpha -  \delta) m$ is at most $e^{-2\delta^2 m}$. 
Therefore,
\eqb
\BB P\left[\tau_{m+1} \leq  K  ,  \sum_{j=0}^K \BB 1_{E_j} \geq  K - (1 - \alpha-\delta) m     \right]  \leq e^{-2\delta^2 m} .
\eqe 
Combining this with~\eqref{eqn-uncond-to-cond-event} gives~\eqref{eqn-uncond-to-cond}.
\end{proof}

We want to apply Lemma~\ref{lem-uncond-to-cond} to the events $\{\mcl Z_k^E \not=\emptyset\}$ and $\{\mcl Z_k^{\frk E} \not=\emptyset\}$. 
However, these events are not $\mcl F_{k+1}$-measurable since for $(z,r) \in\mcl Z_k$, the ball $B_r(z)$ and the $D_h(\cdot,\cdot ; \BB C\setminus \ol{B_r(z)})$-geodesics from $\BB z$ to $\bdy B_r(z)$ are not necessarily contained in $\mcl B_{s_{k+1}}^\bullet$.
To get around this, we need to instead work with a slightly modified event which is $\mcl F_{k+1}$-measurable. 
In particular, we will intersect each of $\{\mcl Z_k^E \not=\emptyset\}$ and $\{\mcl Z_k^{\frk E} \not=\emptyset\}$ with the event $F_k$ of the following lemma.

\begin{lem} \label{lem-holder-balls}
For each $k\in [0,K]_{\BB Z}$, there is an event $F_k \in \sigma\left( \mcl B_{s_{k+1}}^\bullet , h|_{\mcl B_{s_{k+1}}^\bullet} \right)$ with the following properties.  
If $\ep$ is sufficiently small (how small depends only on $a$, $\lambda_4$), then whenever $\mcl E_{\BB r}$ occurs also $ \bigcap_{k=0}^K F_k$ occurs. Moreover, if $F_k$ occurs then $s_{k+1} \leq \tau_{2 \ell\BB r}$ and for each $(z,r) \in \mcl Z_k$ we have $B_{\lambda_4 r}(z) \subset \mcl B_{s_{k+1}}^\bullet$ and the set of $D_h(\cdot,\cdot ; \BB C\setminus \ol{B_r(z)})$-geodesics from $\BB z$ to points of $\bdy B_r(z)$ is determined by $(\mcl B_{s_{k+1}}^\bullet , h|_{\mcl B_{s_{k+1}}^\bullet})$. 
\end{lem}

Lemma~\ref{lem-holder-balls} is a relatively straightforward consequence of the definition of $\mcl E_{\BB r}$. The proof is postponed until Section~\ref{sec-geometric-lemma}.
The event $F_k$ is defined explicitly in Lemma~\ref{lem-holder-balls0} below, but only the properties of the event given in Lemma~\ref{lem-holder-balls} are important for our purposes. 

\begin{lem} \label{lem-nomax-msrble}
Let $F_k$ for $k\in [0,K]_{\BB Z}$ be the event of Lemma~\ref{lem-holder-balls} and let $\mcl F_k$ be the $\sigma$-algebra from~\eqref{eqn-nomax-filtration}. 
Then for $k\in\BB N$,
\eqb \label{eqn-nomax-msrble}
 \left\{\mcl Z_k^E \not=\emptyset \right\} \cap F_k \in \mcl F_{k+1} 
 \quad \text{and} \quad
  \left\{\mcl Z_k^{\frk E} \not=\emptyset \right\} \cap F_k \in \mcl F_{k+1} .
\eqe
\end{lem}
\begin{proof}
By Lemma~\ref{lem-holder-balls}, we have $F_k \in \mcl F_{k+1}$. 
By the definition~\eqref{eqn-good-annulus-set0} we also have $\mcl Z_k \in \mcl F_k \subset\mcl F_{k+1}$.

We now argue that on $F_k$, the set $\mcl Z_k^E$ is determined by $\mcl F_{k+1}$.
Since there are only countably many pairs $(z,r) \in \BB C\times (0,\infty)$ which can possibly belong to $\mcl Z_k^E$, it suffices to show that the event $\{(z,r) \in \mcl Z_k^E\} \cap F_k$ is $\mcl F_{k+1}$-measurable for each such pair $(z,r)$.
Recall from~\eqref{eqn-good-annulus-set} that $\mcl Z_k^E$ is the set of $(z,r) \in \mcl Z_k $ for which $E_r(z) \cap \stabE_{r,k}(z) \cap \{P\cap B_{\lambda_2 r}(z) \not=\emptyset\}$ occurs.
By Lemma~\ref{lem-holder-balls}, if $F_k$ occurs then $B_{\lambda_4 r}(z) \subset \mcl B_{s_{k+1}}^\bullet$ for each $(z,r) \in \mcl Z_k$. 
Since $E_r(z)$ is determined by $h|_{B_{\lambda_4 r}(z)}$ (condition~\ref{item-geo-iterate-msrble}), it follows that $F_k \cap E_r(z) \cap \{(z,r) \in \mcl Z_k\} \in \mcl F_{k+1}$ for each $(z,r) \in \BB C\times (0,\infty)$. 
Moreover, since $P|_{[0,s_{k+1}]} \in \mcl F_{k+1}$ and $P$ does not re-enter $\mcl B_{s_{k+1}}^\bullet$ after time $s_{k+1}$, we have $F_k \cap \{P\cap B_{\lambda_2 r}(z) \not=\emptyset\} \cap \{(z,r) \in \mcl Z_k\} \in \mcl F_{k+1}$ for each $(z,r)$. 
By~\eqref{eqn-ball-stab-event}, each of the events $\stabE_{r,k}(z)$ for $(z,r) \in \mcl Z_k$ is determined by $\mcl F_k$ and the set of $D_h(\cdot,\cdot ; \BB C\setminus \ol{B_r(z)})$-geodesics from $\BB z$ to points of $\bdy B_r(z)$. 
By Lemma~\ref{lem-holder-balls}, it therefore follows that $F_k \cap \stabE_{r,k}(z) \in \mcl F_{k+1}$ for each $(z,r) \in \mcl Z_k$. 
Combining these statements shows that $\{\mcl Z_k^E\not=\emptyset\} \cap F_k \in  \mcl F_{k+1}$. 

Using condition~\ref{item-geo-iterate-msrble} from Theorem~\ref{thm-geo-iterate}, we similarly obtain that $  \left\{\mcl Z_k^{\frk E} \not=\emptyset \right\} \cap F_k \in \mcl F_{k+1}$. 
\end{proof}

\begin{lem} \label{lem-nomax-cond}
Let $\theta$ be as in Proposition~\ref{prop-stab} and let $F_k$ for $k\in \BB N$ be as in Lemma~\ref{lem-holder-balls}.
If $\mcl E_{\BB r}$ occurs, then except on an event of probability $o_\ep^\infty(\ep)$ there are at least $(1 - 4\ep^\theta )K$ values of $k\in [0,K]_{\BB Z}$ for which 
\eqb \label{eqn-nomax-cond}
\BB P\left[ \left\{\mcl Z_k^E \not=\emptyset \right\} \cap F_k \,\big| \, \mcl F_k \right] \geq \frac12 
\eqe
and
\eqb \label{eqn-nomax-pos}
\BB P\left[ \mcl Z_k^{\frk E} \not=\emptyset   \,\big| \, \mcl F_k \right] \geq \ep^{2\nu + o_\ep(1)}  ,
\eqe
where the rate of the $o_\ep(1)$ in~\eqref{eqn-nomax-pos} is deterministic and depends only on $\nu $ and the choice of metric $D$. 
\end{lem}
\begin{proof}
For $k\in\BB N$, let $E_k := \left\{\mcl Z_k^E \not=\emptyset \right\} \cap F_k$.
By Lemma~\ref{lem-nomax-msrble}, we have $E_k \in \mcl F_{k+1}$. 
We may therefore apply Lemma~\ref{lem-uncond-to-cond} with $m = \lfloor  4 \ep^{\theta} K   \rfloor$, $\alpha=1/2$, and $\delta=1/4$ to get that
\eqb \label{eqn-use-uncond}
\BB P\left[  \sum_{k=0}^K \BB 1_{\left( \BB P[E_k \,|\, \mcl F_{k}] \geq 1/2 \right)} \leq (1- 4\ep^{\theta } ) K   ,\, 
\sum_{k=0}^K \BB 1_{E_k} \geq (1 - \ep^\theta) K  \right] = o_\ep^\infty(\ep) .
\eqe
By Proposition~\ref{prop-stab} and Lemma~\ref{lem-holder-balls}, on $\mcl E_{\BB r}$ it holds except on an event of probability $o_\ep^\infty(\ep)$ that $\sum_{k=1}^K \BB 1_{E_k} \geq (1 - \ep^\theta) K$. 
Combining this with~\eqref{eqn-use-uncond} shows that if $\mcl E_{\BB r}$ occurs, then except on an event of probability $o_\ep^\infty(\ep)$ there are at least $(1-4\ep^\theta)K$ values of $k\in [0,K]_{\BB Z}$ for which~\eqref{eqn-nomax-cond} holds. 

On $\mcl E_{\BB r}$, for each $k\in [0,K]_{\BB Z}$ we have $\mcl B_{t_k}^\bullet \subset B_{3\ell \BB r}(\BB z)$ (by~\eqref{eqn-geo-iterate-upper}) and $\BB w \notin B_{3\lambda_4 \ep\BB r}(\mcl B_{t_k}^\bullet)$ (since $|\BB z - \BB w| \geq 4\ell\BB r$).  
By Lemma~\ref{lem-annulus-choose}, whenever these latter conditions hold it holds except on an event of probability $o_\ep^\infty(\ep)$ that 
\eqbn
\BB P\left[ \mcl Z_k^{\frk E} \not=\emptyset   \,\big| \, \mcl F_k \right] \geq \ep^{2\nu + o_\ep(1)} \BB P\left[ \mcl Z_k^E \not=\emptyset   \,\big| \, \mcl F_k \right] - o_\ep^\infty(\ep) .
\eqen
Combining this with~\eqref{eqn-nomax-cond} shows that if $\mcl E_{\BB r}$ occurs, then except on an event of probability $o_\ep^\infty(\ep)$ there are at least $(1 - 4\ep^\theta )K$ values of $k\in [0,K]_{\BB Z}$ for which~\eqref{eqn-nomax-cond} and~\eqref{eqn-nomax-pos} both hold. 
\end{proof}

We now apply the estimate~\eqref{eqn-nomax-pos} to deduce Proposition~\ref{prop-nomax-quant}.

\begin{proof}[Proof of Proposition~\ref{prop-nomax-quant}]
Let $F_k$ be the event of Lemma~\ref{lem-holder-balls}, so that by Lemma~\ref{lem-nomax-msrble} we have $\left\{\mcl Z_k^{\frk E}  =\emptyset \right\} \cap F_k \in \mcl F_{k+1}$. By Lemma~\ref{lem-nomax-cond}, if $\mcl E_{\BB r}$ occurs then except on an event of probability $o_\ep^\infty(\ep)$ there are at least $(1-4\ep^\theta) K$ values of $k\in [0,K]_{\BB Z}$ for which
\eqb \label{eqn-max-quant-prob}
\BB P\left[ \left\{ \mcl Z_k^{\frk E} = \emptyset \right\} \cap F_k \,|\,\mcl F_k \right] \leq  1 -   \ep^{2\nu  + \zeta/2} ,
\eqe
equivalently, there are at most $4\ep^\theta K$ values of $k\in [0,K]_{\BB Z}$ for which
\eqbn
\BB P\left[ \left\{ \mcl Z_k^{\frk E} = \emptyset \right\} \cap F_k \,|\,\mcl F_k \right] \geq  1 -   \ep^{2\nu  + \zeta/2} .
\eqen
By Lemma~\ref{lem-uncond-to-cond} applied with $E_k = \left\{ \mcl Z_k^{\frk E} =\emptyset \right\} \cap F_k$,
$m = \lfloor (1- 4\ep^\theta) K \rfloor$, $\alpha = 1 - \ep^{2\nu + \zeta/2}$, and $\delta = \ep^{2\nu  + \zeta/2}/2$, it follows that if $\mcl E_{\BB r}$ occurs and $\ep$ is sufficiently small, then except on an event of probability at most
\eqb  \label{eqn-max-quant-hoeffding}
\exp\left( - \frac12 \ep^{4\nu  + \zeta} \lfloor (1-4\ep^\theta) K \rfloor \right)
\eqe 
 there are at most
\eqbn
K - (1-\alpha-\delta) m \leq  \left( 1 - \ep^{2\nu + \zeta/2} (1-4\ep^\theta )/ 2 \right) K \leq (1 - \ep^{2\nu+\zeta}) K
\eqen
values of $k \in [0,K]_{\BB Z}$ for which $E_k$ occurs. Equivalently, there are at least $\ep^{2\nu+\zeta} K$ values of $k\in [0,K]_{\BB Z}$ for which either $\mcl Z^{\frk E} \not= \emptyset$ or $F_k$ does not occur. By Lemma~\ref{lem-holder-balls}, on $\mcl E_{\BB r}$ the event $F_k$ occurs for every $k\in [0,K]_{\BB Z}$. 
Since $K \asymp \ep^{-\beta}$ (by~\eqref{eqn-end-time}), if $4\nu < \beta $ then for a small enough choice of $\zeta \in (0,1)$, the quantity~\eqref{eqn-max-quant-hoeffding} is of order $o_\ep^\infty(\ep)$. 
The proposition now follows.
\end{proof}

\begin{proof}[Proof of Theorem~\ref{thm-geo-iterate}]
Assume we are in the setting of the theorem statement with $\nu_* = \frac18(\beta\wedge\theta)$. 
Fix $q > 0$. 
Recall that we have been fixing $\BB z,\BB w \in \BB r U$ with $|\BB z - \BB w| \geq 4\ell\BB r$ throughout this section. 
Proposition~\ref{prop-nomax-quant} implies that if $\mcl E_{\BB r}$ occurs, then for each fixed choice of $\BB z , \BB w \in \left( \ep^q \BB r \BB Z^2 \right) \cap \left(\BB r U\right)$ with $|\BB z -\BB w| \geq 4\ell\BB r$, it holds except on an event of probability $o_\ep^\infty(\ep)$, at a rate which does not depend on $\BB z , \BB w$, or $\BB r$, that there exists $k\in [0,K]_{\BB Z}$ for which the corresponding set $\mcl Z_k^{\frk E}$ of~\eqref{eqn-good-annulus-set'} is non-empty. By~\eqref{eqn-good-annulus-set'}, this means that there exists $z\in\BB C$ and $r \in [\ep^{1+\nu} \BB r , \ep \BB r] \cap\mcl R$ such that $P^{\BB z , \BB w} \cap B_{\lambda_2 r}(z) \not=\emptyset$ and $\frk E_r^{\BB z,\BB w}(z)$ occurs. 

Since the definition of $\mcl E_{\BB r}$ does not depend on $\BB z,\BB w$, we can truncate on $\mcl E_{\BB r}$, then take a union bound over all pairs $\BB z , \BB w \in  \left( \ep^q \BB r \BB Z^2 \right) \cap \left(\BB r U\right)$ with $|\BB z -\BB w | \geq 4 \ell \BB r$, to get that if $\mcl E_{\BB r}$ occurs then the following is true except on an event of probability $o_\ep^\infty(\ep)$. For each such pair $\BB z,\BB w$ that there exists $z\in\BB C$ and $r \in [\ep^{1+\nu} \BB r , \ep \BB r] \cap\mcl R$ such that $P^{\BB z , \BB w} \cap B_{\lambda_2 r}(z) \not=\emptyset$ and $\frk E_r^{\BB z,\BB w}(z)$ occurs. 

Since the parameters in the definition of $\mcl E_{\BB r}$ can be chosen so as to make $\BB P[\mcl E_{\BB r}]$ as close to 1 as we like (Lemma~\ref{lem-reg-event-prob}), we obtain the theorem statement with $4\ell$ in place of $\ell$, which is sufficient since $\ell$ is arbitrary. 
\end{proof}

\subsection{Proofs of geometric lemmas}
\label{sec-geometric-lemma}
 
In this section we prove the geometric lemmas stated in Sections~\ref{sec-stab} and~\ref{sec-uncond-to-cond} whose proofs were postponed to avoid distracting from the main argument, namely Lemmas~\ref{lem-dc-set}, \ref{lem-geodesic-diam}, and~\ref{lem-holder-balls}. The arguments in this section use only the definitions in Sections~\ref{sec-iterate-setup} and~\ref{sec-iterate-reg}. 
In particular, we do not use any of the results in Sections~\ref{sec-stab} or~\ref{sec-uncond-to-cond}.

\begin{figure}[t!]
 \begin{center}
\includegraphics[scale=.8]{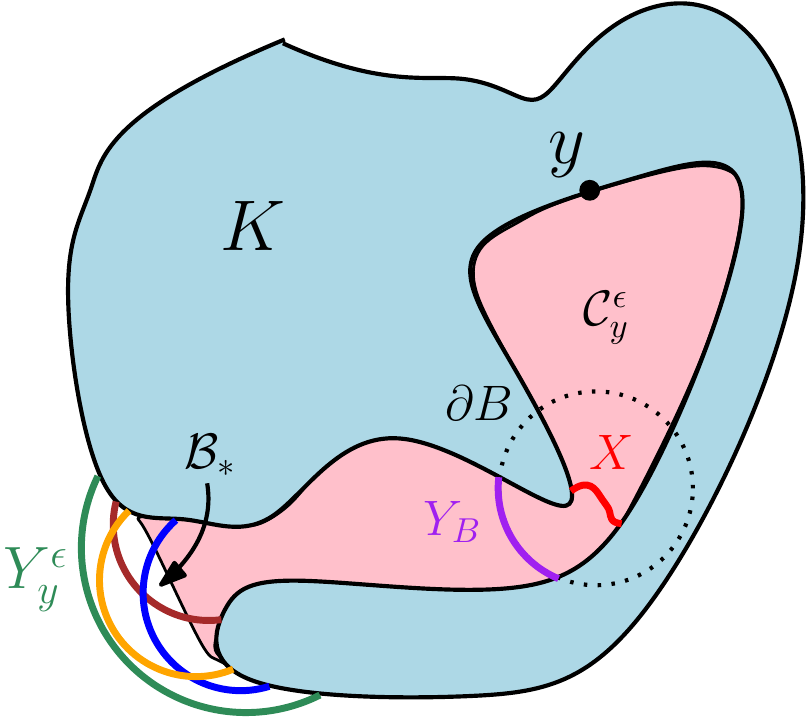}
\vspace{-0.01\textheight}
\caption{Illustration of the proof of Lemma~\ref{lem-dc-set}. 
The set $\mcl C_y^\ep$ is shown in pink. We have shown the boundary of a (non-maximal) ball $B\in\mcl B$ as a dashed line and the associated arc $Y_B\subset\bdy B\setminus K$ in purple. Each set $X$ as in the lemma statement is contained in such a ball $B$ and lies in the bounded connected component $U_B$ of $\BB C\setminus (Y_B\cup K)$. 
Several arcs $Y_B$ for maximal balls $B\in\mcl B_*$ are shown in various colors. Any two such arcs must intersect each other, so the Euclidean diameter of their union is at most $8\ep$. The set $Y_y^\ep$ (green) in the lemma statement is chosen so as to disconnect this union from $\infty$ in $\BB C\setminus K$. 
}\label{fig-dc-set}
\end{center}
\vspace{-1em}
\end{figure} 
  
\begin{proof}[Proof of Lemma~\ref{lem-dc-set}]
See Figure~\ref{fig-dc-set} for an illustration. The proof consists of two main steps. 
\begin{enumerate}
\item We show that there is a \emph{finite} collection of connected sets $X\subset \BB C\setminus \mcl K$ with Euclidean diameter at most $4\ep$ such that each point of $\mcl C_y^\ep$ is contained in the bounded connected component of $\BB C\setminus ( \mcl K \cup X)$ for one of these sets $X$. The sets $X$ can be taken to be appropriate boundary arcs of Euclidean balls of radius $2\ep$.
\item We consider the maximal elements of our finite collection, i.e., those which do not lie in a bounded connected component of any other set in the collection. 
We show that any two maximal elements have to intersect, so the union of the maximal elements has Euclidean diameter at most $8\ep$. We then choose a single connected set (which can be taken to be an arc of a Euclidean ball of radius $8\ep$) which disconnects the union of the maximal elements from $\infty$ in $\BB C\setminus \mcl K$.  
\end{enumerate}
\medskip

\noindent\textit{Step 1: reducing to finitely many arcs of Euclidean balls.}
We will first reduce to considering only a finite collection of sets $X$ as in the statement of the lemma by looking at arcs of Euclidean balls.
Let $\mcl B$ be the set of closed Euclidean balls of the form $B = \ol{B_{2\ep}(z)}$ for $z\in \frac{\ep}{4} \BB Z^2$ with the following properties: $B \cap \bdy \mcl K\not=\emptyset$ and every unbounded connected subset of $\BB C\setminus \mcl K$ whose prime end closure contains $y$ has to intersect $B$.  Since $\mcl K$ is compact, $\mcl B$ is a finite set. 

For $B\in\mcl B$, the set $\bdy B\setminus \mcl K$ is a countable union of open arcs of $\bdy B$. 
Each such arc divides $\BB C\setminus \mcl K$ into a bounded connected component and an unbounded connected component. 
There is one such arc $Y_B$ with the property that $y$ lies on the boundary of the bounded connected component of $\BB C\setminus (\mcl K\cup Y_B )$ and $Y_B $ is not contained in the bounded connected component of $\BB C\setminus (\mcl K\cup X)$ for any other such arc $X\not=Y_B$.  
Note that since $B$ has radius $2\ep$, the arc $Y_B$ is connected and has Euclidean diameter at most $4\ep$. 

For $B\in\mcl B$, let $U_B$ be the bounded connected component of $\BB C\setminus (\mcl K\cup Y_B)$ so that $y\in \bdy U_B$. 
We claim that 
\eqb \label{eqn-ball-arc-cover}
\forall z\in \mcl C_y^\ep, \quad \exists B\in\mcl B \quad \text{such that} \quad z \in U_B .
\eqe 
Indeed, let $X$ be as in the definition of $\mcl C_y^\ep$ for our given $z$ and let $V_X$ be the bounded connected component of $\BB C\setminus X$ with $y$ on its boundary. 
Since $X$ has Euclidean diameter at most $\ep$, we can find $B\in \mcl B$ such that $X$ is contained in the interior of $B$.  
We claim that $V_X\subset U_B$, and hence $z\in U_B$.  
Since $X $ is connected and $X\cap Y_B\subset X\cap \bdy B=\emptyset$, it follows that $X$ is either entirely contained in $U_B$ or $X$ is entirely contained in the unbounded connected component of $\BB C\setminus (\mcl K\cup Y_B)$. 
We claim that $X$ cannot be entirely contained in the unbounded connected component of $\BB C\setminus (\mcl K\cup U_B)$. 
Indeed, by the definition of $X$, each unbounded connected subset of $\BB C\setminus \mcl K$ with $y$ on its boundary must intersect $X $.
Since $X \cap U_B =\emptyset$ and $y\in\bdy U_B$, each unbounded connected subset of $\BB C\setminus \mcl K$ which intersects $U_B$ must intersect $X$. 
This implies that $U_B\subset V_X$, but this cannot happen since $X\subset B$ and by the definition of $Y_B$. 
Therefore $X\subset U_B$, so $V_X\subset U_B$, so~\eqref{eqn-ball-arc-cover} holds.
\medskip

\noindent\textit{Step 2: maximal elements of $\mcl B$.}
We define a partial order on $\mcl B$ by declaring that $B\preceq B'$ if and only if $U_B\subset U_{B'}$. 
Let $\mcl B_*$ be the set of maximal elements of $\mcl B$, i.e., $B_* \in \mcl B_*$ if and only if there is no $B \in \mcl B \setminus \{B_*\}$ such that $B_* \preceq B$.
Since $\mcl B$ is a finite set, for every $B\in\mcl B$ there exists $B_* \in\mcl B_* $ satisfying $B\preceq B_*$. 
  
We claim that if $B_1,B_2 \in\mcl B_*$, then $Y_{B_1}\cap Y_{B_2}\not=\emptyset$. 
Indeed, if $Y_{B_1}\cap Y_{B_2} =\emptyset$ then $Y_{B_1} $ is contained in either $U_{B_2}$ or in the unbounded connected component of $\BB C\setminus ( Y_{B_2} \cup \mcl K)$.
By the maximality of $B_1$, $Y_{B_1}$ must be contained in the unbounded connected component of $\BB C\setminus (\mcl K\cup Y_{B_2})$.
We will now argue that $U_{B_2} \subset U_{B_1}$, which will contradict the maximality of $B_2$. 
Indeed, by the definition of $Y_{B_1}$, every unbounded connected subset of $\BB C\setminus \mcl K$ whose prime end closure contains $y$ has to intersect $Y_{B_1}$.
Since $Y_{B_1}$ is disjoint from $U_{B_2}$ and $y\in \op{Cl}'( U_{B_2})$, it follows that every unbounded connected subset of $\BB C\setminus \mcl K$ which intersects $U_{B_2}$ has to intersect $Y_{B_1}$. 
Therefore, $U_{B_2} \subset U_{B_1}$, which gives the desired contradiction.

Since each set $Y_B$ for $B\in\mcl B$ has Euclidean diameter at most $4\ep$, the preceding paragraph implies that the set $\wt Y_y^\ep := \ol{\bigcup_{B_* \in\mcl B_*} Y_{B_*}}$ is connected and has Euclidean diameter at most $8\ep$. 
Choose a Euclidean ball $\wt B$ of radius at most $8\ep$ which contains $\wt Y_y^\ep$. As in Step 1, there is a unique connected arc $Y_y^\ep$ of $\bdy \wt B\setminus \mcl K$ with the property that $y$ lies on the boundary of the bounded connected component of $\BB C\setminus (\mcl K\cup Y_y^\ep )$ and $Y_y^\ep$ is not contained in the bounded connected component of $\BB C\setminus (\mcl K\cup X)$ for any other such arc $X$. This arc $Y_y^\ep$ has Euclidean diameter at most $16\ep$. 
Then each $Y_{B_*}$ for $B_*\in\mcl B_*$, and hence also each $U_{B_*}$ for $B_*\in\mcl B_*$, is contained in the bounded connected component of $\BB C\setminus (\mcl K\cup Y_y^\ep)$. Since each $z\in \mcl C_y^\ep$ is contained in $U_B$ for some $B\in\mcl B$, and hence in $U_{B_*}$ for some $B_*\in\mcl B_*$, we get that $Y_y^\ep$ satisfies the desired property.  
\end{proof}

We now turn our attention to Lemmas~\ref{lem-geodesic-diam} and~\ref{lem-holder-balls}.
Both lemmas will be proven using the following statement, which in particular gives an explicit definition of the event $F_k$ of Lemma~\ref{lem-holder-balls}.

\begin{figure}[t!]
 \begin{center}
\includegraphics[scale=.8]{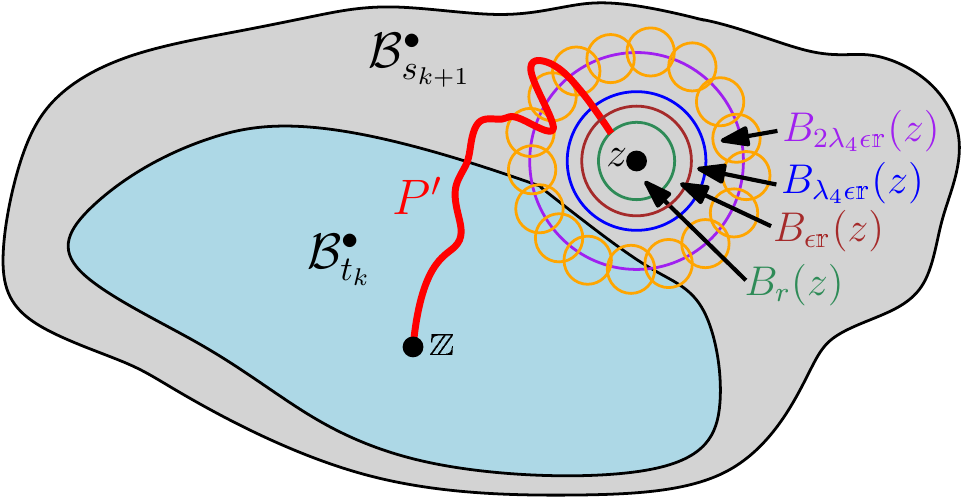}
\vspace{-0.01\textheight}
\caption{
Illustration of the statement and proof of Lemma~\ref{lem-holder-balls0}. 
In order to upper-bound $\sup_{u \in \bdy B_{2\lambda_4 \ep \BB r}(z)} D_h\left( \BB z , u ; \mcl B_{s_{k+1}}^\bullet \setminus \ol{B_{\ep\BB r}(z)} \right)$, we cover $\bdy B_{2\lambda_4 \ep\BB r}(z)$ by Euclidean balls of radius $\ep \BB r/2$ (orange) and upper-bound the $D_h$-diameters of these balls using condition~\ref{item-holder-cont} (H\"older continuity) in the definition of $\mcl E_{\BB r}$. Each of these balls is disjoint from $B_{\ep\BB r}(z)$ and is contained in $\mcl B_{s_{k+1}}^\bullet$, which leads to~\eqref{eqn-holder-balls}. Using Lemma~\ref{lem-holder-balls0} we get an upper bound for the $D_h$-length of the segment of a $D_h(\cdot,\cdot ; \BB C\setminus \ol{B_r(z)})$-geodesic from $\BB z$ to a point of $\bdy B_r(z)$ (such as the one shown in red) stopped at the last time it hits $\bdy B_{2\lambda_4 \ep\BB r}(z)$. This upper bound allows us to prevent such a $D_h$-geodesic from exiting $\mcl B_{s_{k+1}}^\bullet$. These considerations lead to the proofs of Lemmas~\ref{lem-geodesic-diam} and~\ref{lem-holder-balls}. 
}\label{fig-holder-balls}
\end{center}
\vspace{-1em}
\end{figure}

\begin{lem} \label{lem-holder-balls0}
For $k\in [0,K]_{\BB Z}$, let $F_k$ be the event that the following is true. We have $s_{k+1} \leq \tau_{2 \ell\BB r}$ and for each $z \in B_{2\lambda_4 \ep \BB r}(\mcl B_{t_k}^\bullet) \setminus B_{\lambda_4 \ep \BB r}(\mcl B_{t_k}^\bullet)$,  
\eqb \label{eqn-holder-balls}
\sup_{u \in \bdy B_{2\lambda_4 \ep \BB r}(z)} D_h\left( \BB z , u ; \mcl B_{s_{k+1}}^\bullet \setminus \ol{B_{  \ep\BB r}(z)} \right) \leq t_k +c  \ep^\chi \frk c_{\BB r} e^{\xi h_{\BB r}(\BB z)}   ,
\eqe
where $\lambda_4$ is the constant from Theorem~\ref{thm-geo-iterate}, $\chi$ is as in condition~\ref{item-holder-cont} (H\"older continuity) in the definition of $\mcl E_{\BB r}$, and $c > 0$ is constant depending only on $ a , \lambda_4$ (which we do not make explicit).  
If $\mcl E_{\BB r}$ occurs and $\ep$ is sufficiently small (how small depends only on $a, \lambda_4 $), then $F_k$ occurs for each $k\in [0,K]_{\BB Z}$. 
\end{lem}

The reason why we use internal distances in $\mcl B_{s_{k+1}}^\bullet\setminus \ol{B_{\ep\BB r}(z)}$ in~\eqref{eqn-holder-balls} is as follows. Such distances are bounded above by $D_h(\cdot,\cdot ; \BB C\setminus \ol{B_r(z)})$-distances if $r \leq \ep\BB r$ (which is the case if $(z,r) \in \mcl Z_k$), which will be important for controlling $D_h(\cdot,\cdot ; \BB C\setminus \ol{B_r(z)})$-geodesics. Furthermore, such distances are determined by $(\mcl B_{s_{k+1}}^\bullet , h|_{\mcl B_{s_{k+1}^\bullet}})$ by Axiom~\ref{item-metric-local} (locality), which will be important for the proof of Lemma~\ref{lem-holder-balls}. 
We also emphasize that the right side of~\eqref{eqn-holder-balls} is smaller than $s_{k+1}  = t_k +  (\ep^\beta -\ep^{2\beta})\frk c_{\BB r} e^{\xi h_{\BB r}(\BB z)}$ if $\ep$ is small since $\beta < \chi$.

\begin{proof}[Proof of Lemma~\ref{lem-holder-balls0}]
See Figure~\ref{fig-holder-balls} for an illustration of the statement and proof. 
Assume that $\mcl E_{\BB r}$ occurs. 
By~\eqref{eqn-geo-iterate-upper}, we have $s_{k+1}  \leq \tau_{2\ell\BB r}$. 
Hence we just need to check~\eqref{eqn-holder-balls}. 
By the definition~\eqref{eqn-geo-iterate-times} of $t_k$ and $s_{k+1}$ and since $\beta <\chi/\chi' < \chi $ (by~\eqref{eqn-beta-compare}), it holds for small enough $\ep \in (0,1)$ that
\eqb \label{eqn-holder-balls-lower}
D_h(\bdy\mcl B_{t_k}^\bullet , \bdy\mcl B_{s_{k+1}}^\bullet) \geq (\ep^\beta - \ep^{2\beta}) \frk c_{\BB r} e^{\xi h_{\BB r}(\BB z)} >  \ep^\chi \frk c_{\BB r} e^{\xi h_{\BB r}(\BB z)} .
\eqe
Note that $\chi' > \xi(Q+2) \geq 1$, where the last inequality follows, e.g., from the fact that $1 - \xi Q \leq 2\xi$, which is obvious from the definition of LFPP and an estimate for the maximum of $h_\ep^*$ on a bounded open set.

For $z \in B_{2 \lambda_4 \ep \BB r}(\mcl B_{t_k}^\bullet) \setminus B_{\lambda_4\ep\BB r}(\mcl B_{t_k}^\bullet)$,
the Euclidean circle $\bdy B_{2\lambda_4 \ep\BB r}(z)$ intersects $\bdy\mcl B_{t_k}^\bullet$.
We can cover $\bdy B_{2  \lambda_4 \ep \BB r}(z)$ by a $\lambda_4$-dependent constant number of Euclidean balls of the form $B_{\ep \BB r/2}(w)$ for $w\in\bdy B_{2\lambda_4 \ep \BB r}(z)$.  
Note that since $\lambda_4 \geq 1$, the corresponding balls $B_{\ep\BB r}(w)$ are disjoint from $B_{\lambda_4\ep \BB r}(z) \supset B_{\ep\BB r}(z)$. 
By the upper bound for $D_h$-distances from condition~\ref{item-holder-cont} in the definition of $\mcl E_{\BB r}$ and then condition~\ref{item-circle-avg}  (comparison of circle averages) in the definition of $\mcl E_{\BB r}$, each such ball satisfies 
\eqb \label{eqn-holder-balls-bdy}
\sup_{u,v\in B_{\ep \BB r/2}(w)} D_h(u,v ; B_{\ep\BB r}(w)) \leq 2(\ep/2)^\chi \frk c_{\BB r} e^{\xi h_{\BB r}(0)} \preceq \ep^\chi \frk c_{\BB r} e^{\xi h_{\BB r}(\BB z)} ,
\eqe 
with the implicit constant depending only on $a$. 

By summing~\eqref{eqn-holder-balls-bdy} over all such balls $B_{\ep\BB r/2}(w)$, using that $\bdy B_{2\lambda_4 \ep\BB r}(z) \cap \bdy\mcl B_{t_k}^\bullet \not=\emptyset$, 
and comparing to~\eqref{eqn-holder-balls-lower}, we get that for small enough $\ep$ each such ball $B_{\ep\BB r}(w)$ is contained in $\mcl B_{s_{k+1}}^\bullet \setminus \ol{B_{\ep\BB r}(z)}$. 
We deduce that the $D_h\left( \cdot,\cdot; \mcl B_{s_{k+1}}^\bullet \setminus \ol{B_{\ep \BB r}(z)} \right)$-diameter of $\bdy B_{2 \lambda_4 \ep \BB r}(z)$ is at most a $a,\lambda_4$-dependent constant times $\ep^\chi \frk c_{\BB r} e^{\xi h_{\BB r}(\BB z)} $. 
Since $\bdy B_{2 \lambda_4  \ep \BB r}(z)\cap \bdy\mcl B_{t_k}^\bullet \not=\emptyset$, we get that the left side of~\eqref{eqn-holder-balls} is at most $t_k + c \ep^\chi   \frk c_{\BB r} e^{\xi h_{\BB r}(\BB z)}  $ for an appropriate constant $c$. 
\end{proof}

\begin{proof}[Proof of Lemma~\ref{lem-geodesic-diam}]
Assume that $\mcl E_{\BB r} $ occurs and let $P'$ be a $D_h(\cdot,\cdot; \ol{B_r(z)})$-geodesic from $\BB z$ to a point of $\bdy B_r(z)$, as in the statement of the lemma. 
Let $t' \in [t_k , |P'|]_{\BB Z}$ be the last time that $P'$ hits $\bdy B_{2 \lambda_4 \ep \BB r}(z)$. 
Since $\ol{B_r(z)}$ is disjoint from $\mcl B_{t_k}^\bullet$, the segment $P'|_{[0,t_k]}$ is a $D_h$-geodesic and $P'$ does not re-enter $\mcl B_{t_k}^\bullet$ after time $t_k$. 
By~\eqref{eqn-holder-balls} of Lemma~\ref{lem-holder-balls} and since $P'$ is $D_h(\cdot,\cdot ; \BB C\setminus \ol{B_r(z)})$-geodesic, it follows that the $D_h$-length of $P'|_{[0,t']}$ (which equals $t'$) is at most $t_k + c \ep^\chi \frk c_r e^{\xi h_r(\BB z)}$. Therefore, the $D_h$-length of $P'([t_k ,t'])$ is at most $c\ep^\chi \frk c_r e^{\xi h_r(\BB z)}$. 
By conditions~\ref{item-holder-cont} (H\"older continuity) and~\ref{item-circle-avg} (comparison of circle averages) in the definition of $\mcl E_{\BB r}$, the Euclidean diameter of $P'([t_k,t'])$ is at most a $a,\lambda_4$-dependent constant times $\ep^{\chi/\chi'}\BB r$. 
Since $P'([t',|P'|])\subset B_{2\lambda_4 \ep \BB r}(z)$, we obtain~\eqref{eqn-geodesic-diam}. 
\end{proof}

\begin{proof}[Proof of Lemma~\ref{lem-holder-balls}]
Define $F_k$ as in Lemma~\ref{lem-holder-balls0}. That lemma tells us that $\mcl E_{\BB r} \subset \bigcap_{k=0}^K F_k$ for small enough $\ep \in (0,1)$ (depending only on $a,\lambda_4$). 
Furthermore, it is clear from the definition of $F_k$ and Axiom~\ref{item-metric-local} (locality) that $F_k \in \sigma\left(\mcl B_{s_{k+1}}^\bullet, h|_{\mcl B_{s_{k+1}}^\bullet} \right)$. 
Now assume that $F_k$ occurs.
By definition, we have $s_k \leq \tau_{2\ell\BB r}$. 
We consider $(z,r) \in \mcl Z_k$ and check that if $\ep \in (0,1)$ is small enough, then $B_{\lambda_4 r}(z) \subset \mcl B_{s_{k+1}}^\bullet$ and the set of $D_h(\cdot,\cdot ; \BB C\setminus \ol{B_r(z)})$-geodesics from $\BB z$ to points of $\bdy B_r(z)$ is determined by $(\mcl B_{s_{k+1}}^\bullet , h|_{\mcl B_{s_{k+1}}^\bullet})$. 
 
Note that the right side of~\eqref{eqn-holder-balls} satisfies $t_k  + c  \ep^\chi \frk c_{\BB r} e^{\xi h_{\BB r}(\BB z)}   \leq s_{k+1}$.  
Since the left side of~\eqref{eqn-holder-balls} is an upper bound for $\sup_{u\in \bdy B_{2 \lambda_4 \ep \BB r}(z)} D_h(\BB z,u)$, it follows that $\bdy B_{2 \lambda_4 \ep \BB r}(z) \subset \mcl B_{s_{k+1}}^\bullet$. Since $B_{\lambda_4 r}(z) \subset B_{\lambda_4 \ep\BB r}(z)$ (by~\eqref{eqn-good-annulus-set0}) and $\mcl B_{s_{k+1}}^\bullet$ contains every point which it disconnects from $\infty$, we therefore have $B_{\lambda_4 r}(z) \subset \mcl B_{s_{k+1}}^\bullet$. 
 
Finally, we claim that a $D_h(\cdot,\cdot ; \BB C\setminus \ol{B_r(z)})$-geodesic from $\BB z$ to a point of $\bdy B_r(z)$ is the same as a $D_h(\cdot,\cdot ; \mcl B_{s_{k+1}}^\bullet \setminus \ol{B_r(z)})$-geodesic from $\BB z$ to a point of $\bdy B_r(z)$, which gives the desired measurability statement due to Axiom~\ref{item-metric-local} for $D_h$. 
To see this, it suffices to show that if $P'$ is a $D_h(\cdot,\cdot ; \BB C\setminus \ol{B_r(z)})$-geodesic from $\BB z$ to a point of $\bdy B_r(z)$, then $P' \subset \mcl B_{s_{k+1}}^\bullet$.

To this end, let $t$ be the last time that $P'$ hits $\bdy B_{2  \lambda_4 \ep \BB r}(z)$. 
By~\eqref{eqn-holder-balls} and since $P'$ is a $D_h(\cdot,\cdot ; \BB C\setminus \ol{B_r(z)})$-geodesic, it follows that the $D_h$-length of $P'|_{[0,t]}$ (which equals $t$) is at most $t_k +c  \ep^\chi \frk c_{\BB r} e^{\xi h_{\BB r}(\BB z)} < s_{k+1}$. 
Consequently, $P'$ cannot exit $\mcl B_{s_{k+1}}^\bullet$ before time $t$. Since $t$ is the \emph{last} time that $P'$ hits $\bdy B_{2 \lambda_4 \ep  \BB r}(z)$ and the terminal point of $P'$ is contained in $\bdy B_r(z) \subset B_{ \ep \BB r}(z)$, $P'$ cannot exit $\mcl B_{s_{k+1}}^\bullet$ after time $t$, either.
\end{proof}

\section{Forcing a geodesic to take a shortcut}
\label{sec-geodesic-shortcut}

The goal of this section is to prove Proposition~\ref{prop-geo-event}. Throughout, we assume that we are in the setting of Theorem~\ref{thm-weak-uniqueness}, so $D$ and $\wt D$ are two weak $\gamma$-LQG metrics with the same scaling constants. We also let $h$ be a whole-plane GFF and we implicitly assume (by way of eventual contradiction) that the optimal bi-Lipschitz constants $c_*$ and $C_*$ of~\eqref{eqn-max-min-def} satisfy $c_* < C_*$. 

With $\nu_*$ as in Theorem~\ref{thm-geo-iterate}, fix $0 < \mu < \nu \leq \nu_*$ and let $\alpha_* \in (1/2,1)$ and $p_0 \in (0,1)$ be the parameters from Proposition~\ref{prop-attained-good'} for this choice of $\mu$ and $\nu$ (we write $p_0$ instead of $p$ to avoid confusion with another parameter called $p$ below). 
Also fix $\alpha\in [\alpha_* ,1)$ (to be chosen in Lemma~\ref{lem-endpoint-geodesic} just below) and parameters $c_1' ,c_2' $ such that $c_* < c_1' < c_2' < C_*$. 

Let $\mcl R_0$ be the set of $r  > 0$ for which it holds with probability at least $p_0$ that the following is true.
There exists $u \in \bdy B_{\alpha r}(0)$ and $v \in \bdy B_r(0)$ such that 
\eqb \label{eqn-attained-good-event}
\wt D_h(u,v) \leq c_1' D_h(u,v) 
\eqe 
and the $\wt D_h$-geodesic from $u$ to $v$ is unique and is contained in $\ol{\BB A_{\alpha r , r}(0)}$. 
We note that Proposition~\ref{prop-attained-good} implies in particular that for each $\BB r  >0$ one has $\#(\mcl R_0\cap [\ep^{1+\nu} \BB r , \ep \BB r] \cap \{8^{-k} \BB r\}_{k\in\BB N}) \geq\mu\log_8\ep^{-1}$ for small enough $\ep \in (0,1)$. 

\subsection{Outline of the proof of Proposition~\ref{prop-geo-event}}
\label{sec-shortcut-outline}

The main task in the proof of Proposition~\ref{prop-geo-event} is to define the event $E_r(0)$ (which we abbreviate as $E_r$ throughout most of this section). The other events $E_r(z)$ for $z\in\BB C$ will be defined by translation. 
\medskip

\noindent\textbf{Main ideas.}
The basic idea to define $E_r $ is as follows. 
We will define for each pair of points $x' , y' \in \bdy B_{3r}(0)$ a deterministic smooth bump function $\phi $ which takes a large (but independent of $r,x',y'$) value in a long, narrow ``tube" contained in $B_{3r}(0)$ which (almost) contains a path from $x'$ to $y'$ and which vanishes outside of a small neighborhood of this tube. 
Roughly speaking, $E_r$ will be the event that, simultaneously for every choice of $x'$ and $y'$, this tube contains a pair of points $u,v$ such that $\wt D_h(u,v) \leq c_1' D_h(u,v)$ and $|u-v| \asymp r$; and several regularity conditions hold. 
We will show using Proposition~\ref{prop-attained-max} and basic estimates for LQG distances that when $\rho \in (0,1)$ is small (but independent of $r$), $\BB P[E_r]$ is close to 1 for all $r\in \rho^{-1} \mcl R_0$ (Lemma~\ref{lem-geo-event-prob}).

We will then consider a fixed pair of points $\BB z,\BB w \in \BB C \setminus B_{4r}(0)$ and let $x'$ and $y'$ be the first points of $\bdy B_{3r}(0)$ hit by the $D_h$-metric balls grown from $\BB z$ and $\BB w$, respectively. This choice of $x'$ and $y'$ (and hence also the corresponding bump function $\phi$) are random, but are determined by $h|_{\BB C\setminus B_{3r}(0)}$. 
We will show that if $E_r$ occurs and the $D_h$-geodesic between $\BB z$ and $\BB w$ enters $B_{2r}(0)$, then the $D_{h-\phi}$-geodesic between $\BB z$ and $\BB w$ has to stay close to the long narrow tube where $\phi$ is large, and hence has to get close to points $u,v$ with $\wt D_h(u,v) \leq c_1' D_h(u,v)$ and $|u-v| \asymp r$. Essentially, this is because Axiom~\ref{item-metric-f} (Weyl scaling) implies that subtracting $\phi$ makes distances inside the tube much shorter than distances outside. 
If we let $\frk E_r^{\BB z,\BB w}(0)$ be the event that the $D_h$-geodesic gets close to such points $u,v$, then since the conditional laws of $h-\phi$ and $h$ given $h|_{\BB C\setminus B_{3r}(0)}$ are mutually absolutely continuous (and we can add regularity conditions to $E_r$ to control the Radon-Nikodym derivative), we get condition~\ref{item-geo-iterate-compare} in Theorem~\ref{thm-geo-iterate} (with $\lambda_3 = 3$).

We emphasize that the event $E_r$ does \emph{not} include the condition that $P$ stays in the long narrow tube where $\phi$ is large. Indeed, $E_r$ cannot include any conditions which depend on $P$ since $E_r$ needs to be locally determined by $h$. Rather, as explained in the preceding paragraph, if $E_r$ occurs then we can force $P$ to stay in the tube by subtracting the bump function $\phi$ from $h$.
\medskip
 
\noindent\textbf{Section~\ref{sec-geo-event-proof}.}
We give a precise statement of the properties that we need the event $E_r$ and the bump function $\phi$ described above to satisfy.
We then assume the existence of these objects and deduce Proposition~\ref{prop-geo-event}. 
Condition~\ref{item-geo-iterate-dense} of Theorem~\ref{thm-geo-iterate} (with $\mcl R = \rho^{-1}\mcl R_0$) is true in our framework by the definition of $\mcl R_0$ and Proposition~\ref{prop-attained-good}. Conditions~\ref{item-geo-iterate-msrble} and~\ref{item-geo-iterate-prob} are true by assumption (these conditions will be clear from the construction of $E_r$ and $\phi$). Condition~\ref{item-geo-iterate-compare} is proven by comparing the conditional laws of $h$ and $h-\phi$ given $h|_{\BB C\setminus B_{3r}(0)}$, as discussed above. 
The rest of the section is devoted to constructing the event $E_r$ and the bump functions $\phi$. 
\medskip

\noindent\textbf{Section~\ref{sec-endpoint-geodesic}.}
We first show that for any $z\in\BB C$ and $r\in\mcl R_0$, we can find a \emph{deterministic} open ``tube" $V_r(z) \subset B_{3r}(z)$ such that with uniformly positive probability over the choice of $z$ and $r$, there are points $u,v \in V_r(z)$ with the following properties. We have $\wt D_h(u,v) \leq c_1' D_h(u,v) $, $|u-v|\asymp r$, the $\wt D_h$-geodesic from $u$ to $v$ is contained in $V_r(z)$, and any path in $V_r(z)$ between $z-2r$ and $z+2r$ has to get close to each of $u$ and $v$ (Lemma~\ref{lem-deterministic-geodesic}). This is illustrated in Figure~\ref{fig-deterministic-geodesic}. 

To do this, we start with a pair of points $u,v$ as in the definition of $\mcl R_0$, but with $z$ in place of 0. Such a pair of points exists with probability at least $p_0$ by Axiom~\ref{item-metric-translate} (translation invariance). We then extend the $\wt D_h$-geodesic $\wt P$ from $u$ to $v$ to a path $\wt P'$ from $z-2r$ to $z+2r$ by concatenating $\wt P$ with smooth paths. For this purpose, the fact that $\wt P$ is contained in $\ol{\BB A_{\alpha r , r}(z)}$ is useful to ensure that the extra smooth paths intersect $\wt P$ only at $u$ and $v$. We consider the set of squares in a fine grid which intersect $\wt P'$. Since there are only finitely many possibilities for this set of squares, there has to be a deterministic set of squares which equals the set of squares which intersect $\wt P'$ with uniformly positive probability. We define $V_r(z)$ to be the interior of the union of the squares in this set.
\medskip

\noindent\textbf{Section~\ref{sec-highprob-geodesic}.}
We now have an event which satisfies many of the conditions which we are interested in, but it holds only with uniformly positive probability, not with probability close to 1. To get an event which holds with probability close to 1, we consider a small but fixed $\rho \in (0,1)$ and a radius $r\in \rho^{-1}\mcl R_0$. We can find a large number of disjoint balls of the form $B_{\rho r}(z)$ contained in $B_{2r}(0)$ (note that $\rho r\in \mcl R_0$). By the spatial independence properties of the GFF (Lemma~\ref{lem-spatial-ind}), if we make $\rho$ sufficiently small then it holds with high probability that the event of the preceding subsection occurs for a large number of these balls $B_{\rho r}(z)$. We then link up the corresponding sets $V_{\rho r}(z)$ by deterministic paths of squares to find a deterministic open ``tube" $U_r^{x,y}$ joining any two given points of $x,y\in \bdy B_{2r}(0)$ with the following property. With probability close to 1, there are points $u,v \in U_r^{x,y}$ such that $\wt D_h(u,v) \leq c_1' D_h(u,v)$, $|u-v|\asymp r$, the $\wt D_h$-geodesic from $u$ to $v$ is contained in $U_r^{x,y}$, and any path in $U_r^{x,y}$ between $x$ and $y$ has to get close to each of $u$ and $v$ (Lemma~\ref{lem-highprob-geodesic}). See Figure~\ref{fig-highprob-geodesic} for an illustration of this part of the argument. 
\medskip

\noindent\textbf{Section~\ref{sec-E-def}.}
Taking Lemma~\ref{lem-highprob-geodesic} as our starting point, we then build the high-probability event $E_r$ in Proposition~\ref{prop-geo-event} for $r\in \rho^{-1} \mcl R_0$. 
In addition to the aforementioned conditions on the tube $U_r^{x,y}$, we also include extra regularity conditions which will eventually be used to prevent $D_h$-geodesics from staying close to the boundary of $U_r^{x,y}$ without entering it, to get geodesics from $\bdy B_{3r}(0)$ to $\bdy B_{2r}(0)$, and to control the Radon-Nikodym derivative between the conditional law of $h$ and $h-\phi$ (where $\phi$ is the bump function mentioned above) given $h|_{\BB C \setminus B_{3r}(0)}$. 
We also give a precise definition of the bump function $\phi$ which we will subtract from the field: it is equal to a large positive constant on the long narrow tube $U_r^{x,y}$, it is equal to an even larger constant on even narrower tubes which approximate each of the segments $[x,3x/2]$ and $[y,3y/2]$, and it vanishes outside of a small neighborhood of the union of $U_r^{x,y}$ and these two narrower tubes. The definitions of these objects are illustrated in Figure~\ref{fig-internal-geo}.  
\medskip

\noindent\textbf{Section~\ref{sec-bump-function}.} 
We prove that a $D_{h-\phi}$-geodesic is likely to get near points $u,v$ satisfying~\eqref{eqn-attained-good-event}, using the definition of $E_r(0)$ and deterministic arguments to compare various distances. A key point here is that we have set things up so that on $E_r$, the $\wt D_h$-geodesic from $u$ to $v$ is contained in $U_r^{x,y}$ and is far away from the narrow tubes where $\phi$ is larger than it is on $U_r^{x,y}$. 
This allows us to show that subtracting $\phi$ does not change the fact that $\wt D_h(u,v) \leq c_1' D_h(u,v) $.

\begin{remark}
Our proof only shows that the $D_h$-geodesic $P$ gets close to each of the points $u$ and $v$ from~\eqref{eqn-attained-good-event} with positive probability (we then use the triangle inequality to compare the $D_h$-length of a segment of $P$ to $D_h(u,v)$). We do \emph{not} show that $P$ actually merges into the $D_h$-geodesic from $u$ to $v$. We believe that it should be possible to show that $P$ merges into this $D_h$-geodesic, but doing so is highly non-trivial. Indeed, this is closely related to the problem of showing that there are no ``ghost geodesics" for $D_h$ which do not merge into any other $D_h$-geodesics; see~\cite[Section 1.4]{akm-geodesics} for some discussion about the analogous problem in the setting of the Brownian map. Because we do not show that $P$ merges into the $D_h$-geodesic from $u$ to $v$, the arguments of this section do not immediately imply other statements of the form ``if an event occurs for some (random) geodesic with high probability, then with high probability it occurs somewhere along the $D_h$-geodesic between two fixed points".
\end{remark}

\subsection{Proof of Proposition~\ref{prop-geo-event} assuming the existence of events and functions}
\label{sec-geo-event-proof}

In this subsection, we assume the existence of an event $E_r = E_r(0)$ and a collection  $\mcl G_r$ of smooth bump functions $\phi$ which satisfy a few simple properties and deduce Proposition~\ref{prop-geo-event} from the existence of these objects.
The later subsections are devoted to constructing these objects.  
In particular, we will deduce Proposition~\ref{prop-geo-event} from the following proposition.

\begin{prop} \label{prop-geo-event0}
Let $ 0 <\mu < \nu \leq \nu_*$ be as above and let $\BB p \in (0,1)$. 
There exists $\rho \in (0,1)$, depending only on $\BB p , \mu , \nu $, such that for each $r \in \rho^{-1} \mcl R_0$, there is an event $E_r$ and a finite collection $\mcl G_r$ of smooth bump functions, each of which is supported on a compact subset of $\BB A_{r/4,3r}(0)$, with the following properties.
\begin{enumerate}[(A)] 
\item \emph{(Measurability and high probability)} We have $E_r \in \sigma\left( (h - h_{5r}(0) ) |_{\BB A_{r/4,4r}(0)} \right) $ and $\BB P[E_r] \geq \BB p$. \label{item-geo-event0} 
\item \emph{(Bound for Dirichlet inner products)} There is a deterministic constant $\Lambda_0  > 0$ depending only on $\BB p , \mu , \nu ,c_1',c_2' $ such that, writing $(\cdot,\cdot)_\nabla$ for the Dirichlet inner product, it holds on $E_r$ that \label{item-dirichlet-bound}
\eqb \label{eqn-dirichlet-bound}
|(h,\phi )_\nabla |  + \frac12 (\phi  , \phi )_\nabla \leq \Lambda_0 ,\quad\forall \phi\in \mcl G_r .
\eqe
\item \emph{(Subtracting a bump function forces a geodesic to take a shortcut)} Suppose we are given points $\BB z , \BB w \in \BB C\setminus B_{4r}(0)$.
There is a random $\phi \in \mcl G_r$ depending only on $\BB z$, $\BB w$, and $h|_{\BB C\setminus B_{3r}(0)}$ such that the following is true. 
Let $P$ (resp.\ $P^\phi$) be the a.s.\ unique $D_h$- (resp.\ $D_{h-\phi}$-) geodesic from $\BB z$ to $\BB w$. 
There is a deterministic constant $b_0  >0$ depending only on $\BB p , \mu , \nu,c_1',c_2' $ such that if $P\cap B_{2r}(0) \not=\emptyset$ and $E_r$ occurs, then there are times $0  < s < t  <  D_{h - \phi}(\BB z, \BB w) $ and such that \label{item-internal-geo}
\allb \label{eqn-internal-geo0}
&P^\phi(s) , P^\phi(t) \in B_{3r/2}(0) , \quad 
|P^\phi(s) - P^\phi(t)| \geq b_0 r , \notag\\
&\qquad \wt D_{h-\phi }\left(P^\phi(s) , P^\phi(t)\right) \leq c_2'  D_{h- \phi}\left(P^\phi(s) , P^\phi(t)\right) ,\quad \text{and} \notag\\
&\qquad  \wt D_{h- \phi}\left(P^\phi(s) , P^\phi(t)\right) \leq (c_* / C_*) \wt D_{h-\phi}\left(P^\phi(s) , \bdy B_{3r}(0) \right) .
\alle
\end{enumerate}
\end{prop}

The event $E_r$ and the collection of functions $\mcl G_r$ will be defined explicitly in Section~\ref{sec-E-def}; see Section~\ref{sec-shortcut-outline} for an overview of the definitions. The reason why we are able to restrict to a finite collection $\mcl G_r$ of bump functions $\phi$ is that we will break up space into a fine grid and require that the ``tube" where $\phi$ is very large (as referred to in Section~\ref{sec-shortcut-outline}) is a finite union of squares in the grid.  
As explained in Lemma~\ref{lem-E-frkE-compare} just below, Properties~\eqref{item-dirichlet-bound} and~\eqref{item-internal-geo} are used to check condition~\ref{item-geo-iterate-compare} in Theorem~\ref{thm-geo-iterate}. 
The purpose of Property~\eqref{item-dirichlet-bound} is to control the Radon-Nikodym derivative between the conditional laws of $h$ and $h-\phi$ given $h|_{\BB C \setminus B_{3r}(0)}$. 

We now explain how to conclude the proof of Proposition~\ref{prop-geo-event} assuming Proposition~\ref{prop-geo-event0}.
Fix points $\BB z,\BB w\in\BB C\setminus B_{4r}(0)$ and let $P$ be the $D_h$-geodesic from $\BB z$ to $\BB w$, as in Property~\eqref{item-internal-geo}.
We first define the event $\frk E_r^{\BB z,\BB w}(0)$ appearing in Proposition~\ref{prop-geo-event}. 
Let $\frk E_r = \frk E_r^{\BB z , \BB w}(0)$ be the event that there are times $0  < s < t  <  D_{h }(\BB z, \BB w) $ such that 
\allb \label{eqn-frkE-dist}
&P (s) , P  (t) \in B_{3r/2}(0) , \quad 
|P (s) - P (t)| \geq b_0 r , \notag\\
&\qquad \wt D_h(P(s) , P(t)) \leq c_2'  D_h(P (s) , P (t))  ,\quad \text{and} \notag\\
&\qquad \wt D_h(P(s) , P(t)) \leq (c_*/C_*) \wt D_h(P (s) , \bdy B_{3r}(0) ),
\alle 
and 
\allb \label{eqn-frkE-dirichlet}
\exp\left( -   (h ,\phi )_\nabla  +\frac12 (\phi , \phi)_\nabla \right)
\leq \Lambda, \quad \forall \phi \in \mcl G_r ,\quad \text{where} \quad \Lambda := e^{\Lambda_0}  ,
\alle
where $\Lambda_0$ is the constant from Property~\eqref{item-dirichlet-bound} and $b_0$ is the constant from Property~\eqref{item-internal-geo}. 
We note that~\eqref{eqn-frkE-dist} is the same as~\eqref{eqn-internal-geo0} from Property~\eqref{item-internal-geo}, but with $h$ instead of $h-\phi$. 
This condition is the main point of the definition of $\frk E_r$. The extra condition~\eqref{eqn-frkE-dirichlet} is only included to control a Radon-Nikodym derivative when we compare the conditional probabilities of $\frk E_r$ and $E_r$ given $h|_{\BB C\setminus B_{3r}(0)}$. 

\begin{lem} \label{lem-frkE-msrble}
The event $\frk E_r$ is a.s.\ determined by $h|_{B_{3r}(0)}$ and the $D_h$-geodesic $P$ stopped at its last exit time from $B_{3r}(0)$. 
\end{lem}
\begin{proof}  
Recall that each of the functions $\phi \in \mcl G_r$ is supported on $\BB A_{r/4,3r}(0)$. 
Since $\mcl G_r$ is a finite set, it is clear that the condition~\eqref{eqn-frkE-dirichlet} is determined by $h|_{B_{3r}(0)}$. 

To deal with~\eqref{eqn-frkE-dist}, we first observe that the set of pairs of times $s ,t \in [0, D_h(\BB z , \BB w)]$ satisfying $\wt D_h(P(s) , P(t)) \leq (c_*/C_*) \wt D_h(P (s) , \bdy B_{3r}(0) )$ is determined by $P$ stopped at its last exit time from $B_{3r}(0)$ and the internal metric $\wt D_h(\cdot,\cdot ; B_{3r}(0))$.
Indeed, a pair $(s,t)$ belongs to this set if and only if $P(t)$ is contained in the $\wt D_h$-metric ball of radius $(c_* / C_*) \wt D_h(P (s) , \bdy B_{3r}(0) )$ centered at $P(s)$. 
For each such pair of times $s,t$, we have $\wt D_h(P(s) , P(t)) = \wt D_h(P(s) , P(t) ; B_{3r}(0))$. 
Since $P$ is a $D_h$-geodesic, the points $P(s) , P(t)$ and the distance $D_h(P(s) , P(t)) = t-s$ for each such pair of points $s,t$ is determined by $P$ stopped at its last exit time from $B_{3r}(0)$. 
Since $\wt D_h(\cdot,\cdot ; B_{3r}(0))$ is determined by $h|_{B_{3r}(0)}$ (Axiom~\ref{item-metric-local}) we get that the event that there exists times $s ,t \in [0, D_h(\BB z , \BB w)]$ satisfying~\eqref{eqn-frkE-dist} is determined by $h|_{B_{3r}(0)}$ and $P$ stopped at its last exit time from $B_{3r}(0)$.
\end{proof}

We can now check condition~\ref{item-geo-iterate-compare} of Theorem~\ref{thm-geo-iterate} for the above definitions of $E_r = E_r(0)$ and $\frk E_r = \frk E_r^{\BB z,\BB w}(0)$ using the mutual absolute continuity of the laws of $h$ and $h-\phi$. 

\begin{lem} \label{lem-E-frkE-compare}
Assume Proposition~\ref{prop-geo-event0}. 
With $\Lambda$ as in~\eqref{eqn-frkE-dirichlet}, it is a.s.\ the case that 
\allb \label{eqn-E-frkE-compare}
\BB P\left[ E_r \cap \{P\cap B_{2r}(0) \not=\emptyset\} \,|\, h|_{\BB C\setminus B_{3r}(0)}   \right]
\leq  \Lambda \BB P\left[ \frk E_r   \cap \{P \cap B_{2r}(0)\not=\emptyset\} \,|\, h|_{\BB C\setminus B_{3r}(0)}   \right] .
\alle
\end{lem}
\begin{proof}
The occurrence of the events $E_r$ and $\frk E_r$ is unaffected by adding a constant to $h$, so we can assume without loss of generality that $h$ is normalized so that its circle average over $\bdy B_{4r}(0)$, say, is zero. 
By the Markov property of $h$, under the conditional law given $h|_{\BB C\setminus B_{3r}(0)}$, we can decompose $h|_{B_{3r}(0)}$ as the sum of a harmonic function which is determined by $h|_{\BB C\setminus B_{3r}(0)}$ and a zero-boundary GFF on $B_{3r}(0)$ which is independent from $h|_{\BB C\setminus B_{3r}(0)}$.

Let $\phi \in \mcl G_r$ be the smooth bump function from Property~\eqref{item-internal-geo}, which is determined by $h|_{\BB C\setminus B_{3r}(0)}$. 
By a standard Radon-Nikodym derivative calculation for the GFF, 
if we condition $h|_{\BB C\setminus B_{3r}(0)}$ then the conditional law of $h-\phi$ is a.s.\ absolutely continuous with respect to the conditional law of $h$, and the Radon-Nikodym derivative of the former w.r.t.\ the latter is
\eqb
M_h = \exp\left( -   (h ,\phi )_\nabla  - \frac12 (\phi , \phi)_\nabla \right) .
\eqe
Note that since $\phi$ is supported on $B_{3r}(0)$, the Radon-Nikodym derivative $M_h$ depends only on the zero-boundary part of $h|_{B_{3r}(0)}$.

Define the $D_{h-\phi}$-geodesic $P^\phi$ from $\BB z$ to $\BB w$ and the event $\frk E_r^\phi$ in the same manner as $P$ and $\frk E_r$ but with $h-\phi$ in place of $h$. 
By~\eqref{eqn-frkE-dirichlet}, on $\frk E_r $, we have $M_h \leq \Lambda$. 
Therefore,
\allb \label{eqn-annulus-rn}
\BB P\left[ \frk E_r^\phi  \cap \{P^\phi \cap B_{2r}(0)\not=\emptyset\} \,|\, h|_{\BB C\setminus B_{3r}(0)}   \right]
&= \BB E\left[ M_h \BB 1_{\frk E_r  \cap \{P  \cap B_{2r}(0)\not=\emptyset\}} \,|\, h|_{\BB C\setminus B_{3r}(0)}   \right] \notag \\
&\leq \Lambda \BB P\left[ \frk E_r   \cap \{P \cap B_{2r}(0)\not=\emptyset\} \,|\, h|_{\BB C\setminus B_{3r}(0)}   \right] .
\alle

We now claim that
\eqb \label{eqn-frkE-contain}
E_r \cap \{P\cap B_{2r}(0) \not=\emptyset\} \subset \frk E_r^\phi \cap \{P^\phi \cap B_{2r}(0) \not=\emptyset\} .
\eqe
Indeed, Property~\eqref{item-internal-geo} (subtracting a bump function) says that the main condition~\eqref{eqn-frkE-dist} in the definition of $\frk E_r$ is satisfied with $h-\phi$ in place of $h$ whenever $E_r \cap \{P\cap B_{2r}(0) \not=\emptyset\}$ occurs, which implies in particular that $P^\phi \cap B_{2r}(0) \not=\emptyset$ whenever $E_r \cap \{P\cap B_{2r}(0) \not=\emptyset\}$ occurs. 
Furthermore, Property~\eqref{item-dirichlet-bound} (bound for Dirichlet inner products) implies that the Dirichlet energy condition~\eqref{eqn-frkE-dirichlet} in the definition of $\frk E_r$ holds with $h-\phi$ in place of $h$ whenever $E_r$ occurs.
Thus~\eqref{eqn-frkE-contain} holds. 

As an immediate consequence of~\eqref{eqn-frkE-contain}, a.s.\ 
\allb \label{eqn-annulus-mono}
\BB P\left[ E_r \cap \{P\cap B_{2r}(0) \not=\emptyset\} \,|\, h|_{\BB C\setminus B_{3r}(0)}   \right] 
\leq  \BB P\left[ \frk E_r^\phi \cap \{P^\phi \cap B_{2r}(0)\not=\emptyset\} \,|\, h|_{\BB C\setminus B_{3r}(0)}   \right] .
\alle
Combining~\eqref{eqn-annulus-rn} and~\eqref{eqn-annulus-mono} gives~\eqref{eqn-E-frkE-compare}.
\end{proof}

\begin{proof}[Proof of Proposition~\ref{prop-geo-event}, assuming Proposition~\ref{prop-geo-event0}]
Let $\BB p$ be as in Theorem~\ref{thm-geo-iterate} with our given choice of $0 < \mu < \nu \leq \nu_*$ and with the constants
\eqb \label{eqn-constant-choice}
\lambda_1 := 1/4 ,\quad \lambda_2 := 2,\quad \lambda_3 := 3, \quad \lambda_4 := 4,\quad \text{and} \quad \lambda_5 :=5.
\eqe
For $z\in\BB C$, $r \in \rho^{-1} \mcl R_0$, and $\BB z,\BB w\in \BB C\setminus B_{4r}(z)$, let $E_r(z)$ (resp.\ $\frk E_r^{\BB z,\BB w}(z)$) be the event $E_r$ of Proposition~\ref{prop-geo-event0} (resp.\ the event and $\frk E_r^{\BB z +z , \BB w + z}$ defined above) with the field $h(\cdot+z) - h_1(z)\eqD h$ in place of $h$. 

Let $c'' = c''(\alpha,c_1',\mu,\nu) \in (c_* , c_1')$ be chosen as in Proposition~\ref{prop-attained-good'} with $\alpha$ as in Lemma~\ref{lem-endpoint-geodesic} and $c_1'$ in place of $c'$.  
Also let $\mcl R_0$ be defined as in the discussion surrounding~\eqref{eqn-attained-good-event} and let $\mcl R :=  \rho^{-1} \mcl R_0$. By the definition of $\frk E_r^{\BB z,\BB w}(z)$ (in particular,~\eqref{eqn-frkE-dist}), the conditions~\eqref{eqn-geo-event} hold on $\frk E_r^{\BB z,\BB w}(z)$ with $b = b_0$.

If $\BB r >0$ such that $\BB P[\ul G_{\BB r}(c'', \beta) ] \geq \beta$, then Proposition~\ref{prop-attained-good} implies that there exists $\ep_0 = \ep_0(\beta,c_1',\mu,\nu) > 0$ such that for each $\ep\in (0,\ep_0]$,  
\eqb
\#(\mcl R_0 \cap [\ep^{1+\nu} \BB r , \ep \BB r] \cap \{8^{-k} \BB r\}_{k\in\BB N}) \geq\mu\log_8 \ep^{-1}  , 
\eqe
equivalently, 
\eqb
 \#( \mcl R\cap [\ep^{1+\nu} \rho^{-1} \BB r , \ep \rho^{-1} \BB r] \cap \{8^{-k} \rho^{-1}\BB r \}_{k\in\BB N}) \geq\mu\log_8 \ep^{-1} .
\eqe
This shows that condition~\ref{item-geo-iterate-dense} of  Theorem~\ref{thm-geo-iterate} is satisfied with $\rho^{-1} \BB r$ in place of $\BB r$.
By Property~\eqref{item-geo-event0} (measurability and high probability) and Lemma~\ref{lem-frkE-msrble}, conditions~\ref{item-geo-iterate-msrble} and~\ref{item-geo-iterate-prob} of Theorem~\ref{thm-geo-iterate} are satisfied for the events $E_r(z)$ and $\frk E_r^{\BB z,\BB w}(z)$ above. 
By Lemma~\ref{lem-E-frkE-compare}, condition~\ref{item-geo-iterate-compare} of Theorem~\ref{thm-geo-iterate} is also satisfied.
\end{proof}

\subsection{Building a tube which contains a shortcut with positive probability}
\label{sec-endpoint-geodesic}

\begin{figure}[t!]
 \begin{center}
\includegraphics[scale=1]{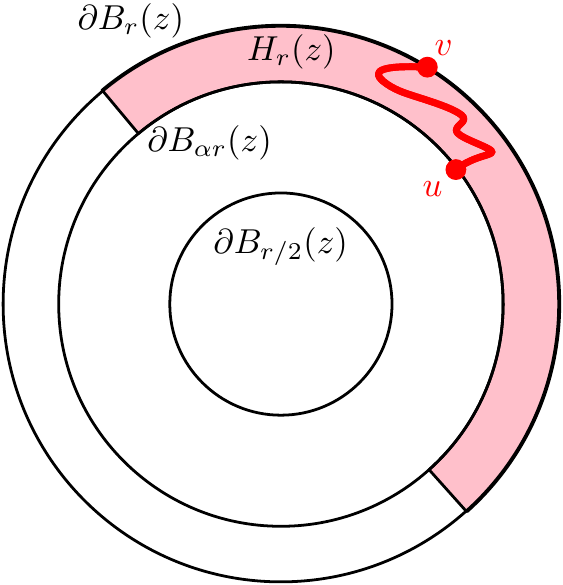}
\vspace{-0.01\textheight}
\caption{Illustration of the statement of Lemma~\ref{lem-endpoint-geodesic}. The lemma asserts that with probability at least $p_0/8$, there is a $D_h$-geodesic (red) between points $u$ and $v$ in the inner and outer boundaries, resp., of the pink half-annulus $H_r(z)$ which is contained in $\ol{H_r(z)}$ and satisfies $\wt D_h(u,v) \leq c_1' D_h(u,v)$. 
The main task of Section~\ref{sec-geodesic-shortcut} is to force a $D_h$-geodesic between two far away points to get near a $\wt D_h$-geodesic like the red one in the picture.  
}\label{fig-endpoint-geodesic}
\end{center}
\vspace{-1em}
\end{figure} 

We now turn our attention to constructing the event $E_r$ and the collection of functions $\mcl G_r$ of Proposition~\ref{prop-geo-event0}, following the strategy outlined in Section~\ref{sec-shortcut-outline}. 
Recall that $\alpha_* \in (1/2,1)$ and $p_0 \in (0,1)$ are the parameters from Proposition~\ref{prop-attained-good'} with $\mu$ and $\nu$ as in Proposition~\ref{prop-geo-event}.

Our goal is to define for each $z\in\BB C$ and each $r\in\mcl R_0$ a deterministic open ``tube" $V_r(z) \subset B_{3r}(z)$ and an event $F_r(z)$ such that $\BB P[F_r(z)]$ is bounded below uniformly over $z$ and $r$, $F_r(z) \in \sigma\left( (h-h_{4r}(z)) |_{B_{3r}(z)} \right)$, and on $F_r(z)$ there are points $u,v\in V_r(z)$ which satisfy~\eqref{eqn-attained-good-event} plus some additional conditions. 
We will define $V_r(z)$ and $F_r(z)$ and prove a lower bound for $\BB P[F_r(z)]$ in Lemma~\ref{lem-deterministic-geodesic}, with Lemma~\ref{lem-endpoint-geodesic} as an intermediate step. We will prove the required measurability in Lemma~\ref{lem-geo-event-local}. 

We define a \emph{half-annulus} of an annulus $A$ to be the intersection of $A$ with a half-plane whose boundary passes through the center of $A$. 
It is easier for us to work with a $\wt D_h$-geodesic which is constrained to stay in a half-annulus rather than a whole annulus. The reason for this is that it allows us to easily find paths from each of the endpoints of the geodesic to points far away from the half-annulus which do not get near the geodesic except at their endpoints (this might be trickier if the geodesic wraps around the whole annulus). 
The following lemma, which is a slight improvement on the condition in the definition of $\mcl R_0$, will allow us to work with a half-annulus rather than a whole annulus. 
 
\begin{lem} \label{lem-endpoint-geodesic}
There exists $\alpha  \in [\alpha_* ,1)$ depending only on $\mu,\nu$ such that for each $r\in \mcl R_0$ and each $z\in\BB C$, there is a deterministic half-annulus $H_r(z) \subset \BB A_{\alpha r , r}(z)$ such that with probability at least $p_0/8$, there exists $u \in \bdy B_{\alpha r}(z)$ and $v \in \bdy B_r(z)$ with the following properties.  
\begin{enumerate}
\item $\wt D_h(u,v) \leq c_1' D_h(u,v)$. 
\item The $\wt D_h$-geodesic from $u$ to $v$ is unique and is contained in $\ol{H_r(z)}$. 
\item $ \wt D_h(u,v) \leq (c_*/C_*)^{2} \wt D_h\left(\BB A_{\alpha r , r}(z) , \bdy B_{2r}(z) \right)$, where $c_*$ and $C_*$ are as in~\eqref{eqn-max-min-def}.  
\end{enumerate}  
\end{lem}
\begin{proof} 
By Axioms~\ref{item-metric-translate} and~\ref{item-metric-coord}, we can find $S > s > 0$ depending only on $p_0$ (and hence only on $\mu,\nu$) such that for each $r>0$, it holds with probability at least $1 - p_0/2$ that the following is true.
\begin{itemize}
\item Any two points of $\BB A_{3r/4, r}(z)$ which are not contained in a single quarter-annulus of $\ol{\BB A_{3r/4,r}(z)}$ lie at $\wt D_h$-length at least $s \frk c_r e^{\xi h_r(z)}$ from each other.
\item $(c_*/C_*)^{2} \wt D_h\left(\BB A_{3 r/4 , r}(z) , \bdy B_{2r}(z) \right) \geq s \frk c_r e^{\xi h_r(z)}$. 
\item $\wt D_h(u,v) \leq S \frk c_r e^{\xi h_r(z)}$ for each $u,v\in \ol{\BB A_{3r/4,r}(z)}$.
\end{itemize}  
Since $\BB A_{\alpha r ,r}(z) \subset \BB A_{3r/4,r}(z)$ for each $\alpha \in [3/4,1)$,  Lemma~\ref{lem-attained-long} applied with the above choice of $s$ and $S$ gives an $\alpha \in [(3/4)\vee \alpha_* ,1)$ depending on $p_0$ and $\alpha_*$ such that for each $r > 0$ it holds with probability at least $1-p_0/2$ that the following is true. For each pair of points $u,v\in \ol{\BB A_{\alpha r , r}(z)}$ such that $\wt D_h(u,v; \BB A_{\alpha r ,r}(0 ) ) = \wt D_h(u,v)$, it holds that $u$ and $v$ are contained in a single quarter-annulus of $\ol{\BB A_{\alpha r , r}(z)}$ and $\wt D(u,v) \leq \wt D_h\left(\BB A_{\alpha r , r}(z) , \bdy B_{2r}(z) \right)$.   
This happens in particular if there is a $\wt D_h$-geodesic from $u$ to $v$ contained in $\ol{\BB A_{\alpha r ,r}(z)}$. 

Combining this with translation invariance (Axiom~\ref{item-metric-translate}) and the definition of $\mcl R_0$ shows that for $r\in\mcl R_0$, it holds with probability at least $p_0/2$ that the conditions in the lemma statement hold but with a \emph{random} quarter-annulus in place of a deterministic half-annulus. This random quarter annulus is a.s.\ contained in one of four possible deterministic half-annuli, so must be contained in one of these four half-annuli with probability at least $p_0/8$. 
We therefore obtain that for an appropriate choice of $H_r(z)$, it holds with probability at least $p_0/8$ that all of the conditions in the lemma statement hold. 
\end{proof}

\begin{figure}[t!]
 \begin{center}
\includegraphics[scale=1]{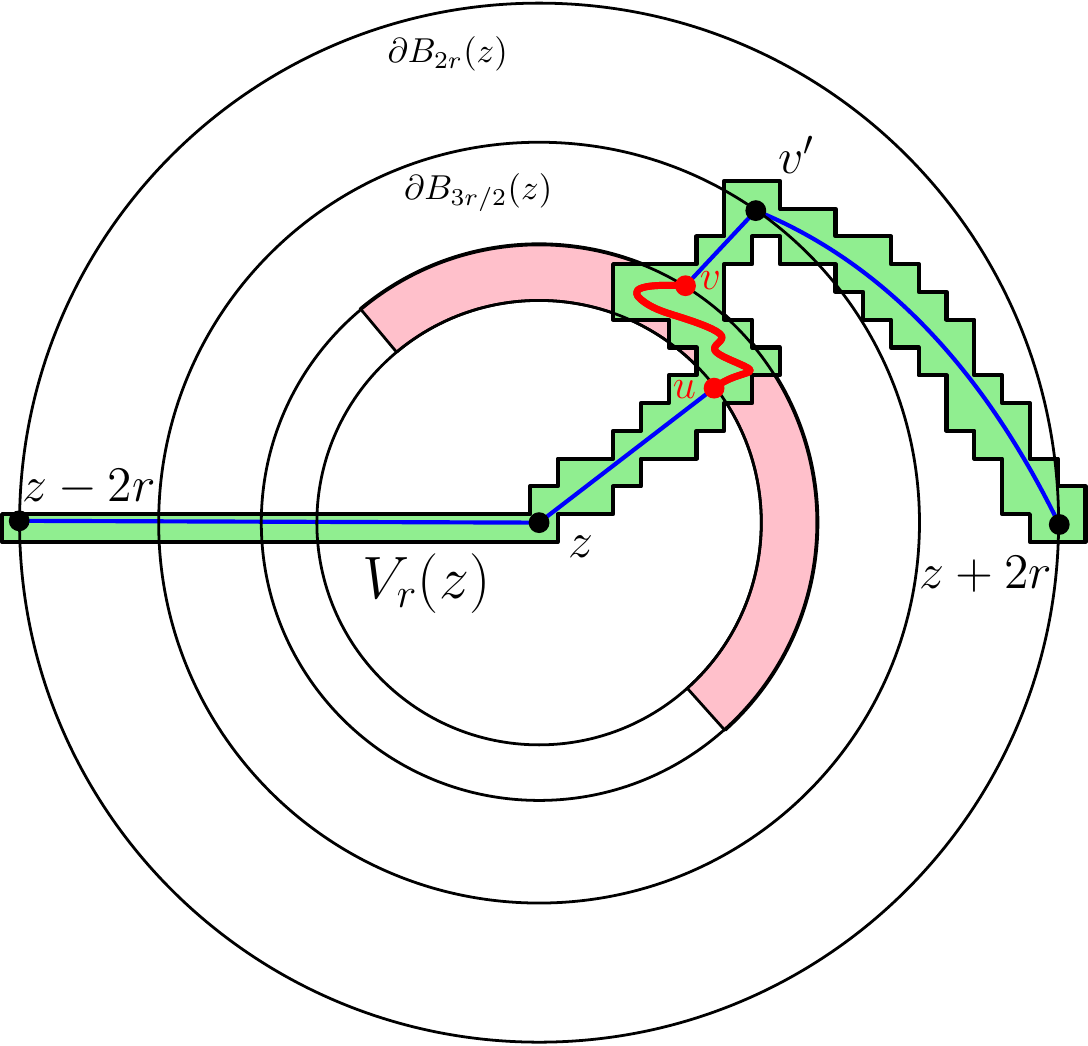}
\vspace{-0.01\textheight}
\caption{Illustration of the statement and proof of Lemma~\ref{lem-deterministic-geodesic}. Building on the setting of Figure~\ref{fig-endpoint-geodesic}, we show that there is a deterministic long narrow ``tube" $V_r(z)$ (light green), which is the interior of the set of $\ep_\dtm r\times \ep_\dtm r$ squares with corners in $\ep_\dtm r \BB Z^2$ which intersect a certain path from $z-2r$ to $z + 2r$, with the following property. With positive probability, every path in the tube from a point near $z-2r$ to a point near $z+2r$ has to get near a pair of points $u,v$ in the tube for which $\wt D_h(u,v) \leq c_1' D_h(u,v)$. We will eventually add a bump function to $h$ which takes a very negative value in such a tube in order to force a geodesic between points which are far away from $B_{2r}(z)$ to get near $u$ and $v$. 
}\label{fig-deterministic-geodesic}
\end{center}
\vspace{-1em}
\end{figure} 

We henceforth assume that $\alpha \in [\alpha_*,1)$ is chosen so that the conclusion of Lemma~\ref{lem-endpoint-geodesic} is satisfied. 
In order to construct deterministic ``tubes" as described in Section~\ref{sec-shortcut-outline}, we will look at unions of squares in a fine grid.
For $\ep > 0$ and $X\subset\BB C$, let 
\eqb \label{eqn-square-set-def}
\mcl S_\ep(X) := \left\{\text{closed $\ep\times \ep$ squares with corners in $\ep \BB Z^2$ which intersect $X$} \right\} .
\eqe
Recall that we have fixed $c_2' > c_1' > c_*$. Choose, in a manner depending only on $c_1',c_2',c_*,C_*$, a small parameter $\eta \in (0,1)$ such that
\eqb \label{eqn-const-error}
\frac{ c_1' (1+2\eta) }{ 1 - 2 c_*^{-1} C_* \eta } < c_2' \quad \text{and} \quad  1+2\eta  < C_* / c_* .
\eqe
The particular choice of $\eta$ in~\eqref{eqn-const-error} will not be used until~\eqref{eqn-scaled-end} below.
For now, the reader should just think of it as a small constant depending on $c_1',c_2'$. 
We also note that $\eta$ is fixed in a way that depends only on $c_1',c_2',c_*,C_*$ (hence only on $c_1',c_2'$ and the choice of $D,\wt D$), so we do not need to explicitly mention the dependence on $\eta$ in what follows.
The following lemma gives us the basic ``building blocks" which will be used to construct $E_r$ in the next two subsections.

\begin{lem} \label{lem-deterministic-geodesic}
There exist small parameters $  b_\dtm , p_\dtm \in (0,1/100)$ depending only on $ \mu,\nu $ and a parameter $\ep_\dtm \in (0,b_\dtm/100)$ depending only on $c_1',c_2',\mu,\nu$ such that for each $z\in\BB C$ and each $r\in \mcl R_0$, there exists a deterministic connected open set $V_r(z) \subset B_{(2+ 2\ep_\dtm)r}(z)$ with the following properties.
The set $V_r(z)$ is the interior of a finite union of squares in $\mcl S_{\ep_\dtm r}(B_{2r}(z))$, $z-2r, z + 2r \in V_r$, and we have $\BB P[F_r(z)]  \geq p_\dtm$, where $F_r(z)$ is the event that the following is true. There are points $u,v\in V_r(z)\cap \ol{B_r(z)}$ with the following properties.
\begin{enumerate} 
\item \emph{(Existence of a shortcut)} We have
\eqb
|u-v| \geq b_\dtm r,\quad \wt D_h(u , v ) \leq c_1' D_h(u  ,v ) ,\quad \wt D_h(u,v) \leq (c_*/C_*)^{2} \wt D_h\left( u  , \bdy B_{2r}(z) \right) ,
\eqe
and the $\wt D_h$-geodesic from $u$ to $v $ is unique and is contained in $V_r(z)  \cap \ol{B_r(z)}$.  \label{item-tube-geo}
\item \emph{(Removing neighborhoods of $u,v$ disconnects $V_r(z)$)} Let $O_u$ be the connected component of $ V_r(z) \cap B_{20\ep_\dtm r}(u)$ which contains $u$ and similarly define $O_v$ with $v$ in place of $u$.  
The connected component of $  V_r(z)   \setminus O_u $ which contains $z-2r$ lies at Euclidean distance at least $\ep_\dtm r $ from the union of the other connected components of $V_r(z) \setminus O_u$. The same is true with $v$ in place of $u$ and $z+2r$ in place of $z-2r$.  \label{item-tube-dc}
\item \emph{(Upper bound for internal diameters of neighborhoods of $u$ and $v$)} Each point of $O_u$ lies at $\wt D_h(\cdot,\cdot ; V_r(z))$-distance at most $\eta \wt D_h(u,v)$ from $u $, and the same is true with $v $ in place of $u $ (here $\eta$ is as in~\eqref{eqn-const-error}). \label{item-tube-internal} 
\end{enumerate}  
\end{lem} 
\begin{proof}
Let $\alpha$ be as in Lemma~\ref{lem-endpoint-geodesic} and set $b_\dtm := 1-\alpha$.  
On the event that points $u \in \bdy B_{\alpha r}(z)$ and $v\in \bdy B_r(z)$ satisfying the conditions on Lemma~\ref{lem-endpoint-geodesic} exist (which happens with probability at least $p_0/8$), choose one such pair of points $(u,v)$ in some measurable manner. Otherwise, let $u = v = 0$. 
On the event $\{u \not=0\}$, let $\wt P$ be the unique $\wt D_h$-geodesic from $u$ to $v$ and let $H_r(z) \subset \BB A_{\alpha r , r}(z)$ be the half-annulus with $\wt P \subset \ol{H_r(z)}$ as in Lemma~\ref{lem-endpoint-geodesic}.
  
We will now extend $\wt P$ to a path $\wt P'$ in $B_{2r}(z)$ from $z-2r$ to $z+2r$ (which will no longer be a $\wt D_h$-geodesic). 
To this end, we first let $v' := (3/2)(v - z) +z \in \bdy B_{3r/2}(z)$ and we let $L_-$ (resp.\ $L_+$) be the linear segment from $z$ to $u$ (resp.\ $v$ to $v'$).
We note that the Euclidean distance between $L_-$ and $L_+$ is at least $b_\dtm r$. 
We can choose a path $\pi_-$ from $z-2r$ to $z$ and a path $\pi_+$ from $v'$ to $z+2r$ in $B_{2r}(z)$ such that the Euclidean distances from $\pi_-\cup \pi_+$ to $H_r(z)$ and from $\pi_- \cup L_-$ to $\pi_+ \cup L_+$ are each at least $b_\dtm r$. 
Let $\wt P'$ be the concatenation of $\pi_- , L_- , \wt P , L_+ , \pi_+$. 
 
Since $|u-v| \geq b_\dtm r$ on the event $\{u\not=0\}$, Axiom~\ref{item-metric-coord} (tightness across scales) together with Lemma~\ref{lem-square-diam} imply that we can find $\ep_\dtm \in (0,b_\dtm/100)$ depending only on $c_1',c_2',\mu,\nu$ such that with probability at least $p_0/9$, the event of Lemma~\ref{lem-endpoint-geodesic} occurs (i.e., $u\not=0$) and also
\eqb \label{eqn-dtm-internal}
\sup_{S\in\mcl S_{\ep_\dtm r}(B_{2r}(z))} \sup_{w_1,w_2 \in S} \wt D_h(w_1,w_2; S) \leq    \frac{\eta}{100} \wt D_h(u,v)   .
\eqe
The number of subsets of $\mcl S_{\ep_\dtm r}(B_{2r}(z))$ is bounded above by a deterministic constant depending only on $\ep_\dtm$. Consequently, we can choose $p_\dtm \in (0,1)$ depending only on $ \mu,\nu,D$ and a deterministic $\mcl K_r(z)   \subset\mcl S_{\ep_\dtm r}(B_{2r}(z))$ such that with probability at least $p_\dtm$, the events of Lemma~\ref{lem-endpoint-geodesic} and~\eqref{eqn-dtm-internal} occur and also 
\eqb \label{eqn-dtm-set}
\mcl K_r(z) = \left\{S \in \mcl S_{\ep_\dtm r}(B_{2r}(z)) : S\cap \wt P' \not=\emptyset \right\} .
\eqe
Let $V_r(z)$ be the interior of the union of the squares in $\mcl K_r(z)$. 
Since $z-2r ,z +  2r \in \wt P'$ and $\wt P'$ is connected, it follows that $V_r(z)$ is connected and contains $z-2r$ and $z+ 2r$.

Henceforth assume that the events of Lemma~\ref{lem-endpoint-geodesic}, \eqref{eqn-dtm-internal}, and~\eqref{eqn-dtm-set} occur. We will check the conditions in the lemma statement with $V_r(z)$ as above. 
\medskip

\noindent\textit{Condition~\ref{item-tube-geo}.} This is immediate from the conditions on $u$ and $v$ from Lemma~\ref{lem-endpoint-geodesic}. 
\medskip

\noindent\textit{Condition~\ref{item-tube-dc}.}   
By the above definitions of $\pi_-$ and $L_-$, the Euclidean $ \ep_\dtm r$-neighborhood of each square of $\mcl S_{\ep_\dtm r}$ which intersects both $B_{2\ep_\dtm r} ( \pi_- \cup L_-) $ and $B_{2\ep_\dtm r}( H_r(z) )$ must be contained in $B_{10\ep_\dtm r}(u)$. 
Furthermore, using that $L_-$ is a linear segment, we get that the $\ep_\dtm$-neighborhood of each such square which intersects $B_{2\ep_\dtm r} ( \pi_- \cup L_-) $ and belongs to $\mcl K_r(z)$ (as defined in~\eqref{eqn-dtm-set}) must be contained in $\ol{O_u}$, with $O_u$ as in the lemma statement. 
Since the Euclidean distance between $\pi_- \cup L_-$ and $\pi_+ \cup L_+$ is at least $b_\dtm r \geq 100 \ep_\dtm r$ and $\wt P \subset H_r(z)$, we see that removing $O_u$ disconnects $ V_r(z) $ into at least two connected components, and the Euclidean distance between the connected component which contains $ z -  2r$ and the union of the other connected components is at least $\ep_\dtm r  $.  
A similar argument applies with $v$ in place of $u$. 
\medskip

\noindent\textit{Condition~\ref{item-tube-internal}.}  
Each point of $O_u$ is contained in a square of $\mcl K_r(z)$ which lies at graph distance at most $40$ from a square which contains $u$ in the adjacency graph of squares of $\mcl K_r(z)$. The same is true with $v$ in place of $u$. It therefore follows from~\eqref{eqn-dtm-internal} that condition~\ref{item-tube-internal} in the lemma statement is satisfied.  
\end{proof}

For $z\in\BB C$ and $r>0$, let $F_r(z)$ be as in Lemma~\ref{lem-deterministic-geodesic}.  In the next subsection, we will use the local independence properties of the GFF (in the form of Lemma~\ref{lem-spatial-ind}) to argue that for a small enough $\rho\in (0,1)$ and for all $r\in \rho^{-1}\mcl R_0$, it is very likely that $F_{\rho r}(z)$ occurs for many points $z\in  B_r(0)$.  To apply the lemma, we will need the following measurability statement. 

\begin{lem} \label{lem-geo-event-local}
For each $z\in\BB C$ and $r>0$, the event $F_r(z)$ is a.s.\ determined by $(h-h_{4r}(z)) |_{B_{3r}(z)}$.
\end{lem} 
\begin{proof} 
First note that the occurrence of $F_r(z)$ is unaffected by scaling each of $D_h$ and $\wt D_h$ by the same constant factor. 
Therefore, Axiom~\ref{item-metric-f} (Weyl scaling) implies that $F_r(z)$ is determined by $h$, viewed modulo additive constant.
So, we only need to show that $F_r(z) \in \sigma\left( h|_{B_{3r}(z)} \right)$. 

We first observe that for $u,v \in \ol{B_r(z)}$, we have $\wt D_h(u,v) \leq (c_*/C_*)^{2} \wt D_h\left( u  , \bdy B_{2r}(z) \right)$ if and only if $v$ is contained in the $\wt D_h$-metric ball of radius $(c_*/C_*)^{2} \wt D_h\left( u  , \bdy B_{2r}(z) \right)$ centered at $v$.
Since this $\wt D_h$-metric ball is contained in $B_{2r}(z)$, we infer from the locality of $\wt D_h$ that the set of $u,v \in B_{2r}(z)$ for which $\wt D_h(u,v) \leq  (c_*/C_*)^{2} \wt D_h\left( u  , \bdy B_{2r}(z) \right)  $ is determined by $ h|_{B_{3r}(z)}$.  
If $\wt D_h(u,v) \leq (c_*/C_*)^{2} \wt D_h\left( u  , \bdy B_{2r}(z) \right)$, then each $\wt D_h$-geodesic from $u$ to $v$ is contained in $B_{2r}(z)$, so the set of $\wt D_h$-geodesics from $u$ to $v$ is the same as the set of $\wt D_h(\cdot,\cdot; B_{2r}(z))$-geodesics from $u$ to $v$.

Furthermore, by the definition~\eqref{eqn-max-min-def} of $c_*$ and $C_*$, we see that
\eqb
\wt D_h(u,v) \leq (c_*/C_*)^{2} \wt D_h\left( u  , \bdy B_{2r}(z) \right)
\Rightarrow 
 D_h(u,v) \leq (c_*/C_*)  D_h\left( u  , \bdy B_{2r}(z) \right) ,
\eqe
so $D_h(u,v)  =D_h(u,v ; \bdy B_{2r}(z))$ whenever $\wt D_h(u,v) \leq (c_*/C_*)^{2} \wt D_h\left( u  , \bdy B_{2r}(z) \right)$. 

By combining these observations with the locality of the metrics $D_h$ and $\wt D_h$, it follows that $F_r(z)$ is determined by $h|_{B_{3r}(z)}$. 
\end{proof}

\subsection{Building a tube which contains a shortcut with high probability}
\label{sec-highprob-geodesic}

\begin{figure}[t!]
 \begin{center}
\includegraphics[scale=.55]{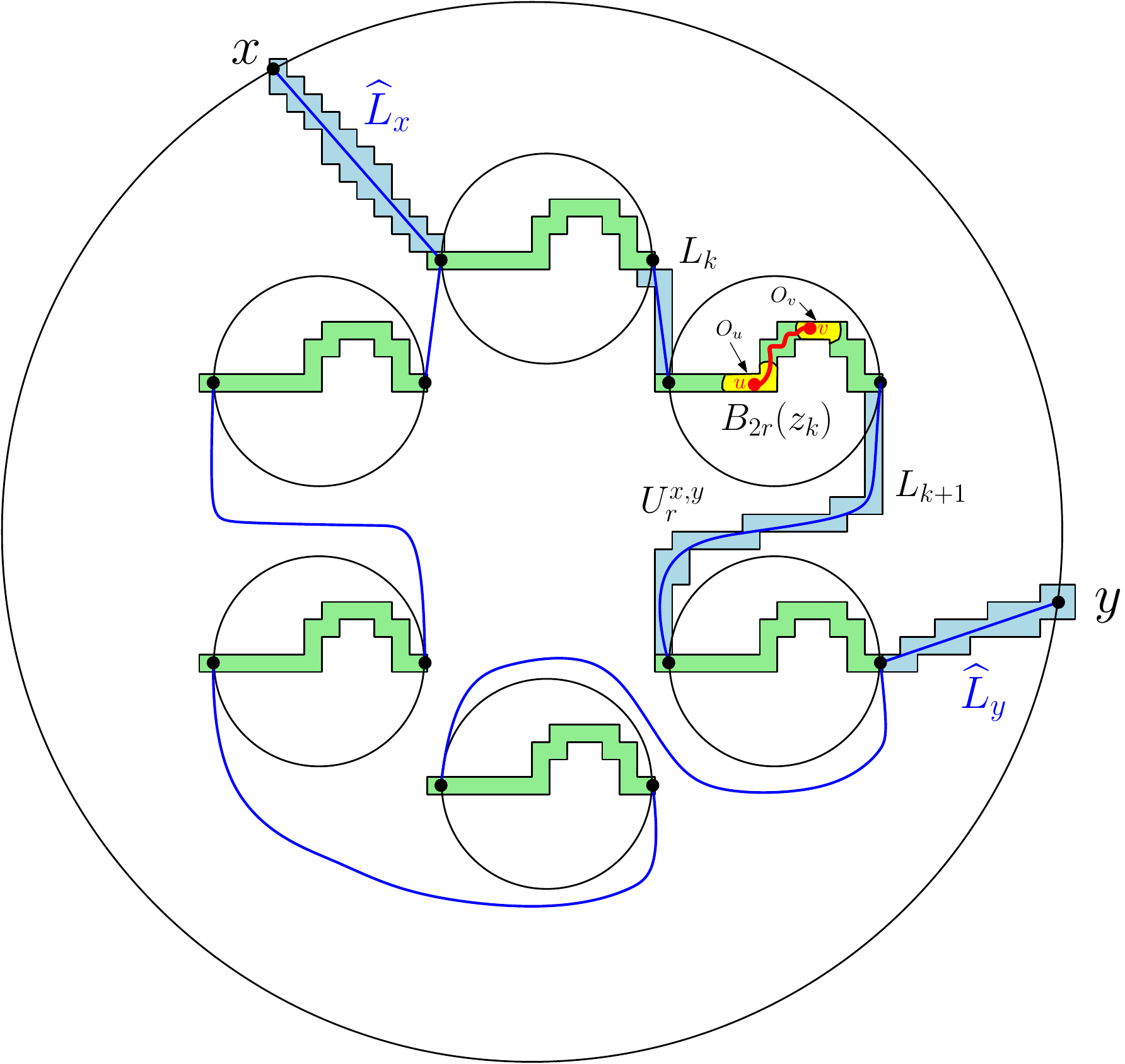}
\vspace{-0.01\textheight}
\caption{Illustration of the statement and proof of Lemma~\ref{lem-highprob-geodesic}.
To get an event with probability close to 1, instead of just an event with uniformly positive probability, we consider a large number of disjoint balls $B_{\rho r}(z)$ centered at a finite set of points $\mcl Z\subset \bdy B_r(0)$ and use Lemma~\ref{lem-spatial-ind} to argue that with high probability, the event $F_{\rho r}(z)$ of Lemma~\ref{lem-deterministic-geodesic} occurs for a suitably ``dense" set of points $z\in\mcl Z$. Then, we link up the tubes $V_{\rho r}(z)$ for $z\in\mcl Z$ (light green) via deterministic paths $L_k$ (blue). For a given choice of points $x,y\in\bdy B_{2r}(0)$ with $|x-y|\geq \delta r$, we define $U_r^{x,y}$ to be the union of the sets $V_{\rho r}(z)$ for points $z \in \mcl Z$ along the counterclockwise arc of $\bdy B_r(0)$ from $x/2$ to $y/2$, the squares of $\mcl S_{\ep_\dtm \rho r}(B_r(0))$ which intersect the deterministic paths joining these sets $V_{\rho r}(z)$, and paths of squares starting from each of $x$ and $y$ (light blue). The sets $O_u$ and $O_v$ from assertion~\ref{item-highprob-dc} are shown in yellow.
}\label{fig-highprob-geodesic}
\end{center}
\vspace{-1em}
\end{figure} 

In the rest of this section, unlike in Section~\ref{sec-endpoint-geodesic}, our events will no longer depend on a parameter $z$. Rather, we will only define events for Euclidean balls centered at 0.
We will now prove a variant of Lemma~\ref{lem-deterministic-geodesic} which holds with probability close to $1$, not just with uniformly positive probability.  This will be accomplished as follows. We fix a small parameter $\rho > 0$ and consider a large number of radius-$\rho r$ balls $B_{\rho r}(z)$ contained in $B_{2r}(0)$ for which the event $F_{\rho r}(z)$ of Lemma~\ref{lem-deterministic-geodesic} occurs with positive probability. We join up the ``tubes" $V_{\rho r}(z)$ for the individual balls into a single large tube, which we will denote by $U_r^{x,y}$. We use Lemma~\ref{lem-spatial-ind} to say that with high probability the event $F_{\rho r}(z)$ occurs for at least one of the small balls, which means that with high probability the tube $U_r^{x,y}$ contains a pair of points $u,v$ as in~\eqref{eqn-attained-good-event}.
See Figure~\ref{fig-highprob-geodesic} for an illustration.

\begin{lem} \label{lem-highprob-geodesic}
For each $p , \delta  \in (0,1)$, there exists $b   , \rho \in (0,1/100)$ depending only on $p ,\delta, \mu,\nu $ and $\ep_0 \in (0,b/100)$ depending only on $c_1',c_2' , p,\delta,\mu,\nu$ such that for each $r\in \rho^{-1} \mcl R_0$ and each $x,y \in \bdy B_{2r}(0)$ with $|x-y| \geq \delta r$, there exists a deterministic connected open set $U_r^{x,y} \subset B_{3r}(0)$ with the following properties.
The set $U_r^{x,y}$ is the interior of a finite union of squares in $\mcl S_{\ep_0 r}(\BB A_{ r/2, 2r}(0) )$, $x,y\in U_r^{x,y}$. Moreover, with probability at least $p$, it holds simultaneously for each $x,y\in \bdy B_{2r}(0)$ with $|x-y| \geq \delta r$ that there are points $u,v\in \BB A_{(1-4\rho) r, (1+4\rho) r}(0)\cap U_r^{x,y}$ with the following properties. 
\begin{enumerate} 
\item \emph{(Existence of a shortcut)} We have
\eqb \label{eqn-highprob-geo}
|u-v| \geq b  r,\quad \wt D_h(u , v ) \leq c_1' D_h(u  ,v ) ,\quad \wt D_h(u,v) \leq (c_*/C_*)^2 \wt D_h\left( u  , \bdy B_{4\rho r}(u) \right) ,
\eqe 
and the $\wt D_h$-geodesic from $u$ to $v$ is unique and is contained in $U_r^{x,y}$.  \label{item-highprob-geo}
\item \emph{(Removing neighborhoods of $u,v$ disconnects $U_r^{x,y}$)} Let $O_u$ be the connected component of $U_r^{x,y} \cap B_{20\ep_0 r}(u)$ which contains $u$ and define $O_v$ similarly with $v$ in place of $u$. 
The connected component of $  U_r^{x,y}  \setminus  O_u$ which contains $x$ lies at Euclidean distance at least $\ep_0 r $ from the union of the other connected components, and the same is true with $v $ in place of $u $ and $y$ in place of $x$.  \label{item-highprob-dc}
\item \emph{(Upper bound for internal diameters of neighborhoods of $u$ and $v$)} Each point of $O_u$ lies at $\wt D_h(\cdot,\cdot ;U_r^{x,y} )$-distance at most $\eta \wt D_h(u,v)$ from $u $, and the same is true with $v $ in place of $u $ (here $\eta$ is as in~\eqref{eqn-const-error}). \label{item-highprob-internal} 
\end{enumerate} 
\end{lem}   
\begin{proof} 
Define the event $F_{\rho r}(z)$ for $z\in \BB C$ and $r\in\rho^{-1} \mcl R_0$ as in Lemma~\ref{lem-deterministic-geodesic}. 
\medskip

\noindent\textit{Step 1: $F_{\rho r}(z)$ occurs for many points $z\in B_{2r}(0)$.} 
Let $n_* \in \BB N$ be chosen so that the conclusion of Lemma~\ref{lem-spatial-ind} is satisfied with $s = 1/3$, $p_\dtm$ in place of $p$, and $1 - \delta (1-p)/100$ in place of $q$. 
Let $\rho :=  (500 n_*)^{-1} \delta  $ and define the set of points
\eqb
\mcl Z := \left\{ r \exp\left( \frac{2\pi \BB i \delta k }{ 100 n_* } \right)   : k \in [1,  100 n_* \delta^{-1}]_{\BB Z} \right\} \subset \bdy B_r(0) .
\eqe
Then the balls $B_{4\rho r}(z)$ for $z\in\mcl Z$ are disjoint and each such ball is contained in $ \BB A_{(1-4\rho) r , (1+4\rho) r}(0)$. 

By Lemmas~\ref{lem-deterministic-geodesic} and~\ref{lem-geo-event-local}, if $r \in \rho^{-1} \mcl R_0$, then each of the events $F_{\rho r}(z)$ for $z\in\mcl Z$ has probability at least $p_\dtm$ and is determined by $(h-h_{4\rho r}(z)) |_{B_{3\rho r}(z)}$.
Each arc $I\subset \bdy B_r(0)$ with Euclidean length at least $\delta r/4$ satisfies $\#(\mcl Z\cap I) \geq n_*$. 
Therefore, Lemma~\ref{lem-spatial-ind} (applied with the whole-plane GFF $h(\cdot/(3\rho r))$ in place of $h$) implies that for each such arc $I$, 
\eqb
\BB P\left[ \text{$\exists z\in \mcl Z \cap I$ such that $F_{\rho r}(z)$ occurs} \right] \geq 1 - \frac{\delta (1-p)}{100 } .
\eqe
We can choose at most $4\pi\delta^{-1}$ arcs of $\bdy B_r(0)$ with Euclidean length $\delta r/4$ in such a way that each arc of $\bdy B_r(0)$ with Euclidean length at least $\delta r/2$ contains one of these arcs. By a union bound, we therefore get that with probability at least $1-(1-p)/4$, 
\eqb \label{eqn-highprob-arcs}
\text{Each arc of $\bdy B_r(0)$ with length at least $\delta r/2$ contains a point $z\in \mcl Z$ s.t.\ $F_{\rho r}(z)$ occurs.}
\eqe
We will show that the statement of the lemma is satisfied with
\eqb \label{eqn-highprob-param}
\ep_0 = \ep_\dtm \rho \quad \text{and} \quad b = b_\dtm \rho .
\eqe
\medskip

\noindent\textit{Step 2: defining $U_r^{x,y}$.} 
Enumerate $\mcl Z = \{z_1,\dots,z_{N} \}$, where $N := \lfloor 100 n_*\delta^{-1} \rfloor$ and  $z_k := r \exp\left( \frac{2\pi \BB i \delta k }{ 100 n_* } \right)$. 
Also set $z_0 := z_N$. 
We now join up the balls $B_{2\rho r}(z_k)$, in a manner which is illustrated in Figure~\ref{fig-highprob-geodesic}. For $k \in [1,N]_{\BB Z}$, choose in a deterministic manner a piecewise linear path $L_k$ from $z_{k-1} + 2\rho r$ to $z_k - 2\rho r$ which is contained in $\BB A_{(1-4\rho)r , (1+4\rho)r}(0)$. We can choose the paths $L_k$ in such a way that the $L_k$'s do not intersect any of the balls $B_{2\rho r}(z)$ for $z\in\mcl Z$ and lie at Euclidean distance at least $\rho r$ from one another.  

Now consider points $x,y\in\bdy B_{2r}(0)$ with $|x-y| \geq \delta r$. By possibly re-labeling, we can assume without loss of generality that the counterclockwise arc of $\bdy B_{2r}(0)$ from $x$ to $y$ is shorter than the clockwise arc. Let $ J \subset \bdy B_r(0)$ by the counterclockwise arc from $x/2$ to $y/2$, so that $J$ has length at least $\delta r/2$. Let $k_x , k_y \in [1,N]_{\BB Z}$ be chosen so that $J \cap \mcl Z = \{z_{k_x} ,\dots, z_{k_y}\}$. 
Let $\wh L_x$ (resp.\ $\wt L_y$) be a smooth path from $x$ to $z_{k_y} - 2r$ (resp.\ from $z_{k_y} + 2r$ to $y$) which does not intersect any of the $B_{2\rho r}(z)$'s for $z\in\mcl Z$ and such that $\wh L_x$ and $\wh L_y$ lie Euclidean distance at least $\rho r$ from each other and from each $L_k$ for $k\in [k_x+1,k_y]_{\BB Z}$. 

Recall that for $X\subset\BB C$, $\mcl S_{\ep_\dtm \rho r}(X)$ denotes the set of closed Euclidean squares of side length $\ep_\dtm \rho r$ with corners in $\ep_\dtm \rho r\BB Z^2$ which intersect $X$. 
With $V_{\rho r}(z)$ as in the definition of $F_{\rho r}(z)$, we define
\eqb \label{eqn-highprob-set-def}
\ol{ U_r^{x,y} } :=  \bigcup_{k = k_x}^{k_y} \ol{V_{\rho r}(z_k)} \cup \bigcup \mcl S_{\ep_\dtm \rho r}\left(\wh L_x \cup \wh L_y \cup  \bigcup_{k=k_x+1}^{k_y} L_k \right)   
\eqe
and we let $U_r^{x,y}$ be the interior of $\ol{U_r^{x,y}}$. Since each $V_r(z_k)$ is the interior of a finite union of squares in $\mcl S_{\ep_\dtm \rho r}(B_{\rho r}(z_k))$, it follows that $U_r^{x,y}$ is the interior of a finite union of squares $\mcl S_{\ep_\dtm \rho r}(\BB A_{r/2,2r}(0) )$.
Since the $V_r(z_k)$'s are connected, it is clear that $U_r^{x,y}$ is connected and contains $x,y$.
We also note that $U_r^{x,y}$ is deterministic. 
\medskip

\noindent\textit{Step 3: checking the conditions for $u$ and $v$.} 
On the event that~\eqref{eqn-highprob-arcs} holds, there is a random $k \in [k_x , k_y]_{\BB Z}$ for which $F_{\rho r}(z_k)$ occurs. If this is the case, choose such a $k$ and point $u,v \in  V_{\rho r}(z_k) \cap \ol{B_{\rho r}(z_k)} $ as in the definition of $F_{\rho r}(z_k)$ in some measurable manner. We will show that for $\ep_0 , b$ as in~\eqref{eqn-highprob-param}, the conditions in the lemma statement hold whenever~\eqref{eqn-highprob-arcs} holds. 
\medskip

\noindent\textit{Condition~\ref{item-highprob-geo}.} Since $B_{2\rho r}(z_k) \subset B_{4\rho r}(u)$ and $V_{\rho r}(z_k) \subset U_r^{x,y}$, it is immediate from Condition~\ref{item-tube-geo} in the definition of $F_{\rho r}(z_k)$ that this condition holds with $b =  b_\dtm \rho$ whenever~\eqref{eqn-highprob-arcs} holds. 
\medskip

\noindent\textit{Condition~\ref{item-highprob-dc}.} Assume~\eqref{eqn-highprob-arcs}.  
Let
\allb
&\ol{W_k(x)}  := \bigcup_{j = k_x}^{k-1} \ol{V_{\rho r}(z_j)} \cup \bigcup \mcl S_{\ep_\dtm \rho r}\left(\wh L_x \cup   \bigcup_{j=k_x+1}^{k} L_j \right) \quad \text{and}  \notag \\ 
&\qquad \ol{W_k(y)}  := \bigcup_{j = k+1}^{k_y} \ol{V_{\rho r}(z_j)} \cup \bigcup \mcl S_{\ep_\dtm \rho r}\left(\wh L_y \cup   \bigcup_{j=k+1}^{k_y} L_j \right)  
\alle
and let $W_k(x)$ and $W_k(y)$ be the interiors of $W_k(x)$ and $W_k(y)$, respectively. 
By~\eqref{eqn-highprob-set-def}, $\ol{U_r^{x,y}} = \ol{W_k(x)} \cup \ol{W_k(y)} \cup \ol{V_{\rho r}(z_k)}$. 
Since $\wh L_x$, $\wh L_y$, and the $L_k$'s for $k\in [k_x+1,k_y]_{\BB Z}$ each lie at Euclidean distance at least $\rho r$ from one another and do not intersect the interiors of the balls $B_{\rho r}(z)$ for $z\in\mcl Z$ and $\ep_\dtm < 1/100$, the sets $W_k(x)$ and $W_k(y)$ lie at Euclidean distance at least $\rho r/2$ from each other and from $B_{\rho r}(z_k)$. 

We have 
\eqb \label{eqn-highprob-inclusion}
U_r^{x,y} \cap B_{20 \ep_\dtm \rho r}(u) = V_{\rho r}(z_k) \cap B_{20 \ep_\dtm \rho r}(u)  , 
\eqe
so the definition of $O_u$ is unaffected if we replace $U_r^{x,y}$ by $V_{\rho r}(z_k)$.  
Furthermore, the connected component of $ U_r^{x,y}  \setminus O_u $ which contains $x$ is the same as the union of $ W_k(x) $ and the connected component of $ V_{\rho r}(z_k)  \setminus  O_u $ which contains $z_k-2 \rho r$; and the union of the other connected components of $ U_r^{x,y}  \setminus O_u $ is the same as the union of $W_k(y))$ and the connected components of $ V_{\rho r}(z_k)  \setminus  O_u $ which do not contain $z_k + 2 \rho r$. 
By condition~\ref{item-tube-dc} in the definition of $F_{\rho r}(z_k)$, we find that these two sets lie at Euclidean distance at least $\ep_\dtm \rho r$ from one another.
\medskip

\noindent\textit{Condition~\ref{item-highprob-internal}.} 
By~\eqref{eqn-highprob-inclusion}, condition~\ref{item-tube-internal} in the definition of $F_{\rho r}(z_k)$ implies that each point of $O_u$ lies at $D_h(\cdot,\cdot ; U_r^{x,y} )$-distance at most $\eta\wt D_h(u,v)$ from $u $. The same is true with $v $ in place of $u $. 
\end{proof}

\subsection{Definition of the event $E_r $ and the bump functions $\mcl G_r$}
\label{sec-E-def}

\newcommand{\Aendpts}{\Delta}
\newcommand{\Adiam}{A}
\newcommand{\Abdy}{\zeta}
\newcommand{\Aacross}{a}
\newcommand{\Atiny}{\theta}
\newcommand{\Aline}{M}

The goal of this subsection is to define the event $E_r$ and the collection of smooth bump functions $\mcl G_r$ appearing in Proposition~\ref{prop-geo-event0}. 
We will also check Properties~\eqref{item-geo-event0} and~\eqref{item-dirichlet-bound} from that proposition (measurability and high probability and bounds for Dirichlet inner products). Property~\eqref{item-internal-geo} (subtracting a bump function) will be checked in Section~\ref{sec-bump-function}.

\begin{figure}[t!]
 \begin{center}
\includegraphics[scale=.8]{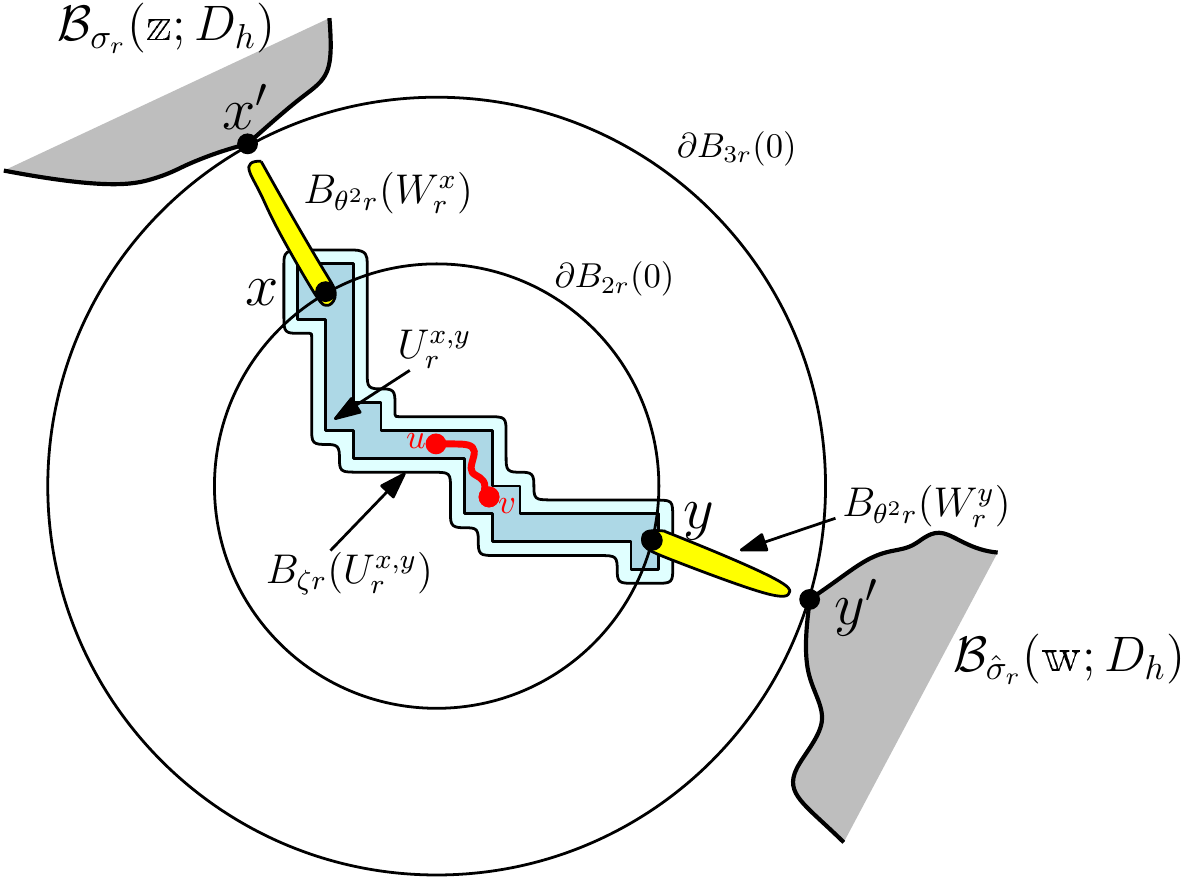} \hspace{5pt}
\includegraphics[scale=1.2]{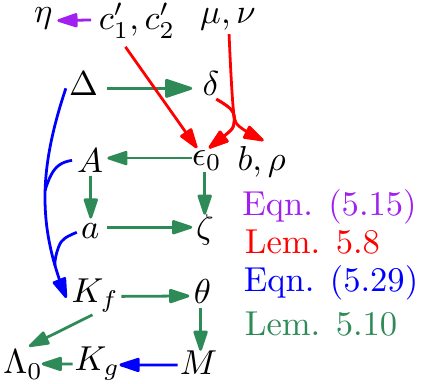} 
\vspace{-0.01\textheight}
\caption{\textbf{Left:} Illustration of the definition of the event $E_r$. 
The blue set in the middle is the set $U_r^{x,y}$ of Lemma~\ref{lem-highprob-geodesic}. The light blue region surrounding it is $B_{\Abdy r}(U_r^{x,y})$, which is the support of the bump function $f_r^{x,y}$. The yellow regions are the supports of the bump functions $g_r^x$ and $g_r^y$, which are used to force $D_h$-geodesics started from points outside of $B_{3r}(0)$ to enter $B_{\zeta r}(U_r^{x,y})$. The figure shows the relevant set for one pair of points $x,y\in\bdy B_{2r}(0)$, but all of the conditions in the event $E_r$ are required to hold \emph{simultaneously} for all pairs of points $x,y\in\bdy B_{2r}(0)$ with $|x-y|\geq \delta r$. This is important since in Section~\ref{sec-bump-function} we will take $x' = (3/2) x$ and $y'  =(3/2) y$ to be the \emph{random} points where the metric balls based at the starting and ending points of a given geodesic (here shown in grey) first hit $\bdy B_{3r}(0)$. 
\textbf{Right:} Schematic diagram of how the various quantities in the definitions of $E_r$ and $\mcl G_r$ are chosen. An arrow between two parameters indicates that one is chosen in a way which depends directly on the other. The colors indicate where the choice is made. Most of the choices in the figure depend on $\BB p$, but this is not illustrated. In the end, all of the parameters depend only on $\BB p , \mu , \nu$ (and the choice of metric). 
}\label{fig-internal-geo}
\end{center}
\vspace{-1em}
\end{figure}

The definitions in this section are illustrated in Figure~\ref{fig-internal-geo}, left.
Before proceeding with the details, we briefly discuss the main ideas involved.
Following Section~\ref{sec-shortcut-outline}, we want to define $\mcl G_r$ to include for each $x,y \in \bdy B_{3r}(0)$ a function $\phi$ which is equal to a large positive constant on the region $U_r^{x,y}$ of Lemma~\ref{lem-highprob-geodesic} and which is supported on the union of a small neighborhood of $U_r^{x,y}$ and two even narrower ``tubes" which approximate the segments $[x,3x/2]$ and $[y,3y/2]$ (shown in yellow in the figure).  
The event $E_r$ will consist of the conditions of Lemma~\ref{lem-highprob-geodesic} plus several regularity conditions discussed below.  

We will eventually consider a fixed pair of points $\BB z,\BB w \in \BB C \setminus B_{4r}(0)$ and choose $x,y\in\bdy B_{2r}(0)$ in such a way that $x' := 3x/2$ and $y' := 3y/2$ are the first points of $\bdy B_{3r}(0)$ hit by the $D_h$-metric balls grown from $\BB z$ and $\BB w$, respectively. Since these points are random, it is important that the conditions in our event hold simultaneously for all possible choices of $x$ and $y$.   We will show in Section~\ref{sec-bump-function} that on $E_r$, subtracting a suitable $\phi \in \mcl G_r$ from the field makes distances in the support of $\phi$ much shorter than distances outside, so the $D_{h-\phi}$-geodesic has to travel through the support of $\phi$ and hence has to get close to the points $u,v$ of Lemma~\ref{lem-highprob-geodesic}.

There are several subtleties involved in this argument which are dealt with via regularity conditions in the definition of $E_r$. 
For example, Lemma~\ref{lem-highprob-geodesic} requires that $|x-y| \geq \delta r$, so we need to ensure that our random metric ball hitting points $x',y'$ are separated. 
This is the purpose of condition~\ref{item-highprob-endpts} in the definition of $E_r$.  
Another difficulty is that it is relatively straightforward to get $D_{h-\phi}$-geodesics into the \emph{support} of $\phi$, but we want such geodesics to actually enter the region $U_r^{x,y}$ where $\phi$ is equal to a large positive constant. 
The reason for this is that we will be comparing ratios of distances via Weyl scaling (Axiom~\ref{item-metric-f}) and it could be that $\phi$ is much smaller on some parts of its support than it is on $U_r^{x,y}$. 
To deal with this, we will include a condition to the effect that paths which stay in a small neighborhood of $\bdy U_r^{x,y}$ without entering $U_r^{x,y}$ are very long (condition~\ref{item-highprob-bdy}).   
We also need functions in $\mcl G_r$ to be supported on $\BB A_{r,3r}(0)$ so we need to make the yellow tubes in Figure~\ref{fig-internal-geo} very close to $x'$ and $y'$ without actually allowing these tubes to contain $x'$ and $y'$ (condition~\ref{item-highprob-line-close}).
The choice of constants involved in these conditions is somewhat delicate, so the event $E_r$ will include several parameters. 

We now commence with the definitions. 
Fix a parameter $\delta \in (0,1)$, to be chosen in a manner depending only on $\BB p  $ in Lemma~\ref{lem-geo-event-prob} below.
Let $\rho , b , \ep_0$ be as in Lemma~\ref{lem-highprob-geodesic} for this choice of $\delta$ and with $p = 1-(1-\BB p)/2$, so that $\rho,b,\ep_0$ depend only on $\delta,\BB p , \mu , \nu$. 
The definitions of $E_r$ and $\mcl G_r$ involve several additional small parameters $\Aendpts \in (0,1)$ and $  \Abdy  , \Aacross ,\Atiny \in (0,\ep_0)$ and large parameters $\Adiam, \Aline,\Lambda_0 >1$ which we will choose in Lemma~\ref{lem-geo-event-prob} below, in a manner depending only on $\BB p , \mu , \nu$. 
See Figure~\ref{fig-internal-geo}, right for a schematic illustration of how the parameters are chosen. 

\subsubsection{Definition of $\mcl G_r$}

We first give the definition of $\mcl G_r$ in terms of the above parameters. 
For each $x,y\in\bdy B_{2r}(0)$ with $|x-y| \geq \delta r$, choose in a deterministic manner depending only on $ U_r^{x,y}$ (not on the particular values of $x$ and $y$) a smooth, compactly supported bump function $  f_r^{x,y} : \BB C\rta [0,1]$ which is identically equal to 1 on $U_r^{x,y}$ and vanishes outside of $B_{ \Abdy r}(U_r^{x,y} )$. 

Since each $U_r^{x,y}$ is the interior of a finite union of squares in $\mcl S_{\ep_0 r}(B_{2r}(0))$, there are at most a finite, $r$-independent number of possibilities for $U_r^{x,y}$ as $x$ and $y$ vary. 
From this and the scale invariance of Dirichlet energy (i.e., $(f(r\cdot), f(r\cdot))_\nabla = (f,f)_\nabla$) it follows that we can arrange that the Dirichlet energy $(f_r^{x,y} , f_r^{x,y})_\nabla$ is bounded above by a constant depending only on $\ep_0,\Abdy$.  
 
If we subtract a large constant multiple of $f_r^{x,y}$ from $h$, then LQG geodesics for the resulting field between points of $U_r^{x,y}$ will tend to stay in $U_r^{x,y}$. 
However, we also need to get geodesics between points of $\BB C\setminus B_{4r}(0)$ into $U_r^{x,y}$. 
For this purpose, we will also subtract even larger constant multiples of bump functions $g_r^x$ and $g_r^y$ which are supported in narrow tubes which approximate the segments $[x,3x/2]$ and $[y,3y/2]$. 
The supports of these bump functions are shown in yellow in Figure~\ref{fig-internal-geo}. 

To define these bump functions, we first define for $x\in \bdy B_{2r}(0)$ the set
\eqb \label{eqn-line-tube-def}
W_r^x = W_r^x(\theta) := \left( \text{Interior of $\bigcup_{S\in \mcl S_{\Atiny r} ( [x , (3/2-\Atiny) x])} S$} \right) \subset \BB A_{r,3r}(0) 
\eqe
where here we recall from~\eqref{eqn-square-set-def} that $\mcl S_{\Atiny r}([x , (3/2-\Atiny) x])$ is the set of $\Atiny r \times \Atiny r$ squares with corners in $\Atiny r \BB Z^2$ which intersect $[x , (3/2-\Atiny) x]$.  
Let $g_r^x : \BB C\rta [0,1]$ be a smooth compactly supported function which is identically equal to 1 on $W_r^x$ and is identically equal to 0 outside of $B_{\Atiny^2 r}(W_r^x) \subset B_{3r}(0)$. 
As in the case of $f_r^{x,y}$ (see the paragraph just above~\eqref{eqn-bump-function-const}), we can arrange that the Dirichlet energy of $g_r^x$ is bounded above by a constant  depending only on $\Atiny,\BB p$. 

We define the large constants
\eqb \label{eqn-bump-function-const}
K_f := \frac{1}{\xi} \log\left(\frac{100 \Adiam}{\Aacross \Aendpts} \right)
\quad \text{and} \quad
K_g := K_f + \frac{1}{\xi} \log\left( \Aline \right) .
\eqe
For each $x,y\in \bdy B_{2r}(0)$ with $|x-y| \geq \delta r$, we define
\eqb \label{eqn-bump-function-def}
\phi_r^{x,y} := K_f f_r^{x,y} + K_g(g_r^x + g_r^y) .
\eqe
Since each of $f_r^{x,y}, g_r^x,g_r^y$ is supported on $\BB A_{r/4,3r}(0)$, so is $\phi_r^{x,y}$. 
We set
\eqb \label{eqn-bump-function-collection}
\mcl G_r := \left\{ \phi_r^{x,y} : x,y\in \bdy B_{2r}(0) ,\: |x-y| \geq \delta r \right\} \cup \{\text{zero function}\} .
\eqe
We emphasize that the definition of $\mcl G_r$ does not depend on the parameter $\Lambda_0$. This will be important when we choose $\Lambda_0$ in Lemma~\ref{lem-geo-event-prob} below.

Recall from the above discussion that the number of possibilities for each of $f_r^{x,y} , g_r^x , g_r^y$ as $x$ and $y$ vary and the Dirichlet energies of each of these functions is bounded above by a constant which does not depend on $r,x$, or $y$. 
Consequently, each of
\eqb \label{eqn-bump-function-bound}
\#\mcl G_r \quad \text{and} \quad \max_{\phi \in \mcl G_r} (\phi,\phi)_\nabla
\eqe
is bounded above by a constant which does not depend on $r$, $x$, or $y$.  

\subsubsection{Definition of $E_r$}

We now define the event $E_r$ appearing in Proposition~\ref{prop-geo-event0}. 

We encourage the reader to skim the list of conditions on a first read and refer back to them as they are used while reading the proof of Lemma~\ref{lem-internal-geo} below. 

With the parameters $\delta, \Aendpts , \Adiam, \Abdy,\Aacross,\Atiny ,\Aline,\Lambda_0 $ as above, we define $E_r$ to be the event that the following is true.
For each $x,y\in\bdy B_{2r}(0)$ with $|x-y| \geq \delta r$, there exists $u,v\in \BB A_{(1-4\rho)r,(1+4\rho)r}(0)$ satisfying the three numbered conditions of Lemma~\ref{lem-highprob-geodesic} and moreover the following additional conditions hold.   
\begin{enumerate}
\setcounter{enumi}{3}
\item For each $x,y\in\bdy B_{2r}(0)$ with $|x -y  |  <  \delta r$,  \label{item-highprob-endpts}
\eqbn
D_h\left( x',y' ; \BB A_{ r , 4r} (0) \right)  \leq \Aendpts \frk c_r e^{\xi h_r(0)} \leq  D_h\left( \bdy B_{2r}(0), \bdy B_{3r}(0) \right) 
,\quad \text{where} \quad x' = \frac32 x \: \text{and}\: y' =\frac32 y.
\eqen 
\item For each $x,y \in \bdy B_{2r}(0)$ with $|x-y| \geq \delta r$ the $D_h$-internal diameter of $U_r^{x,y}$ satisfies  \label{item-highprob-diam}
\eqbn \label{eqn-highprob-diam}
\sup_{w_1,w_2 \in  U_r^{x,y} } D_h\left(w_1,w_2 ; U_r^{x,y} \right) \leq \Adiam \frk c_r e^{\xi h_r(0)} . 
\eqen 
\item For each $x,y \in \bdy B_{2r}(0)$ with $|x-y| \geq \delta r$, the $D_h$-length of every continuous path of Euclidean diameter at least $\ep_0 r/100$ which is contained in $B_{2\Abdy r}(\bdy U_r^{x,y})$ is at least $100 \Adiam \frk c_r e^{\xi h_r(0)} $. \label{item-highprob-bdy}
\item For each $z_1,z_2 \in  \BB A_{r/4,4r}(0)$ such that $  |z_1-z_2| \geq \Abdy r$,   \label{item-highprob-across}
\eqbn \label{eqn-highprob-across}
D_h\left( z_1, z_2 ; \BB A_{r/4,4r}(0) \right)  \geq \Aacross \frk c_r e^{\xi h_r(0)}    .
\eqen
\item With $K_f $ as in~\eqref{eqn-bump-function-const},\label{item-highprob-line-close}
\eqbn \label{eqn-highprob-line-close}
D_h\left(3x/2, (3/2-\Atiny) x ; \BB A_{r,4r}(0) \right) \leq   e^{-\xi K_f} \frk c_r e^{\xi h_r(0)}  , \quad\forall x\in \bdy B_{2r}(0) .
\eqen
\item If we let $W_r^x \subset \BB A_{r,3r}(0)$ be the long narrow tube as in~\eqref{eqn-line-tube-def}, then \label{item-highprob-line-internal}
\eqbn \label{eqn-highprob-line-internal} 
\sup_{w_1,w_2 \in W_r^x} D_h(w_1,w_2 ; W_r^x) \leq \Aline \frk c_r e^{\xi h_r(0)} . 
\eqen
\item With $\mcl G_r$ as in~\eqref{eqn-bump-function-collection}, we have $(h,\phi)_\nabla + \frac12 |(\phi,\phi)_\nabla| \leq \Lambda_0$ for each $\phi\in\mcl G_r$. \label{item-highprob-dirichlet} 
\end{enumerate}

The conditions in the definition of $E_r$ are numbered in such a way that the new parameters involved in each condition depend only on the parameters from the previous conditions. 
We now comment briefly on the purpose of each of the conditions.
As discussed in Section~\ref{sec-shortcut-outline}, to prove Property~\eqref{item-internal-geo} (subtracting a bump function) of Proposition~\ref{prop-geo-event0}, we will grow the $D_h$-metric balls started from $\BB z$ and $\BB w$ until they hit $B_{3r}(0)$. 
We will let $\BB x' $ and $\BB y'$ be their respective hitting points, and we will apply the above conditions with $x = \BB x := (2/3) \BB x'$ and $y = \BB y = (2/3)\BB y'$ (note that $\BB x,\BB y\in \bdy B_{2r}(0)$).

Condition~\ref{item-highprob-endpts} is used to ensure that if $P$ hits $B_{2r}(0)$, then $|\BB x - \BB y| \geq \delta r$ (see Lemma~\ref{lem-hitpt-lower}). 
Condition~\ref{item-highprob-diam} gives us a deterministic upper bound for the $D_h$-diameter of $U_r^{\BB x , \BB y}$ before we subtract the bump function $\phi$. This allows us say that  the $D_{h-\phi}$-diameter of $U_r^{x,y}$ is very small, which is what forces the $D_{h-\phi}$-geodesic $P^\phi$ to enter $U_r^{\BB x , \BB y}$.
Condition~\ref{item-highprob-bdy} prevents $P^\phi$ from staying close to $\bdy U_r^{\BB x , \BB y}$ (in the region where $\phi$ positive, but does not attain its largest possible value) without entering $U_r^{x,y}$ itself. 
Condition~\ref{item-highprob-across} is used to prevent $P^\phi$ from exiting $B_{\Abdy r} ( U_r^{\BB x , \BB y})$ prematurely. 
Conditions~\ref{item-highprob-line-close} and~\ref{item-highprob-line-internal} concern the yellow tubes in Figure~\ref{fig-internal-geo}. These conditions are used to force $P^\phi$ to enter and exit $U_r^{\BB x , \BB y}$ at points near $\BB x$ and $\BB y$, respectively.
Condition~\ref{item-highprob-dirichlet} is used to prove Property~\eqref{item-dirichlet-bound} (bounds for Dirichlet inner products) of Proposition~\ref{prop-geo-event0}.

\subsubsection{Proof of Properties~\eqref{item-geo-event0} and~\eqref{item-dirichlet-bound}}

It is immediate from condition~\ref{item-highprob-dirichlet} in the definition of $E_r$ that Property~\eqref{item-dirichlet-bound} (bounds for Dirichlet inner products) of Proposition~\ref{prop-geo-event0} is satisfied. In the next two lemmas we check the two assertions of Property~\eqref{item-geo-event0} (measurability and high probability). 
 
\begin{lem} \label{lem-geo-event-msrble}
The event $E_r$ is determined by $(h-h_{5r}(0)) |_{\BB A_{r/4,4r}(0)}$
\end{lem}
\begin{proof}
By Axiom~\ref{item-metric-f} (Weyl scaling), the occurrence of $E_r$ is unaffected by adding a real number to $h$, so we only need to show $E_r \in \sigma\left( h|_{\BB A_{r/4,4r}(0)} \right)$. 
The measurability of condition~\ref{item-highprob-geo} follows from exactly the same argument used in the proof of Lemma~\ref{lem-geo-event-local} (this can also be seen from Lemma~\ref{lem-geo-event-local} and the proof of Lemma~\ref{lem-deterministic-geodesic}).
Since $U_r^{x,y} , W_r^x , W_r^y \subset \BB A_{(1/2 - 2\ep_0) r , 3 r}(0)$ and $D_h$ and $\wt D_h$ are local metrics for $h$, the measurability of the other conditions in the definition of $E_r$ follows by inspection and Axiom~\ref{item-metric-local} (locality). 
\end{proof}

\begin{lem} \label{lem-geo-event-prob}
We can choose the parameters $\delta, \Aendpts , \Adiam, \Abdy,\Aacross,\Atiny ,\Aline,\Lambda_0 $ in a manner depending only on $\BB p , \mu,\nu,c_1',c_2'$ in such a way that $\BB P[E_r] \geq \BB p$ for each $r\in \rho^{-1}\mcl R_0$. 
\end{lem}
\begin{proof}
By tightness across scales (Axiom~\ref{item-metric-coord}), we can choose $\Aendpts$ and then $\delta$ in such a way that condition~\ref{item-highprob-endpts} holds with probability at least $1-(1-\BB p)/100$.  
As above, we choose $b,\rho,\ep_0$ as in Lemma~\ref{lem-highprob-geodesic} with the above choice of $\delta$ and with $p = 1-(1-\BB p)/100$ (so that $b,\rho$ depend only on $\BB p,\mu,\nu$ and $\ep_0$ depends only on $\BB p , \mu,\nu,c_1', c_2'$) and define $U_r^{x,y}$ for $x,y \in \bdy B_{2r}(0)$ with $|x-y| \geq \delta r$ as in that lemma.  Then the first four conditions (including the three from Lemma~\ref{lem-highprob-geodesic}) in the definition of $E_r$ occur simultaneously with probability at least $1-2(1-\BB p)/100$.

We will now choose the parameters so as to lower-bound the probabilities of the other conditions in the definition of $E_r$ in numerical order. 
By Lemma~\ref{lem-square-diam}, we can find $C> 0$ depending only on $\ep_0$ (and hence only on $p ,\mu,\nu ,c_1',c_2' $) such that with probability at least $1- (1-\BB p)/100$, we have, with $\mcl S_{\ep_0 r}(\cdot)$ as in~\eqref{eqn-square-set-def},
\eqb \label{eqn-highprob-squares}
\sup_{S\in \mcl S_{\ep_0  r}(B_{2r}(0))} \sup_{w_1,w_2 \in S} D_h(w_1,w_2 ; S) \leq C \frk c_r e^{\xi h_r(0)} .
\eqe 
The total number of squares of $\mcl S_{\ep_0  r}(B_{2r}(0))$ is at bounded above by a constant depending only on $\ep_0$ (and hence only on $\BB p,\mu,\nu,c_1',c_2'  $).
Since each $U_r^{x,y}$ is connected and is the interior of a finite union of such squares, the triangle inequality shows that there is an $\Adiam >1$ depending only on $\BB p ,\mu,\nu $ such that whenever~\eqref{eqn-highprob-squares} holds, also condition~\ref{item-highprob-diam} holds. 
Hence the probability of condition~\ref{item-highprob-diam} is at least $1-(1-\BB p)/100$. 

The set $\bdy U_r^{x,y}$ is the union of some subset of the set of sides of squares in $\mcl S_{\ep_0  r}(B_{2r}(0))$. 
By Lemma~\ref{lem-line-path} (applied with $\Abdy$ in place of $\ep $) and a union bound over all of the sides of all of the squares in $\mcl S_{\ep_0 r}(B_{2r}(0))$, we can choose $\Abdy \in (0,\ep_0/100)$ depending only on $\BB p , \ep_0 , \Adiam$ (and hence only on $  \BB p ,  \mu , \nu ,c_1',c_2' $) such that condition~\ref{item-highprob-bdy} holds with probability at least $1 -  (1-\BB p)/100$. 
 
Since $D_h$ induces the Euclidean topology, we can find $\Aacross \in (0,1)$ depending only on $\BB p , \Abdy$ (and hence only on $ \BB p , \mu , \nu ,c_1',c_2'   $) such that condition~\ref{item-highprob-across} holds with probability at least $1- (1-\BB p)/100$.   

Since the constant $K_f$ of~\eqref{eqn-bump-function-const} depends only on $\Adiam, \Aendpts,\Aacross$, which have already been chosen in a manner depending only on $\BB p , \mu , \nu ,c_1',c_2' $, we can find a small enough $\Atiny \in (0,\Abdy/100)$ depending only on $\BB p , \mu , \nu ,c_1',c_2' $ such that condition~\ref{item-highprob-line-close} holds with probability at least $1-(1-\BB p)/100$.  

Recall from~\eqref{eqn-line-tube-def} that $W_r^x$ is the interior of the union of a set of squares in $\mcl S_{\Atiny r}(B_{3r}(0))$. 
By Axiom~\ref{item-metric-coord} (tightness across scalings) and Lemma~\ref{lem-square-diam}, we can find a sufficiently large $\Aline > 0$ depending only $\Atiny$ (hence only on $\BB p ,\mu,\nu ,c_1',c_2' $) such that  condition~\ref{item-highprob-line-internal} holds with probability at least $1-(1-\BB p)/100$. 

The definition of the set of bump functions $\mcl G_r$ above does not use the parameter $\Lambda_0$. As discussed just after~\eqref{eqn-bump-function-bound}, the number of functions in $\mcl G_r$ and the Dirichlet energies of these functions are each bounded above by constants which depend only on $\BB p ,\mu , \nu ,c_1',c_2' $ and the other parameters which we have already chosen in a manner depending only on $\BB p , \mu , \nu ,c_1',c_2' $. 
Consequently, we can find a constant $\Lambda_0 > 0$ depending only on $\BB p,\mu,\nu,c_1',c_2' $ such that condition~\ref{item-highprob-dirichlet} holds with probability at least $1- (1-\BB p)/100$. 
Combining our above estimates gives the statement of the lemma.
\end{proof}

\subsection{Subtracting a bump function to move a geodesic}
\label{sec-bump-function}

To prove Proposition~\ref{prop-geo-event0}, it remains to check Property~\eqref{item-internal-geo} (subtracting a bump function) for the event $E_r$ and the collection of smooth bump functions $\mcl G_r$ defined above. 
To this end, fix distinct points $\BB z , \BB w \in \BB C\setminus B_{4r}(0)$ and let $P = P^{\BB z,\BB w}$ be the (a.s.\ unique) $D_h$-geodesic from $\BB z$ to $\BB w$.  
We first grow the $D_h$-metric balls until they hit $\bdy B_{3r}(0)$. 
Let $\sigma_r$ (resp.\ $\wh \sigma_r$) be the smallest $s > 0$ for which the $D_h$-metric ball $\mcl B_s(\BB z ;D_h)$ (resp.\ $\mcl B_s(\BB w ; D_h)$) intersects $\ol{B_{3r}(0)}$.
Also let $\BB x'$ (resp.\ $\BB y'$) be a point of $\bdy B_{3r}(0) \cap \mcl B_{\sigma_r}(\BB z ; D_h)$ (resp.\ $\mcl B_{\wh \sigma_r}(\BB w ;D_h)$), chosen in some manner depending only on the appropriate $D_h$-metric ball\footnote{It is in fact not difficult to see that there is a.s.\ a unique intersection point by repeating the argument of \cite[Theorem~1.2]{mq-geodesics}.}, and define the points of $\bdy B_{2r}(0)$
\eqb
\BB x := (2/3) \BB x'  \quad \text{and} \quad \BB y := (2/3) \BB y ' .
\eqe
Note that $\BB x , \BB y \in \sigma\left( h|_{\BB C\setminus B_{3r}(0)} \right)$. 

In the notation~\eqref{eqn-bump-function-def}, we set
\eqb \label{eqn-bump-function-choice}
\phi = \begin{cases}
\phi_r^{\BB x , \BB y} , \quad &\text{if $|\BB x-\BB y| \geq \delta r$} \\
0 ,\quad &\text{otherwise} .
\end{cases}
\eqe
Then $\phi \in \mcl G_r$, as defined in~\eqref{eqn-bump-function-collection}, and $\phi$ is determined by $\BB x , \BB y$ and hence by $h|_{\BB C\setminus B_{3r}(0)}$. 
Hence to prove Property~\eqref{item-internal-geo} it remains only to prove the following. 

\begin{lem} \label{lem-internal-geo}
Let $P^\phi$ be the (a.s.\ unique) $D_{h - \phi}$-geodesic from $\BB z$ to $\BB w$.
If $P\cap B_{2r}(0) \not=\emptyset$ and $E_r$ occurs, then there are times $0  < s < t  <  D_{h - \phi}(\BB z, \BB w) $ such that 
\allb \label{eqn-internal-geo}
&P^\phi(s) , P^\phi(t) \in B_{3r/2}(0) , \quad 
|P^\phi(s) - P^\phi(t)| \geq (b - 40 \ep_0) r , \notag\\
&\qquad \wt D_{h-\phi }\left(P^\phi(s) , P^\phi(t)\right) \leq c_2'  D_{h- \phi}\left(P^\phi(s) , P^\phi(t)\right) ,\quad \text{and} \notag\\
&\qquad  \wt D_{h- \phi}\left(P^\phi(s) , P^\phi(t)\right) \leq (c_* / C_*) \wt D_{h-\phi}\left(P^\phi(s) , \bdy B_{3r}(0) \right) .
\alle
\end{lem}

The rest of this section is devoted to the proof of Lemma~\ref{lem-internal-geo}.
To lighten notation, write 
\eqb \label{eqn-internal-geo-domains}
U = U_r^{\BB x,\BB y} \quad \text{and} \quad \mcl W = B_{\Atiny^2 r} ( W_r^{\BB x})  \cup B_{\Atiny^2 r}( W_r^{\BB y} ) .
\eqe
Throughout, we assume that $E_r$ occurs and $P\cap B_{2r}(0)\not=\emptyset$.   
The proof is an elementary (though somewhat technical) deterministic argument using the conditions in the definition of $E_r$, and is divided into several lemmas.

\begin{lem} \label{lem-hitpt-lower}
We have $|\BB x - \BB y| \geq \delta r$. 
\end{lem}

Lemma~\ref{lem-hitpt-lower} allows us to apply all of the conditions in the definition of $E_r$ with $x = \BB x$ and $y=\BB y$ (note that these conditions hold for all $x,y\in \bdy B_{2r}(0)$ with $|x-y|\geq \delta r$ simultaneously). 
We will use this fact without comment throughout the rest of the proof. 

\begin{proof}[Proof of Lemma~\ref{lem-hitpt-lower}] 
Since $P$ is a $D_h$-geodesic, the $D_h$-distance between the metric balls $\mcl B_{\sigma_r}(\BB z ; D_h)$ and $\mcl B_{\wh \sigma_r}(\BB w ;D_h)$ is equal to the $D_h$-distance traveled by $P$ between the times when it hits these two metric balls. 
Since $P$ enters $B_{2r}(0)$, it must cross between the inner and outer boundaries of $\BB A_{2r,3r}(0)$ at least twice between hitting these two metric balls, so the $D_h$-distance between $\mcl B_{\sigma_r}(\BB z ; D_h)$ and $\mcl B_{\wh \sigma_r}(\BB w ;D_h)$ must be at least $2 D_h(\bdy B_{2r}(0) , \bdy B_{3r}(0))$. Condition~\ref{item-highprob-endpts} in the definition of $E_r$ implies that if $|\BB x - \BB y| < \delta r$ then $D_h(\bdy B_{2r}(0), \bdy B_{3r}(0)) \geq D_h(\BB x', \BB y' ; \BB A_{r,4r}(0))$ which is at least the $D_h$-distance between $\mcl B_{\sigma_r}(\BB z ; D_h)$ and $\mcl B_{\wh \sigma_r}(\BB w ;D_h)$.  This is a contradiction and therefore $|\BB x - \BB y| \geq \delta r$.
\end{proof}

We now prove an upper bound for $D_{h-\phi}(\BB x',\BB y')$. Since $P^\phi$ is a $D_{h-\phi}$-geodesic, this upper bound will allow us to constrain the behavior of $P^\phi$ since $P^\phi$ cannot have any segment whose $D_{h-\phi}$-length is larger than $D_{h-\phi}(\BB x',\BB y')$ (see Lemma~\ref{lem-geomove-restrict} below). 

\begin{lem} \label{lem-geomove-dist}
We have
\eqb \label{eqn-geomove-dist}
D_{h-\phi}\left( \BB x' , \BB y' \right) 
\leq e^{- \xi K_f} (\Adiam +4) \frk c_r e^{\xi h_r(0)}   .
\eqe
\end{lem}
\begin{proof}
By condition~\ref{item-highprob-line-close} in the definition of $E_r$ and since $D_{h-\phi} \leq D_h$,
\eqb \label{eqn-geomove-tip}
D_{h-\phi}\left(\BB x' , W_r^{\BB x} \right) \leq   e^{-\xi K_f} \frk c_r e^{\xi h_r(0)}   \quad \text{and} \quad
D_{h-\phi}\left(\BB y' , W_r^{\BB y} \right) \leq   e^{-\xi K_f} \frk c_r e^{\xi h_r(0)} .
\eqe
By condition~\ref{item-highprob-line-internal}, Axiom~\ref{item-metric-f} (Weyl scaling), and since $\phi \geq K_g$ on each of $W_r^{\BB x}$ and $W_r^{\BB y}$ (with $K_g$ as in~\eqref{eqn-bump-function-const}), 
\eqb \label{eqn-geomove-line}
\text{The internal $D_{h-\phi}$-diameters of $W_r^{\BB x}$ and $W_r^{\BB y}$ are each $\leq e^{-\xi K_g} \Aline \frk c_r e^{ \xi h_r(0)} \leq e^{-\xi K_f} \frk c_r e^{\xi h_r(0)}$. }
\eqe
By condition~\ref{item-highprob-diam}, Axiom~\ref{item-metric-f}, and since $\phi \geq  K_f$ on $U$, 
\eqb \label{eqn-geomove-tube}
\sup_{w_1,w_2 \in  U } D_{h-\phi}\left(w_1,w_2 ; U \right) 
\leq e^{-\xi K_f} \Adiam \frk c_r e^{\xi h_r(0)}  .
\eqe
Since $W_r^{\BB x}$ and $W_r^{\BB y}$ each intersect $U$, we can combine~\eqref{eqn-geomove-tip}, \eqref{eqn-geomove-line}, and~\eqref{eqn-geomove-tube} and use the triangle inequality to get~\eqref{eqn-geomove-dist}.
\end{proof}  

\begin{lem} \label{lem-geomove-restrict}
To lighten notation, let 
\eqbn
\ol P^\phi := P^\phi  \setminus \left( \mcl B_{\sigma_r}(\BB z ; D_h) \cup \mcl B_{\wh \sigma_r}(\BB w ;D_h) \right).
\eqen 
In the notation~\eqref{eqn-internal-geo-domains}, $\ol P^\phi$ is contained in $B_{2\Abdy r}(U \cup \mcl W)$. Furthermore, there is no segment of $\ol P^\phi$ of Euclidean diameter $\geq \ep_0 r/100$ which is contained in $B_{2\Abdy r}(\bdy U) \setminus  \mcl W$. 
\end{lem}
\begin{proof}
Since $\phi$ is supported on $B_{3r}(0)$, the definitions of $\sigma_r$, $\wh\sigma_r$, $\mcl B_{\sigma_r}(\BB z ; D_h)$, and $\mcl B_{\wh \sigma_r}(\BB w ;D_h)$ are unaffected if we replace $h$ by $h-\phi$. 
Since $\ol P^\phi$ is the $D_{h-\phi}$-shortest path between these metric balls, Lemma~\ref{lem-geomove-dist} implies that 
\eqb \label{eqn-geomove-length}
\left(\text{$D_{ h-\phi }$-length of $\ol P^\phi$} \right) \leq e^{- \xi K_f} (\Adiam +4) \frk c_r e^{\xi h_r(0)} .
\eqe  
We will now explain how~\eqref{eqn-geomove-length} together with the definition of $E_r$ allows us to constrain the behavior of $\ol P^\phi$. 

As in the proof of Lemma~\ref{lem-hitpt-lower}, condition~\ref{item-highprob-endpts} in the definition of $E_r$ implies that the $D_h$-distance between $\mcl B_{\sigma_r}(\BB z ; D_h)$, and $\mcl B_{\wh \sigma_r}(\BB w ;D_h)$ is at least $2 \Aendpts \frk c_r e^{\xi h_r(0)}$, which is larger than $e^{- \xi K_f} (\Adiam +4) \frk c_r e^{\xi h_r(0)}$ by the definition~\eqref{eqn-bump-function-const} of $K_f$. 
If $\ol P^\phi$ did not enter the support $B_{\Abdy r}(U) \cup \mcl W$ of $\phi$, then the $D_{h-\phi}$-length of $\ol P^\phi$ would be the same as its $D_h$-length, which must be at least $2 \Aendpts \frk c_r e^{\xi h_r(0)}$.
Hence~\eqref{eqn-geomove-length} implies that $\ol P^\phi$ must enter $B_{\Abdy r}(U) \cup \mcl W$.  

Since $\phi \leq K_f$ outside of $\mcl W$, Axiom~\ref{item-metric-f} (Weyl scaling) together with condition~\ref{item-highprob-bdy} in the definition of $E_r$ implies that the $D_{h - \phi}$-length of every continuous path of Euclidean diameter at least $\ep_0 r/100$ which is contained in $B_{2\Abdy r}(\bdy U) \setminus \mcl W$ is at least $100 e^{-\xi K_f} \Adiam \frk c_r e^{\xi h_r(0)}$. 

It therefore follows from~\eqref{eqn-geomove-length} that the second assertion of the lemma holds. 

We now prove the first assertion of the lemma. 
Since $\phi$ is identically equal to 0 on $\BB C\setminus (B_{\Abdy r}(U) \cup \mcl W)$, condition~\ref{item-highprob-across} in the definition of $E_r$ implies that the $D_{h- \phi}$-length of any curve which is contained in $\BB A_{r/4,4r}(0) \setminus (B_{\Abdy r}(U)  \cup  \mcl W)$ and has Euclidean diameter at least $\Abdy r$ is at least $\Aacross  \frk c_r e^{\xi h_r(0)}$.  
This last quantity is strictly larger than the right side of~\eqref{eqn-geomove-length} by the definition~\eqref{eqn-bump-function-const} of $K_f$.  
It follows that there is no segment of $\ol P^\phi$ of Euclidean diameter at least $ \Abdy r$ which is contained in $\BB A_{r/4,4r}(0) \setminus (B_{\Abdy r}(U) \cup \mcl W)$. 
Each path from $ B_{\Abdy r}(U) \cup \mcl W$ to a point outside of $B_{2\Abdy r}(U \cup \mcl W)$ has a sub-path which is contained in $\BB A_{r/4,4r}(0) \setminus  (B_{\Abdy r}(U) \cup \mcl W)$ and has Euclidean diameter at least $\Abdy r$. Since we know that $\ol P^\phi$ has to hit $B_{\Abdy r}(U) \cup \mcl W$, we infer that $\ol P^\phi$ is contained in $B_{2\Abdy r}(U) \cup \mcl W)$. 
\end{proof}

We now produce the points $0  < s < t  <  D_{h - \phi}(\BB z, \BB w) $ from Lemma~\ref{lem-internal-geo} and check all of the conditions of the lemma except $\wt D_{h-\phi }\left(P^\phi(s) , P^\phi(t)\right) \leq c_2'  D_{h- \phi}\left(P^\phi(s) , P^\phi(t)\right)$ (we will check this last condition in the proof of Lemma~\ref{lem-internal-geo} just below).

\begin{lem} \label{lem-internal-geo0}
There are times $0  < s < t  <  D_{h - \phi}(\BB z, \BB w) $ such that $P^\phi(s) , P^\phi(t) \in B_{3r/2}(0) $, $|P^\phi(s) - P^\phi(t)| \geq (b - 40 \ep_0) r$, and
\eqb \label{eqn-internal-dist-bound}
\wt D_{h-\phi}(P^\phi(s) , P^\phi(t) ) \leq (c_*/C_*)  \wt D_{h -\phi }\left( P^\phi(s) , \bdy B_{3r}(0) \right) .
\eqe 
\end{lem}
\begin{proof}
Recall the points $u,v \in \BB A_{(1-4\rho)r,(1 + 4\rho) r}(0)$ from condition~\ref{item-highprob-geo} in the definition of $E_r $. 
That condition says that the $\wt D_h$-geodesic $\wt P$ from $u$ to $v$ is contained in $U$ and its $\wt D_h$-length is at most $(c_*/C_*)^2 \wt D_h(u , \bdy B_{4\rho r}(u))$. 
The idea of the proof is to use Lemma~\ref{lem-geomove-restrict} to force $P^\phi$ to get close to each of $u$ and $v$, and then to take $s$ and $t$ to be the times at which it does so. 
Since $\phi$ attains its largest possible value on $B_{4\rho r}(u)$ (namely, $K_f$) at every point of $B_{4\rho r}(u) \cap U$ (here we note that $\mcl W$ is disjoint from $B_{3r/2}(0) \supset B_{4\rho r}(u)$), it follows that $\wt P(e^{\xi K_f}\cdot)$ is a $\wt D_{h-\phi}$-geodesic from $u$ to $v$ and 
\eqb \label{eqn-internal-tildeD}
\wt D_{h - \phi}(u,v) = \wt D_{h - \phi}\left(u,v ; U \cap B_{4\rho r}(u) \right)  = e^{-\xi K_f} \wt D_h(u,v) . 
\eqe

Recall from condition~\ref{item-highprob-dc} in the definition of $E_r$ that $O_u$ (resp.\ $O_v$) is the connected component of $U\cap B_{20\ep_0 r}(u)$ which contains $u$.
Since $B_{20\ep_0 r}(u) $ is contained in $B_{3r/2}(0)$, so is disjoint from $\mcl W$, that condition tells us that the connected component of $ (U \cup \mcl W )\setminus O_u$ which contains $\BB x'$ lies at Euclidean distance at least $\ep_0 r$ from the union of the other connected components of $(U\cup \mcl W) \setminus O_u$. 
Since $\Abdy < \ep_0/100$, the $2\Abdy r$-neighborhoods of these two sets lie at Euclidean distance at least $\ep_0 r/2$ from one another.
By Lemma~\ref{lem-geomove-restrict}, $\ol P^\phi = P^\phi  \setminus \left( \mcl B_{\sigma_r}(\BB z ; D_h) \cup \mcl B_{\wh \sigma_r}(\BB w ;D_h) \right)$ cannot exit $B_{2\Abdy r}(U \cup \mcl W)$, so $\ol P^\phi$ must have a segment of Euclidean diameter at least $\ep_0 r/2$ which is contained in  
\eqbn
B_{2\Abdy r}(O_u) \subset O_u \cup \left( B_{2\Abdy r}(\bdy U) \setminus  \mcl W \right)  .
\eqen 
By the other assertion of Lemma~\ref{lem-geomove-restrict}, this segment cannot be entirely contained in $ B_{2\Abdy r}(\bdy U) \setminus  \mcl W $, so $P^\phi$ must enter $O_u$. 
Similarly, $P^\phi$ must enter $O_v$ (and must do so at some time after it enters $O_u$). 

Choose times $0 < s < t < |P^\phi|$ such that $P^\phi(s) \in O_u$ and $P^\phi(t) \in O_v$. 
Then $|P^\phi(s) - u| \leq 20 \ep_0 r$ and $|P^\phi(t) - v| \leq 20 \ep_0 r$, 
By condition~\ref{item-highprob-internal} in the definition of $E_r$,~\eqref{eqn-internal-tildeD}, and the fact that $\phi \equiv K_f$ on $U \setminus \mcl W$, we get that 
\eqb \label{eqn-internal-times}
\wt D_{h-\phi}\left(P^\phi(s) , u ; U \right) \leq \eta  \wt D_{h-\phi}(u,v) 
\quad \text{and} \quad 
\wt D_{h-\phi}\left( P^\phi(t) , v ; U \right) \leq \eta  \wt D_{h-\phi}(u,v) .
\eqe
Since $|u-v| \geq b r$, we have $|P^\phi(s) - P^\phi(t) | \geq (b- 40 \ep_0) r$ and since $u,v \in \ol{B_r(0)}$ we have $P^\phi(s) , P^\phi(t) \in B_{3r/2}(0)$. 

It remains to check the condition~\eqref{eqn-internal-dist-bound}. 
Recall that $\wt D_h(u,v) \leq  (c_*/C_*)^2 \wt D_h(u , \bdy B_{4\rho r}(u))$ and the $\wt D_h$-geodesic from $u$ to $v$ is contained in $U$.
Since $\phi \equiv K_f$ on $U$ and $\phi \leq K_f$ on $B_{4\rho r}(u)$, it follows that 
\eqbn
\wt D_{h-\phi}(u,v) \leq (c_*/C_*)^2 \wt D_{h-\phi}(u , \bdy B_{4\rho r}(u)) \leq (c_*/C_*)^2 \wt D_{h -\phi}\left( P^\phi(s) , \bdy B_{3r}(0) \right).
\eqen
By~\eqref{eqn-internal-times} and the triangle inequality,
\eqbn
\wt D_{h-\phi}\left(P^\phi(s) , P^\phi(t) \right)
\leq (1+2\eta) \wt D_{h-\phi}(u,v) 
\leq (1+2\eta)(c_*/C_*)^2 \wt D_{h-\phi}\left( P^\phi(s) , \bdy B_{3r}(0) \right) ,
\eqen
which is bounded above by the right side of~\eqref{eqn-internal-dist-bound} by the definition~\eqref{eqn-const-error} of $\eta$.  
\end{proof}

\begin{proof}[Proof of Lemma~\ref{lem-internal-geo}]
Let $s$ and $t$ be as in Lemma~\ref{lem-internal-geo0}. By that lemma, it remains only to check that
\eqbn
\wt D_{h-\phi }\left(P^\phi(s) , P^\phi(t)\right) \leq c_2'  D_{h- \phi}\left(P^\phi(s) , P^\phi(t)\right).
\eqen
By~\eqref{eqn-internal-times} and the definitions of $c_*$ and $C_*$,  
\eqb \label{eqn-internal-times'}
 D_{h - \phi}\left(P^\phi(s) , u ; U \right) \leq   c_*^{-1} C_* \eta D_{h - \phi}(u,v) 
\quad \text{and} \quad 
 D_{h - \phi}\left( P^\phi(t) , v ; U \right) \leq  c_*^{-1} C_*  \eta D_{h - \phi}(u,v) .
\eqe 
By the triangle inequality,~\eqref{eqn-internal-times'} implies that
\alb
D_{h - \phi}(u,v) 
&\leq  D_{h - \phi}\left( P^\phi(s) , P^\phi(t)   \right) +   D_{h - \phi}\left(P^\phi(s) , u   \right) + D_{h - \phi}\left(P^\phi(t) , v \right) \notag \\
&\leq D_{h - \phi}\left( P^\phi(s) , P^\phi(t)   \right) + 2 c_*^{-1} C_*  \eta D_{h - \phi}(u,v)
\ale
which re-arranges to give
\eqb \label{eqn-internal-times-D}
D_{h - \phi}(u,v)  \leq \left(1 - 2 c_*^{-1} C_* \eta \right)^{-1}  D_{h - \phi}\left( P^\phi(s) , P^\phi(t)   \right) .
\eqe

Recall that $\phi \leq K_f$ on $\BB C\setminus \mcl W$ and by the last condition in~\eqref{eqn-highprob-geo} we have $D_h(u,v) \leq D_h(u,\mcl W)$. 
It follows from this that each $D_{h-\phi}$-geodesic from $u$ to $v$ is disjoint from $\mcl W$ and $D_{h- \phi}(u,v) \geq e^{-\xi K_f} D_h(u,v)$. 
By combining this with~\eqref{eqn-internal-tildeD} and condition~\ref{item-highprob-geo} in the definition of $E_r $, we get
\eqb \label{eqn-scaled-compare}
 \wt D_{h - \phi}(u,v) \leq c_1'   D_{h - \phi}(u,v) . 
\eqe
By the triangle inequality,
\allb \label{eqn-scaled-end}
\wt D_{h - \phi}\left(P^\phi(s) , P^\phi(t) ; U \right)
&\leq   \wt D_{h- \phi}(u, v ; U) +  \wt D_{h - \phi}\left(P^\phi(s) , u ; U \right) + \wt D_{h - \phi}\left(v , P^\phi(t) ; U \right) \notag \\
&\leq  (1+2\eta) \wt D_{h- \phi}(u, v ) \quad \text{(by \eqref{eqn-internal-tildeD} and \eqref{eqn-internal-times})} \notag \\
&\leq  c_1' (1+2\eta)   D_{h- \phi}(u, v ) \quad \text{(by \eqref{eqn-scaled-compare})}  \notag \\ 
&\leq  \frac{ c_1' (1+2\eta) }{ 1 - 2 c_*^{-1} C_* \eta }  D_{h - \phi}\left( P^\phi(s) , P^\phi(t)   \right)  \quad \text{(by \eqref{eqn-internal-times-D})} \notag \\
&\leq c_2' D_{h-\phi}(P^\phi(s)  , P^\phi(t)) \quad \text{(by the definition~\eqref{eqn-const-error} of $\eta$)} .
\alle
\end{proof}

\section{Proof of Theorem~\ref{thm-weak-uniqueness}}
\label{sec-conclusion}

Assume we are in the setting of Theorem~\ref{thm-weak-uniqueness} and let $h$ be a whole-plane GFF. Also recall the definitions of the optimal bi-Lipschitz constants $c_*$ and $C_*$ from~\eqref{eqn-max-min-def} and the events $\ol G_r(C' , \beta)$ and $\ul G_r(c',\beta)$ from~\eqref{eqn-max-event} and~\eqref{eqn-min-event}.  
We want to show that $c_* = C_*$. 
To do this we will assume that $c_*  <C_*$ and derive a contradiction. 
The following proposition will be used in conjunction with Proposition~\ref{prop-attained-min} to tell us that there are many scales for which the following is true: the pairs $(u,v)$ such that $\wt D_h(u,v) / D_h(u,v)$ is close to $C_*$ are very sparse.

\begin{prop} \label{prop-away-from-max}
Assume that $c_* < C_*$. 
Then there exists $c''  > c_*$, depending only on the values of $c_*$ and $C_*$, such that the following is true.  
If $\beta\in (0,1)$ and $\BB r > 0$ are such that $\BB P[\ul G_{\BB r}(c'',\beta)] \geq \beta$, then for every choice of $\ol\beta \in (0,1)$, one has
\eqb \label{eqn-away-from-max}
\lim_{\delta\rta 0} \BB P\left[ \ol G_{\BB r}(C_* - \delta , \ol\beta )  \right]  = 0
\eqe
at a rate depending only on $\beta,\ol\beta$ (not on $\BB r$). 
\end{prop}
\begin{proof}
Assume $c_* < C_*$. 
Let $\nu_*$ be as in Theorem~\ref{thm-geo-iterate} and
fix parameters $0 < \mu < \nu \leq \nu_*$ and $c_*  < c_1' < c_2' < C_*$ chosen in a manner depending only on $c_*$ and $C_*$.
The proof follows the strategy outlined in the ``main idea" part of the outline in Section~\ref{sec-outline}.
Theorem~\ref{thm-geo-iterate} and Proposition~\ref{prop-geo-event} will allow us to show that if $q > 0$ is fixed, then with probability tending to 1 as $\ep\rta 0$, the following is true. For every pair of points $\BB z,\BB w \in (\ep^q \BB r \BB Z^2) \cap B_{\BB r}(0)  $ with $|\BB z - \BB w | \geq \ol\beta \BB r$, the $D_h$-geodesic $P$ from $\BB z$ to $\BB w$ has to hit a pair of points $P(s) , P(t)$ such that $|P(s) - P(t)| \geq \op{const} \times \ep^{1+\nu} \BB r$ and $\wt D_h(P(s) , P(t)) \leq c_2' D_h(P(s) , P(t))$. This allows us to show that $\wt D_h(\BB z,\BB w) / D_h(\BB z,\BB w)$ is bounded above by $C_*$ minus a $\gamma$-dependent power of $\ep$ for all such pairs of points $\BB z,\BB w$. We can then use H\"older continuity to get the same statement for all pairs of points $\BB z,\BB w \in B_{\BB r}(0)$ with $|\BB z - \BB w|\geq \ol\beta \BB r$ simultaneously.
Choosing $\ep$ to be an appropriate $\gamma$-dependent power of $\delta$ then gives~\eqref{eqn-away-from-max}.  
\medskip

\noindent\textit{Step 1: setup and regularity events.}
Let $c'' = c''(c_1', \mu,\nu)$, $b = b( \mu,\nu) \in (0,1)$, and $\rho = \rho(\mu,\nu) \in(0,1)$ be as in Proposition~\ref{prop-geo-event} with the above choice of $\mu,\nu,c_1',c_2'$.
Also fix $q   > 0$ to be chosen later in a manner depending on $ \beta,\ol\beta$. 

By Theorem~\ref{thm-geo-iterate} applied to the objects of Proposition~\ref{prop-geo-event} and with the above choice of $q$, $\rho^{-1} \BB r$ in place of $\BB r$, $U = B_2(0)$, and $\ell = \rho\ol\beta$, we get the following.
If $\BB r > 0$ is such that $\BB P[\ul G_{\BB r}(c'',\beta)] \geq \beta$, then it holds with probability tending to 1 as $\ep\rta 0$, at a rate depending only on $ q,\ol\beta,\beta,c_1',c_2',\mu,\nu $, that the following is true.
Let $\BB z , \BB w \in \left( \ep^q \rho^{-1} \BB r  \BB Z^2 \right) \cap B_{2\BB r}\left( 0 \right)$ with $|\BB z-\BB w|\geq \ol\beta \BB r$ and let $P  = P^{\BB z,\BB w}$ be the $D_h$-geodesic from $\BB z$ to $\BB w$. 
Then there exists times $0 < s < t < |P |$ such that 
\eqb \label{eqn-away-from-max-times}
|P (s) - P (t) | \geq b \ep^{1+\nu} \rho^{-1} \BB r  
\quad \text{and} \quad
 \wt D_h\left(P (s) , P (t)\right)  \leq c_2' D_h\left( P(s) , P (t)\right)  
\eqe
(in particular, the times $s,t$ arise from a radius $r\in [\ep^{1+\nu} \rho^{-1} \BB r, \ep \rho^{-1} \BB r]$ and a point $z\in\BB C$ for which $\frk E_r^{\BB z,\BB w}(z)$ occurs). 
Henceforth assume that~\eqref{eqn-away-from-max-times} holds for every $\BB z , \BB w \in \left( \ep^q \rho^{-1} \BB r \BB Z^2 \right) \cap B_{2\BB r}\left( 0 \right)$ with $|\BB z-\BB w|\geq \ol\beta \BB r$.

Fix $\chi \in (0,\xi(Q-2))$ and $\chi' > \xi (Q+2)$, as in Lemma~\ref{lem-holder-uniform}. 
By Axiom~\ref{item-metric-coord} (tightness across scales), for each $p\in (0,1)$ we can find a bounded open set $U \subset\BB C$ which contains $B_2(0)$ such that $\BB P[\sup_{u,v\in B_{2\BB r}(0) } D_h(u,v) < D_h( B_{2\BB r}( 0 ) , \BB r \bdy U )] \geq p$ for every $\BB r  >0$. 
On the event of the preceding sentence, every $D_h$-geodesic between two points of $B_{2\BB r}(0)$ is contained in $\BB r U $. 
By applying Lemma~\ref{lem-holder-uniform} with $K= \ol U $ and then sending $p\rta 1$, we get that with probability tending to 1 as $\ep\rta 0$, at a rate which is uniform in $\BB r$, for any two points $z,w\in \BB C$ with $|z-w| \leq (\ep^q \vee (b \ep^{1+\nu}))\rho^{-1} \BB r$ which are either contained in $B_{2\BB r}( 0 )$ or which lie on a $D_h$-geodesic between two points of $B_{2\BB r}(0)$, 
\eqb \label{eqn-away-holder-cont}
 \left| \frac{z-w}{\BB r} \right|^{\chi'} \leq \frk c_{\BB r}^{-1} e^{-\xi h_{\BB r}(0)} D_h(z,w) \leq  \left| \frac{z-w}{\BB r} \right|^\chi  .
\eqe  
Henceforth assume that this is the case. 
\medskip

\noindent\textit{Step 2: bounding $\wt D_h(\BB z,\BB w)/D_h(\BB z,\BB w)$ for points in a fine mesh.}
By~\eqref{eqn-away-from-max-times} and~\eqref{eqn-away-holder-cont}, the times $s$ and $t$ from~\eqref{eqn-away-from-max-times} satisfy 
\eqb \label{eqn-away-times-lower}
t-s = D_h\left(P (s) , P (t)\right)   \geq  (b/\rho)^{\chi'} \frk c_{\BB r} e^{\xi h_{\BB r}(0)} \ep^{(1+\nu)\chi'} .
\eqe 
By the definition~\eqref{eqn-max-min-def} of $C_*$, the $\wt D_h$-lengths of the segments $P|_{[0,s]}$ and $P|_{[t,|P|]}$ are bounded above by $C_* s$ and $C_* (|P|-t)$, respectively. 
Therefore, for each $\BB z , \BB w \in \left( \ep^q \rho^{-1} \BB r \BB Z^2 \right) \cap B_{2\BB r}(0)$ with $|\BB z-\BB w|\geq \ol\beta \BB r$, 
\allb \label{eqn-away-sep}
\wt D_h(\BB z,\BB w) 
&\leq C_* \left( |P| - t + s \right) +  \wt D_h\left(P (s) , P (t)\right) \notag \\
&\leq C_* \left( |P| - t + s \right) +  c_2' (t-s) \quad \text{(by~\eqref{eqn-away-from-max-times})} \notag \\
&\leq C_* D_h(\BB z , \BB w)   - (C_* - c_2') (t-s)  \notag \\
&\leq C_* D_h(\BB z , \BB w)   - (C_* - c_2')(b/\rho)^{\chi'} \frk c_{\BB r} e^{\xi h_{\BB r}(0)} \ep^{(1+\nu)\chi'}   \quad \text{(by~\eqref{eqn-away-times-lower})} .
\alle
\medskip

\noindent\textit{Step 3: transferring from points in a fine mesh to general points.}
If $z,w\in B_{\BB r}(0)$ with $|z-w| \geq \ol\beta\BB r$, then we can find $\BB z , \BB w \in \left( \ep^q \BB r \BB Z^2 \right) \cap B_{2\BB r}\left(0\right)$ such that $|\BB z-\BB w|\geq \ol\beta \BB r$ and $\max\{|z-\BB z| , |w-\BB w|\} \leq 2\ep^q \rho^{-1} \BB r$.
By~\eqref{eqn-away-holder-cont} and the triangle inequality,
\eqb \label{eqn-away-close}
|D_h(\BB z , \BB w) - D_h(z,w)| \leq 2^{2+\chi} \rho^{-\chi}   \frk c_{\BB r} e^{\xi h_{\BB r}(0)} \ep^{q \chi}  , 
\eqe
and the same is true with $\wt D_h$ in place of $D_h$. If we choose $q > \chi'(1+\nu)/\chi$, then~\eqref{eqn-away-close} and~\eqref{eqn-away-sep} together imply that for each $z,w\in B_{\BB r}(0)$ with $|z-w| \geq \ol\beta \BB r$ and each small enough $\ep$, 
\eqb \label{eqn-away-uniform}
\wt D_h(z,w) \leq C_* D_h(z,w)   - a \frk c_{\BB r} e^{\xi h_{\BB r}(0)} \ep^{(1+\nu)\chi'} ,\quad\forall z,w\in B_{\BB r}(0) \quad \text{s.t.} \quad  |z-w| \geq \ol\beta \BB r .
\eqe
where $a > 0$ is a constant depending only on $q,\ol\beta ,c_1',c_2',\mu,\nu$.
\medskip

\noindent\textit{Step 4: choosing $\ep$.}
By Axiom~\ref{item-metric-coord} (tightness across scales), it holds with probability tending to 1 as $\ep\rta 0$, uniformly over all $\BB r > 0$, that $ D_h(z,w) \leq  \frk c_{\BB r} e^{\xi h_{\BB r}(0)} \ep^{-\chi'}$ for each $z,w\in B_{\BB r}(0) $. If this is the case then $a \frk c_{\BB r} e^{\xi h_{\BB r}(0)} \ep^{(1+\nu)\chi'} \geq a \ep^{(2+\nu)\chi'} D_h(z,w)$. Hence~\eqref{eqn-away-uniform} implies that with probability tending to 1 as $\ep\rta 0$, at a rate depending only on $q,\ol\beta, c',\mu,\nu$, 
\eqb \label{eqn-away-end}
\wt D_h(z,w) \leq \left( C_*  - a \ep^{(2+\nu) \chi'} \right)    D_h(z,w)  ,\quad\forall z,w\in B_{\BB r}(0) \quad \text{s.t.}\quad  |z-w| \geq \ol\beta \BB r.
\eqe
Recalling the definition~\eqref{eqn-max-event} of $\ol G_{\BB r}(C_*-\delta,\ol\beta)$, we can choose $\ep$ so that $a\ep^{(2+\nu) \chi'} =\delta$ to get the proposition statement. 
\end{proof}

\begin{proof}[Proof of Theorem~\ref{thm-weak-uniqueness}]
Let $D$ and $\wt D$ be as in Theorem~\ref{thm-weak-uniqueness}, let $h$ be a whole-plane GFF, and define the maximal and minimal ratios $c_*$ and $C_*$ as in~\eqref{eqn-max-min-def}. 
We claim that $c_* = C_*$, i.e., a.s.\ $\wt D_h = c_* D_h$. 
This gives the theorem statement in the case of a whole-plane GFF, which in turn implies the theorem statement for a whole-plane GFF plus a continuous function due to Axiom~\ref{item-metric-f} (Weyl scaling).

It remains to prove that $c_* = C_*$. 
By Proposition~\ref{prop-attained-max} applied with $C'  = C_* -\delta$, there exists $\ol\beta = \ol\beta(\mu,\nu)  \in (0,1)$ and $\ol p = \ol p(\mu,\nu) \in (0,1)$
with the following property. For each $\delta \in (0,1)$, there exists $\ep_0 = \ep_0(\delta,\mu,\nu)  > 0$ such that for each $\ep \in (0,\ep_0]$, there are at least $\mu\log_8 \ep^{-1}$ values of $\BB r \in [\ep^{1+\nu} ,\ep   ] \cap \{8^{-k}   : k\in\BB N\}$ for which 
\eqb \label{eqn-close-to-max}
\BB P\left[\ol G_{\BB r}(C_*-\delta , \ol\beta )\right] \geq  \ol p .
\eqe
We emphasize that $\ol\beta$ and $\ol p$ do not depend on $\delta$. 

We now assume by way of contradiction that $c_* < C_*$ and show that this assumption is incompatible with the conclusion of the preceding paragraph.
To this end, let $c'' \in (c_* , C_*)$ be as in Proposition~\ref{prop-away-from-max}, so that $c''$ depends only on the choice of metrics $D$ and $\wt D$.
Proposition~\ref{prop-attained-min} applied with $c''$ in place of $c'$ shows that there exists $\ul\beta  = \ul\beta(\mu,\nu) \in (0,1)$, $\ul p = \ul p(\mu,\nu) \in (0,1)$, and $\ep_1 = \ep_1(\mu,\nu) > 0$ such that for each $\ep \in (0,\ep_1]$, there are at least $\mu\log_8 \ep^{-1}$ values of $\BB r \in [\ep^{1+\nu} ,\ep   ] \cap \{8^{-k} : k\in\BB N\}$ for which $\BB P[\ul G_{\BB r}(c'',\ul\beta)] \geq \ul p$.

Proposition~\ref{prop-away-from-max} applied with $\beta  =\ul\beta \wedge \ul p$ therefore implies that there exists $\delta = \delta(\mu,\nu) \in (0,1)$ 
such that for each $\ep \in (0,\ep_1]$, there are at least $\mu\log_8 \ep^{-1}$ values of $\BB r \in [\ep^{1+\nu} ,\ep   ] \cap \{8^{-k} : k\in\BB N\}$ for which $\BB P\left[ \ol G_{\BB r}(C_* - \delta , \ol\beta )  \right] \leq \ol p/2$. 
If we take $\mu  > \nu/2$, then this is incompatible with~\eqref{eqn-close-to-max} whenever $\ep \in (0, \ep_0 \wedge \ep_1)$, so we have obtained the desired contradiction.
\end{proof}

\section{Open problems}
\label{sec-open-problems}

\subsubsection*{Dimension calculations}

An important remaining question concerning the LQG metric is the following. 

\begin{prob}[Hausdorff dimension of $\gamma$-LQG] \label{prob-dimension}
Compute the exponent $d_\gamma$ appearing in~\eqref{eqn-xi}, which is the Hausdorff dimension of $\BB C$ with respect to the $\gamma$-LQG metric (this is proven in~\cite{gp-kpz}). 
\end{prob}

Since $\xi = \gamma/d_\gamma$ and $Q=2/\gamma+\gamma/2$, Problem~\ref{prob-dimension} is equivalent to determining the relationship between these two parameters.
The only case in which $d_\gamma$ is known is when $\gamma=\sqrt{8/3}$, in which case $d_{\sqrt{8/3}}=4$.
Due to existing results in the literature, $d_\gamma$ can equivalently be defined in a large number of other equivalent ways, e.g., the following.
\begin{enumerate}
\item For a large class of infinite-volume random planar maps in the $\gamma$-LQG universality class, the number of vertices in the graph distance ball of radius $r$ centered at the root vertex is of order $r^{d_\gamma+o_r(1)}$~\cite[Theorem 1.6]{dg-lqg-dim} and the graph distance traveled by a simple random walk started from the root vertex and run for $n$ steps is of order $n^{1/d_\gamma + o_n(1)}$~\cite{gm-spec-dim,gh-displacement}.
\item For fixed distinct points $z,w\in\BB C$, the Liouville heat kernel (as constructed in~\cite{grv-heat-kernel}) satisfies $\mathsf p_t^\gamma(z,w) = \exp\left( - t^{-\frac{1}{d_\gamma-1} + o_t(1)} \right)$ as $t\rta 0$~\cite[Theorem 1.1]{dzz-heat-kernel}.
\item The optimal H\"older exponent for the $\gamma$-LQG metric w.r.t.\ the Euclidean metric is $\frac{\gamma}{d_\gamma}(Q-2)$ and the optimal H\"older exponent for the Euclidean metric w.r.t.\ the $\gamma$-LQG metric is $\frac{d_\gamma}{\gamma}(Q+2)^{-1}$~\cite[Theorem 1.7]{lqg-metric-estimates}.
\end{enumerate}

The best-known physics prediction for the value of $d_\gamma$ is the Watabiki prediction~\cite{watabiki-lqg},
\eqb \label{eqn-watabiki}
d_\gamma = 1 + \frac{\gamma^2}{4} + \frac14 \sqrt{(4+\gamma^2)^2 + 16\gamma^2} .
\eqe
However, this prediction is known to be false at least for small values of $\gamma$ due to the results of Ding-Goswami~\cite{ding-goswami-watabiki}. 
See~\cite{dg-lqg-dim,gp-lfpp-bounds} for rigorous upper and lower bounds for $d_\gamma$ as well as additional discussion about various possibilities for its value. 
In addition to $d_\gamma$, there are a number of other interesting dimensions related to $\gamma$-LQG which have not yet been computed, for example the following.

\begin{prob}[Geodesic dimension] \label{prob-geo-dimension}
Compute the Euclidean Hausdorff dimension of the $\gamma$-LQG geodesic between two typical points of $\BB C$. 
\end{prob}

\begin{prob}[Ball boundary dimension] \label{prob-bdy-dimension}
Compute the $\gamma$-LQG Hausdorff dimension and the Euclidean Hausdorff dimension of the boundary of a filled $\gamma$-LQG metric ball $\mcl B_s^\bullet(0;D_h)$. 
\end{prob}

In the setting of Problem~\ref{prob-geo-dimension}, the $\gamma$-LQG Hausdorff dimension of a $\gamma$-LQG geodesic is trivially equal to 1. 
The Euclidean dimensions of $\gamma$-LQG geodesics and filled metric ball boundaries are unknown even for $\gamma=\sqrt{8/3}$ and there are not even any conjectures as to their values. The $\sqrt{8/3}$-LQG dimension of the outer boundary of a filled $\sqrt{8/3}$-LQG metric ball is 2~\cite{lqg-tbm2}, but this quantity is not known (even heuristically) for any other value of $\gamma$. 
See~\cite{gp-kpz} for upper bounds for the Euclidean Hausdorff dimension of a $\gamma$-LQG geodesic and for the outer boundary of a filled $\sqrt{8/3}$-LQG metric ball. 

Currently, no explicit lower bounds for any of these quantities are known, although we expect it is not hard to show that they are strictly larger than 1; c.f.~\cite{ding-zhang-geodesic-dim}. 

Another natural random fractal associated with the LQG metric is the boundary of a (non-filled) LQG metric ball (note that this boundary is typically not connected). It is shown in~\cite{gwynne-ball-bdy,lqg-zero-one} that a.s.\ the Hausdorff dimension of the LQG metric ball boundary w.r.t.\ the Euclidean (resp.\ LQG) metric is $2-\xi Q + \xi^2/2$ (resp.\ $d_\gamma-1$). It is also shown in~\cite{lqg-zero-one} that a.s.\ the Hausdorff dimension of the boundary of a filled metric ball w.r.t.\ the Euclidean (resp.\ LQG) metric is strictly smaller than this quantity.

The ``quantum dimension" part of Problem~\ref{prob-bdy-dimension} is closely related to the following question. 

\begin{prob}[$\gamma$-LQG boundary length of metric balls] \label{prob-bdy-measure}
Is there a natural LQG length measure on the boundary of a filled $\gamma$-LQG metric ball? 
\end{prob}
 
In the case when $\gamma=\sqrt{8/3}$, for $s > 0$ the field $h|_{\BB C\setminus\mcl B_s(0;D_h)}$ locally looks like a free-boundary GFF near $\bdy \mcl B_s(0;D_h)$. This allows one to define the $\gamma$-LQG boundary length measure on $\bdy\mcl B_s(0;D_h)$ in the manner of~\cite[Section 6]{shef-kpz}. Alternatively, the length measure on $\bdy\mcl B_s(0;D_h)$ can equivalently be constructed using Brownian surface theory; see~\cite{tbm-characterization,legall-disk-snake}. 
For general $\gamma \in (0,2)$, it is not expected that $h|_{\BB C\setminus\mcl B_s(0;D_h)}$ locally looks like a free-boundary GFF near $\bdy \mcl B_s(0;D_h)$. 
Indeed, if this were the case then the heuristic argument in~\cite[Section 3.3]{qle} would imply that the dimension of $\gamma$-LQG is given by Watabiki's prediction~\eqref{eqn-watabiki}, which we know is false, at least for small $\gamma$, by the results of~\cite{ding-goswami-watabiki}.  
Hence new ideas are required to construct a natural length measure on $\bdy\mcl B_s(0;D_h)$ in this case.

\subsubsection*{Discrete approximations}

Another interesting open problem is to connect the $\gamma$-LQG metric to its discrete counterparts. 

\begin{prob}[Scaling limit of random planar maps] \label{prob-discrete}
Prove Conjecture~\ref{conj-rpm-limit}, which asserts that random planar maps, equipped with their graph distance, converge to the $\gamma$-LQG surface, equipped with the $\gamma$-LQG metric, w.r.t.\ the Gromov-Hausdorff topology.
\end{prob}

One possible approach to Problem~\ref{prob-discrete} is to first prove a scaling limit result for the so-called \emph{mated-CRT maps}, as studied, e.g., in~\cite{gms-tutte,ghs-map-dist,ghs-dist-exponent} using their direct connection to Liouville quantum gravity. One could then try to transfer to other random planar map models by improving on the strong coupling techniques used in~\cite{ghs-map-dist}, which currently only give estimates for distances up to polylogarithmic multiplicative errors. We emphasize, however, that both of these steps are highly non-trivial and are likely to require substantial new ideas.
Another possible approach would be to find some sort of ``combinatorial miracle" which allows one to analyze distances in weighted random planar maps directly (analogous to the Schaeffer bijection~\cite{cv-bijection,schaeffer-bijection,bdg-bijection} for uniform random planar maps). 

A likely easier scaling limit problem is to show universality of the $\gamma$-LQG metric across different approximation schemes. 
One of the most natural approximation schemes is \emph{Liouville graph distance (LGD)}, whereby the distance between two points $z,w\in\BB C$ is defined to be the minimal number of Euclidean balls of $\gamma$-LQG mass $\ep$ whose union contains a path from $z$ to $w$. 

\begin{prob}[Other approximation schemes] \label{prob-lgd}
Show that the LGD metrics, appropriately re-scaled, converge in law to the $\gamma$-LQG metric as $\ep\rta 0$. 
\end{prob}

We expect that the difficulties involved in solving Problem~\ref{prob-lgd} are similar to the difficulties involved in showing that the mated-CRT map converges to $\gamma$-LQG in the metric sense, due to the SLE/LQG representation of the mated-CRT map (see~\cite{gms-tutte,ghs-map-dist}). 

It is shown in~\cite{ding-dunlap-lgd} that LGD, re-scaled by the median distance across a square, is tight and each subsequential limit induces the Euclidean topology. We expect that it is not hard to check that these subsequential limits satisfy Axioms~\ref{item-metric-length0}, \ref{item-metric-local0}, and~\ref{item-metric-coord0} in the definition of the $\gamma$-LQG metric (the latter is just a consequence of the coordinate change formula for the LQG area measure~\cite[Proposition 2.1]{shef-kpz}). One can also obtain a much weaker version of Weyl scaling analogous to the ``tightness across scales" condition (Axiom~\ref{item-metric-coord}) used in our definition of a weak $\gamma$-LQG metric, where one requires that the metrics obtained by adding different constants to the field, then re-scaling appropriately, are tight. 

Hence one possible approach to Problem~\ref{prob-lgd} is to adapt the arguments of this paper and its predecessors to the case when we know that our metric satisfies the coordinate change formula for translations and scalings, but we do not know that it satisfies Weyl scaling. 
However, our arguments are in some ways optimized to work for subsequential limits of LFPP, so there may also be an entirely different argument which is more appropriate for subsequential limits of LGD. 

Theorem~\ref{thm-lfpp} says that the LFPP metrics converge in probability, unlike the case of various approximations of the LQG measure which are known to converge a.s.~\cite{shef-kpz,rhodes-vargas-review,shef-wang-lqg-coord}. 

\begin{prob}[Almost sure convergence of LFPP] \label{prob-lfpp-as}
Can the convergence $\frk a_\ep^{-1} D_h^\ep \rta D_h$ in Theorem~\ref{thm-lfpp} be improved from convergence in probability to a.s.\ convergence?
\end{prob}

\subsubsection*{Metric space structure vs.\ quantum surface structure}

In~\cite{lqg-tbm3}, it is shown that a $\sqrt{8/3}$-LQG surface is a.s.\ determined by its structure as a metric measure space, i.e., the metric measure space $(\BB C , \mu_h , D_h)$ a.s.\ determines its embedding into $\BB C$ and the associated GFF $h$ (modulo conformal automorphisms). Our next problem asks for an extension of this result to the case when $\gamma \in (0,2)$. 

\begin{prob}[Metric measure space structure determines the field] \label{prob-mm-structure}
Show that the field $h$ is a.s.\ determined (modulo rotation and scaling) by the pointed $\gamma$-LQG metric measure space $(\BB C , 0 , \mu_h, D_h)$. 
\end{prob}

Likely the easiest approach to Problem~\ref{prob-mm-structure} is to adapt the arguments of~\cite{gms-poisson-voronoi}, which gives for $\gamma=\sqrt{8/3}$ an explicit way of re-constructing $h$ from $(\BB C ,0, \mu_h , D_h)$ using the adjacency graph of a fine mesh of Poisson-Voronoi cells. 
The arguments of~\cite{gms-poisson-voronoi} are not very specific to the case when $\gamma=\sqrt{8/3}$. The main missing ingredient to extend these arguments to general values of $\gamma$ is the following estimate of independent interest. 

\begin{prob}[Concentration of areas of LQG metric balls] \label{prob-area-conc}
Show that the $\gamma$-LQG area of a $\gamma$-LQG metric ball has superpolynomial concentration, i.e., show that for $C > 1$, 
\eqb \label{eqn-area-conc}
\BB P\left[ C^{-1} \leq \mu_h\left(\mcl B_1(0;D_h) \right) \leq C     \right] = 1 - O_C(C^{-p}) ,\quad\forall p > 0 .
\eqe
\end{prob}

Problem~\ref{prob-area-conc} in the case when $\gamma=\sqrt{8/3}$ follows from known estimates for the Brownian map; see~\cite[Corollary 6.2]{legall-geodesics} and~\cite[Section 4.3]{gms-poisson-voronoi}. 
\medskip

\noindent\textit{Update:} Problems~\ref{prob-mm-structure} and~\ref{prob-area-conc} are solved in~\cite{afs-metric-ball}. 
\medskip

It is shown in~\cite{bss-lqg-gff} that the LQG measure a.s.\ determines the GFF. 
It is also natural to try to recover the LQG measure (and thereby the GFF) from the LQG metric. 

\begin{prob} \label{prob-metric-to-measure}
Does the LQG metric a.s.\ determine the LQG measure? 
More concretely, can the LQG measure be recovered as some sort of Minkowski content measure w.r.t.\ the LQG metric?
\end{prob}

In this paper, we gave a characterization of the $\gamma$-LQG metric in terms of its coupling with the GFF. 
In light of Problem~\ref{prob-mm-structure}, it is natural to ask if there is also a characterization solely in terms of the metric space structure, which does \emph{not} require reference to the GFF.
Such a characterization of the Brownian map (equivalently, the $\sqrt{8/3}$-LQG sphere) is proven in~\cite{sphere-constructions}. 
A purely metric characterization of $\gamma$-LQG could potentially play an important role in a solution to Problem~\ref{prob-discrete}.

\begin{prob}[Metric space characterization] \label{prob-char}
Is there a characterization of $(\BB C , D_h  )$ as a metric space (or of $(\BB C , \mu_h, D_h)$ as a metric measure space), without reference to the GFF and the embedding of this metric space into $\BB C$?
\end{prob}

It is likely that the most natural setting to consider in Problem~\ref{prob-char} is the one where $h$ the field corresponding to a quantum cone or quantum sphere (as defined in~\cite{wedges}) rather than a whole-plane GFF.

\subsubsection*{Additional properties of the LQG metric}

The construction of the $\sqrt{8/3}$-LQG metric in~\cite{lqg-tbm1,lqg-tbm2,lqg-tbm3} yields many special properties of the metric in this case which are not known (and in many cases not expected to hold) for general $\gamma \in (0,2)$. 
For example, one has $d_{\sqrt{8/3}} =4$. Moreover, in the case when $h$ is the GFF associated with a quantum sphere or $\sqrt{8/3}$-quantum wedge, the quantum surfaces obtained by restricting $h$ to the complementary connected components of a $\sqrt{8/3}$-LQG metric ball are conditionally independent quantum disks given their boundary lengths. 
Many further properties can be obtained using the equivalence of $\sqrt{8/3}$-LQG surfaces and Brownian surfaces.
However, there is nothing obviously special about $\gamma=\sqrt{8/3}$ from either of the definitions of the LQG metric given in this paper (the limit of LFPP or the axiomatic definition).

\begin{prob} \label{prob-c0-special}
Can one prove that $d_{\sqrt{8/3}}=4$, the independence properties for complementary connected components of a $\sqrt{8/3}$-LQG metric ball, or any other special property of the $\sqrt{8/3}$-LQG metric directly from the LFPP definition or the axiomatic definition?
\end{prob}

There has been a recent proliferation of exact formulas for quantities related to the $\gamma$-LQG area and boundary length measures for general $\gamma \in (0,2)$, proven using ideas from conformal field theory: see, e.g.,~\cite{krv-dozz,remy-fb-formula,rz-gmc-interval}.  
In the special case when $\gamma=\sqrt{8/3}$, exact formulas for various quantities associated with the $\sqrt{8/3}$-LQG metric can be obtained using its connection to the Brownian surfaces.
Exact formulas for the $\gamma$-LQG metric, if they can be found, could be very useful in attempts to solve most of the other problems listed above. 
 
\begin{prob}[Exact formulas] \label{prob-formulas}
Are there exact formulas for any objects related to the $\gamma$-LQG metric for general $\gamma\in (0,2)$? 
\end{prob}

\begin{prob}[Topology of geodesics] \label{prob-geodesics}
For a general value of $\gamma \in (0,2)$, what is the maximal possible number of $\gamma$-LQG geodesics joining two points in $\BB C$? Is this number finite, and, if so, does it depend on $\gamma$? 
More generally, can one prove results about the possible topologies of the set of $\gamma$-LQG geodesics joining two points in $\BB C$ analogous to the results for the Brownian map in~\cite{akm-geodesics}?
\end{prob}

\noindent\textit{Update:} This problem is solved for $\gamma=\sqrt{8/3}$ in \cite{mq-strong-confluence} using Brownian map based techniques.  Substantial progress on Problem~\ref{prob-geodesics} is made in~\cite{gwynne-geodesic-network}, where it is shown that the results about geodesic networks from~\cite{akm-geodesics} extend verbatim to the case of general $\gamma\in (0,2)$ and that the maximal number of LQG geodesics joining any two points is a.s.\ finite. It is also conjectured in~\cite{gwynne-geodesic-network} that the maximal number of geodesics is $9$, regardless of the value of $\gamma$.
\medskip

Liouville Brownian motion~\cite{grv-lbm,berestycki-lbm} is the natural ``quantum time" parameterization of Brownian motion on an LQG surface. 
If we condition Liouville Brownian motion to travel a macroscopic distance (e.g., from the origin to the unit circle) in a short amount of time, then it is natural to expect that it would roughly follow a path of minimal LQG length. 

\begin{prob}[Liouville Brownian motion and LQG geodesics] \label{prob-lbm}
Does Liouville Brownian motion conditioned to travel a macroscopic (Euclidean or quantum) distance in a short amount of time approximate an LQG geodesic? 
\end{prob}

There is a one-parameter family of infinite-volume $\gamma$-LQG surfaces with boundary called \emph{quantum wedges}, which can be indexed by the \emph{weight} parameter $\frk w > 0$. See~\cite{wedges} for details. 
In~\cite{wedges}, building on~\cite{shef-zipper}, it is shown that one can conformally weld together a weight-$\frk w_1$ quantum wedge and a weight-$\frk w_2$ quantum wedge according to the quantum length measure along their boundaries to get a weight-$\frk w_1 + \frk w_2$ quantum wedge decorated by an SLE$_\kappa(\frk w_1-2;\frk w_2-2)$ curve which corresponds to the gluing interface. 
In~\cite{gwynne-miller-gluing}, it is shown that in the special case when $\gamma=\sqrt{8/3}$, this conformal welding is compatible with the $\sqrt{8/3}$-LQG metric in the following sense: the weight-$(\frk w_1+\frk w_2)$ quantum wedge, equipped with its $\sqrt{8/3}$-LQG metric, is the metric space quotient of the weight-$\frk w_1$ and weight-$\frk w_2$ quantum wedges, equipped with their $\sqrt{8/3}$-LQG metrics, under the same equivalence relation used to define the conformal welding. 

\begin{prob}[Metric gluing of $\gamma$-LQG surfaces] \label{prob-gluing}
Prove metric gluing statements for quantum wedges analogous to the ones in~\cite{gwynne-miller-gluing} for general $\gamma\in (0,2)$. 
\end{prob}

The main missing ingredient needed to solve Problem~\ref{prob-gluing} is suitable estimates for distances between points of $\bdy \BB D$ with respect to the $\gamma$-LQG metric induced by a free-boundary GFF on $\BB D$ (or a variant thereof, like the quantum disk). For $\gamma=\sqrt{8/3}$, the needed estimates are proven in~\cite[Section 3.2]{gwynne-miller-gluing} using results for the Brownian disk.

\subsubsection*{Extensions of the theory}

Throughout this paper, we have neglected the critical case when $\gamma=2$.

\begin{prob}[Critical LQG metric] \label{prob-critical}
Construct a metric on $\gamma$-LQG when $\gamma=2$. 
\end{prob}

See~\cite{shef-deriv-mart,shef-renormalization} for a construction of the $\gamma$-LQG measure for $\gamma=2$. 
One possible approach to Problem~\ref{prob-critical} is to try to take a limit of the $\gamma$-LQG metrics as $\gamma$ increases to 2 (it is shown that the 2-LQG measure is the $\gamma \nearrow 2$ limit of the $\gamma$-LQG measures, appropriately renormalized, in~\cite{aps-critical-lqg-lim}). 
Another (likely more involved) possibility is to adapt the arguments of this paper and its predecessors~\cite{local-metrics,lqg-metric-estimates,gm-confluence} to the critical case, corresponding to LFPP with parameter $\xi = 2/d_2$.
A major difficulty in the critical case is that the 2-LQG metric is not expected to be H\"older continuous w.r.t.\ the Euclidean metric (indeed, the optimal H\"older exponent from~\cite[Theorem 1.7]{lqg-metric-estimates} converges to zero as $\gamma\rta 2^-$), so more refined estimates for the continuity of the metric and for LFPP are likely to be required. 

Recall that our metric for $\gamma \in (0,2)$ is constructed as the limit of LFPP with parameter $\xi = \gamma/d_\gamma$.
Extending further, it is natural to ask what happens when $\xi > 2/d_2$ (it is shown in~\cite[Proposition 1.7]{dg-lqg-dim} that $\gamma\mapsto \gamma/d_\gamma$ is increasing, so $\gamma/d_\gamma  < 2/d_2$). Very recently, it was shown in~\cite{dg-supercritical-lfpp} that LFPP is tight w.r.t.\ the topology on lower semicontinuous functions for all $\xi > 0$. For $\xi > 2/d_2$ every possible subsequential limit is a metric on $\BB C$ which does \emph{not} induce the Euclidean topology. Rather, there is an uncountable, dense, fractal set of ``singular points" whose distance to every other point is infinite. These singular points arise from the thick points of the GFF~\cite{hmp-thick-pts}.

\begin{prob}[LFPP with $\xi  >2/d_2$] \label{prob-central-charge}
Show that LFPP with parameter $\xi  >2/d_2$ converges in law to a limiting metric w.r.t.\ the topology of~\cite{dg-supercritical-lfpp}. 
\end{prob}
 
This metric of Problem~\ref{prob-central-charge} should be related to Liouville quantum gravity with central charge $\mathbf c \in (1,25)$. Note that the central charge associated with $\gamma$-LQG for $\gamma \in (0,2]$ is $\mathbf c = 25 -6(2/\gamma+\gamma/2)^2 \in (-\infty,1]$. 
We refer to~\cite{ghpr-central-charge,dg-supercritical-lfpp} and the references therein for more on LQG with $\mathbf c \in (1,25)$.

The $\gamma$-LQG measure is a special case of a more general theory of random measures called \emph{Gaussian multiplicative chaos} (GMC)~\cite{kahane,rhodes-vargas-review}, which studies limits of regularized versions of ``$e^{\gamma X} \,dz$" for certain Gaussian random distributions $X$. 
Here, $X$ is a random distribution on $\BB R^n$ for some $n\in\BB N$ and $dz$ denotes Lebesgue measure on $\BB R^n$. 

\begin{prob}[More general random metrics] \label{prob-metric-gmc}
Is there a more general theory of random metrics associated with log-correlated random Gaussian distributions analogous to GMC? In particular, can one construct metrics with similar properties to the $\gamma$-LQG metric in higher dimensions? 
\end{prob}

Some of the arguments in the construction of the LQG metric, in this paper as well as~\cite{dddf-lfpp,local-metrics,lqg-metric-estimates,gm-confluence} are specific to the two-dimensional case. The following seem to be the places where the use of two-dimensionality is the most fundamental. 
\begin{itemize}
\item The construction of the LQG metric makes extensive use of the Markov property of the GFF: for an open set $U\subset\BB C$, $h|_U$ decomposes as a zero-boundary GFF in $U$ plus an independent random harmonic function on $U$. This property is not satisfied for log-correlated fields in dimension $\geq 3$, see, e.g.,~\cite{lgf-survey} (note that the GFF is only log-correlated in dimension 2). 
\item The proof of tightness in~\cite{dddf-lfpp}, as well as several proofs in~\cite{lqg-metric-estimates}, use RSW-type arguments which are based on the fact that one can force two paths to intersect each other in dimension 2. 
\item The proof of confluence in~\cite{gm-confluence} is based on a decomposition of the boundary of a filled LQG metric ball into arcs of topological dimension 1, together with an iterative argument where one ``kills off" all but one of the arcs by preventing LQG geodesics from passing through them. In higher dimensions, the boundary of an LQG metric ball cannot be decomposed into sets of dimension 1. In fact, it is plausible that confluence fails in higher dimensions since there is more ``room" for geodesics to move around. 
\end{itemize}

\bibliography{cibiblong,cibib}

\bibliographystyle{hmralphaabbrv}

\end{document}